\documentclass[a4paper,10pt, oneside]{amsart}
\usepackage[utf8]{inputenc}
\usepackage{url}
\usepackage{xcolor}
\definecolor{lightgreen}{rgb}{.20,.60,.22}
\usepackage{pdfpages}
\usepackage{geometry}
\usepackage{microtype}
\usepackage{amssymb,amsmath,amsopn,amsxtra,amsthm,amsfonts}
\usepackage{mathtools}
\usepackage[mathcal]{euscript}
\let\mathscr\mathcal
\usepackage[bb=px]{mathalfa}

\usepackage[backend=biber,style=alphabetic,maxnames=99]{biblatex}
\usepackage{hyperref}
\hypersetup{
	colorlinks=true,
	linkcolor=purple,
	citecolor=lightgreen,
	urlcolor=cyan,
}

\usepackage[capitalise,nameinlink]{cleveref}

\usepackage[T1]{fontenc}

\DeclareFontFamily{U}{min}{}
\DeclareFontShape{U}{min}{m}{n}{<-> udmj30}{}

\usepackage{enumitem}
\setlist[enumerate,1]{label={(\arabic*)},itemsep=\parskip} 
\setlist[itemize,1]{itemsep=\parskip} 
\newlist{thmlist}{enumerate}{2}
\setlist[thmlist,1]{label={\em(\roman*)},ref={(\roman*)},%
	itemsep=\parskip,leftmargin=*,align=left}
\setlist[thmlist,2]{label={\em(\alph*)},ref={(\alph*)},%
	itemsep=\parskip,leftmargin=*,align=left,topsep=0.1cm}
\newlist{defnlist}{enumerate}{2}
\setlist[defnlist,1]{label={(\roman*)},ref={(\roman*)},itemsep=\parskip,%
	leftmargin=*,align=left}
\setlist[defnlist,2]{label={(\alph*)},ref={(\alph*)},itemsep=\parskip,%
	leftmargin=*,align=left,topsep=0.1cm}

\newtheorem*{thm*}{Theorem}
\newtheorem{thmX}{Theorem}
 
\newtheorem{corX}[thmX]{Corollary}

\newtheorem{thm}{Theorem}[section] 
\newtheorem{cor}[thm]{Corollary}
\newtheorem{lem}[thm]{Lemma}
\newtheorem{prop}[thm]{Proposition}

\newtheorem{quest}[thm]{Question}

\theoremstyle{definition}

\newtheorem{de}[thm]{Definition}

\newtheorem{rem}[thm]{Remark}
\newtheorem{warning}[thm]{Warning}

\newtheorem{exam}[thm]{Example}

\newtheorem{notation}[thm]{Notation}
\newtheorem{nota}[thm]{Notation}

\crefname{thm}{Theorem}{Theorems}
\crefname{thmX}{Theorem}{Theorems}
\crefname{cor}{Corollary}{Corollaries}
\crefname{corX}{Corollary}{Corollaries}
\crefname{lem}{Lemma}{Lemmas}
\crefname{prop}{Proposition}{Propositions}
\crefname{thmconstr}{Theorem-Construction}{Theorem-Constructions}
\crefname{propconstr}{Proposition-Construction}{Proposition-Constructions}
\crefname{lemconstr}{Lemma-Construction}{Lemma-Constructions}
\crefname{ax}{Axiom}{Axioms}
\crefname{conj}{Conjecture}{Conjectures}
\crefname{qthm}{Quasi-Theorem}{Quasi-Theorems}
\crefname{qlem}{Quasi-Lemma}{Quasi-Lemmas}
\crefname{quest}{Question}{Questions}
\crefname{claim}{Claim}{Claims}
\crefname{defn}{Definition}{Definitions}
\crefname{de}{Definition}{Definitions}
\crefname{obs}{Observation}{Observations}
\crefname{rem}{Remark}{Remarks}
\crefname{rems}{Remarks}{Remarks}
\crefname{warning}{Warning}{Warnings}
\crefname{exam}{Example}{Examples}
\crefname{exams}{Example}{Examples}
\crefname{constr}{Construction}{Constructions}
\crefname{notation}{Notation}{Notations}
\crefname{nota}{Notation}{Notations}
\crefname{exer}{Exercise}{Exercises}

\renewcommand{\eqref}[1]{(\ref{#1})}

\numberwithin{equation}{section}

\usepackage{tikz}
\usetikzlibrary{matrix}
\usepackage{tikz-cd}

\makeatletter
\renewcommand{\tocsubsection}[3]{%
	\indentlabel{\@ifnotempty{#2}{\ignorespaces#1 #2\hspace{1.5em}}}#3}
\usepackage{xspace}
\usepackage{xifthen}
\usepackage{xparse}

\usepackage{mathtools}

\newcommand{\nc}{\newcommand}
\nc{\renc}{\renewcommand}
\nc{\ssec}{\subsection}
\nc{\sssec}{\subsubsection}
\nc{\on}{\operatorname}
\nc{\term}[1]{#1\xspace}

\nc{\sA}{\ensuremath{\mathcal{A}}\xspace}
\nc{\sB}{\ensuremath{\mathcal{B}}\xspace}
\nc{\sC}{\ensuremath{\mathcal{C}}\xspace}
\nc{\sD}{\ensuremath{\mathcal{D}}\xspace}
\nc{\sE}{\ensuremath{\mathcal{E}}\xspace}
\nc{\sF}{\ensuremath{\mathcal{F}}\xspace}
\nc{\sG}{\ensuremath{\mathcal{G}}\xspace}
\nc{\sH}{\ensuremath{\mathcal{H}}\xspace}
\nc{\sI}{\ensuremath{\mathcal{I}}\xspace}
\nc{\sJ}{\ensuremath{\mathcal{J}}\xspace}
\nc{\sK}{\ensuremath{\mathcal{K}}\xspace}
\nc{\sL}{\ensuremath{\mathcal{L}}\xspace}
\nc{\sM}{\ensuremath{\mathcal{M}}\xspace}
\nc{\sN}{\ensuremath{\mathcal{N}}\xspace}
\nc{\sO}{\ensuremath{\mathcal{O}}\xspace}
\nc{\sP}{\ensuremath{\mathcal{P}}\xspace}
\nc{\sQ}{\ensuremath{\mathcal{Q}}\xspace}
\nc{\sR}{\ensuremath{\mathcal{R}}\xspace}
\nc{\sS}{\ensuremath{\mathcal{S}}\xspace}
\nc{\sT}{\ensuremath{\mathcal{T}}\xspace}
\nc{\sU}{\ensuremath{\mathcal{U}}\xspace}
\nc{\sV}{\ensuremath{\mathcal{V}}\xspace}
\nc{\sW}{\ensuremath{\mathcal{W}}\xspace}
\nc{\sX}{\ensuremath{\mathcal{X}}\xspace}
\nc{\sY}{\ensuremath{\mathcal{Y}}\xspace}
\nc{\sZ}{\ensuremath{\mathcal{Z}}\xspace}

\nc{\bA}{\ensuremath{\mathbf{A}}\xspace}
\nc{\bB}{\ensuremath{\mathbf{B}}\xspace}
\nc{\bC}{\ensuremath{\mathbf{C}}\xspace}
\nc{\bD}{\ensuremath{\mathbf{D}}\xspace}
\nc{\bE}{\ensuremath{\mathbf{E}}\xspace}
\nc{\bF}{\ensuremath{\mathbf{F}}\xspace}
\nc{\bG}{\ensuremath{\mathbf{G}}\xspace}
\nc{\bH}{\ensuremath{\mathbf{H}}\xspace}
\nc{\bI}{\ensuremath{\mathbf{I}}\xspace}
\nc{\bJ}{\ensuremath{\mathbf{J}}\xspace}
\nc{\bK}{\ensuremath{\mathbf{K}}\xspace}
\nc{\bL}{\ensuremath{\mathbf{L}}\xspace}
\nc{\bM}{\ensuremath{\mathbf{M}}\xspace}
\nc{\bN}{\ensuremath{\mathbf{N}}\xspace}
\nc{\bO}{\ensuremath{\mathbf{O}}\xspace}
\nc{\bP}{\ensuremath{\mathbf{P}}\xspace}
\nc{\bQ}{\ensuremath{\mathbf{Q}}\xspace}
\nc{\bR}{\ensuremath{\mathbf{R}}\xspace}
\nc{\bS}{\ensuremath{\mathbf{S}}\xspace}
\nc{\bT}{\ensuremath{\mathbf{T}}\xspace}
\nc{\bU}{\ensuremath{\mathbf{U}}\xspace}
\nc{\bV}{\ensuremath{\mathbf{V}}\xspace}
\nc{\bW}{\ensuremath{\mathbf{W}}\xspace}
\nc{\bX}{\ensuremath{\mathbf{X}}\xspace}
\nc{\bY}{\ensuremath{\mathbf{Y}}\xspace}
\nc{\bZ}{\ensuremath{\mathbf{Z}}\xspace}

\nc{\dA}{\ensuremath{\mathds{A}}\xspace}
\nc{\dB}{\ensuremath{\mathds{B}}\xspace}
\nc{\dC}{\ensuremath{\mathds{C}}\xspace}
\nc{\dD}{\ensuremath{\mathds{D}}\xspace}
\nc{\dE}{\ensuremath{\mathds{E}}\xspace}
\nc{\dF}{\ensuremath{\mathds{F}}\xspace}
\nc{\dG}{\ensuremath{\mathds{G}}\xspace}
\nc{\dH}{\ensuremath{\mathds{H}}\xspace}
\nc{\dI}{\ensuremath{\mathds{I}}\xspace}
\nc{\dJ}{\ensuremath{\mathds{J}}\xspace}
\nc{\dK}{\ensuremath{\mathds{K}}\xspace}
\nc{\dL}{\ensuremath{\mathds{L}}\xspace}
\nc{\dM}{\ensuremath{\mathds{M}}\xspace}
\nc{\dN}{\ensuremath{\mathds{N}}\xspace}
\nc{\dO}{\ensuremath{\mathds{O}}\xspace}
\nc{\dP}{\ensuremath{\mathds{P}}\xspace}
\nc{\dQ}{\ensuremath{\mathds{Q}}\xspace}
\nc{\dR}{\ensuremath{\mathds{R}}\xspace}
\nc{\dS}{\ensuremath{\mathds{S}}\xspace}
\nc{\dT}{\ensuremath{\mathds{T}}\xspace}
\nc{\dU}{\ensuremath{\mathds{U}}\xspace}
\nc{\dV}{\ensuremath{\mathds{V}}\xspace}
\nc{\dW}{\ensuremath{\mathds{W}}\xspace}
\nc{\dX}{\ensuremath{\mathds{X}}\xspace}
\nc{\dY}{\ensuremath{\mathds{Y}}\xspace}
\nc{\dZ}{\ensuremath{\mathds{Z}}\xspace}

\nc{\bbA}{\ensuremath{\mathbb{A}}\xspace}
\nc{\bbB}{\ensuremath{\mathbb{B}}\xspace}
\nc{\bbC}{\ensuremath{\mathbb{C}}\xspace}
\nc{\bbD}{\ensuremath{\mathbb{D}}\xspace}
\nc{\bbE}{\ensuremath{\mathbb{E}}\xspace}
\nc{\bbF}{\ensuremath{\mathbb{F}}\xspace}
\nc{\bbG}{\ensuremath{\mathbb{G}}\xspace}
\nc{\bbH}{\ensuremath{\mathbb{H}}\xspace}
\nc{\bbI}{\ensuremath{\mathbb{I}}\xspace}
\nc{\bbJ}{\ensuremath{\mathbb{J}}\xspace}
\nc{\bbK}{\ensuremath{\mathbb{K}}\xspace}
\nc{\bbL}{\ensuremath{\mathbb{L}}\xspace}
\nc{\bbM}{\ensuremath{\mathbb{M}}\xspace}
\nc{\bbN}{\ensuremath{\mathbb{N}}\xspace}
\nc{\bbO}{\ensuremath{\mathbb{O}}\xspace}
\nc{\bbP}{\ensuremath{\mathbb{P}}\xspace}
\nc{\bbQ}{\ensuremath{\mathbb{Q}}\xspace}
\nc{\bbR}{\ensuremath{\mathbb{R}}\xspace}
\nc{\bbS}{\ensuremath{\mathbb{S}}\xspace}
\nc{\bbT}{\ensuremath{\mathbb{T}}\xspace}
\nc{\bbU}{\ensuremath{\mathbb{U}}\xspace}
\nc{\bbV}{\ensuremath{\mathbb{V}}\xspace}
\nc{\bbW}{\ensuremath{\mathbb{W}}\xspace}
\nc{\bbX}{\ensuremath{\mathbb{X}}\xspace}
\nc{\bbY}{\ensuremath{\mathbb{Y}}\xspace}
\nc{\bbZ}{\ensuremath{\mathbb{Z}}\xspace}

\nc{\mrm}[1]{\ensuremath{\mathrm{#1}}\xspace}
\nc{\mit}[1]{\ensuremath{\mathit{#1}}\xspace}
\nc{\mbf}[1]{\ensuremath{\mathbf{#1}}\xspace}
\nc{\mcal}[1]{\ensuremath{\mathcal{#1}}\xspace}
\nc{\msc}[1]{\ensuremath{\mathscr{#1}}\xspace}
\nc{\mfr}[1]{\ensuremath{\mathfrak{#1}}\xspace}

\renc{\bar}[1]{\overline{#1}}
\nc{\sub}{\subset}
\nc{\too}{\longrightarrow}
\nc{\hook}{\hookrightarrow}
\nc*{\hooklongrightarrow}{\ensuremath{\lhook\joinrel\relbar\joinrel\rightarrow}}
\nc{\hooklong}{\hooklongrightarrow}
\nc{\twoheadlongrightarrow}{\relbar\joinrel\twoheadrightarrow}
\nc{\shiso}{\approx}
\nc{\isoto}{\xrightarrow{\sim}}
\nc{\isofrom}{\xleftarrow{\sim}}
\renc{\ge}{\geqslant}
\renc{\le}{\leqslant}

\nc{\id}{\mathrm{id}}

\DeclareMathOperator{\Hom}{\on{Hom}}
\nc{\uHom}{\underline{\smash{\Hom}}}

\DeclareMathOperator{\End}{\on{End}}

\nc{\uEnd}{\underline{\smash{\End}}}

\renc{\lim}{\varprojlim}
\makeatletter
\newcommand{\colim@}[2]{%
  \vtop{\m@th\ialign{##\cr
    \hfil$#1\operator@font colim$\hfil\cr
    \noalign{\nointerlineskip\kern1.5\ex@}#2\cr
    \noalign{\nointerlineskip\kern-\ex@}\cr}}%
}
\newcommand{\colim}{%
  \mathop{\mathpalette\colim@{\rightarrowfill@\textstyle}}\nmlimits@
}
\makeatother

\nc{\Cofib}{\on{Cofib}}
\nc{\Fib}{\on{Fib}}
\nc{\initial}{\varnothing}
\newcommand{\opp}{\mathrm{op}}

\usepackage[normalem]{ulem}
\nc{\Spc}{\mrm{Spc}}
\nc{\Spt}{\mrm{Spt}}
\nc{\Spec}{\on{Spec}}
\nc{\Stk}{\mrm{Stk}}
\nc{\Sch}{\mrm{Sch}}
\nc{\aff}{\mrm{aff}}
\nc{\A}{\mbf{A}}
\renc{\P}{\mbf{P}}
\nc{\cl}{{\mrm{cl}}}
\nc{\bDelta}{\mathbf{\Delta}}
\nc{\un}{\mathbf{1}}
\nc{\Tot}{\on{Tot}}
\nc{\Cech}{\textnormal{\v{C}}}
\nc{\Mod}{\mrm{Mod}}
\nc{\Qcoh}{\on{Qcoh}}
\nc{\free}{\mrm{free}}
\nc{\ex}{\mrm{ex}}
\nc{\perf}{\mrm{perf}}
\nc{\aperf}{\mrm{aperf}}
\nc{\coh}{\mrm{coh}}
\newcommand{\Cat}{\mrm{Cat}}
\nc{\unitm}{\mbf{1}}
\nc{\sphere}{\mbf{S}}
\nc{\Z}{\mbf{Z}}
\nc{\Map}{\mrm{Map}}
\nc{\map}{\mrm{map}}
\nc{\PrL}{\mathcal{P}\mathrm{r}^\mrm{L}}
\nc{\PrLst}{\mathcal{P}\mathrm{r}^\mrm{L}_\mrm{st}}
\nc{\Motnc}{\mathcal{M}_{\mrm{loc}}}
\nc{\Motadd}{\mathcal{M}_{\mrm{add}}}
\nc{\Einfty}{{\sE_\infty}}
\nc{\E}[1]{{\sE_{#1}}}
\nc{\modmod}{/\!\!/}
\nc{\heart}{\heartsuit}
\nc{\proj}{\mrm{proj}}
\nc{\LL}{\on{L}}
\nc{\K}{\on{K}}
\nc{\G}{\on{G}}
\nc{\GL}{\on{GL}}
\nc{\BGL}{\on{BGL}}
\nc{\M}{\on{M}}
\nc{\KH}{\on{KH}}
\nc{\Alg}{\on{Alg}}
\nc{\CAlg}{\on{CAlg}}
\nc{\cn}{\mrm{cn}}
\nc{\hw}{\mrm{Hw}}
\nc{\htt}{\mrm{Ht}}
\nc{\Fun}{\on{Fun}}
\nc{\Funadd}{\on{Fun}_{\mrm{add}}}
\nc{\Funex}{\on{Fun}_{\mrm{ex}}}
\nc{\Ind}{\on{Ind}}
\nc{\Pro}{\on{Pro}}
\nc{\Kar}{\on{Kar}}
\nc{\Obj}{\on{Obj}}
\nc{\cC}{\mathcal{C}}

\nc{\scr}{\term{simplicial commutative ring}}
\nc{\scrs}{\term{simplicial commutative rings}}

\nc{\Einfring}{\term{$\Einfty$-ring}}
\nc{\Einfrings}{\term{$\Einfty$-rings}}

\nc{\Ering}{\term{$\sE_1$-ring}}
\nc{\Erings}{\term{$\sE_1$-rings}}

\nc{\inftyCat}{\term{category}}
\nc{\inftyCats}{\term{categories}}

\nc{\inftyTop}{\term{topos}}
\nc{\inftyTops}{\term{toposes}}

\nc{\inftyGrpd}{\term{groupoid}}
\nc{\inftyGrpds}{\term{groupoids}}

\def\mc{\mathcal}
\def\mb{\mathbf}

\def\op{\mathrm}

\newcommand{\enu}[1]{ \begin{enumerate}[label=(\arabic*),font=\normalfont]
		#1
	\end{enumerate}
}

\newcommand{\alg}{\operatorname{Alg}}
\newcommand{\calg}{\operatorname{CAlg}}
\newcommand{\cmon}{\operatorname{CMon}}
\newcommand{\modu}{\operatorname{Mod}}

\newcommand{\rmodu}{\operatorname{RMod}}
\newcommand{\amodu}{\operatorname{aMod}}
\newcommand{\mapp}{\operatorname{Map}}
\newcommand{\unmap}{\underline{\operatorname{Map}}}
\newcommand{\unat}{\underline{\operatorname{at}}}

\newcommand{\colimit}{\mathrm{colim}}
\newcommand{\funct}{\operatorname{Fun}}

\newcommand{\ageq}{\ensuremath{\mathcal{A}_{\geq0}}\xspace}

\newcommand{\catsift}{\ensuremath{\widehat{\Cat}_\infty^{\op{sift}}}\xspace}
\newcommand{\catadidem}{\ensuremath{\mathrm{Cat}_{\mathrm{ad}}^{\mathrm{idem}}}\xspace}
\newcommand{\catadidemop}{\ensuremath{\mathrm{Cat}_{\mathrm{ad}}^{\mathrm{idem,op}}}\xspace}
\newcommand{\catadidemomega}{\ensuremath{\mathrm{Cat}_{\mathrm{ad}}^{\mathrm{idem},\omega}}\xspace}
\newcommand{\catadidemomegaone}{\ensuremath{\mathrm{Cat}_{\mathrm{ad}}^{\mathrm{idem},\omega_1}}\xspace}
\newcommand{\catadidemkappa}{\ensuremath{\mathrm{Cat}_{\mathrm{ad}}^{\mathrm{idem},\kappa}}\xspace}

\newcommand{\prl}{\ensuremath{\mc{P}\mathrm{r}^\mathrm{L}}\xspace} 
\newcommand{\prlv}{\ensuremath{\mc{P}\mathrm{r}^\mathrm{L}_{\mathcal{V}}}\xspace}

\newcommand{\pr}{\ensuremath{\mc{P}\mathrm{r}}\xspace}
\newcommand{\prad}{\ensuremath{\mc{P}\mathrm{r}_{\op{ad}}}\xspace}
\newcommand{\prlad}{\ensuremath{\mc{P}\mathrm{r}_{\op{ad}}^\mathrm{L}}\xspace}
\newcommand{\praddbl}{\ensuremath{\mc{P}\mathrm{r}_{\op{ad}}^{\op{dbl}}}\xspace}
\newcommand{\pradat}{\ensuremath{\mc{P}\mathrm{r}_{\op{ad}}^{\op{at}}}\xspace}

\newcommand{\praddblop}{\ensuremath{\mc{P}\mathrm{r}_{\op{ad}}^{\op{dbl,op}}}\xspace}

\newcommand{\prlst}{\ensuremath{\mc{P}\mathrm{r}_{\op{st}}^\mathrm{L}}\xspace}
\newcommand{\prlpst}{\ensuremath{\mc{P}\mathrm{r}^\mathrm{L}_{\op{pst}}}\xspace}

\newcommand{\prst}{\ensuremath{\mc{P}\mathrm{r}_{\op{st}}}\xspace}
\newcommand{\prstdbl}{\ensuremath{\mc{P}\mathrm{r}_{\op{st}}^{\op{dbl}}}\xspace}
\newcommand{\prvdbl}{\ensuremath{\mc{P}\mathrm{r}_{\mc{V}}^{\op{dbl}}}\xspace}
\newcommand{\prtrex}{\ensuremath{\mc{P}\mathrm{r}^{t\text{-rex}}_{\op{st}}}\xspace}
\newcommand{\prtil}{\ensuremath{\mc{P}\mathrm{r}^{t\text{-}iL}_{\op{st}}}\xspace}

\newcommand{\motpst}{\ensuremath{\mathcal{M}\mathrm{ot}_{\mathrm{pst}}}\xspace}
\newcommand{\motpstcn}{\ensuremath{\mathcal{M}\mathrm{ot}^{\mathrm{cn}}_{\mathrm{pst}}}\xspace}
\newcommand{\motspl}{\ensuremath{\mathcal{M}\mathrm{ot}_{\mathrm{spl}}}\xspace}
\newcommand{\motomegaonespl}{\ensuremath{\mathcal{M}\mathrm{ot}_{\omega_{1}\text{-}\mathrm{spl}}}\xspace}
\newcommand{\mot}{\ensuremath{\mathcal{M}\mathrm{ot}}\xspace}
\newcommand{\motsplcn}{\ensuremath{\mathcal{M}\mathrm{ot}_{\mathrm{spl}}^{\mathrm{cn}}}\xspace}
\newcommand{\motomegaonesplcn}{\ensuremath{\mathcal{M}\mathrm{ot}_{\omega_{1}\text{-}\mathrm{spl}}^{\mathrm{cn}}}\xspace}

\newcommand{\Loc}{\ensuremath{\operatorname{Loc}}\xspace}
\newcommand{\Split}{\ensuremath{\operatorname{Split}}\xspace}

\newcommand{\spgeq}{\ensuremath{\operatorname{Sp}_{\geq0}}\xspace}
\newcommand{\spgeqcproj}{\ensuremath{\spgeq^{\mathrm{cproj}}}\xspace}

\newcommand{\hatcat}{\ensuremath{\widehat{\op{Cat}}_{\infty}}\xspace}

\newcommand{\ass}{\ensuremath{\mathrm{CAss}}\xspace}
\newcommand{\assad}{\ensuremath{\mathrm{CAss}_{\op{ad}}}\xspace}

\newcommand{\eonering}{\term{$\mathbb E_1$-ring}}
\newcommand{\eonerings}{\term{$\mathbb E_1$-rings}}

\newcommand{\einfring}{\term{$\mathbb E_\infty$-ring}}
\newcommand{\einfrings}{\term{$\mathbb E_\infty$-rings}}

\newcommand{\infcat}{\term{category}}
\newcommand{\infcats}{\term{categories}}
\newcommand{\syminfcat}{\term{symmetric monoidal category}}
\newcommand{\dualaddinfcat}{\term{dualizable additive category}}
\newcommand{\dualaddinfcats}{\term{dualizable additive categories}}

\newcommand{\shv}{\operatorname{Shv}}

\newcommand{\cflat}{\ensuremath{\mathcal{C}^{\op{flat}}}\xspace}
\newcommand{\cflatw}{\ensuremath{\mcc^{\omega_1\text{-}\op{flat}}}\xspace}
\newcommand{\dflat}{\ensuremath{\mathcal{D}^{\op{flat}}}\xspace}
\newcommand{\dflatw}{\ensuremath{\mcd^{\omega_1\text{-}\op{flat}}}\xspace}

\newcommand{\mca}{\ensuremath{\mathcal{A}}\xspace}
\newcommand{\mcb}{\ensuremath{\mathcal{B}}\xspace}
\newcommand{\mcc}{\ensuremath{\mathcal{C}}\xspace}
\newcommand{\mcd}{\ensuremath{\mathcal{D}}\xspace}
\newcommand{\mce}{\ensuremath{\mathcal{E}}\xspace}
\newcommand{\mcf}{\ensuremath{\mathcal{F}}\xspace}

\newcommand{\mci}{\ensuremath{\mathcal{I}}\xspace}

\newcommand{\mck}{\ensuremath{\mathcal{K}}\xspace}

\newcommand{\mcm}{\ensuremath{\mathcal{M}}\xspace}
\newcommand{\mcn}{\ensuremath{\mathcal{N}}\xspace}

\newcommand{\mcp}{\ensuremath{\mathcal{P}}\xspace}

\newcommand{\mcs}{\ensuremath{\mathcal{S}}\xspace}
\newcommand{\mct}{\ensuremath{\mathcal{T}}\xspace}
\newcommand{\mcu}{\ensuremath{\mathcal{U}}\xspace}
\newcommand{\mcv}{\ensuremath{\mathcal{V}}\xspace}
\newcommand{\mcw}{\ensuremath{\mathcal{W}}\xspace}

\newcommand{\coker}{\op{coker}}
\newcommand{\cofib}{\op{cofib}}
\newcommand{\catper}{\ensuremath{\op{Cat}^{\op{perf}}}\xspace}

\newcommand{\opsp}{\ensuremath{\op{Sp}}\xspace}
\newcommand{\kcn}{\ensuremath{K_{\mathrm{cn}}\xspace}}
\newcommand{\stab}{\ensuremath{\op{Stab}}\xspace}
\newcommand{\calk}{\ensuremath{\op{Calk}}\xspace}
\newcommand{\ind}{\ensuremath{\op{Ind}}\xspace}
\newcommand{\fib}{\ensuremath{\op{fib}}\xspace}

\newcommand{\nuc}{\ensuremath{\op{Nuc}}\xspace}

\newcommand{\nuccsr}{\ensuremath{\op{Nuc}^{CS}(R)}\xspace}
\newcommand{\nucr}{\ensuremath{\op{Nuc}(R)}\xspace}

\newcommand{\nucrgeq}{\ensuremath{\op{Nuc}(R)_{\geq 0}}\xspace}
\newcommand{\modrcpl}{\ensuremath{\op{Mod}_{R}^{\op{cpl}}}\xspace}
\newcommand{\modrcplgeq}{\ensuremath{\op{Mod}_{R,\geq 0}^{\op{cpl}}}\xspace}
\newcommand{\closure}[1]{\overline{#1}}
\begin{document}
		\let\mathbb=\mathbf
	
	\title{Dualizable additive categories}
	\author{Ishan Levy, Jiacheng Liang, Vladimir Sosnilo}

	\maketitle

\vspace{-1.5em}
	
	\begin{abstract}
	We develop a comprehensive theory of dualizable additive categories. We provide several equivalent characterizations, notably identifying them as separated Grothendieck prestable categories satisfying the $\mathrm{AB4}^*$ and $\mathrm{AB6}$ axioms. We establish a connection to almost mathematics by demonstrating that they arise precisely as the categories of connective almost modules over connective $\mathbb{E}_1$-rings. 
	
	Furthermore, we prove that dualizable additive categories are generated by flat objects, and that the passage to flat objects yields an equivalence between dualizable additive categories and compactly assembled additive categories. As a primary application within analytic geometry, we characterize the category $\mathrm{Nuc}(R)_{\geq0}$ of connective nuclear modules (in the sense of Clausen--Scholze \cite{scholze2026lecturesanalyticgeometry}) over an adic $\mathbb{E}_\infty$-ring $R$ via a universal property, identifying it as the additive rigidification of the category of connective complete $R$-modules.
	
	Finally, we construct the universal finitary stable localizing invariant for dualizable additive categories, the presentable stable category $\mathcal{M}\mathrm{ot}_{\mathrm{pst}}$ of prestable motives, and demonstrate that its unit corepresents nonconnective algebraic $K$-theory. We prove that the motives of small additive categories and those of dualizable additive categories generate the same presentable stable subcategory.
	\end{abstract}
		\tableofcontents

	\section*{Introduction}


Algebraic $K$-theory shows up when studying algebraic varieties, infinite groups, $C^*$-algebras, ring spectra, and many other 
objects in different areas of mathematics. The main reason behind this ubiquity is the ability to define it in a uniform way as a 
{\it localizing invariant} of small stable \infcats. Moreover, many fundamental properties of algebraic $K$-theory as a cohomology 
theory, such as the projective bundle formula, Zariski excision and the existence of transfers, also follow from it being a 
localizing invariant. 

An influential idea of Efimov \cite{HoyoisEfimov,efimov2024k} pushed this point of view even further: he observed that an 
arbitrary localizing invariant $E$ naturally extends to a localizing invariant $E^\mathrm{cont}$ of {\it dualizable stable 
\infcats}. This turns out to be especially useful in analytic contexts where the natural stable \infcats associated to analytic spaces are not 
small (or compactly generated), but are dualizable. For example, to each rigid analytic space one associates the \infcat $\mathrm{Nuc}(X)$ of nuclear modules in the sense of Clausen--Scholze \cite{scholze2026lecturesanalyticgeometry}, which is often not
compactly generated, and one can define a version of continuous $K$-theory of $X$ as 
$K^{\mathrm{cont}}(\mathrm{Nuc}(X))$. 

Most of the tools from the theory of small stable \infcats and localizing invariants transfer to the dualizable setting thanks to 
the existence of a resolution 
\[
	\sC \to \Ind(\sD) \to \Ind(\sE)
\]
for any dualizable stable \infcat $\sC$. Here $\sD, \sE$ are small stable \infcats and the functors form a Verdier sequence. 
In particular, the aforementioned continuous $K$-theory of a rigid analytic space can be shown to satisfy Zariski descent and the
projective bundle formula. At the same time, dualizable \infcats provide new tools as well. The limits and the internal hom in 
the \infcat of dualizable stable \infcats $\prstdbl$ are very different from the corresponding constructions in $\PrLst$ and 
they are very interesting even when the input comes from small stable \infcats. 
Efimov used these constructions in \cite{efimov2025localizinginvariantsinverselimits} to prove the isomorphisms 
\[
	K^{\mathrm{cont}}\big(\mathrm{Nuc}(R_I^\wedge)\big) \simeq K^{\mathrm{cont}}\big(\lim^{\on{dbl}}_n \mcd(R/I^n)\big) \simeq \lim K(R/I^n)
\]
for any noetherian ring $R$ with an ideal $I$, showing that the continuous $K$-theory defined this way extends the construction 
that has actually been studied in analytic geometry 
\cite{Wagoner1976,Bloch2013,morrow2016historicaloverviewprocdh,kerzsaitotamme}. 

Often, working solely with stable categories is insufficient; remembering  additional data, such as a $t$-structure, provides a finer geometric perspective.
For example, the data of a $t$-structure is necessary to even state Tannaka duality \cite[Theorem~9.2.0.2]{SAG} and the categorical version of 
Dundas--Goodwillie--Mccarthy theorem \cite[Section~5]{levy2025c} (see also \cite[Theorem~1.0.1]{dgm-elden-vova}).
Some other geometric properties can also be formulated in the language of $t$-structures, and, 
in particular, affineness of $X$ corresponds to the property that the $t$-structure on $\op{QCoh}(X)$ comes from a {\it bounded weight structure} on the subcategory of compact objects. 
When it happens for a compactly generated \infcat with a $t$-structure $\sC$, we have an equivalence 
\[
	\sC \simeq \ind(\op{Stab}(\sA))
\]
where $\sA$ is a small additive \infcat. In other words, $\sC$ is freely generated by a small additive \infcat. This situation is central for \cite{levy2025c}, as it turns out 
many more general stable \infcats can be resolved by these. 
In the dualizable setting the role of such categories is played by {\it dualizable additive categories}. 
The goal of this paper is to develop a theory of dualizable prestable \infcats, analogous to the theory of dualizable stable \infcats,
but taking into account the notion of connectivity:

\begin{de}\label{dfn:dual_add}
	A \emph{dualizable additive \infcat} is a dualizable object in the 
	symmetric monoidal \infcat $\prlad$ of presentable additive \infcats.
\end{de}

We use $\praddbl$ to denote the (non full) subcategory of $\prlad$ consisting of dualizable additive \infcats and internal left adjoints.
For instance, connective modules over a connective \eonering naturally form a dualizable additive 
\infcat, as do additive presheaves on any small additive \infcat. However, as we will see below, 
dualizability accommodates a broader class of examples arising in analytic geometry. 
Another interesting family of examples comes from categories of connective 
sheaves on locally profinite spaces (see  \cref{exm:top_exm}).

One of our main results provides several characterizations of dualizable additive \infcats, analogous to those established in the stable setting (see \cite[Proposition D.7.3.1]{SAG}, \cite{efimov2024k} and \cite{ramzi2024dualizable}).

To state this precisely, we define a colimit-preserving functor $p \colon \mcc \to \mcd$ between presentable prestable \infcats to be a \emph{connective internal localization} (\cref{defconnloc}) if it admits a fully faithful, colimit-preserving right adjoint $p^R$ such that the unit map $x \to p^R p(x)$ is a $\pi_0$-epimorphism.
\begin{thmX}[{\cref{main1}, \ref{basicses}, \ref{main2}}]
	Let $\sC$ be a presentable additive \infcat. Then the following are equivalent:
	\begin{enumerate}
		\item $\sC$ is dualizable additive.
		\item $\sC$ is separated Grothendieck prestable and satisfies $\mathrm{AB4}^*$ and $\mathrm{AB6}$.
		\item $\sC$ is the kernel of a connective internal localization $L$:
		\[
		\sC \to  \mc{P}_\Sigma(\sD) \xrightarrow{L}  \mc{P}_\Sigma(\sE),
		\]
		where $\mcd$ and $\mce$ are some small additive \infcats.
		\item $\sC$ is the kernel of an extension of scalars functor
		$
	  \Mod_{R,\geq0} \xrightarrow{S\otimes_R -}  \Mod_{S,\geq0},
		$
		where $R \to S$ is a $\pi_0$-surjective homological epimorphism of connective \eonerings. Equivalently, there exists a connective \eonering $R$ and an idempotent (two-sided) ideal $I\subset \pi_0R$ such that 
		\[
		\mcc\simeq \amodu_{(R,I)}(\spgeq).
		\]
	\end{enumerate}
\end{thmX}

We note that the final characterization 
exposes a  bridge between dualizable additive \infcats and almost mathematics.
 By a result of Hebestreit and Scholze \cite{hebestreit2024note}, the $\pi_0$-surjective homological epimorphisms out of a connective \eonering $R$ are exactly classified by the idempotent ideals of $\pi_0R$.
 
Dualizable additive categories are generally not compactly generated and often lack 
projective objects entirely. To understand their internal generation, we must shift our focus from projectivity to flatness. 
We demonstrate that these categories are controlled by their flat objects.

\begin{thmX}[{\cref{flatgenerate}, \ref{flcompass}, \ref{main3}}]\label{thmbintro}
	Let $\mcc$ be a dualizable additive \infcat. Then $\mcc$ is generated by $\omega_1$-compact flat objects under small colimits, and the subcategory of flat objects $\mcc^{\op{flat}}$ is a compactly assembled additive \infcat.
	
	 Furthermore, $\mcc$ can be canonically identified with the continuous prestabilization of its flat objects, establishing a categorical equivalence between compactly assembled additive categories and $\praddbl$:
	$$
	\begin{tikzcd}
		\assad\arrow[r, shift left=1ex, "\mcp_{\sqcup,\op{fil}}^{\op{small}}"{name=G}] & \praddbl\arrow[l, shift left=.5ex, "(-)^{\op{flat}}"{name=F}]
		\arrow[phantom, from=F, to=G, , "\scriptscriptstyle\boldsymbol{\sim}"].
	\end{tikzcd}
	$$
\end{thmX}

Having established some of the basic features of dualizable additive categories, we apply them in the setting of analytic geometry, focusing on connective nuclear modules in the sense of Clausen--Scholze.

In tensor-triangular geometry, the condition of being a \emph{compactly-rigidly generated} $tt$-category—where the tensor unit is compact, compact objects are dualizable, and compact objects generate the category under colimits—serves as the foundational framework where finiteness properties and duality behave exceptionally well. In higher category theory, a modern incarnation of this philosophy was introduced in \cite{arinkin2020stack} and further developed in \cite{ramzi2024locally,efimov2025localizinginvariantsinverselimits} under the name of \emph{rigidity}, generalizing the classical rigidity condition from the compactly generated setting to general presentable categories.

Just as additive dualizability is strictly stronger than stable dualizability, \emph{additive rigidity} (rigidity for $\spgeq$-algebras) strictly implies the \emph{stable rigidity} (rigidity for $\opsp$-algebras) of its stabilization. Crucially, additive rigidity provides a robust framework for higher algebra, leading to $t$-structured tensor-triangular geometry \cite{hattt}.

When a presentably symmetric monoidal category lacks this rigidity, one can consider its \emph{rigidification}, the universal rigid approximation mapping into it. As established by Ramzi \cite[Corollary 4.73]{ramzi2024locally}, such rigidifications always exist. 
 It turns out that this purely categorical machinery interacts remarkably well with analytic geometry, recovering a fundamental construction:
\begin{thmX}[{\cref{main5}}]\label{mmain5}
	Let $R$ be a connective adic \einfring. Then the \infcat $\nucrgeq$ of connective nuclear modules (in the sense of Clausen--Scholze) is additively rigid. Furthermore, it can be canonically identified with the additive rigidification of the \infcat $\modrcplgeq$ of connective complete $R$-modules. 	
\end{thmX}
\begin{corX}[{\cref{main6}}]
	For any uniform Tate ring $A$, the \infcat of connective nuclear modules $\nuc(A)_{\geq0}$ is  additively rigid. In particular, $$\nuc(\mathbb{Q}_p)_{\geq0}$$ is  additively rigid.
\end{corX}
We note that because the relevant $t$-structure is right complete (so that its stabilization recovers the full category, i.e., $\opsp(\nucrgeq) \simeq \nucr$), working with the connective part $\nucrgeq$ already captures the whole categorical information of $\nucr$.

These results complement the discussion before \cref{dfn:dual_add} which suggests that stabilizations of dualizable additive \infcats should be thought of as affine. We note that  
$\nuc(X)_{\geq0}$ is not expected to be additively rigid or dualizable unless $X$ is affine. 

Finally, we study the stable localizing motives $\motpst$ of dualizable additive \infcats. We largely follow the approach of \cite{blumberg2013universal}, although certain steps require subtle modifications in the additive context. 
\begin{thmX}[\cref{defmotpst}]
	There exists a  universal finitary localizing invariant to a presentable stable \infcat $$\mcu^{\op{cont}}:\praddbl\to \motpst$$ in the following sense: for any (potentially large) cocomplete stable \infcat \mce, the following functor is an equivalence to the \infcat of finitary localizing invariants
	$$\funct^L(\motpst,\mce)\xrightarrow{\sim}\Loc_\omega(\praddbl,\mce).$$
\end{thmX}
We prove that the motives of small additive \infcats and those of dualizable additive \infcats generate the same presentable stable subcategory. 
Employing a strategy inspired by \cite[Theorem 4.10]{efimov2024k}, we establish the following principle, which relates localizing invariants to the spherical sheaves introduced by Pstrągowski \cite{pstrkagowski2023synthetic} in the context of synthetic spectra:
\[
\{\text{localizing invariants}\} \simeq \{\text{spherical sheaves in the $\mcf\mcf$-topology}\}.
\]
Recall that a spherical sheaf is a functor satisfying descent in a topology and
preserving finite products. The $\mcf\mcf$-topology is the topology generated by singleton
fully faithful functors (see \cref{dfn:fftopology}).
\begin{thmX}[\cref{locinvsphshv}, \ref{smalllargelocinv}]\label{thm:sphericalsheaves}
	Let $\mce$ be a stable \infcat with small colimits, and let 
	$\star \in \{ \op{dbl}, \op{at} \}$. A functor $E \colon \prad^\star \to \mce$ is a 
	localizing invariant if and only if it is a spherical sheaf in the $\mcf\mcf$-topology, 
	i.e. a finite product preserving sheaf in the topology generated by singleton fully 
	faithful functors. This yields a natural equivalence
	\[
	\op{Loc}(\prad^\star,\mce)\simeq \shv_\Sigma(\prad^{\star,\opp};\mce).
	\]
	
	In this topology any object of $\prad^{\op{dbl},\opp}$ admits a cover by of $\prad^{\op{at},\opp}$, 
	so small additive 
	\infcats and dualizable additive \infcats share the same (finitary) localizing invariants:
	\[
	\op{Loc}_{\omega}(\praddbl,\mce)\xrightarrow{\sim} \op{Loc}_{\omega}(\prad^{\op{at}},\mce).
	\]
	Consequently, their motives generate precisely the same presentable stable subcategory of prestable motives.
\end{thmX}

We conclude by demonstrating that the unit of $\motpst$ corepresents nonconnective algebraic $K$-theory.

\begin{thmX}[\cref{main4}]
For any dualizable additive \infcat $\sA$, there is a natural equivalence from the mapping spectrum:
	\[
	\unmap_{\motpst}(\mcu^{\op{cont}}(\spgeq), \mcu^{\op{cont}}(\sA)) \simeq K^{\op{cont}}(\opsp(\sA)).
	\]
\end{thmX}

There is a natural functor $$\motpst \to \Motnc$$ given by sending a dualizable additive \infcat to its stabilization.
Although the units of both categories corepresent algebraic $K$-theory, we doubt that the functor is fully faithful. 
It is known not to be essentially surjective, see \cite[Theorem 7.11]{levy2025c} or \cite[Theorem 7.1]{efimov2025rigidity}. We note that the main result of \cite{levy2025c} shows that
$c$-categories give rise to objects in $\motpst$ that may not come from dualizable additive \infcats. We intend to 
explore a dualizable version of $c$-categories in future work to put these on the same footing.



\subsection*{Outline}{}
In section \ref{section1}, we recall some of the basics of (presentable) prestable \infcats, since dualizable additive categories are examples of these. This section is largely expositional, and most of the results come from \cite{SAG}.

In section \ref{section2}, we begin studying dualizable additive \infcats in more detail, recalling some of the general theory such as in \cite{ramzi2024dualizable} and characterizing dualizable additive categories as separated Grothendieck prestable categories satisfying
$\mathrm{AB4}^*$ and $\mathrm{AB6}$.

In section \ref{section3}, we study additive short exact sequences, which for dualizable additive categories are preserved by localizing invariants. In particular, we prove that every dualizable additive category can be resolved by those that are compact projectively generated, 
and relate dualizable additive categories to almost mathematics. We also provide some examples of how short exact sequences in the setting of small additive categories behave more poorly than for dualizable additive categories, such as not being closed under pullbacks.

In section \ref{section4}, we study flat objects in prestable categories and prove \Cref{thmbintro}.

In section \ref{section5}, we study additive rigidifications, providing criteria for a map to be an additive rigidification, and 
applying this to prove \Cref{mmain5}.

In section \ref{section6}, we study localizing invariants of dualizable additive categories, and prove \Cref{thm:sphericalsheaves} relating them to spherical sheaves in the $\mcf\mcf$-topology.

In section \ref{sec:prestable_motives}, we begin the study of the category of prestable motives, and show that maps from the unit corepresent non-connective $K$-theory. We also discuss a version of the Calkin construction 
for small additive categories, and discuss the $t$-structure on the category of additive motives. Towards the end of the section, we provide many open questions about prestable motives we are curious about.

We also include several appendices. Appendix \ref{appendixA} provides some examples showing that both
additive homological epimorphisms and additive Karoubi projections are not closed under pullbacks.
Appendix \ref{appendixB} develops some of the basics of compact projectively assembled categories and compact projective morphisms.
Appendix \ref{kappaprojgeneration} studies (non-compact) projective generation and large weight structures.
Appendix \ref{appen3} provides complete proofs of two results of Efimov used in Section \ref{section5}.
	\subsection*{Conventions and notations}

\begin{itemize}
	\item We use the word `\infcat' to mean $(\infty,1)$-category. 

	\item We let $\mathcal{S}$ denote the \infcat of anima and $\mathrm{Sp}$ the \infcat of spectra.
	
	\item For each $n \in \mathbb{N}$, $\mathcal{S}_{\leq n}$ denotes the full subcategory of $\mathcal{S}$ spanned by $n$-truncated anima. For instance, $\mathcal{S}_{\leq -1}$ is the full subcategory spanned by $\varnothing$ and $*$, while $\mathcal{S}_{\leq -2}$ consists solely of $*$.
	
	\item For a functor category $\mathrm{Fun}(\mathcal{C},\mathcal{D})$, we employ the superscripts $\op{lex}$, $\op{rex}$, $\op{ex}$, $R$, and $L$ to indicate the full subcategories consisting of functors that preserve finite limits, finite colimits, both finite limits and finite colimits (exact functors), small limits, and small colimits, respectively.
	
	\item Let $\mathcal{C}$ be a \infcat. A sequence $a \to b \to c$ in $\mcc$ equipped with a nullhomotopy of the composite is called a bifiber sequence (or a short exact sequence) if it is simultaneously a fiber and a cofiber sequence.
	
	\item The symbol $\prl$ denotes the \infcat of presentable \infcats and left adjoint functors. By default, it is endowed with the Lurie tensor product as its symmetric monoidal structure.
	
	\item For a regular cardinal $\kappa$, $\prl_{\kappa} \subset \prl$ denotes the (non-full) subcategory consisting of $\kappa$-presentable \infcats and $\kappa$-compactness preserving functors.

	\item For a \infcat $\mcc$, we denote by $\mcc^\kappa$ the full subcategory of $\kappa$-compact objects, and $\mcc^{\op{cproj}}$ denotes the full subcategory of compact projective objects.
	
	\item We denote by $\catper$ the \infcat of small, idempotent-complete stable \infcats, and by $\catadidem$ the \infcat of small, idempotent-complete additive \infcats.
	
	
	\item A symmetric monoidal presentable \infcat $\mathcal{C}$ is said to be presentably symmetric monoidal if its tensor product functor $- \otimes - \colon \mcc \times \mcc \to \mcc$ preserves colimits in each variable separately.
	
	\item The \infcat of commutative monoids in a symmetric monoidal \infcat $\mcc$ is denoted by $\mathrm{CAlg}(\mathcal{C})$, and its objects are referred to as commutative algebras or $\mathbb{E}_{\infty}$-algebras in $\mathcal{C}$. 
	We use \einfrings to refer to commutative algebras in \opsp with its standard symmetric monoidal structure.
	
	\item We write $\mapp(-,-)$ for mapping anima, $\unmap(-,-)$ for enriched mapping objects (provided the enriched context is clear), and $\underline{\Hom}(-,-)$ for internal hom-objects.

\end{itemize}

\section*{Acknowledgments}

Many of our results were inspired by the work of Sasha Efimov. In 2024, at the conference ``Nearby Cycles and Derived Geometry'' in Regensburg, he gave a lecture describing the concept of dualizable weighted 
categories. 
There he explained that Clausen--Scholze's category of nuclear modules over an $I$-complete noetherian ring $R$ can be expressed as the limit of $\Mod_{R/I^n}$ in the category of dualizable weighted 
categories. At the start of our project, it was clear that the categories of dualizable weighted categories and of dualizable additive categories should be equivalent, so in particular we knew there had to be 
a universal property of $\mathrm{Nuc}(R)$ in terms of dualizable additive categories, and it didn't take us long to guess the statement of \cref{mmain5}. 

We also would like to thank Omer Bojan, David Gepner, Lars Hesselholt, Yifan Jin, Ryo Kanda, Akhil Mathew, Maxime Ramzi, Xiangrui Shen, Germán Stefanich, and Christoph Winges for helpful conversations at various stages of this project. We are especially grateful to Yifan Jin for his detailed feedback on an earlier draft of this manuscript. 
Part of this work was carried out while the authors were attending the 2025 International Workshop on Algebraic Topology (IWOAT). We thank the organizers for their hospitality and for providing a stimulating working environment.

The first author was supported by the Clay Research Fellowship. The last author was supported by JSPS KAKENHI Grant Number JP26K16973.

	\section{Preliminaries: prestable \infcats}\label{section1}
	To establish a rigorous foundation for dualizable additive \infcats, we first recall the theory of prestable \infcats. These categories arise precisely as the connective parts of stable \infcats equipped with $t$-structures.
	
	\begin{de}[{See \cite[Proposition~C.1.2.2]{sag}}]\label{defpst}
		A prestable \infcat is a \infcat $\mathcal{C}$ satisfying the following properties:
		
		\begin{enumerate}[label=(\arabic*),font=\normalfont]
			\item Initial and final objects of $\mathcal{C}$ exist and agree (that is, $\mathcal{C}$ has a zero object).
			\item Every map admits a cofiber and every cofiber sequence in $\mathcal{C}$ is also a fiber sequence.
			\item Every map in $\mathcal{C}$ of the form $f: X \rightarrow \Sigma(Y)$ admits a fiber and $\Sigma(Y)$ is the cofiber of its fiber.
		\end{enumerate}
	
	\end{de}
	\begin{prop}[{\cite[Corollary C.1.2.3]{sag}}]\label{prestcriterion}
		Let $\mathcal{C}$ be a \infcat. The following conditions are equivalent:
		\enu{
			\item  The \infcat $\mathcal{C}$ is prestable.
			\item  There exists a fully faithful embedding $\rho: \mathcal{C} \hookrightarrow \mathcal{D}$, where $\mathcal{D}$ is a stable \infcat and the essential image of $\rho$ is closed under finite colimits and extensions.
		}
		
	\end{prop}
	
	\begin{exam}\,
		\enu{
			\item Any stable \infcat is prestable.
			\item Let $\mathcal{C}$ be a stable \infcat equipped with a $t$-structure $\left(\mathcal{C}_{\geq 0}, \mathcal{C}_{\leq 0}\right)$. Then the full subcategory $\mathcal{C}_{\geq 0} \subset \mathcal{C}$ is prestable.
		}
	\end{exam}
	
	Prestable \infcats are closely related to stable \infcats with $t$-structure.
	
	\begin{prop}[{\cite[Proposition C.1.2.9]{sag}}]
		Let $\mathcal{C}$ be a \infcat. The following conditions are equivalent:
		\enu{
			\item The \infcat $\mathcal{C}$ is prestable and admits finite limits.
			\item The \infcat $\mathcal{C}$ is pointed and admits finite colimits, and the canonical map $\rho: \mathcal{C} \rightarrow \mathrm{SW}(\mathcal{C})$ is fully faithful. Moreover, the stable \infcat $\operatorname{SW}(\mathcal{C})$ admits a $t$-structure $\left(\mathrm{SW}(\mathcal{C})_{\geq 0}, \mathrm{SW}(\mathcal{C})_{\leq 0}\right)$ where $\mathrm{SW}(\mathcal{C})_{\geq 0}$ is the essential image of $\rho$.
			\item There exists a stable \infcat $\mathcal{D}$ equipped with a $t$-structure $\left(\mathcal{D}_{\geq 0}, \mathcal{D}_{\leq 0}\right)$ and an equivalence of \infcats $\mathcal{C} \simeq \mathcal{D}_{ \geq 0}$.
		}
		where $\op{SW}(-)$ denotes the Spanier-Whitehead construction.
	\end{prop}
	\begin{rem}[{\cite[Remark C.1.2.10]{sag}}]
		Assume that $\mathcal{C}$ is prestable and admits finite limits. Consider the \infcat $\operatorname{Sp}(\mathcal{C})$ of spectrum objects of $\mathcal{C}$, defined as the homotopy limit of the tower of \infcats
		$$
		\cdots \rightarrow \mathcal{C} \xrightarrow{\Omega} \mathcal{C} \xrightarrow{\Omega} \mathcal{C}
		$$
		which we can identify with
		$$
		\cdots \rightarrow \mathrm{SW}(\mathcal{C})_{\geq-2} \xrightarrow{\tau_{\geq-1}} \mathrm{SW}(\mathcal{C})_{\geq-1} \xrightarrow{\tau_{\geq 0}} \mathrm{SW}(\mathcal{C})_{\geq 0} .
		$$
		From this description, we can identify $\operatorname{Sp}(\mathcal{C})$ with the right completion of the \infcat $\mathrm{SW}(\mathcal{C})$ with respect to its $t$-structure; in particular, the $t$-structure on the \infcat $\operatorname{Sp}(\mathcal{C})$ is right complete.
	\end{rem}
	\begin{notation}
		We use $\prlpst\subset \prl$ to denote the full subcategory of prestable presentable \infcats.
	\end{notation}

\begin{prop}\label{retractpst}
	The full subcategory $\prlpst\subset\prl$ is closed under retracts.
\end{prop}
	\begin{proof}

		Indeed, suppose that there are colimit-preserving functors
		\[
		\mathcal C
		\xrightarrow{i}
		\mathcal D
		\xrightarrow{r}
		\mathcal C
		\]
		such that \(r\circ i\simeq \id_{\mathcal C}\), and assume that
		\(\mathcal D\) is prestable. 
		
		First, \(\mathcal C\) is pointed, because $\prl_*\subset \prl$ is closed under small limits.
		Moreover, since \(i\) and \(r\) preserve finite colimits, they commute with
		suspension. Thus, for every \(X,Y\in\mathcal C\), the canonical map
		\[
		\Map_{\mathcal C}(X,Y)
		\longrightarrow
		\Map_{\mathcal C}(\Sigma X,\Sigma Y)
		\]
		is a retract of the corresponding map in \(\mathcal D\). Since
		\(\Sigma_{\mathcal D}\) is fully faithful, so is
		\(\Sigma_{\mathcal C}\). It follows from
		\cite[Proposition~C.1.2.2]{sag} that the canonical functor
		\[
		j_{\mathcal C}\colon
		\mathcal C\longrightarrow \op{SW}(\mathcal C)
		\]
		is fully faithful.
		
		It remains to verify that its essential image is closed under extensions.
		The functors \(i\) and \(r\) induce exact functors
		\[
		\op{SW}(i)\colon\op{SW}(\mathcal C)\longrightarrow\op{SW}(\mathcal D),
		\qquad
		\op{SW}(r)\colon\op{SW}(\mathcal D)\longrightarrow\op{SW}(\mathcal C),
		\]
		satisfying
		$
		\op{SW}(r)\circ\op{SW}(i)\simeq\id_{\op{SW}(\mathcal C)}.
		$
		They are compatible with the canonical embeddings in the sense that
		\[
		\op{SW}(i)\circ j_{\mathcal C}\simeq j_{\mathcal D}\circ i,
		\qquad
		\op{SW}(r)\circ j_{\mathcal D}\simeq j_{\mathcal C}\circ r.
		\]
		Now let
		\[
		j_{\mathcal C}(X)\longrightarrow E\longrightarrow j_{\mathcal C}(Z)
		\]
		be a cofiber sequence in \(\op{SW}(\mathcal C)\), where \(X,Z\in\mathcal C\).
		Applying \(\op{SW}(i)\), we obtain a cofiber sequence
		\[
		j_{\mathcal D}(iX)
		\longrightarrow
		\op{SW}(i)(E)
		\longrightarrow
		j_{\mathcal D}(iZ)
		\]
		in \(\op{SW}(\mathcal D)\). Since \(\mathcal D\) is prestable,
		\cite[Proposition~C.1.2.2]{sag} implies that the essential image of
		\(j_{\mathcal D}\) is closed under extensions. Consequently, there exists
		an object \(W\in\mathcal D\) such that
		\[
		\op{SW}(i)(E)\simeq j_{\mathcal D}(W).
		\]
		Applying \(\op{SW}(r)\), we conclude that
		\[
		E
		\simeq
		\op{SW}(r)\op{SW}(i)(E)
		\simeq
		\op{SW}(r)j_{\mathcal D}(W)
		\simeq
		j_{\mathcal C}(rW).
		\]
		Thus the essential image of
		\(j_{\mathcal C}\colon\mathcal C\to\op{SW}(\mathcal C)\) is closed under
		extensions. The
		\cite[Proposition~C.1.2.2]{sag} shows that \(\mathcal C\) is prestable.
	\end{proof}
	\begin{prop}
	Let $\mcd\in\prlpst$ be a presentable prestable \infcat. Then for any presentable \infcat \mcc, $\funct^L(\mcc,\mcd)$ is a presentable prestable \infcat.
\end{prop}
\begin{proof} 
	It is standard that $\funct^L(\mcc,\mcd)$ is presentable. We prove that it is prestable.  Let \[ \rho\colon \mcd\longrightarrow \operatorname{Sp}(\mcd) \] denote the canonical stabilization functor. By \cite[Corollary C.1.2.3]{sag}, it suffices to observe that the fully faithful embedding \[ \rho_*\colon \funct^L(\mcc,\mcd) \longrightarrow \funct^L\bigl(\mcc,\operatorname{Sp}(\mcd)\bigr). \tag{1} \] is closed under small colimits.
	\end{proof}
	\begin{prop}[See \cite{sag} C.1.4.1]
		Let $\mathcal{C}$ be a presentable \infcat. The following conditions are equivalent:
		\begin{enumerate}[label=(\alph*),font=\normalfont]
			\item The \infcat $\mathcal{C}$ is prestable and filtered colimits in $\mathcal{C}$ are left exact (AB5).
			\item The \infcat $\mathcal{C}$ is prestable and the functor $\Omega: \mathcal{C} \rightarrow \mathcal{C}$ commutes with filtered colimits.
			\item The \infcat $\mathcal{C}$ is prestable and the functor $\Omega^{\infty}: \operatorname{Sp}(\mathcal{C}) \rightarrow \mathcal{C}$ commutes with filtered colimits.
			\item There exists a presentable stable \infcat $\mathcal{D}$, a $t$-structure $\left(\mathcal{D}_{\geq 0}, \mathcal{D}_{\leq 0}\right)$ on $\mathcal{D}$ which is compatible with filtered colimits, and an equivalence $\mathcal{C} \simeq \mathcal{D}_{\geq 0}$.
			\item The suspension functor $\Sigma_+^\infty:\mathcal{C}\to\operatorname{Sp}(\mathcal{C})$ is fully faithful and its essential image $\operatorname{Sp}(\mathcal{C})_{\geq 0}$ is the connective part of a $t$-structure on $\operatorname{Sp}(\mathcal{C})$ which is compatible with filtered colimits.
		\end{enumerate}
	\end{prop}
	\begin{de}\label{def:AB}
		
		Let $\mcc$ be a presentable prestable \infcat. 
		\enu{
			\item We say $\mcc$ is Grothendieck prestable if it satisfies AB5: filtered colimits in $\mathcal{C}$ are left exact.
			 We let $\op{Groth}\subset \prl$ denote the full subcategory whose objects are Grothendieck prestable \infcats.
			\item We say $\mcc$ satisfies $\mathrm{AB4}^*$ if the inclusion $\mathcal{C}\hookrightarrow \opsp(\mcc)$ is closed under small products.
			\item We say $\mcc$ satisfies $\mathrm{AB6}$ if for any collection of filtered \infcats $J_i, i \in I$, and for any functors $J_i \rightarrow \mathcal{C}, j_i \mapsto x_{j_i}$, the map
			$$
			\varinjlim_{\left(j_i\right)_i \in \prod_i J_i} \prod_{i} x_{j_i} \rightarrow \prod_{i \in I} \varinjlim_{j_i \in J_i} x_{j_i}
			$$
			is an isomorphism.
			
		}
		
	\end{de}
\begin{rem}\label{ab4staradjointable}
	A presentable prestable \infcat $\mcc$ satisfies $\mathrm{AB4}^*$ if and only if the suspension functor $\mathcal{C} \xrightarrow{\Sigma} \mcc$ preserves small products. Equivalently, the commutative square of right adjoints
	$$\begin{tikzcd}
		\prod_i \mcc \arrow[r, "\prod"] & \mcc \\
		\prod_i \mcc \arrow[u, "\Omega"] \arrow[r, "\prod"] & \mcc \arrow[u, "\Omega"]
	\end{tikzcd}$$
	is vertically left adjointable.
	
	Indeed, the horizontal left adjointability of the square
	$$\begin{tikzcd}
		\mcc \arrow[d, Rightarrow, no head] \arrow[r, "\Sigma^{n+1}", dashed, shift left] & \mcc \arrow[l, "\Omega^{n+1}", shift left] \arrow[d, "\Omega"] \\
		\mcc \arrow[r, "\Sigma^{n}", dashed, shift left] & \mcc \arrow[l, "\Omega^n", shift left]
	\end{tikzcd}$$
	implies, upon passing to limits, that the stabilization functor $\Sigma^\infty \colon \mcc \to \opsp(\mcc)$ can be identified with $\lim_n \Sigma^n$.
\end{rem}
	\begin{rem}\label{ab6adjointable}
		In fact, the $\mathrm{AB6}$ condition makes sense for any \infcat admitting small filtered colimits and products, which is equivalent to the vertically left adjointability of the following diagram:
		$$\begin{tikzcd}
			{\prod_i \funct(J_i,\mcc)} \arrow[r, "\prod"]    & {\funct(\prod_i J_i,\mcc)} \\
			\prod_i \mcc \arrow[u, "\delta"] \arrow[r, "\prod"] & \mcc \arrow[u, "\delta"]
		\end{tikzcd} .$$

	\end{rem}
	\begin{lem}\label{bccondconserv}
		Suppose we are given a commutative diagram of \infcats as follows.
		$$\begin{tikzcd}[row sep=scriptsize, column sep=scriptsize]
			& A_1 \arrow[dl, "u_A"'] \arrow[rr, "h_1"] \arrow[dd, "v_{A_1}"' near end] & & B_1 \arrow[dl, "u_B"'] \arrow[dd, "v_{B_1}"] \\
			A_2 \arrow[rr, crossing over, "h_2" near start] \arrow[dd, "v_{A_2}"'] & & B_2 \\
			& C_1 \arrow[dl, "u_C"'] \arrow[rr, "k_1" near start] & & D_1 \arrow[dl, "u_D"] \\
			C_2 \arrow[rr, "k_2"'] & & D_2 \arrow[from=uu, crossing over, "v_{B_2}" near start]\\
		\end{tikzcd}$$
		If $u_B: B_1\to B_2$ is conservative and the left $(A_1A_2C_1C_2)$, right $(B_1B_2D_1D_2)$ and front $(A_2B_2C_2D_2)$ faces are vertically left adjointable, then the back face $(A_1B_1C_1D_1)$ is vertically left adjointable too.
		
	\end{lem}
	\begin{proof}
		The proof is a basic diagram chase. Let $v_{A_1}: A_1 \to C_1$ denote the vertical functor, and let $L_{A_1}: C_1 \to A_1$ be its left adjoint. We use similar notation ($v_{B_1}, L_{B_1}$, etc.) for the other vertical edges. Let $h_1: A_1 \to B_1$ and $k_1: C_1 \to D_1$ denote the horizontal functors on the back face. Let $u_A: A_1 \to A_2$, $u_B: B_1 \to B_2$, etc., denote the depth functors.

		To show that the back face is vertically left adjointable, we must prove that the canonical Beck-Chevalley transformation
		$$ \alpha: L_{B_1} \circ k_1 \to h_1 \circ L_{A_1} $$
		is an equivalence. Since the functor $u_B: B_1 \to B_2$ is conservative, it suffices to show that the post-composed natural transformation $u_B(\alpha)$ is an equivalence.
		
		We can rewrite $u_B \circ L_{B_1} \circ k_1$ through a sequence of natural equivalences obtained by pasting the Beck-Chevalley equivalences of the other faces and the commutativity of the top and bottom faces:
		\begin{align*}
			u_B \circ L_{B_1} \circ k_1 &\simeq L_{B_2} \circ u_D \circ k_1 \quad (\text{by vertically left adjointability of the right face}) \\
			&\simeq L_{B_2} \circ k_2 \circ u_C \quad (\text{by commutativity of the bottom face}) \\
			&\simeq h_2 \circ L_{A_2} \circ u_C \quad (\text{by vertically left adjointability of the front face}) \\
			&\simeq h_2 \circ u_A \circ L_{A_1} \quad (\text{by vertically left adjointability of the left face}) \\
			&\simeq u_B \circ h_1 \circ L_{A_1} \quad (\text{by commutativity of the top face}).
		\end{align*}
		By the 2-categorical coherence of pasting Beck-Chevalley squares, this composite of equivalences canonically identifies with $u_B(\alpha)$. Therefore, $u_B(\alpha)$ is an equivalence. Since $u_B$ is conservative, $\alpha$ is an equivalence, which means the back face is vertically left adjointable.
	\end{proof}
	\begin{prop}\label{conserab6}
		Let $\mcc,\mcd$ be \infcats which admit small filtered colimits and products and $\mcc\to\mcd$ be a conservative functor which preserves small filtered colimits and products. If $\mcd$ satisfies $\mathrm{AB6}$, then so does $\mcc$.
	\end{prop}
	\begin{proof}
		We construct a commutative cube by bridging the $\mathrm{AB6}$ diagrams of $\mcc$ and $\mcd$ via the functor $F$. Let $F_*$ denote the post-composition functor induced by $F$. Consider the following diagram:
		
		$$\begin{tikzcd}[row sep=1.5em, column sep=1em]
			& \prod_i \mcc \arrow[dl, "\prod_i F"'] \arrow[rr, "\prod"] \arrow[dd, "\delta_\mcc"' near end] & & \mcc \arrow[dl, "F"'] \arrow[dd, "\delta_\mcc"] \\
			\prod_i \mcd \arrow[rr, crossing over, "\prod" near start] \arrow[dd, "\delta_\mcd"'] & & \mcd \\
			& \prod_i \funct(J_i, \mcc) \arrow[dl, "\prod_i F_*"'] \arrow[rr, "\prod" near start] & & \funct(\prod_i J_i, \mcc) \arrow[dl, "F_*"] \\
			\prod_i \funct(J_i, \mcd) \arrow[rr, "\prod"'] & & \funct(\prod_i J_i, \mcd) \arrow[from=uu, crossing over, "\delta_\mcd" near start]\\
		\end{tikzcd}$$
		The top and bottom faces commute because limits (in particular, products) in functor categories are computed pointwise, and $F$ explicitly preserves small products. Thus, the diagram is a well-defined commutative cube.  By \cref{ab6adjointable}, it suffices to show the back square is vertically left adjointable.
		Since all conditions of \cref{bccondconserv} are satisfied, we conclude that the back face is vertically left adjointable, as desired. 
	\end{proof}
	\begin{prop}
		The \infcat $\mcs$ of anima satisfies $\mathrm{AB6}$.
	\end{prop}
	\begin{proof}
	It is not hard to verify that the (1-)category $\mcs\mathrm{et}$ of sets satisfies $\mathrm{AB6}$. However, filtered colimits and products in \mcs can be calculated at the level of simplicial sets, so $\mcs$ satisfies $\mathrm{AB6}$.
	\end{proof}
	\begin{cor}\label{compab6}
		Let $\mcc$ be an $\omega$-presentable \infcat. Then it satisfies $\mathrm{AB6}$.
	\end{cor}
	\begin{proof}
		By \cref{conserab6}, it suffices to observe that $$\{\mcc\simeq\funct^{\op{lex}}(\mcc^{\omega,\opp},\mcs)\xrightarrow{ev_x} \mcs | x\in\mcc^\omega\}$$ is jointly conservative and that each functor preserves small filtered colimits and products.
	\end{proof}

	\begin{thm}[See \cite{sag} C.4.2.1]
		The full subcategory $\op{Groth}\subset \prlad$ contains the unit object $\op{Sp}_{\geq0}$ and is closed under Lurie tensor products. Consequently, $\op{Groth}$ inherits a symmetric monoidal structure for which the inclusion $\op{Groth} \hookrightarrow \prlad$ is symmetric monoidal.
	\end{thm}
	
	\begin{de}[See \cite{sag} C.3.1.3]
		Let $\mathcal{C}$ be a presentable stable \infcat. We define a full subcategory $\mathcal{C}_{\geq 0} \subset \mathcal{C}$ to be a \emph{core} if it is closed under small colimits and extensions. 
		
		We will refer to $\pr^{+}_{\op{st}}$ as the \infcat of cored stable \infcats. The objects of $\pr^{+}_{\op{st}}$ are pairs $(\mathcal{C}, \mathcal{C}_{ \geq 0})$, where $\mathcal{C}$ is a presentable stable \infcat and $\mathcal{C}_{\geq 0} \subset \mathcal{C}$ is a core. A morphism from $(\mathcal{C}, \mathcal{C}_{ \geq 0})$ to $(\mathcal{D}, \mathcal{D}_{ \geq 0})$ is given by a colimit-preserving functor $f: \mathcal{C} \rightarrow \mathcal{D}$ satisfying $f\left(\mathcal{C}_{\geq 0}\right) \subset \mathcal{D}_{\geq 0}$.
	\end{de}
	
	\begin{rem}\,
		\enu{	\item Our $\pr^{+}_{\op{st}}$ actually refers to $\op{Groth}^{+}$ in \textup{\cite[Remark C.3.1.3]{sag}}.
			\item \textbf{Warning:} The $\mathcal{C}_{\geq 0}$ in a cored stable \infcat is not necessarily the connective part 
			of an \textit{accessible} $t$-structure unless it is presentable. 
			
			\item If $\mathcal{C}_{\geq 0}$ is presentable, then we call the pair $(\mathcal{C}, \mathcal{C}_{ \geq 0})$ 
			a presentable $t$-category. In fact, we only care about the full subcategory 
			$\pr^{t\text{-rex}}_{\op{st}}\subset\pr^{+}_{\op{st}}$ spanned by those presentable $t$-categories (with right $t$-exact 
			colimit-preserving functors). However, the technical advantage of $\pr^{+}_{\op{st}}$ is that colimits and  limits in 
			it can be calculated easily \textup{\cite[Remark C.3.1.7]{sag}}.
			
		}
	\end{rem}
	\begin{prop}\,\label{colimt}
		\enu{\item The forgetful functor $\pr^{+}_{\op{st}}\to \prlst$ is a bicartesian fibration and the inclusion $\prtrex\subset \pr^{+}_{\op{st}}$ is closed under cocartesian and cartesian liftings and hence induces a sub bicartesian fibration.
			\item $\pr^{+}_{\op{st}}$ admits small colimits and limits and $\pr^{+}_{\op{st}}\to \prlst$ preserves small colimits and limits.
			\item $\prtrex\subset \pr^{+}_{\op{st}}$ is closed under small colimits and limits. In particular, $\prtrex$ admits small colimits and limits, which are preserved by the forgetful functor  $\prtrex\to \prlst$.
		}
	\end{prop}
	\begin{proof}
		(1) Let $U: \pr^{+}_{\op{st}} \to \prlst$ denote the forgetful functor. Given a colimit-preserving functor $f: \mathcal{C} \to \mathcal{D}$ in $\prlst$ and a core $\mathcal{D}_{\geq 0} \subset \mathcal{D}$, the cartesian lift of $\mathcal{D}_{\geq 0}$ along $f$ is given by the preimage $f^{-1}(\mathcal{D}_{\geq 0})$.  Conversely, given a core $\mathcal{C}_{\geq 0} \subset \mathcal{C}$, the cocartesian lift along $f$ is the core in $\mathcal{D}$ generated by $f(\mathcal{C}_{\geq 0})$, which is simply the closure of $f(\mathcal{C}_{\geq 0})$ under small colimits and extensions in $\mathcal{D}$. Thus, $U$ is a bicartesian fibration.
		
		Now we show $\prtrex$ is closed under these liftings. If $\mathcal{D}_{\geq 0}$ 
		is presentable, it is the connective part of an accessible $t$-structure on 
		$\mathcal{D}$. In this case $f^{-1}(\mathcal{D}_{\geq 0})$ is again an accessible 
		(presentable) full subcategory of \mcc and hence forms an accessible $t$-structure 
		on $\mathcal{C}$. Conversely, given an accessible $t$-structure 
		$\mathcal{C}_{\geq 0}$. The core generated by $f(\mathcal{C}_{\geq 0})$ in the 
		presentable \infcat $\mathcal{D}$ is generated by a small set of objects, so its 
		closure under small colimits and extensions remains accessible and therefore 
		presentable. 
		
		(2) It is a standard fact that $\prlst$ admits small limits and colimits. Since $U$ is a bicartesian fibration, we can compute limits and colimits in the total \infcat $\pr^{+}_{\op{st}}$ if the fibers admit them. For any $\mathcal{C} \in \prlst$, the fiber $U^{-1}(\mathcal{C})$ is a poset of cores in $\mathcal{C}$. This poset admits small limits (given by arbitrary intersections of cores) and small colimits (given by the core generated by the union). By the general theory of fibrations (e.g., \cite[Proposition 4.3.1.10]{htt}), $\pr^{+}_{\op{st}}$ admits small limits and colimits, and the forgetful functor $U$  preserves them.
		
		(3) By (2), limits and colimits in $\pr^{+}_{\op{st}}$ are computed by first computing the corresponding limit/colimit of the underlying stable \infcats in $\prlst$, and then taking the appropriate cartesian/cocartesian liftings combined with fiberwise intersections/generations. Since we have established in (1) that $\prtrex$ is closed under both cartesian and cocartesian liftings, and the intersection (resp. colimit closure) of accessible cores remains accessible (hence presentable), $\prtrex \subset \pr^{+}_{\op{st}}$ is closed under small limits and colimits.
	\end{proof}
	\begin{de}
		There is a natural symmetric monoidal structure on $\pr^{+}_{\op{st}}$ given by the construction:
		$$
		\left(\mathcal{C}, \mathcal{C}_{\geq 0}\right) \otimes\left(\mathcal{D}, \mathcal{D}_{\geq 0}\right)=\left(\mathcal{C} \otimes \mathcal{D}, m_{!}\left(\mathcal{C}_{\geq 0}, \mathcal{D}_{\geq 0}\right)\right)
		$$
		where $\mathcal{C} \otimes \mathcal{D}$ is the Lurie tensor product and $m_{!}\left(\mathcal{C}_{\geq 0}, \mathcal{D}_{\geq 0}\right)$ is the smallest full subcategory of $\mathcal{C} \otimes \mathcal{D}$ which is closed under colimits and extensions and contains the objects $m(C, D)$ for each $C \in \mathcal{C}_{\geq 0}$ and $D \in \mathcal{D}_{\geq 0}$.
	\end{de}
	
	\begin{rem}\label{prt-rex}\,
		\enu{
			\item The full subcategory $\pr^{t\text{-rex}}_{\op{st}}\subset\pr^{+}_{\op{st}}$ of presentable $t$-categories is closed under tensor products since $m_{!}\left(\mathcal{C}_{\geq 0}, \mathcal{D}_{\geq 0}\right)$ is presentable if both $\mathcal{C}_{\geq 0}$ and $\mathcal{D}_{\geq 0}$ are presentable.
			\item An object in $\op{CAlg}(\pr^{t\text{-rex}}_{\op{st}})$ can be identified with a $ttt$-\infcat (cf.  \cite{hattt}).
		}
	\end{rem}
	
	\begin{prop}[See {\cite[Proposition C.3.1.1]{sag}}]\label{prestemb}
		Let $\mathcal{C}$ and $\mathcal{D}$ be presentable prestable \infcats. Then the canonical map
		$$
		\opsp\otimes-: \operatorname{Fun}^L(\mathcal{C}, \mathcal{D}) \rightarrow \operatorname{Fun}^L(\operatorname{Sp}(\mathcal{C}), \operatorname{Sp}(\mathcal{D}))
		$$
		is a fully faithful embedding, whose essential image consists of those functors $\operatorname{Sp}(\mathcal{C}) \rightarrow \operatorname{Sp}(\mathcal{D})$ which preserve small colimits and are right $t$-exact (with respect to the canonical $t$-structure).
	\end{prop}
	\begin{rem}
		The original \cite[Proposition C.3.1.1]{sag} requires the Grothendieck condition, but that is actually not used in the proof.
	\end{rem}
	\begin{prop}[See {\cite[Proposition C.3.2.1]{sag}}]\label{t-ex}
		Let $\mathcal{C}$ and $\mathcal{D}$ be Grothendieck prestable \infcats and let $f: \mathcal{C} \rightarrow \mathcal{D}$ be a colimit-preserving functor. Then the following conditions are equivalent:
		\enu{
			\item The functor $f$ is left exact.
			\item The functor $f$ carries $0$-truncated objects of $\mathcal{C}$ to $0$-truncated objects of $\mathcal{D}$.
			\item The induced map $	\opsp\otimes f: \operatorname{Sp}(\mathcal{C}) \rightarrow \operatorname{Sp}(\mathcal{D})$ is left $t$-exact.
		}
	\end{prop}
	
	\begin{cor}\label{pret}\,
		\enu{
			\item The construction $\mathcal{C} \mapsto\left(\operatorname{Sp}(\mathcal{C}),
			 \operatorname{Sp}(\mathcal{C})_{\geq 0}\right)$ determines a fully faithful embedding
			 $$  \prlpst\xhookrightarrow{} \pr^{+}_{\op{st}}$$ from the \infcat of
			   presentable prestable \infcats to the \infcat of cored stable 
			   \infcats.
			\item A pair $(\mathcal{C}, \mathcal{C}_{ \geq 0})$ belongs to the essential image of $\prlpst\hookrightarrow \pr^{+}_{\op{st}}$ if and only if it forms an accessible $t$-structure $\left(\mathcal{C}_{\geq 0}, \mathcal{C}_{\leq 0}\right)$ which is right complete. 
			\item A pair $(\mathcal{C}, \mathcal{C}_{ \geq 0})$ belongs to the essential image of $\op{Groth}\hookrightarrow \pr^{+}_{\op{st}}$ if and only if it forms an accessible $t$-structure $\left(\mathcal{C}_{\geq 0}, \mathcal{C}_{\leq 0}\right)$ which is compatible with filtered colimits and is right complete.
			\item Furthermore, the embedding $\op{Groth}\hookrightarrow \pr^{+}_{\op{st}}$ is symmetric monoidal, hence induces a fully faithful embedding $$\op{CAlg}(\op{Groth})\hookrightarrow \op{CAlg}(\pr^{+}_{\op{st}}).$$
		}
	\end{cor}
	
	\begin{de}[Grothendieck $t$-categories]\label{gro}
		We say a presentable $t$-category $(\mathcal{C},\mathcal{C}_{\geq0})\in \pr^{t\text{-rex}}_{\op{st}}$ is Grothendieck if it lies in the essential image of the embedding $\op{Groth}\hookrightarrow \pr^{t\text{-rex}}_{\op{st}}$.
	\end{de}
	
	\begin{exam}
		If $(\mathcal{A},\mathcal{A}_{\geq0})$ is a  monoidal Grothendieck $t$-category, then for any $R\in \alg(\mathcal{A}_{\geq0})$, the pair $(\modu_R(\mathcal{A}),\modu_R(\mathcal{A})_{\geq0})$ is also a Grothendieck $t$-category (see \cite[Proposition 2.7]{hattt}).
	\end{exam}
	
	\begin{de}[See {\cite[Definition C.1.2.12]{sag}}]
		Let $\mathcal{C}$ be a prestable \infcat which admits finite limits. We say that an object $X \in \mathcal{C}$ is $\infty$-connective if $\tau_{\leq n} X \simeq 0$ for every integer $n$. \\
		We say $\mathcal{C}$ is separated if every $\infty$-connective object of $\mathcal{C}$ is a zero object. \\
		We say $\mathcal{C}$ is complete if it is a homotopy limit of the tower of \infcats:
		$$
		\cdots \rightarrow \tau_{\leq 2} \mathcal{C} \xrightarrow{\tau_{\leq 1}} \tau_{\leq 1} \mathcal{C} \xrightarrow{\tau_{\leq 0}} \tau_{\leq 0} \mathcal{C}=\mathcal{C}^{\heartsuit} .
		$$
		In other words, $\mathcal{C}$ is complete if it is Postnikov complete.
	\end{de}
	
	\begin{rem}\,
		\enu{
			\item If a prestable \infcat $\mathcal{C}$ is complete, then it is separated.
			\item Let $\mathcal{C}$ be a stable \infcat. Then $\mathcal{C}$ is separated if and only if $\mathcal{C} \simeq *$.
		}
	\end{rem}
	
	\begin{prop}
		A product of separated (complete) prestable \infcats is still separated (complete).
	\end{prop}
	\begin{proof}
		It immediately follows from the definition.
	\end{proof}
	\begin{prop}\label{comple}
		Let $\mathcal{C}$ be a prestable \infcat with finite limits. Then:
		\enu{
			\item The canonical $t$-structure $(\op{Sp}(\mathcal{C})_{\geq0},\op{Sp}(\mathcal{C})_{\leq0})$ is hypercomplete if and only if $\mathcal{C}$ is separated.
			\item The canonical $t$-structure $(\op{Sp}(\mathcal{C})_{\geq0},\op{Sp}(\mathcal{C})_{\leq0})$ is left complete if and only if $\mathcal{C}$ is complete.
		}
	\end{prop}
	
	\begin{proof}
		See \cite[Proposition 1.19]{hattt}.
	\end{proof}
	\begin{cor}\label{sepab4*iscomp}
		Let $\mcc$ be a separated presentable prestable \infcat satisfying $\mathrm{AB4}^*$. Then $\mcc$ is complete.
	\end{cor}
	\begin{proof}
		It follows directly by combining \cref{comple} and \cite[Proposition 1.2.1.19]{ha}.
	\end{proof}

	\subsection{Prestable short exact sequences}
	Unlike the stable setting, short exact sequences of prestable \infcats present subtle asymmetries. In this subsection, we analyze how short exact sequences behave when subjected to prestable quotients.
	\begin{lem}Let $F:\mcc\to \mcd$ be a right exact functor between prestable \infcats. Then $F(f)$ is an equivalence if and only if $F(\cofib(f))=0$.
	\end{lem}
	\begin{proof}
		This follows from the fact that a map in a prestable \infcat is an equivalence if and only if its cofiber is zero.
	\end{proof}
	\begin{prop}
		The inclusion $\prlpst\to \prl$ creates fiber sequences.
	\end{prop}
	\begin{proof}
		Let $F:\mcc\to\mcd \in \prlpst$.  Since the inclusion $\ker(F)\hookrightarrow\mcc$ is closed under small colimits and extensions, $\ker(F)$ is prestable too by \cref{prestcriterion}.
	\end{proof}
	
	However, $\prlpst\to \prl$ does not create cofiber sequences in general. That is what we will discuss in the following.
	
	Lurie's \cite[Definition C.2.3.3]{sag} only considers localizing subcategories of a Grothendieck prestable \infcat. But it actually makes sense for any presentable prestable \infcat, as follows.
	\begin{de}
		Let $\mathcal{C}$ be a  presentable prestable \infcat. We will say that a full subcategory $\mathcal{C}_0 \subset \mathcal{C}$ is localizing if it satisfies the following conditions:
		\enu{
			
			\item  The \infcat $\mathcal{C}_0$ is accessible and closed under small colimits in $\mathcal{C}$.
			\item Given a cofiber sequence $C^{\prime} \rightarrow C \rightarrow C^{\prime \prime}$ in $\mathcal{C}$, if any two of the objects $C, C^{\prime}, C^{\prime \prime}$ belong to $\mathcal{C}_0$, then so does the third.
			\item Given a cofiber sequence $C^{\prime} \rightarrow C \rightarrow C^{\prime \prime}$ in $\mathcal{C}$ where $C \in \mathcal{C}_0$ and $C^{\prime \prime} \in \mathcal{C}^{\heartsuit}$, we have $C^{\prime} \in \mathcal{C}_0$.
		}
	\end{de}
	\begin{rem}\,
		\enu{

			\item Any localizing subcategory is prestable too.
			\item The condition (3) in the definition means the full subcategory $\mcc_0$ is closed under subobjects.
		}
	\end{rem}
	\begin{prop}\label{quotientlocalizing}
		Let $\mathcal{C}$ be a presentable prestable \infcat and let $\mathcal{C}_0 \subset \mathcal{C}$ be a full subcategory. The following conditions are equivalent:
		\enu{
			\item The full subcategory $\mathcal{C}_0 \subset \mathcal{C}$ is localizing.
			\item There exists an accessible left exact presentable localization $L: \mathcal{C} \rightarrow \mathcal{D}$ such that $\mathcal{C}_0$ is the full subcategory of $\mathcal{C}$ spanned by the $L$-acyclic objects. (Note \mcd is automatically prestable by \cite[Proposition C.2.3.1]{sag}).
		}
		
	\end{prop}
	\begin{proof}
		This follows from the same argument as the proof of \cite[Proposition C.2.3.8]{sag}; the Grothendieck assumption is not used and can be removed.
	\end{proof}
	We also obtain the following non-obvious statement.
	\begin{cor}
		For any localizing subcategory $\mcc_0$ of a presentable prestable \infcat \mcc, the inclusion $\mcc_0\to\mcc$ is left exact.
	\end{cor}
	\begin{prop}\label{limlex}
		Let $\pr_{\op{pst}}^{\mathrm{L},\op{lex}}\hookrightarrow\prlpst$ denote the (non-full) subcategory of presentable prestable \infcats with left exact left adjoints as morphisms. Then $\pr_{\op{pst}}^{\mathrm{L},\op{lex}}$ admits small limits, which are preserved by the inclusion $\pr_{\op{pst}}^{\mathrm{L},\op{lex}}\hookrightarrow\prl$.
	\end{prop}
	\begin{proof}
		Let $\left\{\mathcal{C}_\alpha\right\}$ be a small diagram of prestable \infcats where the transition maps preserve small colimits and finite limits, and let $\mathcal{C}$ denote a limit of $\left\{\mathcal{C}_\alpha\right\}$ in \prl. It follows that for any presentable \infcat $\mathcal{D}$, a functor $f: \mathcal{D} \rightarrow \mathcal{C}$ preserves small colimits and finite limits if and only if each of the maps $\mathcal{D} \rightarrow \mathcal{C} \rightarrow \mathcal{C}_\alpha$ preserves small colimits and finite limits. 
		
		Now it suffices to show $\mcc$ is prestable. It suffices to observe that \cref{defpst} only involves finite colimits and limits. However, $\{\mcc\to\mcc_\alpha\}$ is jointly conservative and preserves all finite colimits and finite limits, so $\mcc$ is prestable too.
	\end{proof}
	\begin{prop}
		The stabilization functor $\pr_{\op{pst}}^{\mathrm{L},\op{lex}}\xrightarrow{\opsp(-)}\prlst$ preserves small limits.
	\end{prop}
	\begin{proof}
		One can identify $\opsp(-)$ with $\funct_*^{\op{exc}}(\mcs_*^{\op{fin}}, -)$, see \cite[Definition 1.4.2.8]{ha}. But $\funct_*^{\op{exc}}(\mcs_*^{\op{fin}}, -)$ obviously preserves small limits in $\pr_{\op{pst}}^{\mathrm{L},\op{lex}}$.
	\end{proof}
	\begin{prop}\,
		\enu{
			\item Let $i:\mcc\to \mcd \in \prl_{\op{pst}}$ be a fully faithful functor between presentable prestable \infcats such that \mcc is a localizing subcategory. Then $\mcd[W^{-1}]$ is presentable prestable, where $W=\{f|\cofib(f)\in \mcc\}$.
			Furthermore, $$\mcc\to\mcd\to\mcd[W^{-1}]$$ is a bifiber sequence in $\prl_{\op{pst}}$.
		
			\item 	Let $L:\mcd\to\mce \in \prl_{\op{pst}}$ be a localization functor between presentable prestable \infcats. Then the kernel $\ker(L)$ in $\prl$ is prestable too and $\mce\simeq \mcd[W^{-1}]$, where $W=\{f|\cofib(f)\in \ker(L)\}$. Therefore $$\ker(L)\to\mcd\to\mce$$ is a bifiber sequence in $\prl_{\op{pst}}$.
			
		}
	\end{prop}
	\begin{proof}
		(1) This follows from \cref{quotientlocalizing}.	
		\\
		(2) Since the inclusion $\ker(L)\to\mcd$ is closed under small colimits and extensions, $\mcc=\ker(L)$ is prestable too by \cref{prestcriterion}. Applying the fact $L(f)$ is an equivalence if and only if $\cofib(f)\in \mcc$ again, we are done.
	\end{proof}
	\begin{rem}
	Let $i:\mcc\hookrightarrow \mcd \in \prl_{\op{pst}}$ be a fully faithful functor between presentable prestable \infcats such that \mcc is a localizing subcategory.	We have a natural comparison $$\mcd/\mcc\to \mcd[W^{-1}],$$ where $\mcd/\mcc$ denotes the cofiber in $\prl$. But beware that this is not necessarily an equivalence. 
		
			For example, consider the sequence $$\mcd(\mathbb{Z})_{\geq 0}^{p\text{-}\op{nil}}\xhookrightarrow{i} \mcd(\mathbb{Z})_{\geq 0} \xrightarrow{L} \mcd(\mathbb{Z}[1/p])_{\geq 0}.$$
			We have $$\mcd(\mathbb{Z}[1/p])_{\geq 0}\simeq\mcd(\mathbb{Z})_{\geq 0}[W^{-1}]\neq \mcd(\mathbb{Z})_{\geq 0}/ \mcd(\mathbb{Z})_{\geq 0}^{p\text{-}\op{nil}}.$$
			Indeed, one can see that $i^R(\mathbb{Z}[1/p]/\mathbb{Z})=0$, but $\mathbb{Z}[1/p]/\mathbb{Z}$ does not lie in $\mcd(\mathbb{Z}[1/p])_{\geq 0}.$
	\end{rem}

	However, in the following case, we do have $\mcd/\mcc\simeq \mcd[W^{-1}].$
	\begin{prop}
			 Let $i:\mcc\to \mcd \in \prl_{\op{pst}}$ be a fully faithful internal left adjoint functor between presentable prestable \infcats. 
		Then the cofiber $\mce:=\coker(i)$ in \prl   is prestable too and $\mce\simeq \mcd[W^{-1}]$. Furthermore, we have $\mcc\simeq \ker(L)$, where $L:\mcd\to \mce$ denotes the projection and $W=\{f|\cofib(f)\in \mcc\}$. Hence $$\mcc\to\mcd\to\mce$$ is a bifiber sequence in $\prl_{\op{pst}}$.
	\end{prop}
	\begin{proof}
		Since \mce can be identified with the fiber of $i^R$, it follows from \cref{limlex} that \mce is prestable. The inclusion $\mcc\subset \ker(L)$ is obvious and we need to show the converse inclusion. Let $x\in\ker(L)$. Since $$\opsp(\mcc)\xrightarrow{\opsp(i)} \opsp(\mcd)\xrightarrow{\opsp(L)}\opsp(\mce)$$ is a Verdier sequence of stable \infcats, we see there exists $x'\in \opsp(\mcc)$ such that $\opsp(i)(x')=x$. It suffices to show $x'$ is connective, but that follows from $t$-exactness of $\opsp(i^R)$ and $x'=\opsp(i^R)\circ \opsp(i)(x')=\opsp(i^R)(x)$. 
		
		Now note that $L:\mcd\to \mce$ is a localization; then by the prestability of \mce, we have $L(f)$ is an equivalence if and only if $\cofib(f)\in \mcc$.
	\end{proof}
	
	\section{Dualizable additive \infcats}\label{section2}
	Having established the necessary prestable foundations, we introduce the central objects of our study: dualizable additive \infcats. We provide several equivalent characterizations of these categories, notably relating them to Grothendieck's  axioms.
	\begin{de}
		Let $\prad^{\mathrm{L}}$ denote the symmetric monoidal \infcat of presentable additive \infcats with colimit-preserving morphisms and Lurie tensor product. We say a presentable additive \infcat $\mcc$ is \textbf{dualizable additive} if it is dualizable under Lurie tensor product. 
		
		We denote  $\prad^{\op{dbl}}$ as the (non-full) subcategory of $\prad^L$ spanned by dualizable additive \infcats and internal left adjoints, i.e. those left adjoints such that their right adjoints are colimit-preserving.
	\end{de}

	One perspective on additive presentable categories is that they are $\spgeq$-modules:
	\begin{lem}
		$\prlad$ is equivalent to modules over $\spgeq$ in $\prl$, which is an idempotent algebra.
	\end{lem}
	\begin{proof}
		See \cite[Corollary 4.8]{ggn}.
	\end{proof}
Under this equivalence, the general theory of $\mathcal{V}$-linear presentable categories can be specialized to the additive setting by taking $\mathcal{V}=\spgeq$. 
\begin{prop}
	For any \dualaddinfcat $\mcm$, its stabilization $\opsp(\mcm)$ is  dualizable stable.
\end{prop}
\begin{proof}
	It suffices to observe that the base-change functor $\prlad \xrightarrow{\opsp \otimes -} \prlst$ is symmetric monoidal and thus preserves dualizable objects.
\end{proof}
\begin{rem}
	As we will demonstrate in \cref{counterexamdualstnotdualadd}, the converse fails in general, i.e. $\prlad \xrightarrow{\opsp \otimes -} \prlst$ does not reflect dualizable objects.
\end{rem}

We next recall the notion of atomic objects in a \mcv-module.

	\begin{de}
		Let $\mc{V}\in \calg(\prl)$ and $\mc{M}\in\prlv=\modu_{\mc{V}}(\prl)$. 
		
		\enu{\item An object $x \in \mathcal{M}$ is called \textbf{$\mathcal{V}$-atomic}, or simply atomic if the base $\mathcal{V}$ is
			understood, if the $\mathcal{V}$-linear functor $\mathcal{V} \xrightarrow{-\otimes x} \mathcal{M}$ classifies an internal left adjoint, i.e. if $$\unmap_{\mathcal{M}}(x,-): \mathcal{M} \rightarrow \mathcal{V}$$ preserves small colimits and the canonical map $$v \otimes \unmap_{\mathcal{M}}(x, y) \rightarrow \unmap_{\mathcal{M}}(x, v \otimes y)$$ is an equivalence for all $v \in \mathcal{V}, y \in \mathcal{M}$.
			\item The $\mathcal{M}$ is said to be \textbf{$\mathcal{V}$-atomically generated} if the smallest full $\mathcal{V}$-submodule of $\mathcal{M}$ closed under colimits and containing the atomics of $\mathcal{M}$ is $\mathcal{M}$ itself.
		} 
	\end{de}
	
Recall also the following definition from \cite{htt}. 	
\begin{de}
Let $\sC$ be a presentable category. We say that an object $x\in \sC$ is {\bf compact projective} if 
$\Map_{\sC}(x,-)$ preserves sifted colimits. 
\end{de}

	\begin{prop}\label{atcproj}
		Let $\mc{M}\in\modu_{\spgeq}(\prl)\simeq\prlad$. Then an object 
		$x \in \mathcal{M}$ is $\spgeq$-atomic if and only if it is compact projective in $\mathcal{M}$.
	\end{prop}
	
	\begin{proof}
		Because $\spgeq$ is a mode \cite[see][\textsection5]{carmeli2021ambidexterity}, it suffices to check the colimit-preserving property of the internal hom functor by \cite[Example 1.24]{ramzi2024dualizable}. Now we note that the corepresentable functor $\operatorname{Map}_{\mathcal{M}}(x,-): \mathcal{M} \rightarrow \mc{S}$ is the composition of  $\unmap_{\mathcal{M}}(x,-): \mathcal{M} \rightarrow \spgeq$ and $\Omega^{\infty}:\spgeq \rightarrow \mc{S}$. Since both $\mc{M}$ and $\spgeq$ are additive, the connective mapping spectrum functor $$\unmap_{\mathcal{M}}(x,-): \mathcal{M} \rightarrow \spgeq$$ preserves small colimits if and only if it preserves small sifted colimits. Therefore the result follows immediately from  \textup{\cite[ Proposition 1.4.3.9]{ha}} that $\Omega^{\infty}:\spgeq \rightarrow \mc{S}$ is conservative and preserves small sifted colimits.
	\end{proof}
	
	\begin{rem}
		In the additive situation, atomically generated $\spgeq$-modules are precisely compact projectively generated presentable additive \infcats. 
	\end{rem}
	The following theorem of Ramzi gives several useful characterizations of dualizable additive \infcats.
	\begin{thm}[{\cite[Theorem 1.49]{ramzi2024dualizable}}]\label{ram149}
		Let $\mathcal{M} \in \operatorname{Mod}_{\spgeq}(\prl)$ be $\lambda$-presentable for some $\lambda$. The following are equivalent:
		\enu{
			
			\item  $\mathcal{M}$ is a dualizable additive \infcat;
			\item $\mathcal{M}$ is a retract of a compact projectively generated presentable additive \infcat;
			\item There is an internally left adjoint fully faithful embedding $\mathcal{M} \rightarrow \mathcal{N}$ for compact projectively generated presentable additive \infcat $\mathcal{N}$;
			\item The canonical functor $\mathcal{P}_\Sigma(\mathcal{M}^\mu) \rightarrow \mathcal{M}$ admits a left adjoint $\hat{h}$ for some $\mu \geq \lambda$.
		}
	\end{thm}

	\begin{rem}\label{rmkcptassemb}
		Any dualizable additive \infcat $\mcm$ is compactly assembled, and consequently $\omega_1$-presentable; see \cite[Corollary 2.33]{ramzi2024dualizable}. 
		
		Furthermore, by \cite[Remark 1.51]{ramzi2024dualizable}, the left adjoint $\hat{h} \colon \mathcal{M} \rightarrow \mathcal{P}_\Sigma(\mathcal{M}^\mu)$ factors through $\mathcal{P}_\Sigma(\mathcal{M}^{\omega_1})$. This allows us to sharpen condition (4) of \cref{ram149} to the statement that the canonical functor $\mathcal{P}_\Sigma(\mathcal{M}^{\omega_1}) \rightarrow \mathcal{M}$ admits a left adjoint.
	\end{rem}

	We now turn to a different characterization of dualizable additive \infcats, formulated in terms of Grothendieck's axioms.  The next few results prepare the comparison between these axioms on a prestable category and on its stabilization.
	
	\begin{prop}\label{ab6eq}
		Let $\mcc \in \prlpst$ be a presentable  prestable \infcat satisfying $\mathrm{AB4}^*$  (\Cref{def:AB}). Then \mcc satisfies $\mathrm{AB6}$ if and only if $\opsp(\mcc)$ satisfies $\mathrm{AB6}$.
	\end{prop}
	\begin{proof}
		Since the inclusion $\mcc\hookrightarrow \opsp(\mcc)$ is closed under small filtered colimits and products, the ``if'' direction follows from \cref{conserab6}.
		
		For the ``only if'' direction, using \cref{conserab6} again, it suffices to observe that the collection $$\{\tau_{\geq -n}:\opsp(\mcc)\to \opsp(\mcc)_{\geq -n}\simeq \mcc\}$$ is jointly conservative and each $\tau_{\geq -n}$ preserves small filtered colimits and products (due to $\mathrm{AB4}^*$).
	\end{proof}
	
	In order to use $\mathrm{AB4}^*$ effectively, we will need to know that products interact well with geometric realizations. This follows from a standard truncation argument for functors between connective parts of stable \infcats. We recall the following stronger version of \cite[Lemma 1.3.3.11]{ha}. The original statement requires that both $\mc{C}$ and $\mc{C}^{'}$ are left complete. However the left completeness condition on $\mc{C}$ is removable.
	
	\begin{lem}\label{lemma4.3}
		Let $\mc{C}$ and $\mc{C}'$ be stable \infcats equipped with $t$-structures. Then:
		\enu{
			
			\item If $F:\mc{C}_{\geq0}\to\mc{C}_{\geq0}^{'}$ is a functor that preserves finite colimits, then $\tau_{\leq n}\circ F\xrightarrow{\sim}\tau_{\leq n}\circ F\circ\tau_{\leq n}$ is a natural equivalence in $\op{Fun}(\mc{C}_{\geq0},\mc{C}_{[0,n]}^{'})$ for any $n\geq0$.
			\item If $\mc{C}_{\geq0}$ admits geometric realizations and $\mc{C}'$ is left complete, then a functor $F: \mc{C}_{\geq0}\to\mc{C}_{\geq0}^{'}$ preserves finite colimits if and only if it preserves finite coproducts and geometric realizations.
		}
	\end{lem}
	\begin{proof}
		See \cite[Lemma 3.6]{hattt}.
	\end{proof}
	\begin{cor}\label{prodcommutegeom}
		Let \mcc be a separated presentable prestable \infcat satisfying $\mathrm{AB4}^*$. Then geometric realizations  commute with small products in \mcc.
	\end{cor}
	\begin{proof}
		By \cref{sepab4*iscomp} we know \mcc is  complete. Let $I$ be a set. Since \mcc satisfies $\mathrm{AB4}^*$, the product functor $\prod_{i\in I}\opsp(\mcc)\xrightarrow{\prod} \opsp(\mcc)$ is $t$-exact and hence induces a right exact functor $\prod_{i\in I}\mcc\xrightarrow{\prod}\mcc$ of  connective parts. Then by \cref{lemma4.3}(2), $\prod_{i\in I}\mcc\xrightarrow{\prod}\mcc$ preserves geometric realizations.
	\end{proof}
	We next recall the prestable form of the Gabriel--Popescu theorem. For this purpose, one needs a notion of generators adapted to the connective setting. The relevant condition only asks for generation on $\pi_0$, rather than generation under all colimits.
	\begin{de}\label{generasub}
		Let $\mathcal{C}$ be a Grothendieck prestable \infcat. A $0$-generating 
		subcategory for $\mathcal{C}$ is a full subcategory $\mathcal{C}_0 \subset \mathcal{C}$ with the 
		following property: for every object $X \in \mathcal{C}$, there exists a collection of maps 
		$\rho_\alpha: C_\alpha \rightarrow X$, where each $C_\alpha$ belongs to $\mathcal{C}_0$ and 
		the induced map $\oplus_\alpha \pi_0 C_\alpha \rightarrow \pi_0 X$ is an epimorphism in the 
		abelian category $\mathcal{C}^{\heartsuit}$.\footnote{In 
		\cite[Definition C.2.1.1]{SAG}, Lurie, calls a similar notion `generating', but 
		this conflicts with the usual notion of objects generating a category under 
		colimits, so we have chosen $0$-generating instead.}
	\end{de}
	\begin{rem}
		Note that if $\mathcal{C}$ is stable, then any full subcategory $\mathcal{C}_0 \subset \mathcal{C}$ is a $0$-generating subcategory, because $\mcc^\heartsuit\simeq0$ in this case.
	\end{rem}
	\begin{rem}
	Let $\mathcal{C}$ be a Grothendieck prestable \infcat. If $\mathcal{C}_0 \subset \mathcal{C}$ is a full subcategory that generates $\mathcal{C}$ under small colimits, then $\mathcal{C}_0$ is a $0$-generating subcategory of $\mathcal{C}$ in the sense of \cref{generasub}. While the converse does not hold in general, the following theorem shows that it is indeed true whenever $\mathcal{C}$ is separated.
	\end{rem}
	
	\begin{thm}[{\cite[Gabriel--Popescu, Lurie Theorem C.2.1.6]{sag}}]\label{gplthm}
		Let $\mathcal{C}$ be a Grothendieck prestable \infcat and let $\mathcal{C}_0 \subset \mathcal{C}$ be a small $0$-generating subcategory. Assume that $\mathcal{C}$ is separated and that $\mathcal{C}_0$ is closed under finite coproducts in $\mathcal{C}$. Then:
		\enu{
			\item  The inclusion functor $\mathcal{C}_0 \rightarrow \mathcal{C}$ extends to a left exact functor $F: \mathcal{P}_{\Sigma}(\mathcal{C}_0) \rightarrow \mathcal{C}$ which commutes with small colimits.
			\item  The functor $F$ admits a fully faithful right adjoint $G: \mathcal{C} \rightarrow \mathcal{P}_{\Sigma}(\mathcal{C}_0)$.
		}
	\end{thm}
	To prove dualizability from the Grothendieck axioms, we will need to show that the localization supplied by Gabriel--Popescu preserves limits. The main point is preservation of products. We therefore collect two technical lemmas controlling products among sufficiently compact objects and their Ind-completions.
	
	\begin{lem}[{\cite[Proposition 5.4.7.4]{htt}}]\label{compprod}
		Let $\kappa$ be a regular cardinal and $\mcc$ be a $\kappa$-presentable \infcat. If $\lambda\gg \kappa$ is an uncountable regular cardinal such that $\mcc^\kappa$
		is essentially $\lambda$-small, then the full subcategory $\mcc^\lambda\subset \mcc$ is closed under all $\kappa$-small limits in $\mcc$.
	\end{lem}
	\begin{lem}\label{ab6prod}
		Let $\mcb$ be a small \infcat admitting $\kappa$-small products, where $\kappa$ is a regular cardinal. Then:
		\enu{
			\item $\op{Ind}(\mcb)$ admits $\kappa$-small products too.
			\item Let  $\mcc$ be a presentable \infcat satisfying $\mathrm{AB6}$. If $\mcb \to \mcc$ is a functor preserving $\kappa$-small products, then the induced functor $\op{Ind}(\mcb)\xrightarrow{}\mcc$ preserves $\kappa$-small products.
		}
	\end{lem}
	\begin{proof}
		See the proof of \cite[Proposition 1.52]{efimov2024k}.
	\end{proof}
	We will also use stability of the Grothendieck axioms under retracts. The following elementary categorical observation is a convenient way to express this stability in terms of left adjointable squares.
	
	\begin{lem}[{\cite[Lemma 1.48]{ramzi2024dualizable}}]\label{ramzi148}
		In $\Cat^\square:=\funct(\Delta^1\times\Delta^1,\Cat)$, retracts of vertically left adjointable squares are vertically left adjointable.
	\end{lem}
	\begin{prop}\label{retractab6}
		Let \mcc be a presentable \infcat.
		\enu{\item If $\mcc$ is a retraction of \mcd in \prl such that \mcd is prestable and satisfies $\mathrm{AB4}^*$, then \mcc is prestable and satisfies $\mathrm{AB4}^*$ too.
			
			\item If $\mcc$ is a retraction of \mcd in \prl where \mcd satisfies $\mathrm{AB6}$, then \mcc satisfies $\mathrm{AB6}$ too.
			\item If \mcc is prestable, then it is complete if and only if tensoring with the following Postnikov limit diagram is still a limit in \prl.
		}
		$$\begin{tikzcd}
			& \vdots \arrow[d] \\
			& \opsp_{[0,1]} \arrow[d] \\
			\spgeq \arrow[r] \arrow[ru] \arrow[ruu] & \opsp_{=0}
		\end{tikzcd}$$
	\end{prop}
	\begin{proof}
		Part (3) follows, since $\mcc_{\leq n} \simeq \mcc \otimes \opsp_{[0,n]}$. 
		
		For (2), since \mcc is a retract of \mcd in \prl, it is also a retract of \mcd in $\pr^R$. Now the following diagram, which appeared in \cref{ab6adjointable},
		$$\begin{tikzcd}
			{\prod_i \funct(J_i,\mcc)} \arrow[r, "\prod"]    & {\funct(\prod_i J_i,\mcc)} \\
			\prod_i \mcc \arrow[u, "\delta"] \arrow[r, "\prod"] & \mcc \arrow[u, "\delta"]     
		\end{tikzcd} $$ 
		gives a functor $\pr^R\to\pr^{R,\square} $. Therefore  \cref{ramzi148} implies that presentable \infcats satisfying $\mathrm{AB6}$ are closed under retractions in $\pr^R$.
		
		For (1), we first note that \mcc is prestable by \cref{retractpst}. It therefore follows by the same argument as (2), due to \cref{ab4staradjointable}.
	\end{proof}
	We now combine the preceding ingredients to obtain the promised intrinsic characterization of dualizable additive \infcats. One direction follows from the stability of $\mathrm{AB4}^*$ and $\mathrm{AB6}$ under retracts. The converse uses Gabriel--Popescu together with $\mathrm{AB6}$ to show that the resulting localization preserves products, hence admits the internal adjoint required for dualizability.
	
	\begin{thm}\label{main1}
		Let $\mcc\in\prlad$. Then $\mcc$ is dualizable additive if and only if $\mcc$ is a separated Grothendieck prestable \infcat satisfying $\mathrm{AB4}^*$ and $\mathrm{AB6}$.
	\end{thm}
	\begin{proof}
		Let $\mcc$ be a \dualaddinfcat. By \cref{ram149}, $\mcc$ is a retract of $\mcp_\Sigma(\mca)$ in \prlad, where $\mca$ is a small additive \infcat. Note that $\mcp_\Sigma(\mca)$ is complete Grothendieck prestable and satisfies both $\mathrm{AB4}^*$ (since the inclusion $\funct^\times(\mca^\opp,\spgeq) \hookrightarrow \funct^\times(\mca^\opp,\opsp)$ is closed under products) and $\mathrm{AB6}$ (\cref{compab6}). Consequently, $\mcc$ is also complete Grothendieck prestable and satisfies $\mathrm{AB4}^*$ and $\mathrm{AB6}$ by \cref{retractab6}.
		
		For the converse direction, let $\mcc$ be a separated Grothendieck prestable
		 \infcat satisfying $\mathrm{AB4}^*$ and $\mathrm{AB6}$. By \cite[Remark 1.54]{efimov2024k}, $\mcc$ is compactly
		 assembled, and so in particular $\omega_1$-compactly generated.
		 We wish to show the functor
		  $c:\mcp_\Sigma(\mcc^{\omega_1})\to\mcc$ admits a fully faithful left adjoint, 
		  or equivalently show $c$ is limit-preserving. By \cref{gplthm}, we see that 
		  $c$ is a left exact localization, so it suffices to show that $c$ preserves small products.
		
		Let $\kappa$ be an arbitrary uncountable regular cardinal. We want to show that 
		the functor $c:\mcp_\Sigma(\mcc^{\omega_1})\to\mcc$ preserves $\kappa$-small products.
		 Choose a regular cardinal $\lambda$ such that $\lambda\gg  \kappa$ and $\mcc^\kappa$
		is essentially $\lambda$-small. Then $c_\lambda:\mcp_\Sigma(\mcc^{\lambda})\to\mcc$ factors through $c$ via restriction $\mcp_\Sigma(\mcc^{\lambda})\to\mcp_\Sigma(\mcc^{\omega_1})$ (which has a fully faithful left adjoint, and is in particular essentially surjective). So it suffices to show that $c_\lambda:\mcp_\Sigma(\mcc^{\lambda})\to\mcc$ preserves $\kappa$-small products. Firstly, we decompose $c_\lambda$ as the following adjunctions:
		$$\mcp_\Sigma(\mcc^{\lambda})\underset{i}{\stackrel{l}{\rightleftarrows}}\op{Ind}(\mcc^{\lambda})\underset{h}{\stackrel{p}{\rightleftarrows}}\mcc .$$
		Now given a $\kappa$-small set of objects $\{y_i\}$ in $\mcp_\Sigma(\mcc^{\lambda})$, then by \cref{propa5} each $y_i$ can be written as the geometric realization $|z_{i,\bullet}|$ for which every $z_{i,n}$ lies in $\op{Ind}(\mcc^{\lambda})$. By \cref{prodcommutegeom}, we are reduced to showing the functor $\op{Ind}(\mcc^{\lambda})\xrightarrow{p}\mcc$ preserves $\kappa$-small products. That is implied by \cref{ab6prod} due to facts that $\mcc$ satisfies $\mathrm{AB6}$ and that $\mcc^\lambda\subset\mcc$ is closed under $\kappa$-small products by \cref{compprod}.
	\end{proof}
	\begin{rem}\label{rem:Roos}
		Theorem~\ref{main1} is a prestable version of a very classical result about abelian categories. 
		Let $\sA$ be a Grothendieck abelian category with a generator $G$. Gabriel--Popescu theorem provides a localization functor $\mathrm{Mod}_{\op{End}_{\sA}(G)}(\op{Ab}) \stackrel{Q}\to \sA$. 
		Roos's theorem \cite{Roos1965} states that $\sA$ satisfies $\mathrm{AB4}^*$ and $\mathrm{AB6}$ if and only if $Q$ admits a left adjoint (cf. Theorem~\ref{ram149}(4)). 
		
		Moreover, Grothendieck abelian categories satisfying $\mathrm{AB4}^*$ and $\mathrm{AB6}$ are exactly dualizable Grothendieck abelian categories 
		\cite[Corollary~3.7]{kanda2024moduletheoreticapproachdualizablegrothendieck}.
	\end{rem}

	\begin{rem}\label{embedmodcat}
	The argument in Theorem~\ref{main1} also shows that if $\mcc$ is a dualizable additive $\infty$-category and $\mcc_0 \subset \mcc$ is a small $0$-generating subcategory closed under finite coproducts, then the canonical map $\mcp_\Sigma(\mcc_0) \to \mcc$ is a left exact localization that preserves all small limits and colimits, and thus admits a left adjoint.
		
		Indeed, we first choose a large enough $\lambda$ such that $\mcc_0\subset\mcc^\lambda$. We also observe that $c:\mcp_\Sigma(\mcc^{\lambda})\to\mcc$ preserves small products (by the argument above) and factors through  $\mcp_\Sigma(\mcc_0)\to\mcc$ via the restriction functor $\mcp_\Sigma(\mcc^{\lambda})\to \mcp_\Sigma(\mcc_0)$ (which is a limit-preserving localization).
		
		In particular, let $x\in \mcc$ be a generator (which always exists because we can take the direct sum). Then the induced functor\footnote{$\op{End}(x)$ here denotes the mapping \emph{connective} spectrum instead of the mapping spectrum.} $$l_x: \modu_{\op{End}(x)}(\spgeq)\to \mcc$$  is a left exact localization  because $\modu_{\op{End}(x)}(\spgeq)\simeq \mcp_\Sigma(\mcc_x)$ where $\mcc_x$ denotes the full subcategory of $\mcc$ spanned by $\{x^{\oplus^n}|n\geq 0\}$, \cite[see][Remark C.2.1.9]{sag}. Now  $l_x$ admits a left adjoint by the same argument, which implies that any dualizable additive \infcat \mcc admits a strongly continuous embedding $$\mcc\hookrightarrow\modu_{R}(\spgeq)$$ for some connective \eonering $R$.
	\end{rem}
	We next translate the intrinsic characterization into the language of $t$-structures. This identifies precisely which $t$-categories arise from dualizable additive categories.
\begin{de}\label{de:strongly_cocont}
	We denote by $\prtil$ the (non-full) subcategory of $\prtrex$ with the same objects and whose morphisms are the \textbf{strongly $t$-continuous} functors (i.e., right $t$-exact, colimit-preserving functors that admit a $t$-exact, colimit-preserving right adjoint). Equivalently, these morphisms are precisely the internal left adjoints in $\prtrex$.
\end{de}

\begin{cor}\label{dualaddembst}
	The fully faithful embedding $\op{Groth} \hookrightarrow \prtrex$ from \cref{pret} restricts to a fully faithful embedding 
	\[ \praddbl \hookrightarrow \prtil. \]
	Its essential image consists of those separated Grothendieck $t$-categories that satisfy both $\mathrm{AB4}^*$ and $\mathrm{AB6}$.
\end{cor}
	
	\begin{rem}
		The description of the morphisms in this embedding is also compatible with stabilization. In particular,  for a left adjoint functor $F:\mcc\to\mcd$ between dualizable additive \infcats, it is an internal left adjoint if and only if the functor $(\opsp\otimes F)^R\simeq\opsp(F^R)$, i.e. the right adjoint of $\opsp\otimes F :\opsp(\mcc)\to\opsp(\mcd)$, is right $t$-exact and preserves colimits.
	\end{rem}

	\begin{rem}
		The identification $(\opsp\otimes F)^R\simeq\opsp(F^R)$ used above is a special case of the following general observation.	Let $\mcc,\mcd\in\prl$. Let $\funct^{LR}(\mcd,\mcc)$ denote the full subcategory of those functors preserving both colimits and limits. Then the following diagram is commutative.
		$$\begin{tikzcd}
			{\funct^{LR}(\mcd,\mcc)} \arrow[rr, hook] \arrow[d, hook] &  & {\funct^{R}(\mcd,\mcc)} \arrow[d, "\opsp(-)"] \\
			{\funct^{L}(\mcd,\mcc)} \arrow[rr, "\opsp\otimes-"]       &  & {\funct(\op{Sp}(\mcd),\op{Sp}(\mcc))}        
		\end{tikzcd}$$
		To see this, take an arbitrary $g\in\funct^{LR}(\mcd,\mcc)$. Due to the natural equivalence $(\opsp\otimes g)^R\simeq \opsp(g^R)$ from the Lurie tensor product, it suffices to check that the following diagram is horizontally right adjointable:
		$$\begin{tikzcd}
			\mcc \arrow[r, "g^R"', dashed, shift right]                             & \mcd \arrow[l, "g"', shift right]                             \\
			\op{Sp}(\mcc) \arrow[u] \arrow[r, "\op{Sp}(g^R)"', dashed, shift right] & \op{Sp}(\mcd) \arrow[u] \arrow[l, "\op{Sp}(g)"', shift right]
		\end{tikzcd}$$
		which is implied by the fact that $g$ is a limit-preserving left adjoint.
	\end{rem}

	\begin{cor}\label{dualadst}
		Let $\mcc$ be a separated Grothendieck  prestable \infcat satisfying $\mathrm{AB4}^*$. Then \mcc is dualizable additive if and only if $\opsp(\mcc)$ is dualizable stable.
	\end{cor}
	\begin{proof}
	Combining \cref{main1} and \cite[Proposition 1.53]{efimov2024k}, it suffices to show that $\mcc$ satisfies $\mathrm{AB6}$ if and only if $\opsp(\mcc)$ does.	That follows directly from  \cref{ab6eq}.
	\end{proof}
	We conclude the section with some examples and non-examples illustrating the scope of the characterization.

		\begin{exam}\label{exm:top_exm}\,
		\enu{
			\item If $X$ is a locally profinite space, then the \infcat $$\shv(X;\spgeq)$$ of connective sheaves is dualizable additive. Indeed, it is a retract of $\shv(X^+;\spgeq)$ in \prlad, where $X^+$ denotes the one-point compactification of $X$. By assumption, $X^+$ is profinite and hence $\shv(X^+;\spgeq)\simeq \modu_{\mathbb{S}^{X^+}}(\spgeq)$ is dualizable additive.
			\item 
			Let $G$ be an acyclic group, i.e. satisfying $\widetilde{\mathrm{H}}_*(BG;\mathbb{Z}) = 0$. Let $R$ be a connective ring spectrum. In this case $R[G]\to R$ is a 
			$\pi_0$-surjective homological epimorphism, so in particular by \cref{classsub} the kernel 
			\[
			\amodu^{\op{cn}}_{(R[G],I)}=\ker(\modu_{R[G]}^{\op{cn}} \to \modu_{R}^{\op{cn}})
			\]
			is a dualizable additive \infcat. 
		}
	\end{exam}

\begin{rem}\label{counterexamdualstnotdualadd}
	The hypotheses in \cref{main1} are genuinely restrictive. In particular, stable dualizability alone does not imply additive dualizability at the connective level.
	
	For example, for an arbitrary compact Hausdorff space $X$, the category $\shv(X;\spgeq)$ is not necessarily  dualizable additive, as it need not be hypercomplete (see, e.g., \cite[Counterexample 6.5.4.8]{htt}); nevertheless, its stabilization is  dualizable stable. This demonstrates that additive dualizability is a strictly stronger condition than stable dualizability.
\end{rem}

	\begin{rem}
		When $R$ is a regular ring, the Farrell--Jones conjecture \cite[Section~1.6]{FarrellJones1993} predicts that for any torsion-free acyclic group $G$ the map $K(R) \to K(R[G])$ is an equivalence. 
		In particular, if it holds, the map $K(R[G]) \to K(R)$ should also be an equivalence, and we should have $$K^{\op{cont}}(\amodu_{(\mathbb{S}[G],I)})=0.$$
		Moreover, the main results of \cite{BunkeKasprowskiWinges2026} show that the analog of the Farrell--Jones conjecture for an arbitrary lax monoidal finitary localizing invariant $E$ holds in many
		cases as well, which means that $E^{\op{cont}}(\amodu_{(\mathbb{S}[G],I)})$ vanishes. We wonder whether the methods of dualizable additive categories and prestable motives 
		(see \cref{sec:prestable_motives}) can be used to analyze the conjecture in the unknown cases.
	\end{rem}
	\section{Additive short exact sequences}\label{section3}
	In this section, we analyze short exact sequences of $t$-categories, investigating how properties of 
	$t$-structures, such as right completeness and Grothendieck's axioms, behave with respect to fibers, cofibers, and extensions. Building on these results, we then analyze short exact sequences of \dualaddinfcats.

	\subsection{Strongly $t$-continuous short exact sequences}
	We begin by isolating a specific class of localizations that interact harmoniously with $t$-structures. The 
	assumption of a localization being strongly $t$-continuous (see \cref{de:strongly_cocont}) 
	guarantees that connective objects are strictly preserved and reflected under stabilization and quotient operations. 
	
	The key additional condition is a connectivity requirement for the unit of the adjunction; this ensures that the induced localization on connective parts behaves like an internal localization of prestable categories.
	\begin{de}
	Let $F:\mcb\to\mcc \in\prtil$ be a strongly $t$-continuous functor between presentable $t$-categories. We say $F$ is a \textbf{connective internal localization} if $F$ is a localization such that for any 
	$x\in \mcb_{\geq 0}$, the unit map $x\to F^RF(x)$ has connective fiber in \mcb.
	\end{de}
	The following proposition shows that this notion is precisely what is needed to make kernels and cokernels in the strongly $t$-continuous world agree with the expected Verdier constructions. In particular, short exact sequences in $\prtil$ can be described either by fully faithful inclusions or by connective internal localizations.
	
	\begin{prop}\label{shortexactt}
	
		\begin{enumerate}
			\item Let $i:\mca\hookrightarrow \mcb$ be a fully faithful strongly $t$-continuous functor between presentable $t$-categories, i.e. a fully faithful morphism in \prtil. The cofiber $\mcc:=\coker(i)$ in \prtil is created in \prtrex and $p:\mcb\to\mcc$ is a connective internal localization. Furthermore, the induced sequence
			$$\mca\xrightarrow{i}\mcb \xrightarrow{p}\mcc$$ is a bifiber sequence in \prtil.
			\item Let $p:\mcb\xrightarrow{} \mcc$ be a  connective internal localization in \prtil. The fiber $\mca:=\ker(p)$ in \prtil is created in \prtrex and the induced sequence
			$$\mca\xrightarrow{i}\mcb \xrightarrow{p}\mcc$$ is a bifiber sequence in \prtil.
			
		\end{enumerate}
		
	\end{prop}
	\begin{proof}
		(1) Denote by $\mcc:=\coker(i)$ the quotient in \prlst. Then $\mca\xrightarrow{i}\mcb \xrightarrow{p}\mcc$ is a short exact sequence in $(\prlst)^{iL}$. We endow \mcc with a $t$-structure obtained through $\mcc_{\geq 0}$ being generated by $p(\mcb_{\geq 0})$ under small colimits and extensions. Then $(\mcc,\mcc_{\geq 0})$ can be identified with $\coker(i)$ in \prtrex by \cref{colimt}. By the recollement decomposition for arbitrary $x\in\mcb$ $$ii^R(x)\to x \to p^Rp(x),$$ 
		we conclude that $x\in \mcb_{\geq 0} \implies p^Rp(x)\in \mcb_{\geq 0}$. Since $\mcc_{\geq 0}$ is generated by $p(\mcb_{\geq 0})$ under small colimits and extensions, we get $p^R(\mcc_{\geq 0})\subset \mcb_{\geq 0}$. Therefore $p\dashv p^R$ restricts to an internal localization in \prlad
		$$\begin{tikzcd}
			\mcb_{\geq0} \arrow[r, "p|_{\geq0}", shift left=1ex]   & \arrow[l,"p^R|_{\geq0}","\perp"', hook', shift left=1ex] \mcc_{\geq0}
		\end{tikzcd}.$$
		By the recollement decomposition again, $p$ is a connective internal localization as desired.
		We wish to show that $\mcc$ is also the cofiber in \prtil. 
		Let $f:\mcb\to\mce \in \prtil$ be a morphism such that the precomposition with \mca is zero. We need to show that the induced functor $\mcc\to \mce$ is strongly $t$-continuous. That follows from the following factorization and the fact that $p^R$ reflects connective objects.
		$$\begin{tikzcd}
			& \mce \arrow[ld, "f^R"'] \arrow[d, dashed] \\
			\mcb & \mcc \arrow[l, "p^R"]                   
		\end{tikzcd}$$
		Now we wish to show that $\mca\xrightarrow{i}\mcb \xrightarrow{p}\mcc$ is also a fiber sequence in \prtil. By the recollement and the fact that $i$ reflects connective objects, we conclude that $$\mca_{\geq0}\xrightarrow{i}\mcb_{\geq0} \xrightarrow{p}\mcc_{\geq0}$$ is a fiber sequence and hence $\mca\xrightarrow{i}\mcb \xrightarrow{p}\mcc$ is a fiber sequence in \prtrex by \cref{colimt}. Given a morphism $g:\mce \to \mcb \in\prtil$ such that the composition with $\mcc$ is zero, we wish to show that the factorization $\mce \to \mca$ is strongly $t$-continuous. But that follows from the following factorization and the fact that $i^R(\mcb_{\geq 0})=\mca_{\geq 0}$.
		$$\begin{tikzcd}
			\mce                   &                                          \\
			\mca \arrow[u, dashed] & \mcb \arrow[l, "i^R"] \arrow[lu, "g^R"']
		\end{tikzcd}$$
		The argument for (2) is similar to (1) and we leave it to the reader.
	\end{proof}
	\begin{rem}
		It follows from the proof of \cref{shortexactt} that both forgetful functors \begin{center}
			$\prtil\xrightarrow{(\mca,\ageq)\mapsto \mca} \prlst$ \, and \, $\prtil \xrightarrow{(\mca,\ageq)\mapsto \mca_{\geq0}}\prlad$
		\end{center} preserve short exact sequences.
	\end{rem}
	We next record two basic permanence properties for such short exact sequences.
	
\begin{prop}\label{propertycoffib}
		Let $\mca\xhookrightarrow{i}\mcb\xrightarrow{p}\mcc$ be a short exact sequence in \prtil. Then:\enu{\item 
			If $\mcb$ is right complete, then so are \mca and \mcc.
		\item If  $\mcb$ lies in the image of embedding $\praddbl\xhookrightarrow{}\prtil$, then so do \mca and \mcc.	}
\end{prop}
\begin{proof}
	(1) Consider the following diagram.
	$$\begin{tikzcd}
		\mcc \arrow[d, "p^R", hook] \arrow[r] & \lim_n \mcc_{\geq -n} \arrow[d, "\lim_n p^R|_{\geq -n}"] \\
		\mcb \arrow[r]                        & \lim_n \mcb_{\geq -n}                                  
	\end{tikzcd}$$
	Since $p^R$ is  fully faithful $t$-exact and preserves filtered colimits, the diagram above is a pullback square and $\mcc$ is right complete.
	
	The right completeness of \mca follows from the following comparison between Verdier sequences.
		$$\begin{tikzcd}
		\opsp(\mca_{\geq0}) \arrow[d] \arrow[r] & \opsp(\mcb_{\geq0}) \arrow[d] \arrow[r] & \opsp(\mcc_{\geq0}) \arrow[d] \\
		\mca \arrow[r]                                  & \mcb \arrow[r]                          & \mcc                                 
	\end{tikzcd}$$
	(2) Suppose $\mcb$ lies in the image of the embedding $\praddbl\xhookrightarrow{}\prtil$; we wish to prove that both $\mca$ and $\mcc$ are right complete and that both $\mca_{\geq 0}$ and $\mcc_{\geq 0}$ are dualizable additive. The former follows from (1). The latter follows because $\mca_{\geq 0}\hookrightarrow \mcb_{\geq 0}$ is a fully faithful internal left adjoint and $\mcb_{\geq0}\to \mcc_{\geq 0}$ is an internal left adjoint localization, making both of them retracts of $\mcb_{\geq 0}$ in \prlad.
\end{proof}
In order to prove converse permanence statements, we will need a five-lemma type criterion for comparisons of Verdier sequences. The following lemma says that, under suitable adjointability hypotheses, equivalences on the kernel and quotient force an equivalence on the middle term.

	\begin{lem}\label{fivelemst}
		Consider a comparison between Verdier sequences of stable \infcats.
		$$\begin{tikzcd}
			\mathcal{K}_0 \arrow[r, "i_0"] \arrow[d, "f"'] 
			& \mathcal{A}_0 \arrow[r, "p_0"] \arrow[d, "g"'] 
			& \mathcal{B}_0 \arrow[d, "h"'] \\
			\mathcal{K}_1 \arrow[r, "i_1"'] 
			& \mathcal{A}_1 \arrow[r, "p_1"'] 
			& \mathcal{B}_1
		\end{tikzcd}
		$$
		Suppose both squares are horizontally and vertically right adjointable\footnote{One square suffices, see \cite[Proposition A.16]{ramzi2024dualizable}.}. If $f,h$ are equivalences, then so is $g$.
	\end{lem}
	\begin{proof}
		By adjointability, we have the following diagram:
		$$\begin{tikzcd}
			i_0i_0^R \arrow[r] \arrow[d, "\sim"] & \op{Id}_{\mca_0} \arrow[r] \arrow[d] & p_0^Rp_0 \arrow[d, "\sim"] \\
			g^Ri_1i_1^Rg\simeq g^Rgi_0i_0^R \arrow[r]               & g^Rg \arrow[r]                       & g^Rgp_0^Rp_0 \simeq g^R   p_1^Rp_1g           
		\end{tikzcd}$$
		where both the top and bottom are cofiber sequences.
		Thus the unit transformation $\op{Id}_{\mca_0}\to g^Rg$ is an equivalence. Similarly, we can prove that the counit transformation $ gg^R\to\op{Id}_{\mca_0}$ is an equivalence too.
	\end{proof}

\begin{rem}
	
	Note that the vertical adjointability condition in \cref{fivelemst} cannot be removed. For example, given a Verdier sequence of stable \infcats $\mathcal{A} \xrightarrow{i} \mathcal{B} \xrightarrow{p} \mathcal{C}$ such that both $i$ and $p$ admit a right adjoint. Consider the following diagram:
	\[
	\begin{tikzcd}
		\mca \arrow[r] \arrow[d, "\mathrm{id}_{\mathcal{A}}"'] & \mathcal{A} \oplus \mathcal{C} \arrow[r] \arrow[d, "{(i, p^R)}"] & \mathcal{C} \arrow[d, "\mathrm{id}_{\mathcal{C}}"] \\
		\mathcal{A} \arrow[r, "i"']                            & \mcb \arrow[r, "p"']                                             & \mathcal{C}                                       
	\end{tikzcd}
	\]
We see that $(i, p^R)$ is not necessarily an equivalence, despite both squares being horizontally right adjointable.\footnote{We thank Maxime Ramzi for this counterexample.}
\end{rem}
We now apply this five-lemma criterion to right completeness. 

	\begin{prop}\label{fivelemrightcompl}
		Let $\mca\xhookrightarrow{i}\mcb\xrightarrow{p}\mcc$ be a short exact sequence in \prtil. If $\mca,\mcc$ are right complete and the natural functor $\opsp(\mcb_{\geq0})\to\mcb$ is strongly continuous, then $\mcb$ is right complete.
	\end{prop}
	\begin{proof}
		Consider the following diagram:
		$$\begin{tikzcd}
			\opsp(\mca_{\geq0}) \arrow[d, "\sim"] \arrow[r] & \opsp(\mcb_{\geq0}) \arrow[d,"g"] \arrow[r] & \opsp(\mcc_{\geq0}) \arrow[d, "\sim"] \\
			\mca \arrow[r]                                  & \mcb \arrow[r]                        & \mcc.
		\end{tikzcd}$$
		By \cref{shortexactt}, it is a comparison between Verdier sequences and both squares are horizontally right adjointable. 
		
		We wish to prove both squares are vertically right adjointable. 
		By \cite[Proposition A.16]{ramzi2024dualizable}, it suffices to prove it for the left square. Since the 
		adjunction $g\dashv g^R$  restricts to the identity functor on the connective part $\mcb_{\geq0}$ by the 
		dual of \cite[Proposition 1.2.1.17]{ha}, the Beck-Chevalley comparison for the left square becomes an 
		equivalence on objects of $\mca_{\geq0}$. 
		 Since $g^R$ is colimit-preserving by assumption, we conclude that the left square is vertically right adjointable.
		Now invoking \cref{fivelemst}, $g$ is an equivalence and hence \mcb is right complete.
	\end{proof}
	
	We next turn from kernels and cokernels to pullbacks. Since extensions of short exact sequences will later be expressed as pullbacks of localizations, we first recall the following  statement in the stable setting.
	
	\begin{lem}\label{pullbackloc}
		Let $G:\mcb\to\mcc\in (\prlst)^{iL}$ be a localization. Then for any morphism $F:\mca \to\mcc \in (\prlst)^{iL}$, the following diagram 
		$$
		\begin{tikzcd}
			\mca\times_{\mcc}\mcb \arrow[d, "p_{\mcb}"] \arrow[r, "p_{\mca}"] & \mca \arrow[d, "F"] \\
			\mcb \arrow[r, "G"]                                               & \mcc               
		\end{tikzcd}
		$$ 
		is a pullback diagram in $(\prlst)^{iL}$.
	\end{lem}
	\begin{proof}
		Let $\mcd$ denote $\mca\times_{\mcc}\mcb$. By \cite[Lemma 4.1]{ramzi2024dualizable}, it suffices to show that both $p_{\mca}$ and $p_{\mcb}$ are strongly continuous. Because Bousfield localizations are closed under pullback, $p_{\mca}$ admits a fully faithful right adjoint $p_{\mca}^R$, which is given by
		$$p_{\mca}^R(a) = (a,\, G^R F(a),\, \beta).$$ Therefore $p_{\mca}^R$ preserves small colimits and $p_{\mca}$ is strongly continuous. To see that $p_{\mcb}^R$ preserves small colimits, we consider the following diagram of fiber sequences.
	\begin{equation*}
		\begin{tikzcd}
			\mck \arrow[r, "i_1"] \arrow[d, Rightarrow, no head] & \mca\times_{\mcc}\mcb \arrow[d, "p_{\mcb}"] \arrow[r, "p_{\mca}"] & \mca \arrow[d, "F"] \\
			\mck \arrow[r, "i_0"]                                & \mcb \arrow[r, "G"]                                                & \mcc               
		\end{tikzcd}
		\tag{*}
	\end{equation*}
		By the recollement $i_0i_0^R\to \op{Id}_{\mcb}\to G^RG$, it suffices to show that both $p_{\mcb}^RG^R$ and $p_{\mcb}^Ri_0$ preserve small colimits. The former is obvious because $p_{\mcb}^RG^R\simeq p_{\mca}^RF^R$. The latter is more tricky, and it suffices to show $p_{\mcb}^Ri_0\simeq i_1$. We aim to prove that for $k \in i_0(\mathcal{K})$,
		$$p_{\mathcal{B}}^R(k) \simeq (0_{\mathcal{A}}, k, 0).$$
		We verify that the object $X_k = (0_{\mathcal{A}}, k, 0)$ satisfies the universal property of the right adjoint. That is, for any object $Y = (a, b, \phi) \in \mcd$, there is a natural equivalence
		$$\mathrm{Map}_{\mathcal{D}}(Y, X_k) \simeq \mathrm{Map}_{\mathcal{B}}(p_{\mathcal{B}}(Y), k) = \mathrm{Map}_{\mathcal{B}}(b, k).$$
		The left-hand side is given by
		$$\mathrm{Map}_{\mathcal{D}}((a, b, \phi), (0_{\mathcal{A}}, k, 0))
		\simeq \mathrm{Map}_{\mathcal{A}}(a, 0_{\mathcal{A}}) \times_{\mathrm{Map}_{\mathcal{C}}(F a, G k)} \mathrm{Map}_{\mathcal{B}}(b, k).$$
		Because $k \in i_0(\mathcal{K})$, we have $G(k) \simeq 0_{\mathcal{C}}$. Thus, $\mathrm{Map}_{\mathcal{C}}(F a, G k) \simeq \mathrm{Map}_{\mathcal{C}}(F a, 0_{\mathcal{C}})$ is  contractible. The map $\mathrm{Map}_{\mathcal{B}}(b, k) \to \mathrm{Map}_{\mathcal{C}}(F a, 0_{\mathcal{C}})$ is the unique zero map. Therefore the fiber product reduces to
		$$\ast \times_{\ast} \mathrm{Map}_{\mathcal{B}}(b, k) \simeq \mathrm{Map}_{\mathcal{B}}(b, k).$$
	\end{proof}

	\begin{rem}
		The proof also shows that two squares in the diagram $(*)$ are both horizontally and vertically right adjointable by combining with \cite[Proposition 1.62]{ramzi2024dualizable}.
	\end{rem}
	
	\begin{rem}
	The argument of \cref{pullbackloc} also works for any stable base $\mcv\in \calg(\prlst)$. Namely, if $G:\mcb\to\mcc\in (\prlv)^{iL}$ is a localization that is an internal \mcv-left adjoint of \mcv-modules, then for any internal \mcv-left adjoint $F:\mca \to\mcc \in (\prlv)^{iL}$ of \mcv-modules, the following diagram 
	$$
	\begin{tikzcd}
		\mca\times_{\mcc}\mcb \arrow[d, "p_{\mcb}"] \arrow[r, "p_{\mca}"] & \mca \arrow[d, "F"] \\
		\mcb \arrow[r, "G"]                                               & \mcc               
	\end{tikzcd}
	$$ 
	is a pullback diagram in $(\prlv)^{iL}$.
	\end{rem}
	We now refine the preceding stable base-change statement to the setting of $t$-categories. The additional connectivity assumption on the localization ensures that the pullback inherits the expected connective part and that the base-changed localization remains connective.
	\begin{prop}\label{pullbackt}
		Let $G:\mcb\to\mcc\in \prtil$ be a connective internal localization and $F:\mca \to\mcc \in\prtil$ be a morphism. Then $\mca\times_{\mcc}\mcb\to \mca$ is a connective internal localization and the following diagram 
		$$
		\begin{tikzcd}
			\mca\times_{\mcc}\mcb \arrow[d, "p_{\mcb}"] \arrow[r, "p_{\mca}"] & \mca \arrow[d, "F"] \\
			\mcb \arrow[r, "G"]                                               & \mcc               
		\end{tikzcd}
		$$ 
		is a pullback diagram in \prtil, where the connective part on	$\mca\times_{\mcc}\mcb$ is given by $\mca_{\geq0}\times_{\mcc_{\geq0}}\mcb_{\geq0}$. 
	\end{prop}
	\begin{proof}
		Let $\mcd:=\mca\times_{\mcc}\mcb$. We first prove that $p_{\mca}: \mathcal{D} \to \mathcal{A}$ is a connective internal localization. By the proof of \cref{pullbackloc}, $p_{\mca}$ admits a fully faithful right adjoint $p_{\mca}^R$ preserving colimits, which is given by
		$$p_{\mca}^R(a) = (a,\, G^R F(a),\, \beta).$$
		For the connectivity, we need to show that the unit transformation $\mathrm{Id}_{\mathcal{D}_{\geq 0}} \to p_{\mca}^R\, p_{\mca}|_{\mcd_{\geq 0}}$ is pointwise $\pi_0$-surjective. Let $X = (a, b, \alpha) \in \mathcal{D}_{\geq 0}$.  It is equivalent to show that the  fiber $\op{fib}(X\to p_{\mca}^R p_{\mca}(X))$ lies in $\mathcal{D}_{\geq 0}$.
		Because $G$ is a connective internal localization, we have that $$\op{fib}(X\to p_{\mca}^R p_{\mca}(X))\simeq (0,\op{fib}(u_b),0)\in \mcd_{\geq 0},$$ where $u_b: b \to G^R G(b)$ is the unit for $G$.
		
		Now we start to prove the diagram 
		$$
		\begin{tikzcd}
			\mca\times_{\mcc}\mcb \arrow[d, "p_{\mcb}"] \arrow[r, "p_{\mca}"] & \mca \arrow[d, "F"] \\
			\mcb \arrow[r, "G"]                                               & \mcc               
		\end{tikzcd}
		$$ is a pullback in \prtil. By \cref{pullbackloc} it is already a pullback in $(\prlst)^{iL}$. Therefore  it suffices to show the following:
		\begin{enumerate}
			\item Both $p_{\mca}^R$ and $p_{\mcb}^R$ preserve the connective part.
			\item For any presentable $t$-category $\mce$ and a cocone in \prtil, 
			$$\begin{tikzcd}
				\mce \arrow[d, ] \arrow[r, ] & \mca \arrow[d, "F"] \\
				\mcb \arrow[r, "G"]                                               & \mcc               
			\end{tikzcd}$$
			the induced functor $\mce\to\mca\times_{\mcc}\mcb$ admits a right $t$-exact right adjoint.
		\end{enumerate}
		By the previous argument $p_{\mca}^R$ preserves  the connective part. To see $p_{\mcb}^R$ preserves too, we consider the following diagram of fiber sequences.
		$$
		\begin{tikzcd}
			\mck \arrow[r, "i_1"] \arrow[d, Rightarrow, no head] & \mca\times_{\mcc}\mcb \arrow[d, "p_{\mcb}"] \arrow[r, "p_{\mca}"] & \mca \arrow[d, "F"] \\
			\mck \arrow[r, "i_0"]                                & \mcb \arrow[r, "G"]                                               & \mcc               
		\end{tikzcd}
		$$
		By the connective recollement $i_0i_0^R\to \op{Id}_{\mcb}\to G^RG$, it suffices to show that both $p_{\mcb}^RG^R$ and $p_{\mcb}^Ri_0$ preserve  the connective part. The former is obvious because $p_{\mcb}^RG^R\simeq p_{\mca}^RF^R$. For the latter it suffices to show $p_{\mcb}^Ri_0\simeq i_1$, which follows from the proof of \cref{pullbackloc}.
		
		It remains to show (2). Let 
		\[
		H: \mathcal{E} \to \mathcal{D}, \qquad e \mapsto \big(H_A(e),\, H_B(e),\, \gamma\big)
		\] denote the induced functor.
		Since \(H_A\) and \(H_B\) satisfy the strongly $t$-continuous condition (they and their right adjoints \(H_A^R, H_B^R\) preserve colimits and the connective part), by \cref{pullbackloc}, $H$ is strongly continuous  and preserves the connective part.
		Therefore it suffices to show that the right adjoint \(H^R\) preserves the connective part.
		We give an explicit formula for \(H^R\). For \((a,b) \in \mathcal{D}\), using the pullback description of mapping anima, one obtains that \(H^R\) is given in \(\mathcal{E}\) by the homotopy pullback
		\[
		H^R(a,b) \simeq H_A^R(a) \times_{H_B^R G^R G(b)} H_B^R(b),
		\]
		fitting into the diagram
		\[
		\begin{tikzcd}
			H^R(a,b) \arrow[r] \arrow[d] & H_B^R(b) \arrow[d, "H_B^R(u_b)"] \\
			H_A^R(a) \arrow[r] & H_B^R G^R G(b)
		\end{tikzcd}
		\]
		(Here the bottom map arises from the structure morphism \(\alpha:F(a) \simeq G(b)\) in \(\mathcal{D}\) together with the adjoint compatibility \(H_A^R F^R \simeq H_B^R G^R\)).  
		
		Now let \((a,b) \in \mathcal{D}_{\geq 0}\), i.e. \(a \in \mathcal{A}_{\geq 0}\) and \(b \in \mathcal{B}_{\geq 0}\). We must show that
		\[
		H^R(a,b) \in \mathcal{E}_{\geq 0}.
		\]
		To ensure \(H^R(a,b) \in \mathcal{E}_{\geq 0}\), it suffices to show that the fiber $Z=\op{fib}(	H^R(a,b)\to H_A^R(a))$ is connective. Since $Z$ is also the fiber of $H_B^R(u_b)$ by the pullback diagram above, it suffices to observe that $H_B^R(u_b)$ is the image under \(H_B^R\) of a $\pi_0$-surjective map \(u_b: b \to G^R G(b)\) and hence $\pi_0$-surjective.
	\end{proof}
	Having established that connective internal localizations are stable under pullback in $\prtil$, we can now study how structural properties behave under such extensions. The next theorem is the pullback analogue of \cref{propertycoffib}: it shows that right completeness, compatibility with filtered colimits, $\mathrm{AB4}^*$, and dualizable additivity are inherited by the pullback from the two input terms.
	\begin{thm}\label{propertyext}
			Let $G:\mcb\to\mcc\in \prtil$ be a connective internal localization and $F:\mca \to\mcc \in\prtil$. Then:
			\enu{
				\item If $\mca,\mcb$ are right complete, then so is $\mca\times_{\mcc}\mcb$.
					\item If the $t$-structures on $\mca,\mcb$ are compatible with filtered colimits, then so is the $t$-structure on $\mca\times_{\mcc}\mcb$.
				\item If $\mca,\mcb$ satisfy $\mathrm{AB4}^*$, then so does $\mca\times_{\mcc}\mcb$.
				\item If $\mca,\mcb$ 
				lie in the image of the embedding $\praddbl\xhookrightarrow{}\prtil$, then so does $\mca\times_{\mcc}\mcb$.
			} 
	\end{thm}
	\begin{proof}
(1) Since $\mca,\mcb$ are right complete, the following diagram lies in \prtil.
	$$
\begin{tikzcd}
	\opsp(\mca_{\geq0}\times_{\mcc_{\geq0}}\mcb_{\geq0}) \arrow[d] \arrow[r] & \mca \arrow[d, "F"] \\
	\mcb \arrow[r, "G"]                                               & \mcc               
\end{tikzcd}
$$	
By \cref{pullbackt}, the induced functor $\opsp(\mca_{\geq0}\times_{\mcc_{\geq0}}\mcb_{\geq0})\to\mca\times_{\mcc}\mcb$ is thus strongly $t$-continuous. In particular, it is strongly continuous. Applying \cref{fivelemrightcompl} to the short exact sequence $$\mck\xhookrightarrow{i_1}\mca\times_{\mcc}\mcb\xrightarrow{p_A}\mca,$$  
it suffices to show that \mck is right complete.
 But that follows by applying \cref{propertycoffib} to the short exact sequence $\mck\to\mcb\to\mcc$.
 \\
(2) Let $X\in \mcd:=\mca\times_{\mcc}\mcb $ such that $X\simeq\colimit_\alpha X_\alpha$ be a filtered colimit in which each $X_\alpha \in \mcd_{\leq0}$. We wish to prove that $X\in \mcd_{\leq0}$ too. Consider the following recollement decomposition
$$p_A^R p_A^{RR} X_\alpha\to X_\alpha\to i_1^{RR}i_1^R X_\alpha;$$
it suffices to show that both $\colimit_\alpha\, i_1^{RR}i_1^R X_\alpha$ and $\colimit_\alpha \,p_A^R p_A^{RR} X_\alpha$ lie in $\mcd_{\leq0}$.  The latter holds because the $t$-structure on \mca is compatible with filtered colimits. To prove  $\colimit_\alpha\, i_1^{RR}i_1^R X_\alpha \in \mcd_{\leq0}$, we first note that the following two squares are both horizontally and vertically right adjointable from the proof of \cref{pullbackloc}.	$$
\begin{tikzcd}
	\mck \arrow[r, "i_1"] \arrow[d, Rightarrow, no head] & \mca\times_{\mcc}\mcb \arrow[d, "p_{\mcb}"] \arrow[r, "p_{\mca}"] & \mca \arrow[d, "F"] \\
	\mck \arrow[r, "i_0"]                                & \mcb \arrow[r, "G"]                                               & \mcc               
\end{tikzcd}
$$
Therefore by taking right adjoint to $i_1^R\simeq i_0^R p_B$, we have $i_1^{RR}\simeq p_B^R i_0^{RR}$. Consequently, $$\colimit_\alpha\, i_1^{RR}i_1^R X_\alpha\simeq\colimit_\alpha\, p_B^Ri_0^{RR}i_1^R X_\alpha \simeq  p_B^R(\colimit_\alpha i_0^{RR}i_1^R X_\alpha)$$ lies in $\mcd_{\leq0}$, because the $t$-structure on \mcb is compatible with filtered colimits.
\\
(3)
Let $X\in \mcd$ such that $X\simeq\prod_\alpha X_\alpha$ be a small product in which each $X_\alpha \in \mcd_{\geq0}$. We wish to prove that $X\in \mcd_{\geq0}$. Consider the following recollement decomposition
$$i_1i_1^R X_\alpha\to X_\alpha\to p_A^R p_A X_\alpha;$$
it suffices to show that both $\prod_\alpha i_1i_1^R X_\alpha$ and $\prod_\alpha p_A^R p_A X_\alpha$ lie in $\mcd_{\geq0}$. The latter holds because \mca satisfies $\mathrm{AB4}^*$. To prove the former, we first note that $i_1\simeq p_B^R i_0$ by the adjointability. Consequently, $$\prod_\alpha i_1i_1^R X_\alpha\simeq\prod_\alpha p_B^Ri_0 i_1^R X_\alpha \simeq  p_B^R(\prod_\alpha i_0i_1^R X_\alpha)$$ lies in $\mcd_{\geq0}$, because \mcb satisfies $\mathrm{AB4}^*$.
\\
(4) Now assume that $\mca_{\geq 0}, \mcb_{\geq 0}$ are dualizable additive, and that $\mca, \mcb$ are right 
complete. Our goal is to show that the same holds for $\mcd_{\geq 0}$ and $\mcd$. Applying \cref{propertycoffib} to 
the short exact sequence $\mck\to\mcb\to\mcc$, we conclude that $\mck_{\geq0}$ is dualizable additive and that 
$\mck$ is right complete. By \cref{main1} and the previous conclusions, $\mcd_{\geq 0}$ is Grothendieck satisfying 
$\mathrm{AB4}^*$ and it suffices to show that $\mcd$ is right complete and satisfies $\mathrm{AB6}$. The 
right completeness is clear because the connective part of \mcd is created component-wise. For the $\mathrm{AB6}$ 
condition, by \cref{dualadst}, it suffices to show that $\mcd \overset{\underset{\mathrm{def}}{}}{=}\mca\times_{\mcc}\mcb$ is dualizable stable. That follows from \cite[Proposition 1.87(2)]{efimov2024k}.
	\end{proof}
	As an immediate consequence, the preceding statement descends from the ambient category $\prtil$ to the category of dualizable additive \infcats itself. Thus connective internal localizations among dualizable additive categories admit pullbacks without leaving the dualizable additive world.
	\begin{cor}\label{connpullbackofdualadd}
	 Connective internal localizations in \praddbl (see \cref{defconnloc}) are closed under pullback, and such pullbacks are created in $\prl$.
	\end{cor}
	\begin{proof}
		Since $\praddbl\xhookrightarrow{}\prtil$ is a fully faithful embedding, the conclusion follows by combining \cref{pullbackt} and \cref{propertyext}.
	\end{proof}
	We finish the section by considering  compact projective generation. The following elementary criterion will be used to prove that a diagonal pullback along a connective internal localization remains compact projectively generated.
	
	\begin{lem}\label{pgenpst}
		Let $\mcc\in\prlpst$ be a presentable prestable \infcat. Let $S\subset\mcc$ be a set of  projective objects satisfying that for any non-zero $x\in\mcc$, there exists a non-zero morphism $\Sigma^n a\to x$ from some $a\in S$, where $n\geq0$. Then $S$ generates \mcc under small colimits. 
	\end{lem}
	\begin{proof}
	It suffices to show that the set of mapping functors $\{h^s: \mcc\xrightarrow{} \spgeq\mid s\in S\}$ is jointly 
	conservative. Since each $s\in S$ is projective in \mcc, the functor $h^s$ preserves geometric realizations. 
	By \cref{lemma4.3}, each $h^s$ preserves finite colimits. Therefore, it suffices to show that if 
	$h^s(x)=\unmap_{\mcc}(s,x)=0$ for each $s\in S$, then $x=0$. However, that follows from assumption.
	\end{proof}
	We can now prove an additive version of \cite[Proposition 1.88]{efimov2024k}. The point is that, for a connective internal localization, compact projective objects in the source can be paired along the localization to produce enough compact projective objects in the diagonal pullback.
	\begin{prop}\label{diagpullbackcpgen}
		Let $L:\mcc\to\mcd\in\praddbl$ be a connective internal localization of \dualaddinfcats. If $\mcc$ is compact projectively generated\footnote{This implies that \mcd is too.}, then so is the (diagonal) pullback $\mcc\times_{\mcd}\mcc$.
	\end{prop}

\begin{proof}
	Let $\mathcal{E} = \mcc \times_{\mcd} \mcc$. The strategy is to construct explicit compact projective generators for $\mathcal{E}$ out of compact projective objects of $\mcc$ and then apply \cref{pgenpst}. We claim  that $\mathcal{E}$ is generated by the set of objects
	$$\mathcal{G} = \{ (p, q, \alpha) \mid p, q \in \mcc^{\op{cproj}}, \alpha: Lp \xrightarrow{\simeq} Lq \}$$
	where $\mcc^{\op{cproj}}$ denotes the full subcategory of compact projective objects in $\mcc$.
	First, we show that any $G = (p, q, \alpha) \in \mathcal{G}$ is a compact projective object in $\mathcal{E}$. For any object $X = (x, y, f) \in \mathcal{E}$, the mapping connective spectrum in $\mce$ is given by the pullback:
	$$\unmap_{\mathcal{E}}(G, X) \simeq \unmap_\mcc(p, x) \times_{\unmap_\mcd(Lp, Lx)} \unmap_\mcc(q, y).$$
	Since $L$ is a connective internal localization, the induced map
	$$\unmap_\mcc(q, y) \to \unmap_\mcd(Lq, Ly) \simeq \unmap_\mcd(q, L^RLx)$$
	is surjective on $\pi_0$. In \spgeq, colimits commute with pullbacks provided that one of the legs is $\pi_0$-surjective. Since $p$ and $q$ are compact projective in $\mcc$, the functors $\unmap_\mcc(p, -)$ and $\unmap_\mcc(q, -)$ preserve colimits. By the $\pi_0$-surjectivity condition, the pullback also preserves colimits, which implies that $\unmap_{\mathcal{E}}(G, -)$ preserves colimits. Thus, $G$ is compact projective in $\mathcal{E}$.
	
It remains to prove that these compact projective objects generate the category. We verify the hypothesis of \cref{pgenpst} by showing that every non-zero object of $\mathcal{E}$ receives a non-zero map from an iterated suspension of an object of $\mathcal{G}$.

Suppose $X = (x, y, f)\neq0$ is a non-zero object in $\mathcal{E}$. Since $\mcc$ is dualizable additive and in particular separated, any non-zero object has a lowest non-vanishing homotopy group. Thus, there exists a minimal integer $n \geq 0$ such that either $\pi_0(\Omega^n x) \not\simeq 0$ or $\pi_0(\Omega^n y) \not\simeq 0$.
Consider the loop object $\Omega^n X = (\Omega^n x, \Omega^n y, \Omega^n f) \in \mathcal{E}$. 
By our minimal choice of $n$, there exist compact projective objects $p, q \in \mcc^{\op{cproj}}$ and maps $u \colon p \to \Omega^n x$ and $v \colon q \to \Omega^n y$, at least one of which is non-zero.
 (If, say, $\pi_0(\Omega^n y) = 0$, we simply choose $q = 0$ and $v = 0$.)

In $\mcd$, we have the composite map $f \circ Lu: Lp \to L(\Omega^n x) \simeq L(\Omega^n y)$. Because $L$ is a connective internal localization, the map
\[
\unmap_\mcc(p, \Omega^n y) \to \unmap_\mcd(Lp, L\Omega^n y)
\]
is $\pi_0$-surjective. Therefore, we can lift $f \circ Lu$ to a map $u': p \to \Omega^n y$ in $\mcc$. Similarly, we can lift $f^{-1} \circ Lv$ to a map $v': q \to \Omega^n x$.

We define the direct sum $c = p \oplus q \in \mcc^{\op{cproj}}$. The object $G_c = (c, c, \id_{Lc})$ belongs to $\mathcal{G}$. We can arrange the maps constructed above into matrices to define a morphism $\Phi: G_c \to \Omega^n X$ in $\mathcal{E}$:
\[
\Phi = \left( \begin{pmatrix} u \\ v' \end{pmatrix}, \begin{pmatrix} u' \\ v \end{pmatrix} \right).
\]
By construction, these maps intertwine with $\Omega^n f$, making $\Phi$ a well-defined morphism. Furthermore, since at least one of $u$ or $v$ is non-zero, the morphism $\Phi$ is non-zero, as desired.

\end{proof}

	\subsection{Dualizable additive short exact sequences}
Now we turn to short exact sequences of \dualaddinfcats. The results of the previous subsection allow us to transport the formalism of strongly $t$-continuous short exact sequences to \dualaddinfcats. 

We begin by spelling out the corresponding notions intrinsically in the prestable setting.

	\begin{de}\label{defconnloc}
	Let \mcc be a presentable prestable \infcat.
		\enu{
			\item  We say a full subcategory $i:\mcc'\subset \mcc$ is a \textbf{closed subcategory} if $i$ admits a colimit-preserving right adjoint. 
			\item Let $p: \mcc\to \mcd$ be a functor. We say $p$ is a \textbf{connective internal localization} 
			if $p$ admits a fully faithful colimit-preserving right adjoint $p^R$ such that for any $x\in\mcc$ the unit map $x\to p^Rp(x)$ is a $\pi_0$-epimorphism.	
		}
	\end{de}
	These definitions are chosen so that, after applying stabilization, they recover the connective internal localizations considered before in $\prtil$.
	
	\begin{rem}
	
Note that every closed subcategory and every connective internal localization of $\mcc$ is automatically presentable prestable. Furthermore, if $\mcc$ is a dualizable additive \infcat, it follows from \cref{ram149} that these are also dualizable additive.
	\end{rem}
	The next proposition is the additive analogue of \cref{shortexactt}. It says that, for dualizable additive \infcats, closed subcategories and connective internal localizations give rise to genuine fiber-cofiber sequences.
	
	\begin{prop}\label{basicses}
		
		\begin{enumerate}
		
			\item 	Let $i:\mcc\hookrightarrow\mcd$ be a fully faithful internal left adjoint of dualizable additive \infcats. Let $\mce=\coker(i)$. Then  the cofiber sequence $$\mcc\xrightarrow{i}\mcd\xrightarrow{p}\mce$$ is also a fiber sequence in both $\prlad$ and $\praddbl$.
			\item 	Let $L:\mcd\to\mce$ be a connective internal localization of dualizable additive \infcats.  Let $\mcc=\ker(L)$. Then  the fiber sequence $$\mcc\xrightarrow{i}\mcd\xrightarrow{p}\mce$$ is also a cofiber sequence in both $\prlad$ and $\praddbl$.
		\end{enumerate} 
	In both cases, we have a natural cofiber sequence $ii^R\to \op{Id}_\mcd\to p^Rp$ in $\funct^L(\mcd,\mcd)$.
	\end{prop}
	\begin{proof}
		Since we have the natural inclusion $\praddbl\xhookrightarrow{\opsp\otimes-}\prtil$, it follows by \cref{shortexactt} and \cref{propertyext}.
	\end{proof}
	Passing to compact projective generators recovers a familiar exactness statement for small additive \infcats. In particular, fully faithful functors between small idempotent-complete additive \infcats have Verdier-style cokernels which also compute the corresponding fibers.
	
	\begin{cor}\label{addkarseq}
		Let $F:\mca\to\mcb$ be a fully faithful additive functor between small idempotent-complete additive \infcats. Then  $$\mca\to\mcb\to\coker(F)$$ is a fiber-cofiber sequence in $\catadidem$.
	\end{cor}
As in the stable theory of Verdier localizations, a short exact sequence may be encoded either by its kernel or by its quotient. The preceding proposition gives this correspondence in the dualizable additive setting.

	\begin{prop}\label{connlocandclosub}
		Let $\mcc\in\praddbl$. Then we have a natural equivalence of posets $$\{\text{closed subcategories of }\mcc\}\simeq\{\text{connective internal localizations }p:\mcc\to \mcd\} .$$
	\end{prop}
	\begin{proof}
		It suffices to observe the following equivalences
		$$\funct^{\op{clsub}}(\Delta^1,\praddbl)\xleftarrow{\sim}\funct^{\op{bifib}}(\Delta^1\times\Delta^1,\praddbl)\xrightarrow{\sim}\funct^{\op{cn}\text{-}\op{loc}}(\Delta^1,\praddbl),$$
		which are induced by \cref{basicses}.
	\end{proof}

	\subsection{Almost \infcats and dualizable kernels}

The algebraic realization of dualizable additive \infcats is closely tied to almost mathematics. The guiding idea is that closed subcategories of connective module categories are controlled by $\pi_0$-surjective idempotent localizations of connective rings, and these localizations are classified by idempotent ideals.

 We first recall the relevant classification theorem.

	\begin{thm}[{\cite[Theorem B]{hebestreit2024note}}]\label{almostalg}
		Let $A$ be a connective $\mathbb{E}_k$-ring with $1 \leq k \leq \infty$. Consider the full subcategory $\mathrm{LQ}^k_A$ of $ \mathrm{Alg}_{\mathbb{E}_k}(\mathrm{Sp})_{A /}$ spanned by the maps $\varphi: A \rightarrow B$ for which
		\enu{
			\item  the multiplication $B \otimes_A B \rightarrow B$ is an equivalence, i.e. $\varphi$ is a localisation,
			\item $B$ is connective, and
			\item $\pi_0(\varphi): \pi_0 A \rightarrow \pi_0 B$ is surjective.
		}
		Then the functor
		$$
		\mathrm{LQ}^k_A \longrightarrow\left\{I \subset \pi_0 A \mid I^2=I\right\}, \quad \varphi \longmapsto \operatorname{ker}\left(\pi_0 \varphi\right)
		$$
		is an equivalence of categories, where we regard the target as a poset via the inclusion ordering. The inverse image of some $I \subset \pi_0(A)$ can be described more directly as $A / I^{\infty}$, where
		$$
		I^{\infty}=\lim _{n \in \mathbb{N}^{\mathrm{op}}} J_I^{\otimes_A n}
		$$
		with $J_I \rightarrow A$ the fibre of the canonical map $A \rightarrow \mathrm{H}\left(\pi_0(A) / I\right)$. This inverse system stabilises on $\pi_i$ for $n>i+1$.
		
		Furthermore, the image of the (fully faithful) restriction functor $$\operatorname{LMod}_{A / I^{\infty}}(\spgeq)\rightarrow \operatorname{LMod}_A(\spgeq)$$ consists exactly of those modules whose homotopy groups are killed by $I$.
		
	\end{thm}
	\begin{rem}\label{pi0almost}
		In the proof of \cite[Theorem B]{hebestreit2024note}, they also show that $\pi_0(A/I^\infty)=\pi_0A/I$, but note that $\pi_0I^\infty\neq I$ in general.
	\end{rem}

	\begin{de}
		Let $R$ be a connective \eonering and $I\subset \pi_0R$ be an idempotent (two-sided) ideal. We define the \infcat $\amodu_{(R,I)}(\spgeq)$ of connective almost $(R,I)$-modules  to be the kernel of $\modu_R(\spgeq)\to \modu_{R/I^\infty}(\spgeq)$.
	\end{de}
	We next compare the ring-theoretic localizations appearing in \cref{almostalg} with connective internal localizations of the corresponding module categories. This is a Morita-theoretic translation of the preceding classification.
		\begin{lem}\label{classsub}
		Let $R$ be a connective \eonering. Then the assignment $$(R\to S)\mapsto (\modu_{R}(\spgeq)\to\modu_{S}(\spgeq))$$ induces an equivalence of posets
		$$\theta:\mathrm{LQ}_R\xrightarrow{\simeq} \{\text{connective internal localizations } \modu_{R}(\spgeq)\to \mcd\}$$ where $\mathrm{LQ}_R$ denotes the full subcategory  of $ \mathrm{Alg}(\spgeq)_{R /}$ spanned by the maps $\varphi: R \rightarrow S$ for which
		\enu{
			\item  the multiplication $S \otimes_R S \rightarrow S$ is an equivalence, i.e. $\varphi$ is idempotent, and
			\item $\pi_0(\varphi): \pi_0 R \rightarrow \pi_0 S$ is surjective.
		}
	\end{lem}
	\begin{proof}
		Because the Morita embedding $\alg(\spgeq)\hookrightarrow \prl_{\op{ad},*/}$
		\cite[see][\textsection 4.8.5]{ha} induces an embedding 
		$$\alg(\spgeq)_{R/}\hookrightarrow \prl_{\op{ad},R/} $$ we see that $\theta$ 
		is fully faithful. For the essential surjectivity, we first observe that a 
		connective internal localization $p:\modu_{R}(\spgeq)\to \mcd$ satisfies Barr--Beck--Lurie
		 conditions \cite[Theorem 4.7.3.5]{ha}. Therefore there exists a connective 
		 \eonering map $f:R\to S$ such that $p=f_!$. Since $p$ is a connective internal localization, $f$ is 
		 idempotent and $\pi_0f$ is surjective.
	\end{proof}
	
	Combining this Morita-theoretic classification with the kernel--quotient correspondence from \cref{connlocandclosub}, we obtain a concrete description of all closed subcategories of a connective module category by almost mathematics.

	\begin{cor}\label{classdualclosesubcat}
		Let $R$ be a connective \eonering. Then the assignment $I\mapsto \amodu_{(R,I)}(\spgeq)$ gives an equivalence of posets
		$$\left\{\text{two-sided ideals } I \subset \pi_0 R \mid I^2=I\right\} \simeq \{\text{closed subcategories of }\modu_R(\spgeq)\} .$$
	\end{cor}
	\begin{proof}

		This follows from combining \cref{classsub}, \cref{connlocandclosub} and \cref{almostalg}.

	\end{proof}
	We can now identify arbitrary dualizable additive \infcats. Indeed, by the embedding theorem from \cref{section2}, every such category appears as a closed subcategory of a connective module category; the preceding corollary then identifies this closed subcategory with a category of connective almost modules:
	
		\begin{thm}\label{main2}
		Let $\mcc\in \prlad$. Then \mcc is dualizable additive if and only if there exists a connective \eonering $R$ and an idempotent ideal $I\subset \pi_0R$ such that $$\mcc\simeq \amodu_{(R,I)}(\spgeq).$$
	\end{thm}
	\begin{proof}
		By \cref{embedmodcat}, \mcc is a closed subcategory of  $\modu_R(\spgeq)$ for some connective \eonering $R$. Now by  \cref{classdualclosesubcat} we are done.
	\end{proof}
	
\begin{rem}
	\cref{main2} provides a refinement of \cite[Theorem 2.10.16]{krause2024sheaves}. While the latter states that any dualizable stable \infcat can be realized as the kernel of a homological epimorphism between $\mathbb{E}_1$-rings, our result implies that any dualizable additive \infcat can in fact be realized as the kernel of a $\pi_0$-surjective homological epimorphism between connective $\mathbb{E}_1$-rings.
\end{rem}
We finally discuss kernels of arbitrary morphisms in $\praddbl$. Unlike connective internal localizations, a general morphism need not have its ordinary kernel dualizable.

	\begin{lem}\label{dualker1}
		Let $\mcv\in\calg(\prl)$ and $f: \mathcal{M} \rightarrow \mathcal{N}$, $i:\mcn_0\hookrightarrow\mcn$ be maps in $\prvdbl$ such that $i$ is fully faithful. The dualizable pullback $(\mcm\times_\mcn\mcn_0)^{\mcv\text{-}\mathrm{dbl}}$ in $\prvdbl$ can be identified with the largest dualizable $\mcv$-submodule of $\mcm$ contained in the pullback $\mcm\times_\mcn\mcn_0$ in $\prl_{\mcv}$ and for which the inclusion into $\mcm$ is an internal left adjoint.
	\end{lem}
		\begin{proof}
		See \cite[Lemma 4.5.3]{Zariski_Bal}.
	\end{proof}
	Specializing this general description to the additive base $\spgeq$ gives the following concrete interpretation of dualizable kernels in $\praddbl$.
	
	\begin{cor}\label{dualkernel}
		Let $f: \mathcal{M} \rightarrow \mathcal{N}$ be a map in $\praddbl$. The dualizable kernel $\ker^d(f)$ can be identified with the largest dualizable additive subcategory of $\mcm$ contained in $\ker(f)$ and for which the inclusion functor into $\mcm$ is an internal
		left adjoint.
	\end{cor}

	The almost-module description allows us to compute these dualizable kernels explicitly for maps of connective \eonerings. The answer is obtained by extracting the largest idempotent ideal in the kernel on $\pi_0$:
	
	\begin{thm}\label{dualker}
		Let $R\to S$ be a map of connective \eonerings. Then the dualizable kernel can be identified with the \infcat of connective almost $(R,I_0)$-modules $$\ker^d(\modu_R(\spgeq)\to \modu_S(\spgeq))\simeq \amodu_{(R,I_0)}(\spgeq),$$ where $I_0$ is the maximal idempotent (two-sided) ideal of $\pi_0R$ such that  $I_0\subset \ker(\pi_0(R)\to \pi_0(S))$.
	\end{thm}
	\begin{proof}
		By \cref{dualkernel}, the dualizable kernel $\ker^d(\modu_R(\spgeq)\to \modu_S(\spgeq))$ can be identified with the maximal dualizable additive full subcategory of $$\ker(\modu_R(\spgeq)\to \modu_S(\spgeq))$$ such that the inclusion to $\modu_R(\spgeq)$ is an internal left adjoint. 
		
		Let us denote $$I=\ker(\pi_0(R)\to \pi_0(S)).$$ Given an idempotent ideal $J\subset \pi_0R$. By the characterization of $\modu_{R/J^\infty}(\spgeq)$  in \cref{almostalg}, we see that the functor $\modu_R(\spgeq)\to \modu_S(\spgeq)$ factors through the localization $\modu_R(\spgeq)\to \modu_{R/J^\infty}(\spgeq)$ if and only if $J\subset I$, as desired.
	\end{proof}
\begin{exam}
	For example, the dualizable additive kernel of the localization $$p:\mcd(\mathbb{Z})_{\geq0}\to \mcd(\mathbb{Z}[p^{-1}])_{\geq0}$$ is $0$ instead of $\mcd(\mathbb{Z})_{\geq0}^{p\text{-nil}}$.
\end{exam}
\subsection{Small additive short exact sequences}

To  motivate the necessity of the dualizable additive framework more, we briefly digress to examine the pathologies inherent in small additive \infcats. Specifically, we demonstrate through counterexamples that additive homological epimorphisms and Karoubi projections fail to be closed under pullback, a flaw  rectified by transitioning to the dualizable additive setting.

\begin{thm}[{\cite[Theorem 2.1.6]{dgm-elden-vova}}]\label{addquot}
	Let $\mathcal{B}$ be a small additive \infcat and $\mathcal{A} \subset \mathcal{B}$ be an additive full subcategory. 
	Define
	$\mathrm{W}:=\left\{f \oplus \op{id}_b: a \oplus b \rightarrow a^{\prime} \oplus b\right.$ where $f: a \rightarrow a^{\prime} \in \mathcal{A}$ and $b\in\mcb\}$. Then $\mathcal{B}[\mathrm{W}^{-1}]$ can be identified with the cofiber of $\mca\to\mcb$ in the \infcat of (small) additive \infcats $\Cat_{\op{ad}}$.
\end{thm}
\begin{nota}
	We will denote the quotient (the cofiber) in $\Cat_{\op{ad}}$ by $\op{quot}(\mathcal{A} \to \mathcal{B})$.
\end{nota}

\begin{de}
	Let $\mcc_0 \subset \mcc$ be a full subcategory of a \infcat \mcc. We say $\mcc_0$ is an idem-dense subcategory if the retraction closure of $\mcc_0$ in \mcc is \mcc itself.
\end{de}
\begin{rem}
	Note that an idem-dense subcategory is equivalent to a subcategory $\mcc_0$  which has the same idempotent completion as \mcc.
\end{rem}
\begin{exam}\label{densepullback}
	Let $\mca\hookrightarrow \mcb$ be a fully faithful additive functor between small additive \infcats.	If $\mathcal{A} \subset \mathcal{B}$ is idem-dense, the inclusion
	$$
	\mathcal{E} \times_{\mathcal{B}} \mathcal{A}\hookrightarrow\mce
	$$
	obtained by the pullback along an additive functor\footnote{even if $F$ is between idempotent-complete additive \infcats} $F \colon \mathcal{E} \to \mathcal{B}$ between small additive \infcats  needs not be idem-dense in $\mathcal{E}$. 
	
	For example, let $R = \mathbb{Z} \times \mathbb{Z}$ and  
$
	\mathcal{B} = \mathrm{Proj}(R)
	$
	be the additive category of finitely generated projective $R$-modules (hence idempotent-complete), and let
	$
	\mathcal{A} = \mathrm{Free}(R) \subset \mathcal{B}
	$
	be the full subcategory on finitely generated free $R$-modules. Therefore $\mathcal{A}$ is idem-dense in $\mathcal{B}$.
	Consider the (non-free) projective $R$-module
	$
	P_0 = \mathbb{Z} \times \{0\}.
	$
	Let
	$
	\mathcal{E} := \mathbf{add}(P_0)
	$
	be the additive full subcategory generated by $P_0$. Its objects are exactly the finite direct sums $P_0^{\oplus n} \cong \mathbb{Z}^n \times \{0\}$, and $\mathcal{E} \simeq \mathrm{Proj}(\mathbb{Z}) \simeq \mathrm{Free}(\mathbb{Z})$, hence $\mathcal{E}$ is idempotent-complete. Let $F \colon \mathcal{E} \hookrightarrow \mathcal{B}$ be the inclusion functor. 
	
	We claim the pullback $\mathcal{E} \times_{\mathcal{B}} \mathcal{A}$ is zero and hence not idem-dense in \mce. Since $F$ is an inclusion, the pullback $\mathcal{E} \times_{\mathcal{B}} \mathcal{A}$ identifies with the intersection $\mathcal{E} \cap \mathcal{A}$ inside $\mathcal{B}$. Objects of $\mathcal{E}$ have the form $\mathbb{Z}^n \times \{0\}$, while objects of $\mathcal{A}$ have the form
	\[
	R^m \cong (\mathbb{Z} \times \mathbb{Z})^m \cong \mathbb{Z}^m \times \mathbb{Z}^m.
	\]
	If $\mathbb{Z}^n \times \{0\} \cong \mathbb{Z}^m \times \mathbb{Z}^m$ as $R$-modules, comparing the second components forces $m=0$, whence the object is $0$. Thus
	\[
	\mathcal{E} \times_{\mathcal{B}} \mathcal{A} \simeq 0.
	\]
	
\end{exam}
\begin{rem}
	Unlike the stable case, the idempotent completion functor $$\Cat_{\op{ad}}\xrightarrow{(-)^{\natural}} \catadidem$$ does not admit a left adjoint, because by \cref{densepullback} idem-dense inclusions are not closed under pullback.
\end{rem}
\begin{de}
	Let $$\mca \hookrightarrow\mcb \to \mcc$$ be a sequence in \catadidem. We say that it is an additive Karoubi sequence if it is a bifiber sequence in \catadidem.
\end{de}
\begin{rem}
	By \cref{addkarseq}, if $\mca \to \mcb$ is fully faithful and the sequence $$\mca \hookrightarrow\mcb \to \mcc$$ is a cofiber sequence in \catadidem, then it is an additive Karoubi sequence.
\end{rem}
	\begin{de}
		We say a morphism $\mcb\to\mcc\in \catadidem$ is an additive Karoubi projection if it appears in an additive Karoubi sequence; we say that $\mcb\to\mcc\in \catadidem$ is an additive homological epimorphism if $\mcp_{\Sigma}(\mcb)\to\mcp_{\Sigma}(\mcc)$ is a connective internal localization.
	\end{de}
	\begin{rem}
		An additive Karoubi projection is always an additive homological epimorphism. An additive homological epimorphism is
		an additive Karoubi projection if and only if the kernel of the localization $\mcp_{\Sigma}(\mcb)\to\mcp_{\Sigma}(\mcc)$ is compact projectively generated.
	\end{rem}

\begin{rem}\,
	\enu{
		\item 	Additive homological epimorphisms are not closed under pullback, see \cref{exa1} for a counterexample.
\item 	Additive Karoubi projections are not closed under pullback, see \cref{exa2} for a counterexample.}

\end{rem}

\section{Flat generation}\label{section4}
A natural question is whether every dualizable additive \infcat admits enough projectives. This fails in general; see \cite[Example 3.1.24]{stefanich2023classification} for a counterexample consisting of a nontrivial \dualaddinfcat in which the only projective object is zero. However, we establish in this section that every dualizable additive \infcat admits enough flat objects.

We begin by recalling the relevant categorical notions of epimorphism and purity. These notions will provide the language in which flat objects are defined and compared across prestable \infcats.
	\begin{de}
		Let $\mcc$ be a prestable \infcat with finite limits.
		\enu{
			\item
			We say a morphism $f$ in \mcc is an \textbf{effective epimorphism} if $\pi_0(f)$ is an epimorphism in the abelian category $\mcc^\heartsuit$.
			\item We say a morphism $f$ in \mcc is a \textbf{split effective epimorphism} if $f$ admits a section.
			\item We say a morphism $f$ in \mcc is a \textbf{pure epimorphism} if $f$ is generated under filtered colimits by split effective epimorphisms when viewed as an object in $\mcc^{\Delta^1}$.
		}
	\end{de}
	This definition is designed so that purity is detected after testing against compact projective objects. We record a few elementary stability properties which will be used repeatedly below.
	\begin{rem}
		Pure epimorphisms are closed under finite sums and filtered colimits (that exist).
		Note that our notion of pure epimorphisms does not always agree with other ones in the literature, 
		such as in \cite[Definition 3.1]{martini2025locally}, where $f$ must be a filtered colimit of split effective epimorphisms (as opposed to being generated under filtered colimits by these).
		We suggest calling a map satisfying this stronger property a \textit{strictly pure epimorphism}.
	\end{rem}
	The following result, due to Stefanich, will be used as a basic tool for recognizing effective epimorphisms in the prestable setting.
	\begin{prop}[{\cite[Proposition 2.1.6 and Theorem 6.1.7]{stefanich2023derived}}]
		Let $\mcc$ be a prestable \infcat with finite limits. Then:
		\enu{ \item A morphism $f$ in \mcc is an effective epimorphism in \mcc if and only if the \v{C}ech nerve of $f$ is a colimit diagram.
			\item Let $f : X \to Y$ and $g : Y \to Z$ be morphisms in \mcc. If $f$ is an effective epimorphism,
			then $g$ is an effective epimorphism if and only if $g \circ f$ is an effective epimorphism.
			\item Effective epimorphisms are closed under pullbacks and pushouts along any map.
			\item Effective epimorphisms are closed under retracts.
		}
	\end{prop}
	We can now define the intrinsic notion of flatness which will be used throughout the rest of the section.
	\begin{de}\label{flatdef}
		Let $\mcc$ be a Grothendieck prestable \infcat. We say an object $x\in \mcc$ is \textbf{(categorically) flat} if any effective epimorphism $u\to x$ is a pure epimorphism (see \cite{martini2025locally} for a discussion at the 1-categorical level for strictly pure epimorphisms).
	\end{de}
	The terminology is justified by comparison with the standard notion of flatness for modules, as well as with its higher-categorical refinements in \cref{flatequiv}.
	\begin{rem}
	For a symmetric monoidal Grothendieck prestable \infcat $\mcc\in\op{CAlg}(\op{Groth})$, categorical flatness is in general \emph{different} from $\otimes$-flatness; see \cite{Estrada2012LocallyFP} for an example of a symmetric monoidal Grothendieck abelian category which is generated by $\otimes$-flat objects but has no nonzero categorically flat objects. In the additive rigid case, however, the two notions coincide, as will be shown in the forthcoming work \cite{gentc} of the second author.
	\end{rem}

	The following cancellation property is a version of \cite[Lemma 2.5]{martini2025locally} for our notion of pure epimorphisms.
	\begin{lem}\label{lem2.5}
		Let \mcc be a Grothendieck prestable \infcat. If $f: x \rightarrow y$ and $g: y \rightarrow z$ are morphisms such that $g \circ f$ is a pure epimorphism, then $g$ is a pure epimorphism.
	\end{lem}
\begin{proof}
	Let $\mathcal E\subseteq \mathcal C^{\Delta^1}$ be the class of
	morphisms $h:x\to z$ with the following property: for every
	factorization
	\[
	x\xrightarrow{f}y\xrightarrow{g}z,
	\qquad h=g\circ f,
	\]
	the morphism $g$ is a pure epimorphism.
	
	We claim that $\mathcal E$ contains all split effective epimorphisms
	and is closed under filtered colimits. Since the pure epimorphisms form
	the smallest class containing the split effective epimorphisms and
	closed under filtered colimits, this will prove the result.
	
	First, suppose that $h=g\circ f$ admits a section $s:z\to x$. Then
	$f\circ s:z\to y$ is a section of $g$, since
	\[
	g\circ f\circ s=h\circ s=\id_z.
	\]
	Thus $g$ is a split effective epimorphism, and hence is pure.
	
	Now let
	\[
	h\simeq \operatorname*{colim}_{i\in I}h_i,
	\qquad
	h_i:x_i\to z_i,
	\]
	be a filtered colimit of morphisms belonging to $\mathcal E$, and let
	$
	h=g\circ f
	$
	be a factorization through an object $y$. Set
	$
	y_i:=y\times_z z_i.
	$
	The structural morphism $h_i\to h$ induces a factorization
	\[
	x_i\longrightarrow y_i\xrightarrow{g_i}z_i
	\]
	of $h_i$. Since $h_i\in\mathcal E$, the morphism $g_i$ is pure.
	
	Since filtered colimits commute with finite limits in a Grothendieck
	prestable $\infty$-category,
	\[
	\operatorname*{colim}_{i\in I}y_i
	\simeq
	y\times_z\operatorname*{colim}_{i\in I}z_i
	\simeq y.
	\]
	Hence
	\[
	g\simeq\operatorname*{colim}_{i\in I}g_i
	\]
	in $\mathcal C^{\Delta^1}$. Since pure epimorphisms are closed under
	filtered colimits, $g$ is pure. Therefore $\mathcal E$ is closed under
	filtered colimits, and the result follows.
\end{proof}
We next examine the functoriality of purity and flatness under adjunctions. This will allow us to transport flat objects along the internal left adjoints which occur naturally among dualizable additive \infcats; see {\cite[Lemma 2.6]{martini2025locally}} for a result closely related to the following proposition:
	\begin{prop}
		Let $F: \mathcal{C} \rightleftarrows \mathcal{D}: G$ be an adjunction between Grothendieck prestable \infcats. The following assertions hold.
		\enu{
			\item $F$ preserves pure epimorphisms.
			\item If $F$ is fully faithful and $G$ preserves effective epimorphisms, then $F$ preserves flat objects.
			\item If $F$ is fully faithful and $G$ preserves filtered colimits, then $F$ reflects flat objects.
		}

	\end{prop}
	\begin{proof}
		Assertion (i) is clear since $F$ preserves retracts and filtered colimits.
		
		(ii) Let $x \in \mathcal{C}$ be a flat object and let $p: y \rightarrow F(x)$ be any effective epimorphism. Note that the unit $\lambda: \op{Id}_{\mathcal{C}} \Rightarrow G \circ F$ is a natural isomorphism. We denote by $\epsilon$ the counit $F \circ G \Rightarrow \op{Id}_{\mathcal{D}}$. By hypothesis, the map $G(p): G(y) \rightarrow (G \circ F)(x) \simeq x$ is an effective epimorphism, hence a pure epimorphism since $x$ is flat. It then follows from assertion (i) that $(F \circ G)(p):(F \circ G)(y) \rightarrow(F \circ G \circ F)(x) \cong F(x)$ is a pure epimorphism. But this last map decomposes as $(F \circ G)(y) \xrightarrow{\epsilon} y \xrightarrow{p} F(x)$. It then follows from \cref{lem2.5} that $p$ is a pure epimorphism. Therefore, $F(x)$ is a flat object.
		
		(iii) First note that $F$ preserves effective epimorphisms.
		Indeed, in a prestable $\infty$-category, a morphism $p$ is an effective
		epimorphism if and only if it is $\pi_0$-epimorphic.
		We next note that $G$ preserves pure
		epimorphisms, because it preserves filtered colimits. 
		
		Now let $x\in\mathcal C$ be such that $F(x)$ is flat, and let
		\[
		p:x'\longrightarrow x
		\]
		be an effective epimorphism. Then $F(p)$ is an effective epimorphism in
		$\mathcal D$. Since $F(x)$ is flat, $F(p)$ is pure. Hence $GF(p)$ is pure in
		$\mathcal C$.
		
		Since $F$ is fully faithful, the unit
		\[
		\eta:\operatorname{Id}_{\mathcal C}\xrightarrow{\sim}GF
		\]
		is an equivalence. By naturality, we have a commutative square
		\[
		\begin{tikzcd}
			x' \ar[r,"p"] \ar[d,"\eta_{x'}"']
			&
			x \ar[d,"\eta_x"]
			\\
			GF(x') \ar[r,"GF(p)"']
			&
			GF(x).
		\end{tikzcd}
		\]
		The vertical maps are equivalences, so $p$ is equivalent to $GF(p)$ in
		$\mathcal C^{\Delta^1}$. Therefore $p$ is pure. Hence $x$ is flat, and
		$F$ reflects flat objects.
	\end{proof}
	For dualizable additive \infcats, the hypotheses above are automatically satisfied in the situations of interest, yielding the following preservation and reflection statement.
	\begin{cor}\label{reflectflat}
		Let $F:\mcc\to\mcd$ be an internal left adjoint between dualizable additive \infcats. Then $F$ preserves flat objects. If furthermore $F$ is fully faithful, then $F$ reflects flat objects.
	\end{cor}

	The next proposition provides the fundamental compatibility check for the definition of flatness.
	\begin{prop}\label{flatequiv}
		Let $R$ be a connective \eonering. Then an object $M\in\modu_{R}(\spgeq)$ is flat in the sense of \cref{flatdef} if and only if $M$ is flat in the sense of \cite[Definition 7.2.2.10]{ha}.
	\end{prop}
	\begin{proof}
		The ``if'' direction follows from \cite[Lazard's Theorem 7.2.2.15(2)]{ha}. For the
		 ``only if'' direction, let $M\in\modu_{R}(\spgeq)$ be a (categorically) flat object.
		  It suffices to show that $-\otimes_RM: \rmodu_{R}(\spgeq)\to \spgeq$ preserves discrete 
		  objects by \cite[Theorem 7.2.2.15(5)]{ha}. Let $P\to M$ be an effective epimorphism 
		  where $P$ is projective. Then $P\to M$ is a pure epimorphism by categorical flatness 
		  of $M$. Let $K$ denote its kernel and $N$ be an arbitrary discrete $R$-module. Then 
		  $$N\otimes_RK\to N\otimes_RP\to N\otimes_RM$$ is a cofiber sequence in $\opsp$ such 
		  that $N\otimes_RP\to N\otimes_RM$ is a pure epimorphism. Note that pure epimorphism 
		  induces epimorphisms of homotopy groups at all degrees. So by long exact sequences of homotopy groups, we are done.
	\end{proof}
	For dualizable additive \infcats, flatness admits a useful dual characterization. This reformulation will be the main technical bridge between categorical flatness and the duality formalism developed earlier.
	\begin{prop}
		Let $\mcc$ be a dualizable additive \infcat. Then $x\in\mcc$ (viewed as a morphism $\spgeq\to\mcc$ in \prlad ) is flat if and only if the dual functor $x^\vee:\mcc^\vee\to \spgeq$ is left exact.
	\end{prop}
	\begin{proof}
		By \cref{flatequiv} we have seen that they agree in the case $\mcc=\modu_{R}(\spgeq)$ where $R$ is a connective \eonering, because the relative tensor product exhibits $\modu_{R}(\spgeq)^\vee \simeq \rmodu_{R}(\spgeq)$, \cite[see][Remark 4.8.4.8]{ha}.
		
		For the general case,  by \cref{main2} there exists a closed embedding $\mcc\hookrightarrow \modu_{R}(\spgeq)$ where $R$ is a connective \eonering. We will say an object $x\in \mcc$ is dually flat if $x^\vee:\mcc^\vee\to \spgeq$ is left exact. Then by \cref{reflectflat} it suffices to show that a fully faithful internal left adjoint $F:\mcc\to\mcd$ between \dualaddinfcats preserves and reflects dually flat objects, which is directly implied from definition.
	\end{proof}
	As an immediate consequence, flat objects enjoy the expected closure properties.
	\begin{cor}
		Let $\mcc$ be a \dualaddinfcat. Then
		flat objects in it are closed under finite sums and filtered colimits.
		
	\end{cor}
	\begin{proof}
		Let $x\in\mcc$. We observe that $x^\vee:\mcc^\vee\to \spgeq$ can be decomposed as $\mcc^\vee \xrightarrow{1\otimes x}\mcc^\vee \otimes \mcc\xrightarrow{ev} \spgeq$. In other words, we have $x^\vee= ev(-, x)$, so flat objects in \mcc are closed under finite sums and filtered colimits.
	\end{proof}
	The next result identifies flat objects in the compact projectively generated case. This is the higher-categorical analogue of Lazard's theorem:
	\begin{prop}[Lazard's Theorem]\label{lazardthm}
		Let $\mcc_0$ be a small additive \infcat. Then the full subcategory of $\mcc=\mcp_\Sigma(\mcc_0)$ spanned by flat objects can be identified with $\op{Ind}(\mcc_0)$. It can be also identified with $\mcp_{\sqcup}^{\sqcup,\op{fil}}(\mcc_0)$.
	\end{prop}
	\begin{proof}
		We learn the following argument from \cite[Proposition 2.2.22]{stefanich2023classification}. By the previous corollary, it suffices to show that any flat object $M\in\mcc$ can be written as a filtered colimit of objects in $\mcc_0$. We wish to show that the functor $M:\mcc_0^\opp\to\mcs$ defines an ind-object of $\mcc_0$. Equivalently, we will prove that $M:\mcc_0^\opp\to\mcs$  defines a pro-object of $\mcc_0^\opp$. 
		
		By additivity, $M$ admits a unique additive lifting $M':\mcc_0^\opp\to\spgeq$. Then it is not hard to show that the dual functor $M^\vee: \mcp_\Sigma(\mcc_0^\opp)\to \spgeq$ is left Kan extended from $M'$. Let $p: \mathcal{E} \rightarrow \mcp_\Sigma(\mcc_0^\opp)$ be the left fibration associated to the functor $\Omega^\infty M^\vee: \mcp_\Sigma(\mcc_0^\opp)\to\mc{S}$. Then the base change of $p$ to $\mcc_0^\opp$ is the left fibration classifying $M$. We have to show that every finite diagram $G: \mathcal{I} \rightarrow \mathcal{E} \times_{\mcp_\Sigma(\mcc_0^\opp)} \mcc_0^\opp$ admits a left cone. The fact that $M$ is flat implies that the functor $M^\vee$ is left exact, and therefore $\mce$ is cofiltered and $G$ extends to a left cone $G^{\triangleleft}: \mathcal{I}^{\triangleleft} \rightarrow \mathcal{E}$.

		Let $\overline{N}=(N,y) $ be the value of $G^{\triangleleft}$ at the cone point, where $N\in \mcp_\Sigma(\mcc_0^\opp)$ and $y\in \Omega^\infty M^\vee(N)$. To show that $G$ extends to a left cone in $\mathcal{E} \times_{\mcp_\Sigma(\mcc_0^\opp)} \mcc_0^\opp$ it is enough to prove that $\overline{N}$ receives a map from an object in $\mathcal{E} \times_{\mcp_\Sigma(\mcc_0^\opp)} \mcc_0^\opp$. Now we observe that $N\simeq\op{colim}_\alpha h_{X_\alpha}$,  and hence $\pi_0M^\vee(N)\simeq \op{colim}_\alpha \pi_0M^\vee(h_{X_\alpha})$ in $\op{Ab}$, where  each $X_\alpha\in \mcc_0^\opp$. Therefore there exists an object $X\in\mcc_0^\opp$, $x\in \pi_0M^\vee(h_X)$ and a morphism $h_X\to N \in \funct(\Delta^1,\mcp_\Sigma(\mcc_0^\opp))$ which sends $x$ to $y$, i.e., there exists a morphism $(h_X,x)\to (N,y)$ in \mce.
	\end{proof}
	We now pass to an arbitrary dualizable additive \infcat. Combining the preceding description and structural results, we obtain the expected sufficient flat objects:
	\begin{thm}\label{flatgenerate}
		Let $\mcc\in\praddbl$ be a \dualaddinfcat. Then $\mcc$ is generated by $\omega_1$-flats (meaning $\omega_1$-compact flat objects) under small colimits.
	\end{thm}
	\begin{proof}
		We learn the following argument from \cite[Lemma 2.4]{quillen1997module}. Let $\mcc'\subset \mcc$ be the smallest subcategory which contains $\omega_1$-flats and is closed under small colimits. We wish to show $\mcc'=\mcc$.   
		
		By \cref{main2}, there exists a connective \eonering $R$, an idempotent ideal $I\subset \pi_0R$ and a closed embedding $i:\mcc \hookrightarrow\modu_{R}(\spgeq)$ such that $\mcc\simeq \amodu_{(R,I)}(\spgeq)$. We have $i^R=I^\infty\otimes_R(-)$, where $I^\infty$ is viewed as an $(R,R)$-bimodule. Now let $x\in \mcc^{\omega_1}$ be an $\omega_1$-compact object. Then $i(x)$ is $\omega_1$-compact in $\modu_{R}(\spgeq)$ and hence there exists an $\omega_1$-projective left $R$-module $P$ and an effective epimorphism $g:P\to i(x)$ in $\modu_{R}(\spgeq)$. 
		Since the induced map $$g':I^\infty \otimes_R P \rightarrow I^\infty \otimes_R i(x) \xrightarrow{\sim} i(x)$$ is also effective epimorphic, there is a map $f': P \rightarrow I^\infty \otimes_R P$ such that $g' f'=g$. Let $F$ be the colimit of the system $\left(P_n\right)_{n \geq 0}$ over $i(x)$, where $P_n=P$ and $P_n \rightarrow P_{n+1}$ is the composition $$f:P\xrightarrow{f'}I^\infty \otimes_R P\to P$$ for all $n$. Since $F$ is a sequential colimit of $\omega_1$-projectives, it is $\omega_1$-flat in $\modu_{R}(\spgeq)$. Also $I^\infty \otimes_R F\simeq F$, since for each $n\geq0$ the map $f:P_n\to P_{n+1}$ factors through $I^\infty \otimes_R P_{n+1}$. Therefore $F\in \mcc$ and $F$ is $\omega_1$-flat in \mcc by \cref{reflectflat}. Finally, the induced map $F \rightarrow i(x)$ given by $g: P_n \rightarrow i(x)$ is effective epimorphic. Now because \mcc is $\omega_1$-presentable, we have actually shown that the full subcategory of $\omega_1$-flats in \mcc forms a $0$-generating subcategory of \mcc, and hence  $\mcc'=\mcc$ by \cref{gplthm}.
	\end{proof}

		\begin{rem}
		By \cref{flatgenerate}, any object in a \dualaddinfcat \mcc can be written as a geometric realization of flat objects (via flat resolution). One might ask if  $\mcc$ is freely generated by flat objects under geometric realizations, namely if one has an equivalence $\mcc\simeq \mcp_{\emptyset}^{\Delta^{\opp}}(\mcc^{\op{flat}})$.
		 However, this is not true, because $\mcp_{\emptyset}^{\Delta^{\opp}}(\mcc^{\op{flat}})$ is projectively generated\footnote{Every object $X\in \mcc^{\op{flat}}$ is projective in $\mcp_{\emptyset}^{\Delta^{\opp}}(\mcc^{\op{flat}})$.}, while a \dualaddinfcat is  in general not. 
	\end{rem}
\subsection{Dualizable versus compactly assembled additive \infcats}
	
	In this subsection, we show that the full subcategory of flat objects is compactly assembled. Furthermore, this yields a canonical equivalence between the category of compactly assembled additive categories and that of dualizable additive categories, with the inverse given by continuous prestabilization.
	
	We begin by relating the flatness condition to the compact projective exhaustibility. This comparison is the key input for showing that the category of flat objects has a compactly assembled structure.
	
	\begin{prop}\label{flatcpexhau}
		Let $\mathcal{C}\in\praddbl$ be a \dualaddinfcat and $x\in\mcc$. Then the following are equivalent:
		\enu{
			\item $x$ is $\omega_1$-flat.
			\item $x$ is \spgeq-atomically exhaustible (see \cref{atomicexhau}).
			\item $x$ is compact projectively exhaustible.
			
		}
	\end{prop}
	\begin{proof}
		The implication  $(2)\Rightarrow(3) $ is clear; it suffices to show $(1)\Rightarrow(2) $ and $(3)\Rightarrow(1) $.
		
		For $(3)\Rightarrow(1) $, we first choose a compact projective exhaustion $x\simeq \colim_n x_n$ and a strongly continuous embedding $i:\mcc\to\modu_{R}(\spgeq)$. Then $i(x)\simeq\colim_n i(x_n) $ is a compact projective exhaustion in $\modu_{R}(\spgeq)$. Because for any $n$, the object $i(x_{n+1})\in \modu_{R}(\spgeq)=\mcp_\Sigma(\op{Free}^{\op{fin}}_R)$ can be written as a sifted colimit $i(x_{n+1})\simeq\colimit_I P_i$ of finite free modules by \cref{siftrep2}, the map $i(x_n)\to i(x_{n+1})$ factors through a finite free module $P_i$. Therefore $i(x)$ is equivalent to a sequential colimit of finite free modules and $i(x)$ is $\omega_1$-flat. Because $i$ reflects $\omega_1$-flats, $x$ is $\omega_1$-flat.
		
		For $(1)\Rightarrow(2)$, let $x\in \mcc^{\omega_1}$ be an $\omega_1$-flat object. Then $i(x)$ is $\omega_1$-flat in $\modu_{R}(\spgeq)$ and hence $i(x)\simeq \colim_n P_n$ is a sequential colimit of finite free modules. Since we also have $i(x)\simeq \colim_n ii^R(P_n)$, the map $P_0 \rightarrow i(x) \simeq \colim_n ii^R(P_n)$ factors through some $ii^R(P_{j'_0})$. Using compactness again, there is a $j_0 \geq j_0^{\prime}$ such that the composite $P_{0} \rightarrow i i^R(P_{j_0}) \rightarrow P_{j_0}$ is equivalent to the transition map. Thus, for any $n$, we find a $j_{n+1} > j_n$ such that $P_{j_n} \rightarrow P_{j_{n+1}}$ factors through the counit $i i^R(P_{j_{n+1}}) \rightarrow P_{j_{n+1}}$. Therefore $x$ can be written as a sequential colimit $\colim_ni^R(P_{j_n})$ such that each transition map factors through a compact projective object in $\modu_{R}(\spgeq)$, which shows $x$ is \spgeq-atomically exhaustible.
	\end{proof}

This characterization immediately gives a control on the size of the full subcategory of flat objects.

\begin{prop}\label{flatacc}
	Let $\mcc$ be a \dualaddinfcat. Then $\mcc^{\op{flat}}$ is $\omega_1$-accessible. Moreover, we have equivalences
	\[
	\mcc^{\omega_1\text{-}\op{flat}} \simeq (\mcc^{\op{flat}})^{\omega_1} 
	\quad \text{and} \quad 
	\op{Ind}_{\omega_1}(\mcc^{\omega_1\text{-}\op{flat}}) \simeq \mcc^{\op{flat}}.
	\]
	In particular, the inclusion $\cflat \hookrightarrow \mcc$ preserves and reflects $\omega_1$-compact objects.
\end{prop}
\begin{proof}
	It is clear that $\mcc^{\omega_1\text{-}\op{flat}}\subset(\mcc^{\op{flat}})^{\omega_1} $. 
	Because $\mcc^{\omega_1\text{-}\op{flat}}$ is idempotent complete, it suffices to show that $\mcc^{\omega_1\text{-}\op{flat}}$ generates $\mcc^{\op{flat}}$ under $\omega_1$-filtered colimits. Choose a strongly continuous embedding $i:\mcc \to \modu_{R}(\spgeq)$ where $R$ is a connective \eonering. We then basically follow the same idea in the proof of \cite[Proposition 2.27]{ramzi2024dualizable}.
	
	Let $m$ be a flat object in \mcc. We denote $\mcn=\modu_{R}(\spgeq)$. Since $i(m)$ is flat in \mcn, we can write $i(m) \simeq \operatorname{colim}_J n_j$ for some $\omega_1$-filtered poset diagram $n_{\bullet}: J \rightarrow \mcn^{\omega_1\text{-}\op{flat}}$, which exists because for $\mcn$ it is true that $\op{Ind}_{\omega_1}(\mcn^{\omega_1\text{-}\op{flat}})\simeq\mcn^{\op{flat}}$. We thus have $m \simeq \operatorname{colim}_J i^R\left(n_j\right)$. Let $\operatorname{Down}_\omega(J)$ denote the poset of downwards closed filtered subsets of $J$ that have a countable cofinal subset - note that this is a filtered subset and we have an inclusion $J \rightarrow \operatorname{Down}_\omega(J), j \mapsto\{\leq j\}$. Left Kan extend the diagram $i^R\left(n_{\bullet}\right)$ to $\operatorname{Down}_\omega(J)$ to get a diagram $\tilde{n}_{\bullet}: \operatorname{Down}_\omega(J) \rightarrow \mathcal{C}$.
	
	We claim that the subposet $P \subset \operatorname{Down}_\omega(J)$ spanned by those $F$ such that $\tilde{n}_F$ is $\omega_1$-flat in \mcc is cofinal and $P$ is $\omega_1$-filtered. Once we know this, we will be done: it will follow that $m$ is an $\omega_1$-filtered colimit of $\omega_1$-flat objects.
	
	To prove this cofinality, since $\op{Down}_\omega(J)$ is a filtered poset, it suffices to produce, for every $F_0$, a bigger $F$ with the desired property. Now $F_0 \subset J$ has a countable cofinal subset, and $J$ is $\omega_1$-filtered, so we can find a $j_0 \in J$ which is an upper bound for $F_0$. Thus we may assume $F_0=\left\{\leq j_0\right\}$.
	
	In this case, $\tilde{n}_{j_0}=i^R(n_{j_0})$. Now we note that since $n_{j_0}$ is $\omega_1$-compact, the map $n_{j_0} \rightarrow m \simeq \operatorname{colim}_J i i^R(n_k)$ factors through some $i i^R(n_{j_1^{\prime}})$. Using again $\omega_1$-compactness, there is a $j_1 \geq j_1^{\prime}, j_0$ such that the composite $n_{j_0} \rightarrow i i^R(n_{j_1}) \rightarrow n_{j_1}$ is the map induced by $n_{\bullet}$. Thus, for any $j_0$, we find a $j_1 \geq j_0$ such that $n_{j_0} \rightarrow n_{j_1}$ factors through the counit $i i^R(n_{j_1}) \rightarrow n_{j_1}$. We can then iterate from $j_1$ to find a $j_2$, and then a $j_n$ and so on, and we can let $F$ be the downwards closure $\left\{j_n, n \in \mathbb{N}\right\}$. For this $F$, we will have $i(\tilde{n}_F)=\operatorname{colim}_{j \in F} ii^R\left(n_j\right)=\operatorname{colim}_n ii^R\left(n_{j_n}\right) \simeq \operatorname{colim}_n (n_{j_n})$. The last equivalence follows from the fact that in the following diagram, if there exists a dotted lift making the lower right triangle commute
	$$\begin{tikzcd}[row sep=large, column sep=large]
		ii^R(n_0) \arrow[r, "\epsilon_{n_0}"] \arrow[d, "ii^R(f)"'] & n_0 \arrow[d, "f"] \arrow[ld, dotted] \\
		ii^R(n_1) \arrow[r, "\epsilon_{n_1}"'] & n_1
	\end{tikzcd}$$
	then the upper left triangle automatically commutes by adjunction: maps $ii^R(n_0) \to ii^R(n_1)$
	are entirely determined by the composition $ii^R(n_0) \to ii^R(n_1) \to n_1$.
	Therefore $i(\tilde{n}_F)$ is a sequential colimit of $\omega_1$-flats in \mcn, so it is still $\omega_1$-flat in \mcn. However  $i$ reflects $\omega_1$-flats so $\tilde{n}_F$ is  $\omega_1$-flat in \mcc.
	
	It remains to show that $P$ is $\omega_1$-filtered. Since it is cofinal in the poset $\op{Down}_\omega(J)$, it suffices to show it for the
	latter, which can be found in the proof of \cite[Corollary 2.34]{ramzi2024dualizable}.
\end{proof}
The preceding proposition may be reformulated in several equivalent cocompletion languages:
	\begin{rem}\label{altercocomp}
		Let  ``$\omega_1\text{-}\op{fil}$'' denote  the collection of small $\omega_1$-filtered diagrams and ``$\op{fil}_{\omega_1}$'' denote  the collection of $\omega_1$-small filtered diagrams.
		For any small additive \infcat $\mca_0$, because $\op{Ind}_{\omega_1}(\mca_0)\subset\mcp_\Sigma(\mca_0)$ is closed under finite products, it is additive and by the same reason of \cref{ab6prod} we have $\op{Ind}_{\omega_1}(\mca_0) \simeq \mcp_{\sqcup}^{\sqcup,\omega_1\text{-}\op{fil}}(\mca_0)$. If furthermore $\mca_0$ admits $\omega_1$-small filtered colimits (equivalently admits sequential colimits), then $$\op{Ind}_{\omega_1}(\mca_0)\simeq   \mcp_{\op{fil}_{\omega_1}}^{\op{fil}}(\mca_0)\simeq
		\mcp_{\sqcup,\op{fil}_{\omega_1}}^{\sqcup,\op{fil}}(\mca_0) ,$$ where the first equivalence follows from \cite[\href{https://kerodon.net/tag/0694}{0694} and \href{https://kerodon.net/tag/0695}{0695}]{lurie2025kerodon} and the second follows similarly from \cref{ab6prod}.
		Therefore, for any \dualaddinfcat \mcc we have $$\cflat\simeq\op{Ind}_{\omega_1}(\cflatw)\simeq \mcp_{\sqcup}^{\sqcup,\omega_1\text{-}\op{fil}}(\cflatw)\simeq  \mcp_{\op{fil}_{\omega_1}}^{\op{fil}}(\cflatw)\simeq \mcp_{\sqcup,\op{fil}_{\omega_1}}^{\sqcup,\op{fil}}(\cflatw).$$ 
	\end{rem}

The proof of compact assembledness will require a small amount of algebra concerning idempotent ideals and their matrix analogues.
	\begin{lem}\label{matrixlem1}
		Let $A$ be a unital associative ring and $I=I^2$ be an idempotent two-sided ideal of $A$. Then for any $n\in\mathbb{N}$, the matrix ideal $\mathbb{M}_n(I)\subset \mathbb{M}_n(A)$ is idempotent too, i.e. $\mathbb{M}_n(I)=\mathbb{M}_n(I)^2$.
	\end{lem}
	\begin{proof}
		We claim every $X \in \mathbb{M}_n(I)$ can be written as $X=\sum_{i=1}^N A_i B_i$ with $A_i, B_i \in \mathbb{M}_n(I)$. This can be shown as follows:
		
		For all $1 \leq i, j \leq n$ and $a \in I$ let $E_{i j}(a) \in \mathbb{M}_n(I)$ denote the matrix with $a$ as the $(i, j)$-th entry and all other entries $0$. Because every matrix in $\mathbb{M}_n(I)$ is the sum of such matrices it suffices to show the statement for these matrices. Because $I^2=I$ there exist $a_1, b_1, \ldots, a_N, b_N \in I$ with $a=\sum_{k=1}^N a_k b_k$. Thus
		$$
		E_{i j}(a)=E_{i j}\left(\sum_{k=1}^N a_k b_k\right)=\sum_{k=1}^N E_{i j}\left(a_k b_k\right)=\sum_{k=1}^N E_{i 1}\left(a_k\right) E_{1 j}\left(b_k\right) .
		$$
	\end{proof}
	This matrix observation provides the following factorization criterion for maps between finite free modules in the almost setting.
	\begin{lem}\label{factorize}
		Let $R$ be a connective \eonering and $I\subset \pi_0R$ be an idempotent (two-sided) ideal. Given a map $f:R^m\to R^n$ of finite free left modules, then:
		\enu{
			\item  $f$ factors through an object in $\amodu_{(R,I)}(\spgeq)$ if and only if the associated $m\times n$ matrix $A_f\in \mathbb{M}_{m\times n}(\pi_0R)$  lies in $\mathbb{M}_{m\times n}(I)$.
			\item If $f$ factors through an object in $\amodu_{(R,I)}(\spgeq)$, then $f$ can be decomposed as two maps $R^m \xrightarrow{g} R^k \xrightarrow{h} R^n$ such that both $g,h$ factor through  $\amodu_{(R,I)}(\spgeq)$.
		}
	\end{lem}
	\begin{proof}
		(1)  Suppose $f$ factors as $R^m \xrightarrow{\alpha} X \xrightarrow{\beta} R^n$ for some object $X \in \amodu_{(R,I)}(\spgeq)$.  The morphism $f$ induces a map $\pi_0(f)$ of left $\pi_0R$-modules, which factors as $$(\pi_0 R)^m \xrightarrow{\pi_0(\alpha)} \pi_0(X) \xrightarrow{\pi_0(\beta)} (\pi_0 R)^n.$$ Since the image of $\pi_0(\alpha)$ is contained in $\pi_0(X)=I \cdot \pi_0(X)$, the image of $\pi_0(f)$ must lie in $I \cdot (\pi_0 R)^n = I^n$. Thus, the matrix $A_f \in \mathbb{M}_{m\times n}(\pi_0 R)$ associated to $f$ has all its entries in $I$.
		
		 Conversely, assume $A_f \in \mathbb{M}_{m\times n}(I)$. 
		Then $f$ factors through $(I^\infty)^n$ by the following diagram: $$\begin{tikzcd}
			& R^m \arrow[ld, dashed] \arrow[d, "f"] \arrow[rd, "0"] &                \\
			(I^\infty)^n \arrow[r] & R^n \arrow[r]                                    & (R/I^\infty)^n
		\end{tikzcd}$$	
		(2) Suppose $f$ factors through an object in $\amodu_{(R,I)}(\spgeq)$. By part (1), its matrix $A_f$ has all entries in $I$, meaning $A_f \in \mathbb{M}_{m\times n}(I)$. Because $I$ is an idempotent ideal, the matrix ideal $\mathbb{M}_n(I) \subset \mathbb{M}_n(\pi_0R)$ is an idempotent ideal too by \cref{matrixlem1}. Viewing $\mathbb{M}_{m\times n}(I)$ as a right module over $\mathbb{M}_n(\pi_0R)$, we also have $\mathbb{M}_{m\times n}(I) = \mathbb{M}_{m\times n}(I) \cdot \mathbb{M}_n(I)$ by the same argument as \cref{matrixlem1}. Therefore, $A_f$ can be written as a finite sum of products of matrices:
		$$ A_f = \sum_{l=1}^p B_l C_l $$
		where each $B_l \in \mathbb{M}_{m\times n}(I)$ and $C_l \in \mathbb{M}_{n\times n}(I)$. 
		We can rewrite this sum of products as a single matrix multiplication by organizing $B_l$ and $C_l$ into block matrices. Let $k = p \cdot n$. Define $B \in \mathbb{M}_{m\times k}(I)$ as the horizontal concatenation of $B_l$, and $C \in \mathbb{M}_{k\times n}(I)$ as the vertical concatenation of $C_l$:
		$$ B = \begin{pmatrix} B_1 & B_2 & \cdots & B_p \end{pmatrix}, \quad C = \begin{pmatrix} C_1 \\ C_2 \\ \vdots \\ C_p \end{pmatrix} $$
		Then we precisely have $A_f = B \cdot C$. 
		
		Since $R^m, R^k$, and $R^n$ are finite free left modules, this algebraic decomposition strictly lifts to a composition of $R$-module morphisms $R^m \xrightarrow{g} R^k \xrightarrow{h} R^n$ such that $A_g = B$ and $A_h = C$. Since both matrices $A_g$ and $A_h$ have entries exclusively in $I$, applying the ``if'' direction of part (1) guarantees that both $g$ and $h$ factor through objects in $\amodu_{(R,I)}(\spgeq)$.
	\end{proof}
	We now recall the categorical framework in which the flat subcategory will live. 
	\begin{de}
		Let $\hatcat^{\op{fil}}$ be the \infcat of large \infcats with small filtered colimits and whose morphisms are  filtered colimit-preserving functors. We say an object $\mcc$ in it is compactly assembled if there exists an $\omega$-accessible \infcat $\mcd=\op{Ind}(\mcd^\omega)$ such that $\mcc$ is a retract of \mcd in $\hatcat^{\op{fil}}$. 
		
		Let $F: \mcc\to \mcd$ be a functor between compactly assembled \infcats. We say $F$ is compactly assembled if it preserves filtered colimits and compact maps. 
		
		We denote $\ass\subset\hatcat^{\op{fil}}$ as the (non-full) subcategory spanned by compactly assembled \infcats and compactly assembled functors. We denote $\assad\subset\ass$ as the (non-full) subcategory spanned by compactly assembled additive \infcats and  compactly assembled additive functors.
	\end{de}
	The following elementary characterization will be useful for verifying compact assembledness in additive contexts.
	\begin{lem}\label{retraassad}
		Let $\widehat{\op{Cat}}_{\infty,\op{ad}}^{\op{fil}}$ be the \infcat of large additive \infcats with small filtered colimits and whose morphisms are filtered colimit-preserving additive functors. Then an object $\mcc$ in $\widehat{\op{Cat}}_{\infty,\op{ad}}^{\op{fil}}$ is compactly assembled if and only if there exists an $\omega$-accessible additive \infcat $\mcd=\op{Ind}(\mcd^\omega)$ such that $\mcc$ is a retract of \mcd in $\widehat{\op{Cat}}_{\infty,\op{ad}}^{\op{fil}}$.
	\end{lem}
	\begin{proof}
		The ``if'' direction is clear; it suffices to show the ``only if'' direction. Let $\mcc\in\widehat{\op{Cat}}_{\infty,\op{ad}}^{\op{fil}}$. If \mcc is compactly assembled, then $c:\op{Ind}(\mcc^{\omega_1})\to\mcc$ admits a fully faithful left adjoint $\hat{h}$. Since adjoint functors are automatically additive, they exhibit \mcc as a retract of $\op{Ind}(\mcc^{\omega_1})$  in $\widehat{\op{Cat}}_{\infty,\op{ad}}^{\op{fil}}$.
	\end{proof}
The preceding accessibility and factorization results combine to give compact assembledness of the flat subcategory.	
	\begin{thm}\label{flcompass}
		Let \mcc be a \dualaddinfcat. Then \cflat is a compactly assembled additive \infcat.
	\end{thm}
	\begin{proof}[First proof]\!\!\!\footnote{We thank Germán Stefanich for pointing out this easier argument.}
		By \cref{embedmodcat} and \cref{flatgenerate}, we obtain a hat Yoneda functor $$\mcc\xhookrightarrow{\hat{h}}\mcp_\Sigma(\mcc^{\omega_1\text{-}\op{flat}}).$$ However, both it and its right adjoint $\mcp_\Sigma(\mcc^{\omega_1\text{-}\op{flat}})\to \mcc$ preserve flat objects, so we get a retract $$\mcc^{\op{flat}}\to \mcp_\Sigma(\mcc^{\omega_1\text{-}\op{flat}})^{\op{flat}}\simeq\ind(\mcc^{\omega_1\text{-}\op{flat}})\to \mcc^{\op{flat}}$$ in $\widehat{\op{Cat}}_{\infty,\op{ad}}^{\op{fil}}$. Consequently, $\mcc^{\op{flat}}$ is compactly assembled additive.
	\end{proof}
	We also give a second proof, which is slightly more constructive and will be useful for the compact exhaustion statement below.
	\begin{proof}[Second proof]\label{secproof}
		Since $\cflat$ is $\omega_1$-accessible, it suffices to show that any $\omega_1$-flat object $x\in\mcc$ is compactly exhaustible internally in $\cflat$ by \cite[Theorem 2.39]{ramzi2024dualizable}. Beware that this does not follow immediately from the existence of a compact projective exhaustion in $\mcc$, because the exhaustion appearing in \cref{flatgenerate} does not necessarily lie in \cflat. We need a more clever strategy.
		
		Now choose a strongly continuous embedding $i:\mcc \to \modu_{R}(\spgeq)$, which induces an identification $\mcc=\amodu_{(R,I)}(\spgeq)$, where $I\subset \pi_0R$ is an idempotent two-sided ideal. Then $X:=i(x)$ is an $\omega_1$-flat left $R$-module and hence can be written as a sequential colimit $X\simeq\colim_n X_n$ of finite free left $R$-modules. Since we also have $X\simeq\colim_n ii^R(X_n)$, each map $X_i \to X_j$ in the sequence eventually factors through an object in \mcc. Without loss of generality, we can assume all maps $X_n \to X_{n+1}$ in the sequence factor through \mcc.
		
		For each $n$, by \cref{factorize} we can decompose $X_n\to X_{n+1}$ into a composition $$X_n=X_{n,0}\to X_{n,1} \to X_{n+1}$$ of two maps such that $X_{n,1}$ is finite free and both maps factor through \mcc. We can then iterate from $X_{n,1}$ to find a $X_{n,2}$, and then a $X_{n,i}$ and so on, which produces an $\mathbb{N}^{\triangleright}$-diagram
		$$
		\begin{tikzcd}
			{X_n=X_{n,0}} \arrow[r, "{f_{n,0}}"] \arrow[rd] & {X_{n,1}} \arrow[d] \arrow[r, "{f_{n,1}}"] & {X_{n,2}} \arrow[ld] \arrow[r, "{f_{n,2}}"] & \cdots \arrow[lld] \\
			& X_{n+1}                                    &                                             &                   
		\end{tikzcd}$$
		such that each $f_{n,i}$ in the diagram factors through \mcc. We denote $Y_n :=\colim_i X_{n,i}$. Then our previous sequence $\{X_n\}$ can be decomposed as
		$$\cdots\to X_n\to Y_n\to X_{n+1}\to Y_{n+1}\to\cdots$$
		Then each $Y_n$ lies in the image of \cflatw and each map $Y_n\to Y_{n+1}$ is a compact projective map by design and hence compact in \cflat, therefore $\{i^R(Y_n)\}$ becomes a compact exhaustion of $x$ in \cflat.
	\end{proof}
	
	\begin{rem}\label{compexhaus}
		Although the second proof above seems more technical, it shows that for any $\omega_1$-flat object $x$ in \mcc, we can find a compact exhaustion $\{y_n\}$ of $x$ in \cflat such that each $y_n$ is $\omega_1$-flat and each transition map is compact projective in \mcc. That produces the following useful corollary.
	\end{rem}
	
	\begin{cor}\label{leftadjtable}
		Let \mcc be a \dualaddinfcat. Then the following diagram is vertically left adjointable. 
		$$
		\begin{tikzcd}
			\mathcal{C}^{\op{flat}} \arrow[r,hook, "i"]                                                            & \mcc                                             \\
			\mathrm{Ind}(\mathcal{C}^{\omega_1\text{-}{\op{flat}}}) \arrow[u, "p_{1}"] \arrow[r,hook, ""'] & \mcp_\Sigma(\mathcal{C}^{\omega_1\text{-}{\op{flat}}}) \arrow[u, "p_{2}"']
		\end{tikzcd}$$
	\end{cor}
	
	\begin{proof}
		By \cref{compexhaus}, for any $x\in\cflatw$, we can choose a compact exhaustion $\{y_n\}$ of $x$ in \cflat such that $\{y_n\}$ is a compact projective exhaustion of $x$ in \mcc. Then
		the canonical map $\hat{h}_{2} \circ i \rightarrow i \circ \hat{h}_{1}$ is an equivalence on \cflatw, and therefore (by filtered colimit preservation and \cref{flatacc}) on the whole \cflat.
	\end{proof}
	The second proof of \cref{flcompass} also clarifies how compact maps in the flat category compare with compact projective maps in the ambient dualizable additive category:
	\begin{rem}\label{cprojmapfldecomp}
		For a \dualaddinfcat \mcc, the inclusion $i:\cflat\hookrightarrow \mcc$ sends compact maps to compact projective maps by \cref{leftadjtable}. Therefore for a map in $\cflat$, it is a compact map in $\cflat$ if and only if it is a compact projective map in \mcc. In particular, by combining with \cite[Remark 2.56]{ramzi2024dualizable}, any compact projective map between flat objects in \mcc can be decomposed into two compact projective maps of flat objects in \mcc.
	\end{rem}
It remains unclear whether the same decomposition phenomenon holds for arbitrary compact projective maps in the ambient category.
	\begin{quest}
		In a \dualaddinfcat \mcc, is any compact projective map a composite of two such maps?
	\end{quest}
	We now turn to the reconstruction of a dualizable additive \infcat from its  flat objects. For this purpose we introduce the notion of continuous prestabilization.
	\begin{de}
		Let $\mca$ be a potentially large \infcat with finite coproducts and small filtered colimits. We denote by $\mcp_{\sqcup,\op{fil}}^{\op{small}}(\mca)$ the formal cocompletion of \mca without changing finite coproducts and small filtered colimits. That induces a left adjoint (see \cite[\textsection5.3.6]{htt}) $$\hatcat^{\sqcup,\op{fil}}\to \hatcat^{\op{small}}.$$ 
		
		When $\mca$ is a compactly assembled additive \infcat, we call $\mcp_{\sqcup,\op{fil}}^{\op{small}}(\mca)$ the \textbf{continuous prestabilization} of \mca.
	\end{de}
	This construction admits several equivalent descriptions, which make its relation to compact objects and filtered colimits more transparent:
	\begin{rem}
		By \cref{altercocomp}, for any $\omega_1$-accessible additive \infcat \mca with small filtered colimits, we have $$\mcp_{\sqcup,\op{fil}}^{\op{small}}(\mca)\simeq \mcp_{\sqcup,\op{fil}_{\omega_1}}^{\op{small}}(\mca^{\omega_1})\simeq \funct^{\times,\op{cofil}_{\omega_1}}(\mca^{\omega_1,\opp},\mcs)\simeq \funct^{\times,\op{cofil}}(\mca^{\opp},\mcs)$$ where the third one denotes the full subcategory of those functors which preserve finite products and $\omega_1$-small cofiltered limits.
	\end{rem}
	We can now state the main  theorem of this section. It says that passing to flat objects and taking continuous prestabilization are inverse constructions.
	\begin{thm}\label{main3}
		For any  \dualaddinfcat \mcc, we have $$\mcc\simeq \mcp_{\sqcup,\op{fil}}^{\op{small}}(\mcc^{\op{flat}}) .$$ For any compactly assembled additive \infcat \mca,  the continuous prestabilization $\mcp_{\sqcup,\op{fil}}^{\op{small}}(\mca)$ is dualizable additive and $$\mca\simeq \mcp_{\sqcup,\op{fil}}^{\op{small}}(\mca)^{\op{flat}} .$$
		
		Furthermore, they induce the following categorical equivalence between compactly assembled additive \infcats and \dualaddinfcats   
		$$
		\begin{tikzcd}
			\assad\arrow[r, shift left=1ex, "\mcp_{\sqcup,\op{fil}}^{\op{small}}"{name=G}] & \praddbl\arrow[l, shift left=.5ex, "(-)^{\op{flat}}"{name=F}]
			\arrow[phantom, from=F, to=G, , "\scriptscriptstyle\boldsymbol{\sim}"] .
		\end{tikzcd}$$
	\end{thm}
	\begin{proof}
		First, we show both functors $\mcp_{\sqcup,\op{fil}}^{\op{small}}$ and $(-)^{\op{flat}}$ are well-defined. 
		
		For $(-)^{\op{flat}}$, since we have proved that the full subcategory of flat objects is compactly assembled, it remains to show that for any internal left adjoint $F:\mcc\to\mcd$ between \dualaddinfcats, $F^{\op{flat}}:\cflat\to \mathcal{D}^{\op{flat}}$ is compactly assembled. Since it is clear that $F^{\op{flat}}$ is additive and preserves filtered colimits, by \cite[Corollary 2.43]{ramzi2024dualizable} it suffices to show that the following diagram is vertically left adjointable.
		$$
		\begin{tikzcd}[row sep=large, column sep=large]
			\mathcal{C}^{\op{flat}} \arrow[r, "F"]                                                            & \mathcal{D}^{\op{flat}}                                              \\
			\mathrm{Ind}(\mathcal{C}^{\omega_1\text{-}{\op{flat}}}) \arrow[u, "p_{1}"] \arrow[r, "\mathrm{Ind}(F)"'] & \mathrm{Ind}(\mathcal{D}^{\omega_1\text{-}{\op{flat}}}) \arrow[u, "p_{2}"']
		\end{tikzcd}$$
		By \cref{compexhaus}, for any $x\in\cflatw$ we can find a compact exhaustion of $x$ in \cflat such that each transition map is compact projective in \mcc.  Then by \cref{preserveprojmap1}, $F$ sends such a compact exhaustion to a compact exhaustion in $\mathcal{D}^{{\op{flat}}}$. Using this, one proves that the canonical map $\hat{h}_{2} \circ F \rightarrow \operatorname{Ind}(F) \circ \hat{h}_{1}$ is an equivalence on \cflatw, and therefore (by filtered colimit preservation and \cref{flatacc}) on the whole \cflat.
		
		For $\mcp_{\sqcup,\op{fil}}^{\op{small}}$, we first observe it restricts to a (left adjoint) functor  
		$$\mcp_{\sqcup,\op{fil}}^{\op{small}}: \widehat{\op{Cat}}_{\infty,\op{ad}}^{\op{fil}}\to \widehat{\op{Cat}}_{\infty,\op{ad}}^{\op{small}}.$$
		Therefore, if $\mca$ is a compactly assembled additive \infcat, i.e. a retract  in $\widehat{\op{Cat}}_{\infty,\op{ad}}^{\op{fil}}$ of $\op{Ind}(\mcb_0)$ for some small additive \infcat $\mcb_0$, then $\mcp_{\sqcup,\op{fil}}^{\op{small}}(\mca)$ is a retract of $\mcp_\Sigma(\mcb_0)$ in $\widehat{\op{Cat}}_{\infty,\op{ad}}^{\op{small}}$ and hence dualizable additive. It remains to show that for a compactly assembled additive functor $H:\mca \to\mcb \in\assad$, $\mcp_{\sqcup,\op{fil}}^{\op{small}}(H)$ is an internal left adjoint between \dualaddinfcats. By  \cite[Corollary 2.43]{ramzi2024dualizable} again, we see that $H$ is a retract of $\op{Ind}(H): \op{Ind}(\mca^{\omega_1})\to \op{Ind}(\mcb^{\omega_1})$ in $\widehat{\op{Cat}}_{\infty,\op{ad}}^{\op{fil}}$. Thereby $\mcp_{\sqcup,\op{fil}}^{\op{small}}(H)$ is a retract of $\mcp_\Sigma(H): \mcp_\Sigma(\mca^{\omega_1})\to \mcp_\Sigma(\mcb^{\omega_1})$ in $\widehat{\op{Cat}}_{\infty,\op{ad}}^{\op{small}}$. The latter is obviously an internal left adjoint, thus $\mcp_{\sqcup,\op{fil}}^{\op{small}}(H)$ is so by \cite[Lemma 1.47]{ramzi2024dualizable}.
		
		Now we show that both functors are inverse to each other. 
		
		Since $\cflat\subset\mcc$ is closed under filtered colimits and finite coproducts, we have a  natural comparison
		$$\phi:\mcp_{\sqcup,\op{fil}}^{\op{small}}(\mcc^{\op{flat}})\to \mcc . $$
		Since $\hat{h}_1:\cflat\to \op{Ind}(\cflatw)$ is a  fully faithful left adjoint internally in $\widehat{\op{Cat}}_{\infty,\op{ad}}^{\op{fil}}$, it induces a fully faithful internal left adjoint between \dualaddinfcats $$\mcp_{\sqcup,\op{fil}}^{\op{small}}(\mcc^{\op{flat}})\xrightarrow{\mcp_{\sqcup,\op{fil}}^{\op{small}}(\hat{h}_1)} \mcp_{\sqcup,\op{fil}}^{\op{small}}(\op{Ind}(\cflatw))\simeq \mcp_\Sigma(\cflatw) .$$
		We claim the following diagram commutes.
		$$
		\begin{tikzcd}
			{\mcp_{\sqcup,\op{fil}}^{\op{small}}(\mcc^{\op{flat}})} \arrow[rd, "{\mcp_{\sqcup,\op{fil}}^{\op{small}}(\hat{h}_1)}"'] \arrow[rr, "\phi"] &                      & {\mcc} \arrow[ld, "\hat{h}_\mcc"] \\
			& \mcp_\Sigma(\cflatw) &                              
		\end{tikzcd}$$
		It suffices to show that it commutes when restricted to \cflat, but that follows from \cref{leftadjtable}. Now we have $\phi$ is fully faithful because $\mcp_{\sqcup,\op{fil}}^{\op{small}}(\hat{h}_1)$ is fully faithful. Furthermore, $\phi$ is essentially surjective because \mcc is generated by \cflat under small colimits. Therefore $\phi$ is an equivalence.
		
		Now let \mca be a compactly assembled additive \infcat. 	Note that $h:\mca\to\mcp_{\sqcup,\op{fil}}^{\op{small}}(\mca)$ is fully faithful due to \cite[Proposition 5.3.6.2.(3)]{htt}. We wish to show that $h$ induces an equivalence $$h:\mca\simeq \mcp_{\sqcup,\op{fil}}^{\op{small}}(\mca)^{\op{flat}} .$$ 
		Since $\hat{h}_\mca:\mca\to \op{Ind}(\mca^{\omega_1})$ is fully faithful left adjoint internally in $\widehat{\op{Cat}}_{\infty,\op{ad}}^{\op{fil}}$, it induces a  fully faithful internal left adjoint of \dualaddinfcats $$\mcp_{\sqcup,\op{fil}}^{\op{small}}(\mca)\xrightarrow{\mcp_{\sqcup,\op{fil}}^{\op{small}}(\hat{h}_\mca)} \mcp_{\sqcup,\op{fil}}^{\op{small}}(\op{Ind}(\mca^{\omega_1})) \simeq\mcp_\Sigma(\mca^{\omega_1}) .$$
		Since it sends $\mca$ to flat objects $\op{Ind}(\mca^{\omega_1})\simeq\mcp_{\Sigma}(\mca^{\omega_1})^{\op{flat}}$, by \cref{preserveprojmap1}, all objects in \mca are flat and we have $$\mca\subset \mcp_{\sqcup,\op{fil}}^{\op{small}}(\mca)^{\op{flat}}.$$
		We also have the following commutative diagram, which is horizontally right adjointable by construction.
		$$\begin{tikzcd}[row sep=large, column sep=7em]
			\mca \arrow[d, "h"] \arrow[r, "\hat{h}_\mca"]                                                                & \op{Ind}(\mca^{\omega_1}) \arrow[d, "i"] \\
			{\mcp_{\sqcup,\op{fil}}^{\op{small}}(\mca)} \arrow[r, "{j:=\mcp_{\sqcup,\op{fil}}^{\op{small}}(\hat{h}_\mca)}"] & \mcp_\Sigma(\mca^{\omega_1})            
		\end{tikzcd}$$
		Now given an $\omega_1$-flat object $x\in\mcp_{\sqcup,\op{fil}}^{\op{small}}(\mca)$, we denote $j=\mcp_{\sqcup,\op{fil}}^{\op{small}}(\hat{h}_\mca)$. Then $j(x)$ is $\omega_1$-flat in $\mcp_\Sigma(\mca^{\omega_1}) $ too and can be written as a sequential colimit $\colim_n h_{a_n}$ in $\mcp_\Sigma(\mca^{\omega_1})$ and hence actually in $\op{Ind}(\mca^{\omega_1})$, where each $a_n\in \mca^{\omega_1}$ and each transition map factors through $jj^R(h_{a_n})$, as we did in \cref{flatcpexhau}. Note however that $j^R(h_{a_n})=\hat{h}_\mca^R(h_{a_n})= a_n$ and $jj^R(h_{a_n})=\hat{h}_\mca\hat{h}_\mca^R(h_{a_n})=\hat{h}_\mca(a_n)$; therefore, we have the following lifting diagram for each $n$.
		$$\begin{tikzcd}
			\hat{h}_\mca(a_n) \arrow[d] \arrow[r] & h_{a_n} \arrow[d] \arrow[ld, dashed] \arrow[rd] &                  \\
			\hat{h}_\mca(a_{n+1}) \arrow[r]       & h_{a_{n+1}} \arrow[r]                           & \colim_n h_{a_n}
		\end{tikzcd}$$
		Therefore $\{a_n\}$ is a compact exhaustion in \mca and $j(x)=\colim_n h_{a_n}\simeq \colim_n \hat{h}_\mca(a_n)\simeq \hat{h}_\mca(\colim_n a_n)$. That implies in fact $x\in \mca$ and we get $\mcp_{\sqcup,\op{fil}}^{\op{small}}(\mca)^{\omega_1\text{-}\op{flat}}\subset \mca$. And by \cref{flatacc} we get $\mcp_{\sqcup,\op{fil}}^{\op{small}}(\mca)^{\op{flat}}\subset \mca$, and hence $\mcp_{\sqcup,\op{fil}}^{\op{small}}(\mca)^{\op{flat}}= \mca$.
	\end{proof}
	As an application of this reconstruction theorem, full faithfulness of an internal left adjoint can be detected on flat objects, and in fact already on \(\omega_1\)-flat objects:
	\begin{cor}
		Let $F:\mcc\to \mcd$ be an internal left adjoint of \dualaddinfcats. If the restriction $F|_{\omega_1\mathrm{\text{-}flat}}:\cflatw\to\mcd^{\omega_1\mathrm{\text{-}flat}}$ is fully faithful, then $F$ is fully faithful.
	\end{cor}
	\begin{proof}
		Since $F|_{\mathrm{flat}}:\cflat\to\mcd^{\mathrm{flat}}$ is a compactly assembled additive functor by the proof of \cref{main3}, the following diagram is vertically left adjointable.
		$$
		\begin{tikzcd}[row sep=large, column sep=5em]
			\cflat \arrow[d, "\hat{h}"', dashed, shift right] \arrow[r, "F|_{\mathrm{flat}}"] & \dflat \arrow[d, "\hat{h}"', dashed, shift right] \\
			\op{Ind}(\cflatw) \arrow[r, "\op{Ind}(F|_{\omega_1\mathrm{\text{-}flat}})"] \arrow[u,  shift right]                   & \op{Ind}(\dflatw) \arrow[u,  shift right]                  
		\end{tikzcd}$$
		However, both $\hat{h}$ are  fully faithful left adjoints internally in $\widehat{\op{Cat}}_{\infty,\op{ad}}^{\op{fil}}$. Applying $\mcp_{\sqcup,\op{fil}}^{\op{small}}$ to the above diagram, we get the following diagram
		$$
		\begin{tikzcd}[row sep=large, column sep=5em]
			\mcc \arrow[d, "\hat{h}_{\mcc}"', dashed, shift right] \arrow[r, "F"] & \mcd \arrow[d, "\hat{h}_{\mcd}"', dashed, shift right] \\
			\mcp_\Sigma(\cflatw) \arrow[r, "\mcp_{\Sigma}(F|_{\omega_1\mathrm{\text{-}flat}})"] \arrow[u,  shift right]                   & \mcp_\Sigma(\dflatw) \arrow[u,  shift right]                  
		\end{tikzcd}$$
		where $\hat{h}_{\mcc}$ and $\hat{h}_{\mcd}$ are fully faithful. Also, the bottom functor is fully faithful too by assumption. Consequently,  the top functor $F$ is fully faithful, as desired.
	\end{proof}
	\section{$\nuc(R)_{\geq0}$ as additive rigidification}\label{section5}
	In this section, we develop basic properties of additive rigidification, and show that the \infcat of connective nuclear modules in the sense of Clausen--Scholze \cite{scholze2026lecturesanalyticgeometry}
 can be characterized as the additive rigidification of the \infcat of connective complete  modules. 
 
 The argument proceeds in two steps. We first establish a general comparison between stable rigidification and additive rigidification under suitable compatibility with the connective part. We then verify these hypotheses for complete modules over (connective) adic \einfrings, using the theory of nuclear objects.

\subsection{Additive rigidity}
Rigidification serves as the universal mechanism for ensuring that compactness and relative compactness coincide within a monoidal \infcat, firstly introduced in \cite{arinkin2020stack}. 
In this subsection, we introduce a useful criterion for identifying the additive rigidification with the stable rigidification.

We begin by recalling the relative form of rigidity, since the additive case will be obtained by taking the base to be \(\spgeq\). The following definition is adapted from \cite[Definition 4.36]{ramzi2024locally}.

\begin{de}
	Let $\mcv\in\calg(\prl)$ and $\mcw\in\operatorname{CAlg}_{\mcv}\simeq\calg(\prl)_{\mcv/}$. We say that $\mcw$ is \emph{$\mcv$-rigid} (or \emph{rigid over $\mcv$}) if its unit object $\mb{1}\in \mcw$ is $\mcv$-atomic, and the multiplication functor $m \colon \mcw \otimes_{\mcv} \mcw \to \mcw$ is an $(\mcw \otimes_{\mcv} \mcw)$-internal left adjoint. 
	
	In particular, when $\mcv=\spgeq$, we say that $\mcw\in\calg(\prlad)$ is \emph{additively rigid} if it is rigid over $\spgeq$.
\end{de}
	\begin{rem}
		By \cite[Lemma 4.57]{ramzi2024locally}, a  rigid algebra \mcw over \mcv is automatically dualizable over \mcv.
	\end{rem}
	We next formulate the corresponding universal construction. Thus rigidification is characterized not by an explicit model, but by its universal mapping property against rigid algebras.
	\begin{de}
	Let $\mcv\in\calg(\prl)$. Let $f: \overline{\mathcal{W}} \rightarrow \mathcal{W} \in \operatorname{CAlg}_{\mcv}=\calg(\prl)_{\mcv/}$ be a map. We say that $f$ witnesses $\overline{\mathcal{W}}$ as a rigidification of $\mathcal{W}$ over \mcv, if $\overline{\mathcal{W}}$ is a rigid $\mathcal{V}$-algebra and for any rigid commutative $\mathcal{V}$-algebra $\mathcal{W}_0$, postcomposition with $f$ induces an equivalence:
		$$
		\operatorname{Fun}_{\mathcal{V}}^{L, \otimes}\left(\mathcal{W}_0, \overline{\mathcal{W}}\right) \xrightarrow{\sim} \operatorname{Fun}_{\mathcal{V}}^{L, \otimes}\left(\mathcal{W}_0, \mathcal{W}\right)
		$$
	\end{de}
	The rigidification always exists, as the following theorem shows.
\begin{thm}[{\cite[Corollary 4.73]{ramzi2024locally}}]
	Let $\mcv \in \calg(\prl)$. Every algebra $\mcw \in \calg_{\mcv}$ admits a rigidification $\op{Rig}_{\mcv}(\mcw)$. Consequently, this induces an adjunction:
	\begin{tikzcd}
		\calg_{\mcv}^{\op{rig}} \arrow[r, hook, shift left=1ex] & \calg_{\mcv} \arrow[l, shift left=.5ex, "\op{Rig}_{\mcv}(-)"]
	\end{tikzcd}
\end{thm}

\begin{nota}
	When $\mcv = \opsp$, we write $(-)^{\op{rig}} \coloneq \op{Rig}_{\opsp}(-)$. Similarly, when $\mcv = \spgeq$, we write $(-)^{\op{arig}} \coloneq \op{Rig}_{\spgeq}(-)$.
\end{nota}
	We shall need a few elementary compatibilities of internal homs with monoidal adjunctions. The first one is a standard projection-type formula.
	\begin{lem}\label{internalhom1}
		Let $F:\mcb\to\mcc$ be a  symmetric monoidal functor between closed symmetric monoidal  \infcats. Suppose that $F$ admits a  right adjoint $F^R$. Then for any $x\in \mcb$ and $y\in \mcc$, we have an identification for the internal hom $$\underline{\Hom}_{\mcb}(x,F^Ry)\simeq F^R\underline{\Hom}_{\mcc}(Fx,y).$$
	\end{lem}
	\begin{proof}
		For any test object $b \in \mcb$, we evaluate the mapping anima into the left-hand side:
		$$ \operatorname{Map}_{\mcb}(b, \underline{\Hom}_{\mcb}(x,F^Ry)) \simeq \operatorname{Map}_{\mcb}(b \otimes x, F^Ry) \simeq \operatorname{Map}_{\mcc}(F(b \otimes x), y) \simeq$$
		$$  \operatorname{Map}_{\mcc}(F(b) \otimes F(x), y) \simeq \operatorname{Map}_{\mcc}(F(b), \underline{\Hom}_{\mcc}(Fx, y)) \simeq$$
		$$  \operatorname{Map}_{\mcb}(b, F^R\underline{\Hom}_{\mcc}(Fx, y)). $$
		Since this natural equivalence holds for all objects $b \in \mcb$, Yoneda's lemma implies the desired canonical equivalence $\underline{\Hom}_{\mcb}(x,F^Ry) \simeq F^R\underline{\Hom}_{\mcc}(Fx,y)$.
	\end{proof}
	Rigidification is closely related to trace class morphisms. We recall the notion in the form needed below.
	\begin{de}
	Let $\mathcal{C}$ be a symmetric monoidal \infcat. A map $f: x \rightarrow y$ is called \textbf{trace class} if there exists an object $d$ together with a pairing $d \otimes x \rightarrow \mathbf{1}$ and a map $\mathbf{1} \rightarrow y \otimes d$ such that $f$ factors as $x \rightarrow y \otimes d \otimes x \rightarrow y$.

	\end{de}
	\begin{rem}
		If $\mathcal{C}$ is closed symmetric monoidal, then this condition is equivalent to the requirement that the map $\mathbf{1} \rightarrow 	\underline{\Hom}_{\mathcal{C}}(x, y)$ classifying $f$ lifts through the canonical map
		$$
		\underline{\Hom}_{\mathcal{C}}\left(x, \mathbf{1}\right) \otimes y \rightarrow \underline{\Hom}_{\mathcal{C}}(x, y).
		$$
	\end{rem}
		\begin{prop}\label{traceclassfactorization}
		Let $\mathcal{C}$ be a symmetric monoidal category, and let
		$
		f:x\to y
		$
		be a trace-class morphism represented by  data
		$
		\eta:\mathbf{1}\to d\otimes y$ and $
		\epsilon:x\otimes d\to \mathbf{1}.
		$
		Suppose that there are morphisms $h:d_0\to d$, $g:y_0\to y$ and a factorization
		\[
		\eta:
		\mathbf{1}
		\xrightarrow{\eta_0}
		d_0\otimes y_0
		\xrightarrow{h\otimes g}
		d\otimes y.
		\]
		Then the morphism
		$
		f_0:x\to y_0
		$
		defined by the composite
		$
		x\otimes\mathbf{1}
		\xrightarrow{\id_x\otimes\eta_0}
		x\otimes d_0\otimes y_0
		\xrightarrow{\epsilon_0\otimes\id_{y_0}}
		\mathbf{1}\otimes y_0
		$
		is trace class, where $\epsilon_0=\epsilon\circ (\id_x\otimes h)$.
		In particular, $f$ factors through the trace-class morphism $f_0:x\to y_0$.
		
	\end{prop}
	
	\begin{proof}
		The morphisms
		\[
		\eta_0:\mathbf{1}\to d_0\otimes y_0,
		\qquad
		\epsilon_0:x\otimes d_0\to\mathbf{1}
		\]
		form trace-class data for $f_0$, so $f_0$ is trace class by definition.
		Moreover, the factorization follows from the following diagram
		$$\begin{tikzcd}[column sep=5em] & x\otimes d_0\otimes y_0 \arrow[r,"\epsilon_0\otimes\id_{y_0}"] \arrow[d,"\id_x\otimes h\otimes g"] & y_0 \arrow[d,"g"] \\ x \arrow[ru,"\id_x\otimes\eta_0"] \arrow[r,"\id_x\otimes\eta"'] & x\otimes d\otimes y \arrow[r,"\epsilon\otimes\id_y"'] & y . \end{tikzcd}$$
	\end{proof}
	In order to compare trace class maps in a stable category with trace class maps in its connective part, we first record a basic connectivity observation.
	\begin{lem}

		Let  $\mcc=\opsp(\mcp_{\Sigma}(\mca))$, where $\mca$ is a small additively \syminfcat, let $p \in \mca, y \in \mcc_{\geq 0}$. Then: 
		\enu{\item The internal hom $\underline{\Hom}_{\mcc}(p,y)$ is connective in \mcc.
		\item If $f:p\to y$ is trace class in $\mcc$, it is trace class in $\mcc_{\geq 0}$.}

	\end{lem}
	\begin{proof}
		(1) By compact projective generation, it suffices to observe that for any $q\in \mca$, the mapping spectrum $$\unmap_{\mcc}(q,\underline{\Hom}_{\mcc}(p,y))\simeq\unmap_{\mcc}(q\otimes p,y) $$ is connective.
		
	(2)	It suffices to observe that the internal hom $\underline{\Hom}_{\mcc}(p,y)\in\mcc$ is connective and hence coincides with $\underline{\Hom}_{\mcc_{\geq0}}(p,y)$.
	\end{proof}
	The previous lemma allows one to descend trace class factorizations to the connective part after precomposition with compact projective maps. We formulate this in the following general form.
	\begin{lem}\label{conntrcl}
		Let $\mcb\in \calg(\prlst)$ be equipped with an accessible $t$-structure that is compatible with the monoidal structure.  Suppose that there exists a fully faithful symmetric monoidal functor $i:\mcb\hookrightarrow \mcc$  whose underlying functor is strongly $t$-continuous and that $\mcc\simeq\opsp(\mcp_{\Sigma}(\mca))$ for some  small additively \syminfcat $\mca$. 
		
		If $x\xrightarrow{f} y\xrightarrow{g} z$ are maps in $\mcb_{\geq0}$ such that $f$ is a compact projective map (see \Cref{atomicmaps}) in $\mcb_{\geq0}$ and $g$ is trace class in $\mcb$, then the composite $gf$ is trace class in $\mcb_{\geq0}$.
	\end{lem}
	\begin{proof}
		Since $i$ is strongly $t$-continuous, $i(x)\xrightarrow{i(f)}i(y)$ is a compact projective map in $\mcc_{\geq 0}$. Therefore it factors through some compact projective $p\in \mca$ and we obtain the following commutative diagram,
		$$\begin{tikzcd}
			{\underline{\Hom}_{\mcb}(y,\mb{1})} \arrow[rr, "\sim"] \arrow[dd] &                                                       & {i^R\underline{\Hom}_{\mcc}(iy,i\mb{1})} \arrow[d]                  \\
			& {\tau_{\geq0}\underline{\Hom}_{\mcb}(x,\mb{1})} \arrow[ld] & {i^R\underline{\Hom}_{\mcc}(p,i\mb{1})} \arrow[d] \arrow[l, dashed] \\
			{\underline{\Hom}_{\mcb}(x,\mb{1})} \arrow[rr, "\sim"]            &                                                       & {i^R\underline{\Hom}_{\mcc}(ix,i\mb{1})}                           
		\end{tikzcd}$$
		where the equivalences follow from \cref{internalhom1}. The existence of the dotted arrow follows from connectivity of $\underline{\Hom}_{\mcc}(p,i\mb{1})$, which holds because of compact projectivity 
		of $p$. 
		In particular, the natural map $\underline{\Hom}_{\mcb}(y,\mb{1})\to \underline{\Hom}_{\mcb}(x,\mb{1})$ factors through $\tau_{\geq0}\underline{\Hom}_{\mcb}(x,\mb{1})$.
		Therefore, we obtain the following commutative diagram.
		$$\begin{tikzcd}
			& {\underline{\Hom}_{\mcb}(y,\mb{1})\otimes z} \arrow[d] \arrow[ldd, dashed] \arrow[r] & {\underline{\Hom}_{\mcb}(y,z)} \arrow[d] \\
			& {\underline{\Hom}_{\mcb}(x,\mb{1})\otimes z} \arrow[r]                               & {\underline{\Hom}_{\mcb}(x,z)}           \\
			{\tau_{\geq0}\underline{\Hom}_{\mcb}(x,\mb{1})\otimes z} \arrow[ru] \arrow[r] & {\tau_{\geq0}\underline{\Hom}_{\mcb}(x,z)} \arrow[ru]                                &                                         
		\end{tikzcd}$$
		Note that $\tau_{\geq0}\underline{\Hom}_{\mcb}(x,z)\simeq \underline{\Hom}_{\mcb_{\geq 0}}(x,z)$, hence $x\xrightarrow{gf} z $ is trace class in $\mcb_{\geq 0}$, as desired.
	\end{proof}
	The following theorem is the main abstract result of this subsection. It shows that, under a strong \(t\)-continuity hypothesis, the connective part of the stable rigidification already satisfies the universal property of additive rigidification.
	\begin{thm}\label{addrig}
		Let $\mcc=\opsp(\mcp_{\Sigma}(\mca))$, where $\mca$ is a small additively symmetric monoidal \infcat. Suppose that the inclusion from the stable rigidification 
		$\mcc^\op{rig}\hookrightarrow \mcc$ is strongly $t$-continuous, where the $t$-structure on $\mcc^\op{rig}$ is given by $\mcc^\op{rig}_{\geq 0}=\mcc^\op{rig}\cap \mcc_{\geq0}$. Then  $\mcc^\op{rig}_{\geq 0}$ is additively rigid, and it can be identified with the additive rigidification $\mcc_{\geq0}^\op{arig}$ of $\mcc_{\geq0}$. 
	\end{thm}
	\begin{proof}
		We first show that  $\mcc^\op{rig}_{\geq 0}$ is additively rigid. 
		By \cref{propertycoffib}, we see that $\mcc^\op{rig}_{\geq 0}$ is dualizable 
		additive and the $t$-structure on $\mcc^{\op{rig}}$ is right complete. Following 
		the same approach as in \cite[Proposition 4.15]{ramzi2024locally}, it suffices 
		to show that any compact projective map $x\to z$ in $\mcc^\op{rig}_{\geq 0}$ is 
		trace class in $\mcc^{\op{rig}}_{\geq 0}$ (note that being trace class in $\mcc^{\op{rig}}_{\geq 0}$ is stronger than in $\mcc^{\op{rig}}$). 
		
		By \cite[Addendum 2.42]{ramzi2024dualizable}, it actually suffices to show that enough compact projective maps in $\mcc^\op{rig}_{\geq 0}$ are trace class, thus we can assume that $x\to z$ is a compact projective map between flat objects.
		Since $\mcc^\op{rig}_{\geq 0}$ is dualizable additive, we can decompose $x\to z$ into two compact projective maps $x\to y \to z$ in $\mcc^\op{rig}_{\geq 0}$ (see  \cref{cprojmapfldecomp}).  Let $u:\mcc^\op{rig}_{\geq 0}\hookrightarrow\mcc^\op{rig}$ denote the inclusion. Since $\mcc^\op{rig}_{\geq 0}$ is Grothendieck prestable and $\mcc^\op{rig}$ is right complete, the right adjoint $u^R$ preserves filtered colimits. Since both categories are compactly assembled, $u$ preserves compact maps, and hence $y\to z$ is a compact map in $\mcc^\op{rig}$. Since $\mcc^\op{rig}$ is stably rigid, $y\to z$ is trace class in $\mcc^\op{rig}$. Applying \cref{conntrcl} to the inclusion $\mcc^\op{rig}\hookrightarrow\mcc$, we conclude that $x\to z$ is trace class in $\mcc^\op{rig}_{\geq 0}$, as desired.
		
		Now we show that $\mcc^\op{rig}_{\geq 0}$ can be identified with the additive rigidification $\mcc_{\geq0}^\op{arig}$.  Since the unit of $\mcc_{\geq0}$ is compact projective, by the proof of \cite[Corollary 4.73]{ramzi2024locally}, $\mcc_{\geq0}^\op{arig}$ can be identified with the largest additively rigid full subcategory of $\mcc_{\geq0}$, thus $$\mcc^\op{rig}_{\geq 0}\subset \mcc_{\geq0}^\op{arig}.$$ 
		On the other hand, $\mcc^\op{rig}$ is the largest stably rigid full subcategory of $\mcc$, thus $$\opsp(\mcc_{\geq0}^\op{arig})\subset \mcc^\op{rig}.$$ Since both inclusions $\mcc^\op{rig}\hookrightarrow\mcc$ and $\opsp(\mcc_{\geq0}^\op{arig})\hookrightarrow\mcc$ create connective objects, and since both categories are right complete, we conclude that $\mcc_{\geq0}^\op{arig}=\mcc^\op{rig}_{\geq 0}$.
	\end{proof}
	
	Now we prove a useful result in the projectively generated case. We now recall the $\kappa$-projective generation condition and refer the reader to \cref{kappaprojgeneration} for more details. 
	\begin{de}
		Let $\mcc\in\calg(\prlad)$ and $\kappa$ be a regular cardinal. We say \mcc is a \textbf{$\kappa$-projectively presentably symmetric monoidal additive \infcat} if \mcc is $\kappa$-projectively generated, i.e., $\mcc\simeq\mcp_{\Sigma_\kappa}(\mcc^{\kappa\text{-}\op{proj}})$, and the full subcategory $\mcc^{\kappa\text{-}\op{proj}}\subset \mcc$ is closed under tensor products and contains the unit object.
	\end{de}
The reduction requires one more elementary observation: trace class maps can be detected after a fully faithful monoidal localization.

	\begin{lem}\label{locinternalhom}
		Let $L:\mcc\to\mcd$ be a  symmetric monoidal functor between closed symmetric monoidal \infcats that admits a fully faithful right adjoint $L^R$. Then for any $x,y\in \mcd$, we have an identification for the internal hom $$\underline{\Hom}_{\mcc}(L^Rx,L^Ry)\simeq L^R\underline{\Hom}_{\mcd}(x,y).$$
		
		Furthermore, if $L^R$ preserves the unit and $f\in \mcd $ is a map such that $L^R(f)$ is trace class in \mcc, then $f$ is trace class in \mcd.
	\end{lem}
	\begin{proof}
		The canonical identification  $\underline{\Hom}_{\mcc}(L^Rx,L^Ry) \simeq L^R\underline{\Hom}_{\mcd}(x,y)$ follows directly from \cref{internalhom1}.
		
		Now assume that $L^R$ preserves the unit and $f:x\to y$ be a map such that $L^R(f)$ is trace class in \mcc. By the following diagram, we see that $f$ is trace class in \mcd:
		$$\begin{tikzcd}
			{\underline{\Hom}_{\mcc}(L^R(x),\mb{1})\otimes L^R(y)} \arrow[d, "\sim"] \arrow[r] & {\underline{\Hom}_{\mcc}(L^R(x),L^R(y))} \arrow[d, "\sim"] \\
			{L^R\underline{\Hom}_{\mcd}(x,\mb{1})\otimes L^R(y)} \arrow[r] \arrow[d]                                     & {L^R\underline{\Hom}_{\mcd}(x,y)}                                                  \\
			{L^R(\underline{\Hom}_{\mcd}(x,\mb{1})\otimes y)} \arrow[ru]                                               &                                                                         
		\end{tikzcd}$$
	\end{proof}

	We now show that additive rigidification is insensitive to replacing the \(\kappa\)-projective presentation by its finite-coproduct cocompletion.
	\begin{prop}\label{addrigidificationkappaproj}
		
		Let $\mcc\in\calg(\prlad)$ be a $\kappa$-projectively presentably symmetric monoidal additive \infcat, i.e. $\mcc\simeq\mcp_{\Sigma_\kappa}(\mcc^{\kappa\text{-}\op{proj}})$ for some regular cardinal $\kappa$. The following functor between the additive rigidifications is an equivalence
		$$\mcp_{\Sigma}(\mcc^{\kappa\text{-}\op{proj}})^{\op{arig}}\xrightarrow{\sim}\mcc^{\op{arig}}.$$
	\end{prop}
	\begin{proof}
		When $\kappa=\omega$, it is obvious. Now we assume that $\kappa>\omega$. By \cite[Lemma 4.71]{ramzi2024locally}, the additive rigidification of \mcc can be identified with that of $\mcp_{\Sigma}(\mcc^{\kappa})$. Therefore, it suffices to show that for  any morphism $F:\mce \to \mcp_{\Sigma}(\mcc^{\kappa})$ in $\calg(\prlad)$ from an additively rigid \infcat $\mce$, we have $F$ factors through $\mcp_{\Sigma}(\mcc^{\kappa\text{-}\op{proj}})$. 
		
		By \cite[Corollary 4.14]{ramzi2024locally}, any compact projective map in \mce is trace class. Since \mce is generated by compact projectively exhaustible objects under small colimits, it suffices to show that any compact projective trace class  map in $\mcp_{\Sigma}(\mcc^{\kappa})$ factors through $\mcp_{\Sigma}(\mcc^{\kappa\text{-}\op{proj}})$.
		By \cite[Addendum 2.42]{ramzi2024dualizable}, it actually suffices to show that the composite of any four compact projective trace class  maps of \emph{flat objects} $$V\to W\to X\to Y \to Z$$ in $\mcp_{\Sigma}(\mcc^{\kappa})$  factors through $\mcp_{\Sigma}(\mcc^{\kappa\text{-}\op{proj}})$.
		
		Note that by \cref{lazardthm} any flat object in $\mcp_{\Sigma}(\mcc^\kappa)$ is a  filtered colimit of representable objects in $\mcc^\kappa$. Since the unit of $\mcp_{\Sigma}(\mcc^{\kappa})$ is compact projective, by \cref{traceclassfactorization} we can obtain  factorizations in it as follows, where $P,Q$ are compact projective (namely $P=h(p),Q=h(q)$ for some $p,q\in \mcc^{\kappa}$) and $P\to Q$ is trace class in $\mcp_{\Sigma}(\mcc^{\kappa})$:
		$$\begin{tikzcd}
			& P \arrow[d] \arrow[r, dashed] & Q \arrow[d] \\
			X \arrow[r] \arrow[ru, dashed] & Y \arrow[r]                   & Z          
		\end{tikzcd}$$
		Applying \cref{locinternalhom} to the adjunction $$c: \mcp_{\Sigma}(\mcc^\kappa)\rightleftarrows \mcc: h,$$ 
		we see that $p\to q\in \mcc^{\kappa}$ is trace class in \mcc.
		By the same argument, $V\to X$ factors through $h(r)\to h(s)$ for some map $r\to s \in \mcc^{\kappa}$ that is trace class in \mcc.  Since the unit of \mcc is projective, any trace class map in $\mcc$ is a projective map. Note that any object in $\mcc^{\kappa}$ can be written as a geometric realization of objects in $\mcc^{\kappa\text{-}\op{proj}}$; therefore, we  can obtain  factorizations in \mcc as follows, where $a$ lies in $\mcc^{\kappa\text{-}\op{proj}}$.
		$$\begin{tikzcd}
			& a \arrow[d] \arrow[rrd] &             &   \\
			r \arrow[r] \arrow[ru, dashed] & s \arrow[r]                     & p \arrow[r] & q
		\end{tikzcd}$$
		In particular, the composite map $V\to Z$ factors through $h(a)\in\mcp_{\Sigma}(\mcc^{\kappa\text{-}\op{proj}})$, as desired.
	\end{proof}
	
\subsection{Nuclear objects}	To apply the preceding criterion, we need a way to recognize stable rigidifications in examples. This is provided by the theory of nuclear objects.
	\begin{de}

		Let $\mcc\in \calg(\prl_{\op{st},\omega})$. We say $M\in \mcc$ is a \textbf{nuclear} object if for any compact object $P\in \mcc$, the following trace comparison is an equivalence 
		$$P^\vee \otimes M\to \underline{\Hom}_{\mcc}(P,M), $$  where $P^\vee$ denotes $\underline{\Hom}_{\mcc}(P,\mb{1})$.
		
We say $M\in \mcc$ is a \textbf{basic nuclear} object if $M$ can be written as a sequential colimit in which every transition map is trace class.
	\end{de}
	A useful feature of nuclear objects is that they are generated by basic nuclear objects.
	\begin{prop}[{\cite[Proposition 7.11]{aoki2025very}}]\label{basicnucgen}
		Let $\mcc\in \calg(\prl_{\op{st},\omega})$. Then $\nuc(\mcc)\subset \mcc$ is generated by basic nuclear objects under filtered colimits.
	\end{prop}
	We also recall two permanence properties of nuclear objects which will be used implicitly below.
	\begin{rem}
		The natural inclusion $$\nuc(\mcc)\hookrightarrow\mcc$$  is closed under small colimits and tensor products by \cite[Theorem 8.6]{clausen2026condensed}. We always have the inclusion $\mcc^{\op{rig}}\subset \nuc(\mcc)$ by \cite[Proposition 1.27]{efimov2025localizinginvariantsinverselimits}.
	\end{rem}
We next give a practical criterion for verifying the hypotheses of \cref{addrig}. It is formulated in terms of the comparison between stable trace objects and nuclear objects.
\begin{de}\label{destab}
	Let $$\stab(-):\catadidem\to\catper$$ denote the left adjoint to the forgetful functor. It can be identified with $\opsp(\mcp_{\Sigma}(-))^\omega$.
\end{de}
	\begin{prop}\label{addrigcriterion} 
		Let $\mcc=\opsp(\mcp_{\Sigma}(\mca))$, where $\mca$ is a small additively symmetric monoidal \infcat. The assumption of \cref{addrig} holds when \mcc satisfies the condition of \cite[Proposition 1.30]{efimov2025localizinginvariantsinverselimits}; namely, for every\footnote{Every object in \mca suffices.} $x\in \stab(\mca)$,
		$$
		x^{\op{tr}}=\unmap_{\mcc}(\mb{1},(-)^\vee\otimes x)\in \ind(\stab(\mca))\simeq \mcc
		$$
		is a nuclear object of \mcc.
	\end{prop}
	\begin{proof}
		By \cite[Proposition 1.30]{efimov2025localizinginvariantsinverselimits}, we have $\nuc(\mcc)=\mcc^{\op{rig}}$ and the right adjoint $i^R$ to the inclusion $\nuc(\mcc)\xhookrightarrow{i} \mcc$ is determined by $ii^R(c) \simeq c^{\op{tr}}=\unmap_{\mcc}(\mb{1},(-)^\vee\otimes c)$. Since for any (compact projective) object $a\in \mca$, the internal dual $a^\vee\in\mcc$ is connective, we see that $i^R$ is right $t$-exact, as desired.
	\end{proof}
	Since later examples involve localizations and module categories over nuclear algebras, we record the following base-change property of nuclear objects.
	\begin{lem}[{\cite[Lemma 1.11]{fedeli2023topological}}]\label{lemnucAmodule}
		Let $\mcc\in \calg(\prl_{\op{st},\omega})$ and $A\in\calg(\nuc(\mcc))$. Then we have a natural equivalence $$\modu_A(\nuc(\mcc))\xhookrightarrow{\sim} \nuc(\modu_A(\mcc)).$$
	\end{lem}
	\begin{proof}
		We first verify the inclusion $\modu_A(\nuc(\mcc))\xhookrightarrow{} \modu_A(\mcc)$ factors through $ \nuc(\modu_A(\mcc))\subset \modu_A(\mcc)$.
		This follows from \cref{basicnucgen}, because
		$\mcc\xrightarrow{A\otimes-}\modu_A(\mcc)$ in the following diagram sends basic nuclear objects to basic nuclear objects:
		$$\begin{tikzcd}
			\nuc(\mcc) \arrow[r, hook] \arrow[d, "A\otimes-"] & \mcc \arrow[d, "A\otimes-"] \\
			\modu_A(\nuc(\mcc)) \arrow[r, hook]               & \modu_A(\mcc)              
		\end{tikzcd}$$
		
		We now show that the inclusion $\modu_A(\nuc(\mcc))\xhookrightarrow{} \nuc(\modu_A(\mcc))$ is essentially surjective. Let $M\in \nuc(\modu_A(\mcc))$. We need to show that the underlying object of $M$ is nuclear in \mcc. Given a compact object $P\in\mcc^\omega$, we have
		$$
		\begin{aligned}
			\underline{\Hom}_{\mcc}(P,M)
			&\simeq \underline{\Hom}_{\modu_A(\mcc)}(A\otimes P,M) \\
			&\simeq \underline{\Hom}_{\modu_A(\mcc)}(A\otimes P,A)\otimes_A M \\
			&\simeq \underline{\Hom}_{\mcc}(P,\mb{1})\otimes A\otimes_A M \\
			&\simeq \underline{\Hom}_{\mcc}(P,\mb{1})\otimes M,
		\end{aligned}
		$$
		as desired.
	\end{proof}
	Combining this compatibility with the comparison theorem above gives an additive rigidification statement for module categories.
	\begin{cor}\label{modnucaddrig}
		Let $\mcc=\opsp(\mcp_{\Sigma}(\mca))$, where $\mca$ is a small additive symmetric monoidal \infcat. Suppose that the inclusion $\mcc^\op{rig}\hookrightarrow\nuc(\mcc)$ is an equivalence and that the inclusion 
		$\mcc^\op{rig}\hookrightarrow \mcc$ is strongly $t$-continuous\footnote{In particular, both conditions can be guaranteed when the condition of \cref{addrigcriterion} holds.}. Then for any $A\in\calg(\mcc_{\geq0}^\op{rig})$, the inclusion 
		$$\modu_A(\mcc_{\geq0}^\op{rig})\hookrightarrow \modu_A(\mcc_{\geq0})$$ 
		exhibits $\modu_A(\mcc_{\geq0}^\op{rig})$ as the additive rigidification of $\modu_A(\mcc_{\geq0})$.
	\end{cor}
	\begin{proof}
		By \cref{lemnucAmodule}, the following inclusion is an equivalence
		$$\modu_A(\nuc(\mcc))\xhookrightarrow{\sim} \nuc(\modu_A(\mcc)).$$
		In particular, $\nuc(\modu_A(\mcc))$ is rigid over \opsp and hence the following inclusion is an equivalence
		$$\modu_A(\mcc)^\op{rig}\xhookrightarrow{\sim} \nuc(\modu_A(\mcc)).$$
		Since by assumption the inclusion 
		$\mcc^\op{rig}\hookrightarrow \mcc$ is strongly $t$-continuous, the inclusion 
		$$\modu_A(\mcc)^\op{rig}\simeq\modu_A(\mcc^\op{rig})\hookrightarrow \modu_A(\mcc)$$ is so. Therefore the \infcat $\modu_A(\mcc)$ satisfies the assumption of \cref{addrig} which concludes the proof.
	\end{proof}
	
\subsection{Nuclear (solid) modules}

We now apply the abstract theory of additive rigidification to connective complete modules over adic \einfrings. The goal is to identify the connective nuclear category as the additive rigidification of the  category of connective complete modules.
	\begin{de}[{\cite[Definition 8.1.1.5]{sag}}]
		An \textbf{adic \einfring} is a \emph{connective} \einfring $R$ together with a topology on
		the commutative ring $\pi_0R$, which admits a finitely generated ideal $I\subset \pi_0R$ of definition. We say $R$ is a \textbf{complete}  adic \einfring if $R\simeq R^\wedge_I$.
		
	We denote by $\modrcpl \subset \modu_R(\opsp)$ the full subcategory of ($I$-)complete $R$-modules. Note that this definition does not depend on the choice of $I$ (see \cite[Remark 7.3.1.2]{sag}).
	\end{de}
	We recall the basic structural properties of the category of complete modules which will be used below.
	\begin{rem}
		Let $R$ be an adic \einfring. The inclusion $\modrcpl\subset \modu_R(\opsp)$ is an accessibly reflective subcategory that is compatible with the monoidal structure by \cite[\textsection7.3.5]{sag}, hence $\modrcpl$ admits a natural presentably symmetric monoidal structure, which we denote by $-\widehat{\otimes}_R -$.
		
		The inclusion $\modrcpl\subset \modu_R(\opsp)$ is  closed under truncations and connective covers by \cite[Proposition 7.3.4.4]{sag}, hence \modrcpl inherits an (accessible) $t$-structure which is compatible with the monoidal structure.
	\end{rem}
	\begin{rem}\label{withoutlosscomplete}
		Note that for an adic \einfring $R$, we always have $\modrcpl\simeq \modu_{R^\wedge_I}^{\op{cpl}}$ (see \cite[Lemma 2.4.13]{gregoric2021moduli}). Therefore, when studying complete modules, one may assume without loss of generality that $R$ is complete.
	\end{rem}
	The first step is to understand the connective part of the category of complete modules in terms of projective generators.
	\begin{prop}\label{wstructureonmodrcpl}
		Let $R$ be an adic \einfring. Then \modrcpl is generated by $\omega_1$-projective objects (in the connective part) as a localizing subcategory, and the connective part $\modrcplgeq \subset\modrcpl$ forms a hypercomplete weight structure (see \cref{wstructure}). 
	\end{prop}
	\begin{proof}
		By \cref{tcatiswcat}, it suffices to show that \modrcpl is generated by $\omega_1$-projective objects as a localizing subcategory. Since the completion functor $$\modu_R\to \modrcpl$$ has a $t$-exact right adjoint which preserves $\omega_1$-filtered colimits, it preserves $\omega_1$-projective objects (in the connective part). Note that $\modu_R$ is generated by $\omega_1$-projective objects as a localizing subcategory, so is \modrcpl.
		
	\end{proof}
	This immediately gives an explicit projective presentation of the  category of connective complete modules.
	\begin{cor}\label{modrcplomega1projgen}
		Let $R$ be an adic \einfring. Then we have $$\modrcplgeq\simeq \mcp_{\Sigma_{\omega_1}}(\op{Proj}_R^{\op{cpl},\omega_1}).$$ In particular, the inclusion $\op{Proj}_R^{\op{cpl},\omega_1}\hookrightarrow \modrcplgeq$ is a dense subcategory (in the sense of \cite[Definition 20.4.1.1]{sag}).
	\end{cor}
	\begin{proof}
		By \cref{wstructureonmodrcpl}, the connective part $\modrcplgeq\subset \modrcpl$ forms a hypercomplete weight structure. Therefore, by	\cref{hycompleteimpliespsigamak}, \modrcplgeq is generated by $\omega_1$-projective objects under small colimits. By \cref{kcprojgen}, we conclude that $\modrcplgeq\simeq \mcp_{\Sigma_{\omega_1}}(\op{Proj}_R^{\op{cpl},\omega_1}).$
	\end{proof}
	We next compare the stable hull of $\omega_1$-projective complete modules with the category of $\omega_1$-compact complete modules.
	\begin{lem}
		Let $R$ be an adic \einfring.	The functor $$\stab(\op{Proj}_R^{\op{cpl},\omega_1})\to \modu_R^{\op{cpl},\omega_1}$$ induced by the inclusion $\op{Proj}_R^{\op{cpl},\omega_1}\hookrightarrow\modu_R^{\op{cpl},\omega_1}$ is fully faithful. 
	\end{lem}
	\begin{proof}
		By \cite[Lemma 6.6]{winges2024localisation}, $\stab(\op{Proj}_R^{\op{cpl},\omega_1})$ is the smallest stable subcategory of itself containing $\op{Proj}_R^{\op{cpl},\omega_1}$. Therefore it suffices to show that the restriction $\op{Proj}_R^{\op{cpl},\omega_1}\hookrightarrow\modu_R^{\op{cpl},\omega_1}$ is fully faithful on the \emph{mapping spectra} of $\stab(\op{Proj}_R^{\op{cpl},\omega_1})$. However,  objects of $\op{Proj}_R^{\op{cpl},\omega_1}$ are projective in both $\stab(\op{Proj}_R^{\op{cpl},\omega_1})$ and $\modu_R^{\op{cpl},\omega_1}$, hence all their mapping spectra are connective. So, it reduces to showing full faithfulness on mapping anima, which is obvious.
	\end{proof}
	This comparison allows us to locate the dualizable and compact objects of complete modules inside the stable hull of  $\omega_1$-projective complete modules.
	\begin{prop}\label{modrcpldsubsetstab}
		Let $R$ be an adic \einfring.	We have the inclusions $$\modu_R^{\op{cpl},\omega}\subset \modu_R^{\op{cpl},d}\subset \op{Stab}(\op{Proj}_R^{\op{cpl},\omega_1})$$ inside \modrcpl.
	\end{prop}
	\begin{proof}
		By \cite[Lemma 4.53]{ramzi2024locally}, the first inclusion follows from the local rigidity of $\modrcpl$ over \opsp.
		
		Now we prove the second inclusion. By virtue of \cref{withoutlosscomplete}, we may assume that $R$ is complete without loss of generality.	By \cite[Theorem 11.3]{naumann2024separable}, we have $\modu_R^{\op{cpl},d}\simeq \op{Perf}(R)$. In particular, the stable subcategory $$\modu_R^{\op{cpl},d}\hookrightarrow \modrcpl$$ can be identified with the thick subcategory generated by the unit $R$. Therefore, it suffices to observe $R$ lies in  $\op{Stab}(\op{Proj}_R^{\op{cpl},\omega_1})$.
	\end{proof}

	\begin{rem}
		In particular, the inclusion $$\op{Stab}(\op{Proj}_R^{\op{cpl},\omega_1})\hookrightarrow \modrcpl$$  is a dense full subcategory by \cite[Remark 20.4.1.11]{sag}, because \modrcpl is compactly generated by combining \cite[Proposition 7.1.1.12(e) and Proposition 7.3.1.7]{sag}.
	\end{rem}
	
We are now ready to recall the notion of nuclear modules used by Efimov \cite{efimov2025localizinginvariantsinverselimits} in this context, though with slightly different notation.

\begin{de}
	Let $R$ be an adic \einfring. We define
	\[
	\nucr \coloneq \nuc\big(\ind(\op{Stab}(\op{Proj}_R^{\op{cpl},\omega_1}))\big).
	\]
	Note that $\ind(\op{Stab}(\op{Proj}_R^{\op{cpl},\omega_1}))$ can be canonically identified with $\opsp(\mcp_{\Sigma}(\op{Proj}_R^{\op{cpl},\omega_1}))$. We write
	\[
	\nucrgeq \coloneq \nucr \cap \ind(\op{Stab}(\op{Proj}_R^{\op{cpl},\omega_1}))_{\geq 0}\simeq \nucr \cap\mcp_{\Sigma}(\op{Proj}_R^{\op{cpl},\omega_1}),
	\]
	which equips $\nucr$ with an accessible $t$-structure compatible with its symmetric monoidal structure.
\end{de}

\begin{rem}
	Our $\nucr$ corresponds to $\nuccsr$ in \cite{efimov2025localizinginvariantsinverselimits}. Note that Efimov reserves the notation $\nucr$ for $\nuc\big(\ind(\modu_R^{\op{cpl},\omega_1})\big)$, which can be identified with the stable rigidification of $\modrcpl$, but differs from our $\nucr$.
\end{rem}
	The abstract criterion from the first subsection can now be applied to this nuclear category.
	
	\begin{thm}\label{nucadrig}
		Let $R$ be an adic \einfring. Then $\nucrgeq$ is additively rigid, and the inclusion $$\nucrgeq\xhookrightarrow{i}\mcp_{\Sigma}(\op{Proj}_R^{\op{cpl},\omega_1})$$  exhibits \nucrgeq as the additive rigidification of $\mcp_{\Sigma}(\op{Proj}_R^{\op{cpl},\omega_1})$.
	\end{thm}
	
	\begin{proof} 
		Let $P=(\bigoplus_{\mathbb{N}} R)^{\wedge}\in \op{Proj}_R^{\op{cpl},\omega_1}$, and let $\mcd=\op{Stab}(\op{Proj}_R^{\op{cpl},\omega_1})$.	Combining \cref{addrig} and \cref{addrigcriterion}, it suffices to show that $\ind(\mcd)$ satisfies the condition in \cite[Proposition 1.30]{efimov2025localizinginvariantsinverselimits}. Since $P$ generates $\op{Proj}_R^{\op{cpl},\omega_1}$ under finite sums and retracts, by \cite[Proposition 1.33]{efimov2025localizinginvariantsinverselimits} it is enough to show that $P^{\opp,\vee}\in \ind(\mcd^\opp)$ is nuclear.  However, this follows from  \cref{popveenuclear} (cf. \cite[Proposition 7.4]{efimov2025localizinginvariantsinverselimits}). 
	\end{proof}
	We also record the following explicit computation, which explains the relation with the formulas appearing in Efimov's work.
	\begin{rem}
		In \cite[Corollary 7.7]{efimov2025localizinginvariantsinverselimits}, Efimov further gives an explicit formula $$
		\unmap_{\ind(\mcd)}(P,ii^R(P))\simeq \big((\prod_{\mathbb{N}} R)\otimes_R (\bigoplus_{\mathbb{N}} R)\big)^\wedge.
		$$	However, both \cite[Proposition 7.4 and Corollary 7.7]{efimov2025localizinginvariantsinverselimits} rely on an argument written in German \cite[Satz 3.8]{andreychev2023k}. For the convenience of the reader, we provide complete proofs of both results in \cref{popveenuclear} and \cref{appenc1}.
	\end{rem}
	
	It remains to pass from  \(\mcp_{\Sigma}(\op{Proj}_R^{\op{cpl},\omega_1})\) back to the actual  category of connective complete modules.
	\begin{thm}\label{main5}
		Let $R$ be an adic \einfring. Then $\nucrgeq$ can be identified with the additive rigidification of $\modrcplgeq$. 	
	\end{thm}
	\begin{proof}
		Let $\mcc=\modrcplgeq$. By \cref{nucadrig}, $\nuc(R)_{\geq0}$ can be identified with the additive rigidification of $\mcp_{\Sigma}(\mcc^{\omega_1\text{-}\op{proj}})$.
		Therefore,  it suffices to show that the following functor is an equivalence $$\mcp_{\Sigma}(\mcc^{\omega_1\text{-}\op{proj}})^{\op{arig}}\xrightarrow{}\mcc^{\op{arig}}.$$ However, that follows from combining 	\cref{addrigidificationkappaproj} and
		\cref{modrcplomega1projgen}.
	\end{proof} 

We finish by spelling out a consequence for localizations of complete adic rings, which includes the basic examples of uniform Tate rings.
	\begin{cor}\label{main6}
		Let $R$ be a complete adic \einfring and let $S\subset\pi_0R$ be a submonoid. We denote $\mcc=\ind(\stab(\op{Proj}_R^{\op{cpl},\omega_1}))$ and define
		\[
		\nuc(R[S^{-1}]):=\nuc(\modu_{R[S^{-1}]}(\mcc)).
		\]
		Then 	$\nuc(R[S^{-1}])_{\geq0}$ is additively rigid.
		
		Consequently, for any uniform Tate ring $A$, $\nuc(A)_{\geq0}$ is additively rigid; in particular, $$\nuc(\mathbb{Q}_p)_{\geq0}$$ is additively rigid.
		
	\end{cor}

	\begin{proof}
		 It is straightforward to see that $R[S^{-1}]$ lies in $\nuc(R)_{\geq0}$, because by \cref{modrcpldsubsetstab} we have a fully faithful symmetric monoidal embedding $\modu_R(\opsp)\simeq\ind(\modu_R^{\op{cpl},d})\hookrightarrow \nuc(R).$
		 It follows from \cref{modnucaddrig} that $\nuc(R[S^{-1}])_{\geq0}$ is additively rigid.
	\end{proof}

		The preceding corollary covers many analytic examples, but it leaves open the extent to which additive rigidity holds in greater generality:
	\begin{quest}
		Is $\nuc(A)_{\geq0}$ additively rigid for an arbitrary analytic ring $A$? More generally, when is $\nuc(X)_{\geq0}$ additively rigid for an analytic stack $X$?
	\end{quest}

\section{Localizing invariants and spherical sheaves}\label{section6}
Having characterized the internal structure of dualizable additive \infcats, we now turn to their global functorial behavior and localizing invariants. Extending Pstrągowski's theory of spherical sheaves \cite{pstrkagowski2023synthetic}, we establish an identification between localizing invariants and spherical sheaves.

We begin by recalling the categorical properties of the ambient categories of dualizable and atomically generated additive \infcats. These properties ensure that the usual additivity and localization arguments can be carried out internally in this setting.
	\begin{thm}[{\cite[Theorem 3.1, 3.14]{ramzi2024dualizable}}]\label{praddblkprl}
		$\prad^{\op{dbl}}$ is itself $\omega_1$-presentably symmetric monoidal and the symmetric monoidal (non-full) inclusion $\praddbl\to \prl_{\op{ad},\omega_1}$ preserves and reflects $\kappa$-compact objects for any uncountable regular cardinal $\kappa$.\footnote{The reflection property here comes from the comonadicity, which is proved in \cite{ramzi2024dualizable} too.} 
	\end{thm}
	Besides the full category of dualizable additive \infcats, we will also need its compact projectively generated part.
	\begin{de}
		We define   $\pradat\subset \praddbl$ to be  the  full subcategory of those atomically generated (i.e. compact projectively generated) presentable additive \infcats. 
	\end{de}

	\begin{rem}
		Note that we have a natural equivalence $\catadidem\xrightarrow[\sim]{\mcp_{\Sigma}(-)}\pradat$. That is similar to the stable case $\catper\xrightarrow[\sim]{\ind(-)}\prst^{\op{at}}$.
	\end{rem}
	The next basic structural fact is that these ambient categories themselves retain the semi-additive nature.
	\begin{prop}
		Both \pradat and $\praddbl$ are  semi-additive \infcats.
	\end{prop}
	\begin{proof}
		It suffices to show that both (non-full) inclusions $\pradat\to \prl$ and $\praddbl \to \prl$ create finite coproducts and finite products, because $\prl$ is semi-additive.
		
		For the finite coproducts, it follows from \cite[Proposition 1.62]{ramzi2024dualizable}\footnote{The argument in \cite[Proposition 1.62]{ramzi2024dualizable} also proves $\pradat\to \prl$ creates small colimits.}. 
		For the finite products, it suffices to show that for two colimit-preserving functors $\mcb\xrightarrow{F}\mcc, \mcb\xrightarrow{G}\mcd$ between presentable additive \infcats, $\mcb\xrightarrow{(F,G)}\mcc\times \mcd$ is an internal left adjoint if and only if both $F,G$ are. 	Let $H = (F, G): \mcb \to \mcc \times \mcd$. Since $\mcb$ is an additive \infcat, finite products and finite coproducts in $\mcb$ canonically coincide (denoted by $\oplus$). The right adjoint $H^R: \mcc \times \mcd \to \mcb$ can be explicitly computed by $H^R(C, D) \simeq F^R(C) \oplus G^R(D)$.
			 In the presentable additive \infcat $\mcb$, the direct sum functor $\oplus: \mcb \times \mcb \to \mcb$ is a finite colimit, and thus commutes with all small colimits. Therefore, the composite functor $H^R(C, D) \simeq F^R(C) \oplus G^R(D)$ preserves small colimits if and only if both $F^R: \mcc \to \mcb$ and $G^R: \mcd \to \mcb$ individually preserve small colimits. 
			
	\end{proof}
	We next record a simple stability property of internal left adjoints under cofiber sequences of functors.
	\begin{lem}\label{functorext}
	Let $f\to t\to c$ be a cofiber sequence in $\funct(\mcc,\mcd)$ where \mcc and \mcd are presentable prestable \infcats. 
	\enu{
		\item If both $f,c$ preserve colimits, then so does $t$.
		\item If both $f,c$ are internal left adjoints, then so is $t$.}
\end{lem}
\begin{proof}
	(1) Suppose $f$ and $c$ preserve small colimits. Let $x\simeq\colimit_\alpha x_\alpha$ be a colimit diagram in \mcc. It follows from the following diagram that $\colimit_\alpha t(x_\alpha)\simeq t(\colimit_\alpha x_\alpha)$.
	$$\begin{tikzcd}
		\colimit_\alpha f(x_\alpha) \arrow[d, "\sim"] \arrow[r] & \colimit_\alpha t(x_\alpha) \arrow[d] \arrow[r] & \colimit_\alpha c(x_\alpha) \arrow[d, "\sim"] \\
		f(\colimit_\alpha x_\alpha) \arrow[r]                  &  t(\colimit_\alpha x_\alpha) \arrow[r]          &  c(\colimit_\alpha x_\alpha)                 
	\end{tikzcd}$$
	
	(2) Assume $f$ and $c$ are internal left adjoints. By part (1), $t$ preserves colimits. By \cref{prestemb}, $t$ is an internal left adjoint if and only if the stable right adjoint $T^R=\opsp(t)^R$ preserves colimits and the connective part. Since $f$ and $c$ are internal left adjoints, their stable extensions $F$ and $C$ are internal left adjoints satisfying that $F^R$ and $C^R$ preserve small colimits and the connective part. The cofiber sequence of left adjoints $F \to T \to C$ yields a fiber sequence of right adjoints $$C^R \to T^R \to F^R$$ in $\funct(\operatorname{Sp}(\mcd), \operatorname{Sp}(\mcc))$. Because $\operatorname{Sp}(\mcc)$ is stable, this fiber sequence is equivalently a cofiber sequence. Therefore, $T^R$ preserves colimits and the connective part too, as desired.
\end{proof}
To encode cofiber sequences uniformly, we use the standard \(S_2\)-construction.
\begin{nota}
	Let $\mathcal{C}$ be a pointed \infcat with finite colimits. We let $S_2 \mathcal{C} \subset \operatorname{Fun}(\square, \mathcal{C})$ denote the full subcategory spanned by those squares which are cofiber sequences, i.e. pushout squares where the bottom left term is a zero object. 
	
	We define three functors $f, t, c: S_2 \mathcal{C} \rightarrow \mathcal{C}$ that take a cofiber sequence $x \rightarrow y \rightarrow z$ to $x, y, z$ respectively. Note that there is a canonical cofiber sequence of functors $f \rightarrow t \rightarrow c$ from $S_2 \mathcal{C}$ to $\mathcal{C}$.  Note that the natural projection $$S_2\mcc\to\mcc^{\Delta^1}$$ given by $(x\to y\to z)\mapsto (x\to y)$ is a categorical equivalence.
\end{nota}
The elementary functors out of \(S_2\mathcal C\) satisfy a useful chain of adjunctions, which will be used to compare different forms of split exactness.
\begin{prop}\label{adjss2}
	Let $\mathcal{C}$ be a prestable \infcat with finite limits. Define the following seven functors between $S_2\mathcal{C}$ and $\mathcal{C}$:
	\begin{enumerate}[label=(\arabic*),font=\normalfont]
		\item $c: S_2\mathcal{C} \to \mathcal{C}$ given by $c(x \to y) = y \cup_x 0$ (the cofiber).
		\item $i_2: \mathcal{C} \to S_2\mathcal{C}$ given by $i_2(a) = (0 \to a)$.
		\item $t: S_2\mathcal{C} \to \mathcal{C}$ given by $t(x \to y) = y$ (target evaluation).
		\item $i_1: \mathcal{C} \to S_2\mathcal{C}$ given by $i_1(a) = (a \xrightarrow{\operatorname{id}} a)$.
		\item $f: S_2\mathcal{C} \to \mathcal{C}$ given by $f(x \to y) = x$ (source evaluation).
		\item $i_0: \mathcal{C} \to S_2\mathcal{C}$ given by $i_0(a) = (a \to 0)$.
		\item $\Omega c: S_2\mathcal{C} \to \mathcal{C}$ given by $\Omega c(x \to y) = x \times_y 0$ (the fiber).
	\end{enumerate}
	Then there exists a sequence of six adjunctions:
	$$ c \dashv i_2 \dashv t \dashv i_1 \dashv f \dashv i_0 \dashv \Omega c .$$
\end{prop}

\begin{proof}
	We establish this chain of adjunctions by verifying the natural equivalences of the associated mapping anima. Let $(x \to y)$ be an object in $S_2\mathcal{C}$ and $a$ be an object in $\mathcal{C}$.
	
	1. $c \dashv i_2$ :
	\begin{align*}
		\operatorname{Map}_{\mathcal{C}}(c(x \to y), a) &\simeq \operatorname{Map}_{\mathcal{C}}(y \cup_x 0, a) \\
		&\simeq \operatorname{Map}_{\mathcal{C}}(y, a) \times_{\operatorname{Map}_{\mathcal{C}}(x, a)} \operatorname{Map}_{\mathcal{C}}(0, a) \\
		&\simeq \operatorname{Map}_{\mathcal{C}}(y, a) \times_{\operatorname{Map}_{\mathcal{C}}(x, a)} *
	\end{align*}
	This is precisely the anima of maps from $y$ to $a$ such that the composition $x \to y \to a$ is equipped with a nullhomotopy. On the other hand, the mapping anima $\operatorname{Map}_{S_2\mathcal{C}}((x \to y), i_2(a))$ consists of commutative squares from $(x \to y)$ to $(0 \to a)$, which is equivalent to requiring the composition $x \to y \to a$ to factor through $0$. Thus, the two mapping anima are naturally equivalent.
	
	2. $i_2 \dashv t$ :
	\begin{align*}
		\operatorname{Map}_{S_2\mathcal{C}}(i_2(a), (x \to y)) &\simeq \operatorname{Map}_{S_2\mathcal{C}}((0 \to a), (x \to y)) \\
		&\simeq \operatorname{Map}_{\mathcal{C}}(0, x) \times_{\operatorname{Map}_{\mathcal{C}}(0, y)} \operatorname{Map}_{\mathcal{C}}(a, y) \\
		&\simeq * \times_{*} \operatorname{Map}_{\mathcal{C}}(a, y) \simeq \operatorname{Map}_{\mathcal{C}}(a, y)
	\end{align*}
	This naturally matches $\operatorname{Map}_{\mathcal{C}}(a, t(x \to y))$.
	
3. $t \dashv i_1$ :
	\begin{align*}
		\operatorname{Map}_{S_2\mathcal{C}}((x \to y), i_1(a)) &\simeq \operatorname{Map}_{S_2\mathcal{C}}((x \to y), (a \xrightarrow{\operatorname{id}} a)) \\
		&\simeq \operatorname{Map}_{\mathcal{C}}(x, a) \times_{\operatorname{Map}_{\mathcal{C}}(x, a)} \operatorname{Map}_{\mathcal{C}}(y, a) \\
		&\simeq \operatorname{Map}_{\mathcal{C}}(y, a)
	\end{align*}
	This naturally matches $\operatorname{Map}_{\mathcal{C}}(t(x \to y), a)$.
	
	4. $i_1 \dashv f$ :
	\begin{align*}
		\operatorname{Map}_{S_2\mathcal{C}}(i_1(a), (x \to y)) &\simeq \operatorname{Map}_{S_2\mathcal{C}}((a \xrightarrow{\operatorname{id}} a), (x \to y)) \\
		&\simeq \operatorname{Map}_{\mathcal{C}}(a, x) \times_{\operatorname{Map}_{\mathcal{C}}(a, y)} \operatorname{Map}_{\mathcal{C}}(a, y) \\
		&\simeq \operatorname{Map}_{\mathcal{C}}(a, x)
	\end{align*}
	This naturally matches $\operatorname{Map}_{\mathcal{C}}(a, f(x \to y))$.
	
5. $f \dashv i_0$ :
	\begin{align*}
		\operatorname{Map}_{S_2\mathcal{C}}((x \to y), i_0(a)) &\simeq \operatorname{Map}_{S_2\mathcal{C}}((x \to y), (a \to 0)) \\
		&\simeq \operatorname{Map}_{\mathcal{C}}(x, a) \times_{\operatorname{Map}_{\mathcal{C}}(y, 0)} \operatorname{Map}_{\mathcal{C}}(x, 0) \\
		&\simeq \operatorname{Map}_{\mathcal{C}}(x, a) \times_{*} * \simeq \operatorname{Map}_{\mathcal{C}}(x, a)
	\end{align*}
	This naturally matches $\operatorname{Map}_{\mathcal{C}}(f(x \to y), a)$.
	
	6. $i_0 \dashv \Omega c$ :
	\begin{align*}
		\operatorname{Map}_{S_2\mathcal{C}}(i_0(a), (x \to y)) &\simeq \operatorname{Map}_{S_2\mathcal{C}}((a \to 0), (x \to y)) \\
		&\simeq \operatorname{Map}_{\mathcal{C}}(a, x) \times_{\operatorname{Map}_{\mathcal{C}}(a, y)} \operatorname{Map}_{\mathcal{C}}(0, y) \\
		&\simeq \operatorname{Map}_{\mathcal{C}}(a, x) \times_{\operatorname{Map}_{\mathcal{C}}(a, y)} * \\
		&\simeq \operatorname{Map}_{\mathcal{C}}(a, x \times_y 0)
	\end{align*}
	This is precisely $\operatorname{Map}_{\mathcal{C}}(a, \Omega c(x \to y))$ by the universal property of the fiber (pullback).
\end{proof}

\begin{rem}
 We remark that because $\mathcal{C}$ is only assumed to be prestable rather than stable, the chain of adjunctions stops here. The cofiber functor $c$ does not generally preserve limits, hence lacks a left adjoint, and the fiber functor $\Omega c$ does not generally preserve colimits, hence lacks a right adjoint.
\end{rem}
We also need a finiteness property of the arrow category, ensuring that the \(S_2\)-construction remains inside the atomically generated world when the original category does.
\begin{prop}
	Let $\mcc\in\praddbl$.	If $\mcc$ is compact projectively generated, then so is $\mcc^{\Delta^1}$.
\end{prop}
\begin{proof}
	Let $S$ be a small set of compact projective generators for $\mcc$. Consider the two evaluation functors $ev_0, ev_1: \mcc^{\Delta^1} \to \mcc$ given by evaluating an arrow at its domain and codomain, respectively. 
	These evaluation functors admit left adjoints $L_0, L_1: \mcc \to \mcc^{\Delta^1}$, which can be explicitly described as follows. For any object $c \in \mcc$,
	$$ L_0(c) = (c \xrightarrow{\operatorname{id}} c) \quad \text{and} \quad L_1(c) = (0 \to c). $$
	Because $ev_i$ preserves small sifted colimits, if $c$ is a compact projective object in $\mcc$, the composite functor $\operatorname{Map}_\mcc(c, ev_i(-))$ preserves small sifted colimits. This implies that the images $L_0(c)$ and $L_1(c)$ are compact projective objects in $\mcc^{\Delta^1}$.
	
	Now, define $S' = \{ L_0(s) \mid s \in S \} \cup \{ L_1(s) \mid s \in S \}$. Since $S$ is small, $S'$ is a small set of compact projective objects in $\mcc^{\Delta^1}$. Since $\{ev_i|i=0,1\}$ is a jointly conservative collection, we see that $S'$ generates $\mcc^{\Delta^1}$. 
This proves that $\mcc^{\Delta^1}$ is compact projectively generated.
\end{proof}
The next point is more subtle: in the prestable setting, not every natural transformation between internal left adjoints determines an internal left adjoint into the arrow category. The following proposition gives the precise criterion needed later.
\begin{prop}\label{dualaddintervalpower}
	Let $\theta:f_1\to f_2$ be a natural transformation in $\funct^{LL}(\mcc,\mcd)$, where $\mcc, \mcd$ are presentable prestable \infcats. Then the following conditions are equivalent:
	\enu{
		\item The functor $\cofib(\theta)$ is an internal left adjoint from \mcc to \mcd.
		\item The induced transformation on corresponding right adjoints $f_2^R\to f_1^R$ is pointwise effective epimorphic.
		\item The induced functor $\mcc\to \mcd^{\Delta^1}$ is an internal left adjoint.
	}
\end{prop}
\begin{proof}
	Let $F: \mcc \to \mcd^{\Delta^1}$ be the induced functor given by $F(C) = (f_1(C) \xrightarrow{\theta_C} f_2(C))$. We first explicitly compute its right adjoint $F^R: \mcd^{\Delta^1} \to \mcc$. For any object $(X \xrightarrow{\alpha} Y) \in \mcd^{\Delta^1}$ and $C \in \mcc$, the definition of morphisms in the arrow category yields natural equivalences of mapping anima:
	\begin{align*}
		\operatorname{Map}_{\mcc}(C, F^R(X \xrightarrow{\alpha} Y)) &\simeq \operatorname{Map}_{\mcd^{\Delta^1}}(F(C), (X \xrightarrow{\alpha} Y)) \\
		&\simeq \operatorname{Map}_{\mcd^{\Delta^1}}((f_1(C) \to f_2(C)), (X \xrightarrow{\alpha} Y)) \\
		&\simeq \operatorname{Map}_{\mcd}(f_1(C), X) \times_{\operatorname{Map}_{\mcd}(f_1(C), Y)} \operatorname{Map}_{\mcd}(f_2(C), Y).
	\end{align*}
	Using the adjunctions $f_1 \dashv f_1^R$ and $f_2 \dashv f_2^R$, this mapping anima is canonically equivalent to:
	$$ \operatorname{Map}_{\mcc}(C, f_1^R(X)) \times_{\operatorname{Map}_{\mcc}(C, f_1^R(Y))} \operatorname{Map}_{\mcc}(C, f_2^R(Y)) \simeq \operatorname{Map}_{\mcc}\left(C, f_1^R(X) \times_{f_1^R(Y)} f_2^R(Y)\right). $$
	By Yoneda's lemma, this yields a canonical identification $F^R(X \xrightarrow{\alpha} Y) \simeq f_1^R(X) \times_{f_1^R(Y)} f_2^R(Y)$. Similarly, let $ \opsp(\mcc)\xrightarrow{F_i=\opsp\otimes f_i}\opsp(\mcd)$ denote the stabilization of $f_i$,  and let $F_1\xrightarrow{\opsp\otimes \theta} F_2$ denote the stabilization of the transformation $\theta$; then the right adjoint $H^R$ of the functor $\opsp(\mcc)\xrightarrow{H=\opsp\otimes F} \opsp(\mcd^{\Delta^1})\simeq \opsp(\mcd)^{\Delta^1}$ is given by $$H(X\to Y)= F_1^R(X) \times_{F_1^R(Y)} F_2^R(Y).$$
	We are now ready to prove the equivalences:
	
$(2) \implies (3)$: By \cref{prestemb}, the functor $F$ is an internal left adjoint if and only if the right adjoint $H^R=(\opsp\otimes F)^R: \opsp(\mcd)^{\Delta^1}\to \opsp(\mcc)$ is colimit-preserving and right $t$-exact. That follows from the explicit formula $H^R(X \xrightarrow{} Y) \simeq F_1^R(X) \times_{F_1^R(Y)} F_2^R(Y)$.

	$(3) \implies (1)$: If $F$ is an internal left adjoint, then since the cofiber functor $c: \mcd^{\Delta^1} \to \mcd$ is also an internal left adjoint (its right adjoint $i_2$ preserves colimits because limits and colimits in $\mcd^{\Delta^1}$ are computed pointwise), their composition $\cofib(\theta) = c \circ F$ is automatically an internal left adjoint.
	
	$(1) \implies (2)$: Suppose $\cofib(\theta)$ is an internal left adjoint. 	Now, recall the cofiber functor $c = \operatorname{cofib}: \mcd^{\Delta^1} \to \mcd$. As established previously in \cref{adjss2}, its right adjoint is the target inclusion $i_2: \mcd \to \mcd^{\Delta^1}$ given by $i_2(Y) = (0 \to Y)$. The functor $\cofib(\theta)$ factors canonically as $c \circ F$. Therefore, its right adjoint evaluates as:
	$$(\cofib(\theta))^R(Y)\simeq F^R(i_2(Y))\simeq F^R(0\to Y)\simeq f_1^R(0)\times_{f_1^R(Y)}f_2^R(Y)$$
	$$\simeq 0\times_{f_1^R(Y)}f_2^R(Y)\simeq\operatorname{fib}\bigl(f_2^R(Y)\to f_1^R(Y)\bigr).$$ Since the functor $\cofib(\theta)$ is an internal left adjoint indicates that the stable right adjoint $(\opsp\otimes\cofib(\theta))^R: \opsp(\mcd)\to \opsp(\mcc)$ is colimit-preserving and right $t$-exact, we deduce that
	the functor $$\operatorname{fib}\Bigl(F_2^R(-) \to F_1^R(-)\Bigr)\simeq \bigl(\cofib(F_1\to F_2)\bigr)^R\simeq (\opsp\otimes\cofib(\theta))^R$$ is colimit-preserving and right $t$-exact. Thus the map $f_2^R(Y) \to f_1^R(Y)\in\mcc$ must be an effective epimorphism for each $Y\in \mcd$.
\end{proof}
This criterion explains the extent to which \(\mathcal D^{\Delta^1}\) behaves like an interval object for dualizable additive \infcats.
\begin{rem}
	\cref{dualaddintervalpower} above indicates that, unlike in the stable case, the following inclusion is in general proper  $$\funct^{LL}(\mcc,\mcd^{\Delta^1})\subsetneqq \funct(\Delta^1, \funct^{LL}(\mcc,\mcd)),$$ when $\mcc, \mcd$ are \dualaddinfcats. That means $\mcd^{\Delta^1}$ is not the rigorous interval power in \praddbl, but it happens to be our desired interval power when it comes to splitting/localizing invariants.
\end{rem}
We can now formulate the two classes of invariants considered in this section.

\begin{de}\label{deflocinv}
	Let $\mce$ be an additive \infcat, and let $E: \prad^\star \to \mce$ be a functor where $\star \in \{ \op{dbl}, \op{at} \}$.
	\begin{enumerate}[label=(\arabic*), font=\normalfont]
		\item We say that $E$ is a \textbf{splitting invariant} if it preserves finite products and satisfies the following additivity condition: for any cofiber sequence $f \to t \to c$ in $\operatorname{Fun}(\mathcal{C}, \mathcal{D})$ consisting of internal left adjoints\footnote{By \cref{functorext}, $t$ is automatically an internal left adjoint if both $f$ and $c$ are.} where $\mcc,\mcd\in \prad^\star$, there is an equivalence $E(t) \simeq E(f)+ E(c)$ in $\pi_0\mapp_{\mce}(E(\mcc),E(\mcd))$.
		
		\item Assume further that $\mce$ admits finite limits. We say that $E$ is a \textbf{localizing invariant} if it carries fiber-cofiber sequences in $\prad^\star$ to fiber sequences in $\mce$.
	\end{enumerate}
\end{de}
	The following proposition gives several equivalent ways to recognize the splitting condition. In particular, it reduces additivity for arbitrary cofiber sequences of functors to the universal \(S_2\)-sequence.
	\begin{prop}\label{splinvchar}
		Let $E: \prad^\star \rightarrow \mathcal{E}$ be a finite-product-preserving functor where $\star \in \{ \op{dbl}, \op{at} \}$ to an additive \infcat \mce. The following are equivalent:
		\enu{
			
			\item $E$ is a splitting invariant;
			\item For any \dualaddinfcat $\mathcal{C}$, and for the canonical cofiber sequence of functors
			
			$$
			f \rightarrow t \rightarrow c: S_2 \mathcal{C} \rightarrow \mathcal{C},
			$$
			we have	$
			E(t) \simeq E(f)+E(c) ;
			$
			
			\item For any \dualaddinfcat $\mathcal{C}$, $E$ applied to $S_2 \mathcal{C} \xrightarrow{(f, c)} \mathcal{C} \times \mathcal{C}$ yields an equivalence;
			
			\item  For any left split\footnote{The ``left split'' means $i,p$ admit left adjoints.} exact sequence $\mathcal{K} \xrightarrow{i} \mathcal{C} \xrightarrow{p} \mathcal{D}$ in $\prad^\star$,  $E$ applied to $\mathcal{C} \xrightarrow{(i^L, p)} \mathcal{K} \times \mathcal{D}$ yields an equivalence.
			
			\item  For any right split\footnote{The ``right split'' means $i,p$ admit right adjoints such that $i^R,p^R$ are internal left adjoints.} exact sequence $\mathcal{K} \xrightarrow{i} \mathcal{C} \xrightarrow{p} \mathcal{D}$ in $\prad^\star$, $E$ applied to $\mathcal{C} \xrightarrow{(i^R, p)} \mathcal{K} \times \mathcal{D}$ yields an equivalence.
		}
	\end{prop}

	\begin{proof}
	We follow the arguments of \cite[Lemma B.7]{ramzi2024dualizable}, adapting them to the context of left and right split exact sequences in $\prad^\star$. Note that (2) is a special case of (1) evaluated on the canonical sequence on $S_2 \mathcal{C}$. 
	We will show that $(1) \Rightarrow (4)$, $(1) \Rightarrow (5)$, $(4) \Rightarrow (3)$, $(5) \Rightarrow (3)$, $(3) \Rightarrow (2)$, and finally $(2) \Rightarrow (1)$.
		
		\textbf{$(1) \Rightarrow (4)$:} Consider a left split exact sequence $\mathcal{K} \xrightarrow{i} \mathcal{C} \xrightarrow{p} \mathcal{D}$. By definition, $i$ and $p$ admit left adjoints $i^L$ and $p^L$. The exactness implies that we have a canonical cofiber sequence of endofunctors on $\mathcal{C}$:
		$$ i \circ i^L \to \operatorname{id}_{\mathcal{C}} \to p^L \circ p. $$
		Applying the splitting invariant $E$, we obtain $E(\operatorname{id}_{\mathcal{C}}) \simeq E(i)E(i^L) + E(p^L)E(p)$. Let $(i, p^L): \mathcal{K} \times \mathcal{D} \to \mathcal{C}$ be the functor $(k, d) \mapsto i(k) \oplus p^L(d)$. Since $E$ preserves finite products (and thus direct sums), the equation above translates to:
		$$ E(\operatorname{id}_{\mathcal{C}}) \simeq E((i, p^L)) \circ E((i^L, p)). $$
		Conversely, evaluating the composition $(i^L, p) \circ (i, p^L)$ on an object $(k, d) \in \mathcal{K} \times \mathcal{D}$ yields $(i^L(i(k) \oplus p^L(d)), p(i(k) \oplus p^L(d)))$. Since $i$ is fully faithful ($i^L i \simeq \operatorname{id}$), $p p^L \simeq \operatorname{id}$, and the cross terms vanish by exactness ($i^L p^L \simeq 0$ and $p i \simeq 0$), we have $(i^L, p) \circ (i, p^L) \simeq \operatorname{id}_{\mathcal{K} \times \mathcal{D}}$. Therefore, $E((i, p^L))$ and $E((i^L, p))$ are mutually inverse, proving $(4)$.
		
		\textbf{$(1) \Rightarrow (5)$:} Similarly, for a right split exact sequence, $i$ and $p$ admit internal left adjoint right adjoints $i^R$ and $p^R$. We have the cofiber sequence of endofunctors $p^R \circ p \to \operatorname{id}_{\mathcal{C}} \to i \circ i^R$. Analogous matrix calculus via $E$ shows that $E((i^R, p))$ is an equivalence with inverse $E((i, p^R))$, establishing $(5)$.
		
		\textbf{$(5) \Rightarrow (3)$:} To see explicitly how (3) is a special case of both (4) and (5), we construct two distinct split exact sequences involving $S_2\mathcal{C}$ and apply the chain of adjunctions $c \dashv i_2 \dashv t \dashv i_1 \dashv f \dashv i_0 \dashv \Omega c$ in \cref{adjss2}.
		Consider the following sequence:
		$$ \mathcal{C} \xrightarrow{i_1} S_2\mathcal{C} \xrightarrow{c} \mathcal{C} $$
		First, we verify this is an exact sequence. For any $a \in \mathcal{C}$, $i_1(a) = (a \xrightarrow{\operatorname{id}} a)$. Applying the cofiber functor $c$ yields $c(a \xrightarrow{\operatorname{id}} a) = a \cup_a 0 \simeq 0$. Thus, $c \circ i_1 \simeq 0$.
	 By our adjunction chain, the right adjoint of the identity embedding $i_1$ is precisely the source evaluation functor $f$. Thus, $i_1^R = f$.
		Applying the formula from (5), the functor $\mathcal{C} \xrightarrow{(i^R, p)} \mathcal{K} \times \mathcal{D}$ precisely becomes:
		$$ S_2\mathcal{C} \xrightarrow{(f, c)} \mathcal{C} \times \mathcal{C} $$
		Condition (5) guarantees this is an equivalence, which exactly recovers condition (3).
		
		\textbf{$(4) \Rightarrow (3)$:}
		Consider an alternative sequence:
		$$ \mathcal{C} \xrightarrow{i_2} S_2\mathcal{C} \xrightarrow{f} \mathcal{C} $$
		We verify this is exact. For any $a \in \mathcal{C}$, $i_2(a) = (0 \to a)$. Applying the source evaluation functor $f$ yields $f(0 \to a) = 0$. Thus, $f \circ i_2 \simeq 0$.
	 By our adjunction chain, the left adjoint of the target embedding $i_2$ is precisely the cofiber functor $c$. Thus, $i_2^L = c$.
		Applying the formula from (4), the functor $\mathcal{C} \xrightarrow{(i^L, p)} \mathcal{K} \times \mathcal{D}$ evaluates to:
		$$ S_2\mathcal{C} \xrightarrow{(c, f)} \mathcal{C} \times \mathcal{C} $$
		Condition (4) guarantees that $(c, f)$ is an equivalence. Up to the canonical isomorphism that swaps the two factors of the product $\mathcal{C} \times \mathcal{C}$, it again recovers condition (3).
		
		\textbf{$(3) \Rightarrow (2)$:} Assume $E((f,c)): E(S_2 \mathcal{C}) \to E(\mathcal{C} \times \mathcal{C})$ is an equivalence. We know that the functor $i_0 \oplus i_1: \mathcal{C} \times \mathcal{C} \to S_2 \mathcal{C}$ given by $(x, y) \mapsto (x \to x \oplus y \to y)$ provides a section for $(f, c)$. Therefore, $E(i_0 \oplus i_1)$ must be the inverse of $E((f, c))$. We can now calculate $E(t)$ by decomposing the identity:
		\begin{align*}
			E(t) &\simeq E(t) \circ E(\operatorname{id}_{S_2 \mathcal{C}}) \\
			&\simeq E(t) \circ E(i_0 \oplus i_1) \circ E((f, c)) \\
			&\simeq E(t) \circ (E(i_0) \oplus E(i_1)) \circ E((f, c)) \\
			&\simeq (E(\operatorname{id}_{\mathcal{C}}) \oplus E(\operatorname{id}_{\mathcal{C}})) \circ E((f, c)) \\
			&\simeq E(f) + E(c),
		\end{align*}
		which proves $(2)$. Note that we used $t \circ i_0 \simeq \operatorname{id}_{\mathcal{C}}$ and $t \circ i_1 \simeq \operatorname{id}_{\mathcal{C}}$.
		
		\textbf{$(2) \Rightarrow (1)$:} Suppose we have a cofiber sequence of internal left adjoint functors $F \to G \to H$ from $\mathcal{C} \to \mathcal{D}\in \prad^\star$. By \cref{dualaddintervalpower}, this data is equivalent to a single internal left adjoint $\Phi: \mathcal{C} \to S_2 \mathcal{D}$ such that postcomposing $\Phi$ with the canonical cofiber sequence $f \to t \to c$ on $S_2 \mathcal{D}$ recovers our original sequence $F \to G \to H$. Applying $E$, we have $E(G) \simeq E(t) \circ E(\Phi)$. By condition $(2)$, we substitute $E(t) \simeq E(f) + E(c)$, yielding:
		$$ E(G) \simeq (E(f) + E(c)) \circ E(\Phi) \simeq E(f \circ \Phi) + E(c \circ \Phi) \simeq E(F) + E(H). $$
		Thus, $E$ converts cofiber sequences of morphisms into sums, making it a splitting invariant.
	\end{proof}
	We next record the additive version of the Eilenberg swindle.
	\begin{prop}[The Eilenberg swindle]\label{swindle}
			Let $E: \prad^\star\to\mce$ be a splitting invariant to an additive \infcat where $\star \in \{ \op{dbl}, \op{at} \}$. If $\mca$ is a small additive \infcat with countable coproducts or with countable  products, then $E(\mcp_\Sigma(\mca))=0$.
	\end{prop}
	\begin{proof}
 Let $f: \mcp_\Sigma(\mca) \to \mcp_\Sigma(\mca)$ be the functor given by the countable coproduct (resp. product) of the identity functor, i.e., $f \simeq \bigoplus_{\mathbb{N}} \operatorname{id}_{\mcp_\Sigma(\mca)}$ (resp. $f \simeq \prod_{\mathbb{N}} \operatorname{id}_{\mcp_\Sigma(\mca)}$). 
		
		Notice that taking the direct sum with the identity functor shifts the countable sum (resp. product), giving a canonical equivalence of functors $f \oplus \operatorname{id}_{\mcp_\Sigma(\mca)} \simeq f$. Because $E$ is a splitting invariant, and $f$ is formed by colimit-preserving functors, applying $E$ yields the algebraic equation:
		$$ E(f) + E(\operatorname{id}_{\mcp_\Sigma(\mca)}) \simeq E(f \oplus \operatorname{id}_{\mcp_\Sigma(\mca)}) \simeq E(f) \quad \text{in } \pi_0\operatorname{End}_{\mce}(E(\mcp_\Sigma(\mca))). $$
		Since $\mce$ is an  additive \infcat, we can cancel $E(f)$ from both sides. This forces $E(\operatorname{id}_{\mcp_\Sigma(\mca)}) \simeq 0$. Since $E$ is an additive functor, $E(\mcp_\Sigma(\mca)) \simeq 0$ in $\mce$.
	\end{proof}
	We now compare the two notions of invariance. The next proposition shows that localizing invariance is stronger than splitting invariance.
	\begin{prop}
		Let $E: \prad^\star\to\mce$ be a localizing invariant to an additive \infcat with finite limits  where $\star \in \{ \op{dbl}, \op{at} \}$. Then $E$ is a splitting invariant. 
	\end{prop}
	\begin{proof}
	Let $\mathcal{K} \xrightarrow{i} \mathcal{C} \xrightarrow{p} \mathcal{D}$ be a left split exact sequence in $\prad^\star$. Because $E$ is a localizing invariant, it sends the bifiber sequence $\mathcal{K} \to \mathcal{C} \to \mathcal{D}$ to a fiber sequence in $\mce$:
	$$ E(\mathcal{K}) \xrightarrow{E(i)} E(\mathcal{C}) \xrightarrow{E(p)} E(\mathcal{D}). $$
	Since the sequence is left split, $i$ and $p$ admit left adjoints.
	The left adjoint $i^L$ induces a map $E(i^L): E(\mathcal{C}) \to E(\mathcal{K})$ which splits the fiber sequence in $\mce$, meaning $E(i^L) \circ E(i) \simeq \operatorname{id}_{E(\mathcal{K})}$. In any additive \infcat, a split fiber sequence canonically decomposes the middle object into a direct sum\footnote{By Yoneda embedding, it is reduced to the case $\spgeq$.}. Thus, the maps $E(i^L)$ and $E(p)$ induce an equivalence:
	$$ E(\mathcal{C}) \xrightarrow{} E(\mathcal{K})\times E(\mathcal{D}). $$
	In particular, $E$ preserves finite products by taking $ \mcc=\mck\times \mcd$.
	Therefore by \cref{splinvchar}, $E$ is a splitting invariant.
	\end{proof}
	The preceding discussion expresses localizing invariants in terms of exactness properties. We now reinterpret these exactness properties sheaf-theoretically, using a singleton Grothendieck topology in the spirit of Pstrągowski.
		\subsection{Digression: singleton Grothendieck topology}
	We generalize the theory of additive sites (with a singleton topology) that appeared in \cite[\textsection2]{pstrkagowski2023synthetic} to the semi-additive setting. A striking result is that  \emph{localizing invariants are  exactly spherical sheaves}.

	We begin with a few generalities on localizations and singleton-generated topologies.
	\begin{de}[Strongly saturated class]
		Let $\mathcal{E}$ be a cocomplete \infcat. We shall say that a class of maps $S \subset \mathcal{E}$ is strongly saturated if the following conditions hold:
		\enu{\item $S$ contains the isomorphisms and it has the 2-out-of-3 property;
			\item $S$ is closed under colimits.	
		} 
	\end{de}

\begin{rem}
		The definition of strongly saturated class in \cite[Definition 5.5.4.5]{htt} includes the additional condition that the class $S$ be closed under cobase change. The following lemma, however, shows that the additional condition is automatic.
\end{rem}
	\begin{lem}[{\cite[Lemma 2.2.5]{anel2022left}}]
		Let $S$ be a class of maps in a \infcat   $\mathcal{E}$ with finite colimits. If $S$ contains the isomorphisms and the full subcategory $S \subset \mathcal{E}^{\Delta^1}$ is closed under pushouts, then $S$ is closed under cobase change.
	\end{lem}
	We now introduce the semi-additive analogue of spherical presheaves.
	\begin{de}
		Let $\mathcal{C}$ be a small  \infcat with finite coproducts and \mce be a \infcat. We say a functor $\mcc^\opp\to \mce$ is a \textbf{spherical presheaf} if it preserves finite products.
	\end{de}
To define the singleton topology uniformly, we keep track of a class of morphisms stable under base change.
	\begin{de}\label{goodpair}
		Let $\mcc$ be a small \infcat with pullbacks and $Q$ be a collection of morphisms. We say $Q$ is a good collection if $Q$ contains all equivalences and is closed under pullbacks, i.e., if $f:x\to y\in Q$, then the morphism $x\times_y z\to z$ obtained by the pullback along any morphism $z\to y$ still lies in $Q$.
		
		We call such a pair $(\mcc,Q)$ a \textbf{good pair}. If moreover \mcc is semi-additive, then we call such a pair a \textbf{semi-additive good pair}.
	\end{de}
	We recall the minimal language of quasi-topologies needed to make this construction precise.
\begin{de}
	A \textbf{quasi-topology} $\tau$ on a \infcat $\mathcal{C}$ assigns to every $X \in \mathcal{C}$ a collection $\tau(X)$ of sieves on $X$, called $\tau$-sieves, such that, for every $f: Y \rightarrow X, f^* \tau(X) \subset \tau(Y)$.
\end{de}
Given a good pair, we now define the singleton quasi-topology generated by the morphisms in \(Q\).
	\begin{de}
	Let $(\mcc,Q)$ be a good pair. We define the quasi-topology $\tau_Q$ on \mcc as follows: a sieve $\mcc^{(0)}_{/X}\subset\mcc_{/X}$ of $X$ lies in $\tau_Q(X)$ if and only if it can be generated by a single morphism $(Z\to X)\in Q$.
	
	We say that a presheaf $F\in \mcp(\mcc)$ is a $Q$-sheaf if for every $X \in \mathcal{C}$ and every $R \in \tau_Q(X)$ (viewed as a subsheaf of $h_X$), the restriction map $\operatorname{Map}_{\mcp(\mcc)}(h_X, F) \rightarrow \operatorname{Map}_{\mcp(\mcc)}(R, F)$ is an equivalence.
	\end{de}
	Although \(\tau_Q\) is only a quasi-topology, it defines the same sheaf theory as the Grothendieck topology it generates. Moreover, the sheaf condition can be checked on Čech nerves of the singleton covers:
	\begin{prop}
		Let $(\mcc,Q)$ be a good pair. Let $\bar{\tau}_Q$ denote the smallest Grothendieck topology containing $\tau_Q$. Then $\shv_{\tau_Q}(\mcc)=\shv_{\bar{\tau}_Q}(\mcc)$.
		
		Moreover, a presheaf $F\in\mcp(\mcc)$ is a $Q$-sheaf if and only if for any $(X\to Y)\in Q$, the augmented cosimplicial diagram $\Delta^\triangleleft\xrightarrow{F(X^\bullet/Y)} \mcs$ induced by the \v{C}ech nerve
	\[
	\begin{tikzcd}
		F(Y) \arrow[r] & 
		F(X) \arrow[r, shift left=1.2ex] \arrow[r, shift right=1.2ex] & 
		F(X \times_Y X) \arrow[l] \arrow[r, shift left=2ex] \arrow[r] \arrow[r, shift right=2ex] & 
		F(X \times_Y X \times_Y X) \arrow[l, shift left=1ex] \arrow[l, shift right=1ex] \arrow[r, phantom, "\cdots"] & \null
	\end{tikzcd}
	\]
		is a limit diagram.
	\end{prop}
	\begin{proof}
		The first statement follows from \cite[Corollary C.2]{hoyois2015quadratic}. For the second statement, it suffices to observe that the sieve $R_f \hookrightarrow h_Y$ generated by a morphism $f: X \to Y$ can be identified with the geometric realization of the \v{C}ech nerve in $\mcp(\mcc)$. That follows from the following factorization:
		\[
		\begin{tikzcd}[column sep=large, row sep=large]
			\cdots \arrow[r, shift left=1.5ex] \arrow[r, description] \arrow[r, shift right=1.5ex] & 
			h_{X \times_Y X} \arrow[r, shift left=1ex, "d^0"] \arrow[r, shift right=1ex, "d^1"'] & 
			h_X \arrow[r, two heads] \arrow[rd, "h_f"'] & 
			R_f\simeq|h_{X^\bullet/Y}| \arrow[d, hook] \\
			& & & h_Y
		\end{tikzcd}
		\]
	\end{proof}
	For later applications to invariants valued in an arbitrary additive \infcat, we will use the following tensor product compatibility:
		\begin{rem}
		Let $(\mcc,Q)$ be a semi-additive good pair. For any
		presentable additive \infcat $\mce$, by a similar argument to \cite[Proposition 1.3.1.7]{sag} we have a canonical  equivalence:
		$$\shv_\Sigma(\mcc;\spgeq)\otimes\mce\xrightarrow[]{\sim} \shv_\Sigma(\mcc;\mce).$$
	\end{rem}
The key point of the singleton topology is that, for spherical presheaves, the sheaf condition admits a purely exactness-theoretic formulation.
		\begin{thm}\label{sphshvcriterion}
		Let $(\mcc,Q)$ be a semi-additive good pair and \mce be an additive \infcat. A spherical presheaf $X:\mcc^\opp\to\mce$ is a $Q$-sheaf if and only if for every fiber sequence $F\to B\to A$ in \mcc with $B\to A$ in $Q$, the induced sequence
		$X(A)\to X(B)\to X(F)$
		is a fiber sequence in \mce.
	\end{thm}
	
	\begin{proof}
		The ``if'' direction is proved by the same argument as in \cite[Theorem 2.8]{pstrkagowski2023synthetic}.
		
		For the ``only if'' direction we only need a small modification of the original argument in \cite[Theorem 2.8]{pstrkagowski2023synthetic}. The potential issue is that the iterated fiber sequence in \mcc
		\[
		F\times_F\cdots\times_F F \longrightarrow (B\times_A B)\times_B\cdots\times_B(B\times_A B)\longrightarrow B\times_A\cdots\times_A B
		\]
		need not split in the semi-additive setting, even though the last map admits a retract. However, since the target \infcat \mce is additive, applying $X(-)$ to this sequence yields a split fiber sequence. 
	\end{proof}

	\begin{rem}
		The semi-additivity hypothesis is essential in this argument and cannot be weakened to pointedness, as the identification $X(F \times F) \simeq X(F) \oplus X(F)$ is crucially required in the original proof \cite[Theorem 2.8]{pstrkagowski2023synthetic}.
	\end{rem}
Since the sheaf condition is expressed by finite limits, it is compatible with filtered colimits:
	\begin{thm}
		Let $(\mcc,Q)$ be a semi-additive good pair. Then the inclusion $$\shv_\Sigma(\mcc;\spgeq)\hookrightarrow \mcp_\Sigma(\mcc;\spgeq)$$ preserves filtered colimits. In particular, $\shv_\Sigma(\mcc;\spgeq)$ is a compactly generated prestable \infcat.
	\end{thm}
	\begin{proof}
		By \cref{sphshvcriterion}, being a spherical sheaf is a condition that can be
		described using finite limits. It follows that the inclusion $\shv_\Sigma(\mcc;\spgeq)\hookrightarrow \mcp_\Sigma(\mcc;\spgeq)$ is closed under filtered colimits, because $\mcp_\Sigma(\mcc;\spgeq)$ is Grothendieck prestable.
	\end{proof}
	We next discuss functoriality of spherical sheaves with respect to maps of singleton sites.
		\begin{de}
		Let $(\mcc,Q)$ and $(\mcd,R)$ be two good pairs. 
		We say $f:(\mcc,Q)\to (\mcd,R)$ is a morphism of good pair if the functor $f:\mcc \to \mcd$ preserves pullbacks and good morphisms. If both $(\mcc,Q)$ and $(\mcd,R)$ are semi-additive good pairs, then we say $f$ is a \textbf{spherical morphism} if $f$ further preserves finite products\footnote{It is enough to preserve zero object since we have required the pullback-preserving condition.}.
	\end{de}
	The following technical lemma ensures that localizations of presheaf categories preserve sphericity under a mild finite-product hypothesis.
		\begin{lem}\label{LSpreservesphe-semi}
		Let $\mathcal{C}$ be a small semi-additive  \infcat, $S$ a small set of morphisms in $\mcp(\mathcal{C})$ and $$L^S: \mcp(\mathcal{C}) \rightarrow S^{-1} \mcp(\mathcal{C})$$ the associated localization functor taking values in $S$-local presheaves. If $S$ consists only of morphisms of spherical presheaves and $L^S$ preserves finite products, then $L^S$ takes spherical presheaves to spherical $S$-local presheaves.
	\end{lem}
	\begin{proof}
		Since $\mcp_{\Sigma}(\mathcal{C})$ is presentable and $S$ consists of morphisms of spherical presheaves, it follows that there also exists an associated localization on the spherical presheaf \infcat, which we denote by $$L_{\Sigma}^S: \mcp_{\Sigma}(\mathcal{C}) \rightarrow S^{-1} \mcp_{\Sigma}(\mathcal{C}).$$ Our goal is to show $L^S$ already preserves spherical presheaves, so that $L^S$ and $L_{\Sigma}^S$ coincide on spherical presheaves.
		
		By the above, it is enough to prove that any spherical presheaf admits an $L^S$-equivalence into an $S$-local spherical presheaf. Since $L_{\Sigma}^S$ exists, it certainly admits an $L_{\Sigma}^S$-equivalence into an $S$-local spherical presheaf. As $L_{\Sigma}^S$-equivalences form the smallest strongly saturated class of morphisms of spherical presheaves containing $S$, it is enough to show that $L^S$-equivalences also form a strongly saturated class of morphisms of spherical presheaves and they contain $S$. The 2-out-of-3 property is clear. 
		
		Since the inclusion $\mcp_{\Sigma}(\mathcal{C}) \hookrightarrow \mcp(\mathcal{C})$ preserves sifted colimits, we deduce that $L^S$-equivalences are closed under sifted colimits in $\mcp_{\Sigma}(\mathcal{C})^{\Delta^1}$. It is thus enough to show that $L^S$-equivalences are closed under finite coproducts.
		
		Since $\mcp_{\Sigma}(\mathcal{C})$ is semi-additive by \cite[Corollary 2.4]{ggn}, so is $\mcp_{\Sigma}(\mathcal{C})^{\Delta^1}$. As finite coproducts and products coincide in semi-additive \infcats, we deduce that it is enough to know that $L^S$-equivalences of spherical presheaves are closed under finite products. This follows immediately from the
		assumption that $L^S$ preserves finite products.
	\end{proof}
	This lemma allows us to restrict the usual sheaf adjunctions to spherical sheaves.
	\begin{prop}\label{ressphe}
		Let $f: (\mathcal{C},Q) \rightarrow (\mathcal{D},R)$ be a spherical morphism of semi-additive good pairs. Then, the induced adjunction $$f_! \dashv f^* : \shv(\mathcal{C}) \rightleftarrows \shv(\mathcal{D})$$ restricts to one on the \infcats of spherical presheaves $$f_! \dashv f^*|_{\shv_\Sigma} : \shv_\Sigma(\mathcal{C}) \rightleftarrows \shv_\Sigma(\mathcal{D}).$$ Here, $f_!= L \circ \operatorname{Lan}_f$, where $\operatorname{Lan}_f$ is the left Kan extension of the composite $\mcc \rightarrow \mcd \rightarrow \mcp(\mcd)$ and $L$ is the sheafification, and $f^*$ is given by precomposition. 
	\end{prop}
	\begin{proof}
		Since $f$ preserves good morphisms and pullbacks, it induces an adjunction of the above form on sheaf \infcats
		$$f_! \dashv f^* : \shv(\mathcal{C}) \rightleftarrows \shv(\mathcal{D}).$$ We only have to verify that $f_!, f^*$ take spherical sheaves to spherical sheaves.
		
		In the case of the former, we first observe that $\operatorname{Lan}_f$ takes representables to representables and preserves sifted colimits, so that it takes spherical presheaves to spherical presheaves by \cite[Proposition 5.5.8.14]{htt}. Also, $L$  preserves sphericity by \cref{LSpreservesphe-semi} because $L$ is left-exact. Therefore $f_!= L \circ \operatorname{Lan}_f$ preserves spherical sheaves.
		
		On the other hand, $f^*$ is given by a precomposition along a  semi-additive functor, so it clearly takes spherical sheaves to spherical sheaves.
	\end{proof}
	We will also need a criterion ensuring that a fully faithful inclusion of sites induces an equivalence on sheaf categories.
	\begin{lem}[{\cite[Lemma C.3]{hoyois2015quadratic}}]\label{basisequi}
		Let $u: \mathcal{C} \hookrightarrow \mathcal{D}$ be a fully faithful functor between small \infcats. Let $\tau$ and $\rho$ be quasi-topologies on $\mathcal{C}$ and $\mathcal{D}$, respectively. Suppose that:
		\enu{
			\item Every $\tau$-sieve is generated by a cover $\left\{U_i \rightarrow X\right\}$ such that:
			\begin{enumerate}
				\item   the fiber products $U_{i_0} \times_X \cdots \times_X U_{i_n}$ exist and are preserved by $u$;
				\item 	 $\left\{u\left(U_i\right) \rightarrow u(X)\right\}$ is a $\bar{\rho}$-cover in $\mathcal{D}$.
			\end{enumerate}
			\item For every $X \in \mathcal{C}$ and every $\rho$-sieve $R \hookrightarrow u(X), u^*(R) \hookrightarrow X$ is a $\bar{\tau}$-sieve in $\mathcal{C}$.
			\item Every $X \in \mathcal{D}$ admits a $\bar{\rho}$-cover $\left\{U_i \rightarrow X\right\}$ such that the fiber products $U_{i_0} \times_X \cdots \times_X U_{i_n}$ exist and belong to the essential image of $u$.
		}
		Then the adjunction\footnote{The functor $u^*$ denotes the restriction along $u$ and the functor $u_*$ denotes the right Kan extension along $u$.} $$u^* \dashv u_* : \mcp(\mathcal{D}) \rightleftarrows \mcp(\mathcal{C})$$ restricts to an equivalence of \infcats $$\operatorname{Shv}_\rho(\mathcal{D}) \simeq \operatorname{Shv}_\tau(\mathcal{C}).$$ In particular, the functor $u_!:\operatorname{Shv}_\tau(\mathcal{C})\to \operatorname{Shv}_\rho(\mathcal{D})$, the left adjoint to $u^*|_{\shv}$, is an equivalence too.
	\end{lem}

	\begin{rem}
		Moreover, \cite[Lemma C.3]{hoyois2015quadratic} proves that a presheaf on $\mathcal{D}$ is a $\rho$-sheaf if and only if it is the right Kan extension of a $\tau$-sheaf on $\mathcal{C}$. 
	\end{rem}
	For the singleton topologies considered here, the preceding basis criterion takes the following concrete form.
	\begin{thm}\label{sphshvequi}
	Let $u: (\mathcal{C},Q) \hookrightarrow (\mathcal{D},R)$ be a fully faithful morphism of good pairs. Suppose that:
	\enu{
		\item For any $f: X\to u(Y)\in R$, there exists a map $u(Z)\to X$ in \mcd such that the composition $Z\to Y$ lies in $Q$.
		\item Every $X \in \mathcal{D}$ admits a map $v:u(U) \rightarrow X$ such that $v$ lies in $R$ and $U\in \mcc$.
	}
	Then the induced adjunction
	$$u_! \dashv u^* : \shv(\mathcal{C}) \rightleftarrows \shv(\mathcal{D})$$
	is an equivalence.
	
	Furthermore, if $u: (\mathcal{C},Q) \hookrightarrow (\mathcal{D},R)$ is a spherical morphism of semi-additive good pairs, then it restricts to an equivalence between \infcats of spherical presheaves
	$$u_! \dashv u^*|_{\shv_\Sigma} : \shv_\Sigma(\mathcal{C}) \rightleftarrows \shv_\Sigma(\mathcal{D}).$$
	\end{thm}
\begin{proof}
	One can verify that the conditions in \cref{basisequi} hold by assumption, therefore the adjunction
	$u_! \dashv u^* : \shv(\mathcal{C}) \rightleftarrows \shv(\mathcal{D})$
	is an equivalence.
	
	Now suppose that $u: (\mathcal{C},Q) \hookrightarrow (\mathcal{D},R)$ is a spherical morphism of semi-additive good pairs. By \cref{ressphe}, we see that the adjunction
	$u_! \dashv u^* : \shv(\mathcal{C}) \rightleftarrows \shv(\mathcal{D})$
	restricts to an adjunction
	$$
	u_! \dashv u^*|_{\shv_\Sigma} : \shv_\Sigma(\mathcal{C}) \rightleftarrows \shv_\Sigma(\mathcal{D}).
	$$
	Since $u_! \dashv u^*|_{\shv_\Sigma}$ is a restriction of an equivalence, it is an equivalence too.
\end{proof}

	\subsection{Localizing invariants as spherical sheaves}
	Applying the semi-additive site machinery developed above, we formulate our main  theorem in this section. We prove that   localizing invariants  are precisely spherical sheaves with respect to the singleton topology generated by fully faithful morphisms.
	
	The first step is to verify that the relevant categories of additive \infcats, equipped with fully faithful morphisms, form semi-additive good pairs.
	\begin{prop}\label{dualaddsaddpair}
	Let $\mcf\mcf$ denote the class of fully faithful morphisms. Then both $(\praddblop,\mcf\mcf)$ and $(\prad^{\op{at,op}},\mcf\mcf)$ form semi-additive good pairs in the sense of \cref{goodpair}.
	\end{prop}
	\begin{proof}
		Since both (non-full) inclusions $\pradat\to \prlad$ and $\praddbl\to\prlad$ create pushouts, it suffices to show that fully faithful morphisms in $\prlad$ are closed under pushouts. In other words, it suffices to show that Bousfield localizations in $\pr^R$ are closed under pullbacks. That, however, is classical (see \cite[Lemma 2.7]{ramzi2025every} for an argument).
	\end{proof}

	\begin{de}\label{dfn:fftopology}
		We refer to these topologies on $\praddblop$ and $\prad^{\op{at,op}}$ as the $\mcf\mcf$-topology. 
	\end{de}
	
	\begin{rem}
		Note that, rigorously speaking, $(\praddblop,\mcf\mcf)$ and $(\prad^{\op{at,op}},\mcf\mcf)$ are not small sites, but it does not affect much. One only needs to enlarge the universe appropriately and the same theory works.
	\end{rem}
	We can now state the promised identification between localizing invariants and spherical sheaves.
	\begin{thm}\label{locinvsphshv}
		Let $E: \prad^\star \to \mce$ be a functor to an additive \infcat $\mce$, where $\star \in \{ \op{dbl}, \op{at} \}$. Then $E$ is a localizing invariant if and only if it is a spherical sheaf in the $\mcf\mcf$-topology.
	\end{thm}
	\begin{proof}
		Since by \cref{dualaddsaddpair} both $(\praddblop,\mcf\mcf)$ and $(\prad^{\op{at,op}},\mcf\mcf)$ are semi-additive good pairs, the result follows immediately from \cref{sphshvcriterion}.
	\end{proof}
	\begin{rem}
		\cref{locinvsphshv} also applies to \catper and \prstdbl. That is, a functor $E:\catper\to \mce$ or $E:\prstdbl\to \mce$ to an additive \infcat \mce is a localizing invariant\footnote{in the sense that it sends short exact sequences to fiber sequences} if and only if it is a spherical sheaf in the $\mcf\mcf$-topology (cf. \cite[Proposition 4.17]{efimov2024k}).
	\end{rem}
	We finish by comparing localizing invariants on all dualizable additive \infcats with those defined only on the atomically generated ones.
	\begin{thm}\label{smalllargelocinv}
		Let \mce be a (potentially large) additive \infcat with finite limits. Then the following forgetful functor is an equivalence:
		$$\op{Loc}(\praddbl,\mce)\xrightarrow{\sim} \op{Loc}(\prad^{\op{at}},\mce).$$
		Furthermore, assume that $\mathcal{E}$ is a stable \infcat, and let $\mathcal{K}$ be a collection of small \infcats such that \mce admits $\mathcal{K}$-shaped colimits. Let $\operatorname{Loc}_{\mathcal{K}}(\praddbl,\mathcal{E})$ denote the full subcategory of $\operatorname{Loc}(\praddbl,\mathcal{E})$ spanned by localizing invariants preserving $\mathcal{K}$-shaped colimits. Then the above equivalence restricts to an equivalence $$\operatorname{Loc}_{\mathcal{K}}(\praddbl,\mathcal{E}) \xrightarrow{\sim} \operatorname{Loc}_{\mathcal{K}}(\prad^{\op{at}}, \mathcal{E}).$$
	\end{thm}
	\begin{proof}
		Embedding $\mce$ into a very large additive \infcat, we can assume that \mce admits large colimits and limits\footnote{e.g. one can take $\widehat{\mcp}_\Sigma(\mce)$.}. Since \mce is additive, by \cref{sphshvcriterion} the vertical arrows in the following diagram are equivalences: 
		$$\begin{tikzcd}
			{\op{Loc}(\praddbl,\mce)} \arrow[r]                                                                                     & {\op{Loc}(\prad^{\op{at}},\mce)}                                                                                      \\
			{\funct^{\op{lar}\text{-}R}\big(\widehat{\shv}_\Sigma(\praddblop)^\opp,\mce\big)} \arrow[u, "\sim"'] \arrow[r] & {\funct^{\op{lar}\text{-}R}\big(\widehat{\shv}_\Sigma(\prad^{\op{at,op}})^\opp,\mce\big)} \arrow[u, "\sim"']
		\end{tikzcd}
		$$
		where $\widehat{\shv}$ denotes the (very large) \infcat of $\mcf\mcf$-sheaves valued in the (very large) \infcat $\widehat{\mcs}$ of large anima.
		Therefore, it suffices to show that the natural functor
		$$u_!|_{\widehat{\shv}} :\widehat{\shv}_\Sigma(\prad^{\op{at,op}}) \longrightarrow \widehat{\shv}_\Sigma(\praddblop)$$
		is an equivalence, where $u$ denotes the inclusion $\prad^{\op{at,op}}\hookrightarrow\praddblop$. However, that follows from 
		 \cref{sphshvequi}, because $u$ satisfies all conditions in \cref{sphshvequi} by considering the inclusion $\mcc\xhookrightarrow{\hat{h}}\mcp_{\Sigma}(\mcc^{\omega_1})$ for any $\mcc\in\praddbl$.
		 
		It remains to check that this equivalence is compatible with the additional requirement of preserving \(\mathcal K\)-shaped colimits. Now assume that \mce is stable and admits \mck-shaped colimits. We mimic the argument in \cite[Corollary 2.76]{ramzi2024dualizable}. Let $E\in \op{Loc}(\prad^{\op{at}},\mce)$. It suffices to prove that for any small \infcat $K\in\mck$, $E$ preserves $K$-shaped colimits if and only if $E^{\text{cont}}\in \op{Loc}(\praddbl,\mce)$ does. Since $\pradat\hookrightarrow\praddbl$ preserves small colimits, one direction is clear: if $E^{\text{cont}}$ preserves them, then so does $E$.
		 
		 Conversely, suppose $E$ preserves $K$-shaped colimits and let $\mathcal{M}_{\bullet}: K \rightarrow \praddbl$ be a $K$-diagram. Consider the natural localization sequence $\mathcal{M}_k \rightarrow \mcp_{\Sigma}(\mathcal{M}_k^{\omega_1}) \rightarrow \mcp_{\Sigma}(\mathcal{M}_k^{\omega_1})/\mcm_k$ and consider its colimit: it gives a localization sequence in \praddbl $$\operatorname{colim}_K \mathcal{M}_k \rightarrow \operatorname{colim}_K \mcp_{\Sigma}(\mathcal{M}_k^{\omega_1}) \rightarrow \operatorname{colim}_K \mcp_{\Sigma}(\mathcal{M}_k^{\omega_1})/\mcm_k.$$
		 Since $E^{\text{cont}}$ vanishes on the middle \infcat by \cref{swindle} and assumption, it follows that $$E^{\text{cont}}(\operatorname{colim}_K \mathcal{M}_k) \simeq \Omega E^{\text{cont}}(\operatorname{colim}_K \mcp_{\Sigma}(\mathcal{M}_k^{\omega_1})/\mcm_k).$$ Now each $\mcp_{\Sigma}(\mathcal{M}_k^{\omega_1})/\mcm_k$ is compact projectively generated, therefore so is the colimit, and therefore 
		 $$
		 E^{\text{cont}}\big(\operatorname{colim}_K \mcp_{\Sigma}(\mathcal{M}_k^{\omega_1})/\mcm_k\big)  \simeq \operatorname{colim}_K E^{\text{cont}}\big( \mcp_{\Sigma}(\mathcal{M}_k^{\omega_1})/\mcm_k\big)
		 $$ by assumption.
		 From this it follows that $E^{\text{cont}}\left(\operatorname{colim}_K \mathcal{M}_k\right) \simeq \operatorname{colim}_K E^{\text{cont}}\left(\mathcal{M}_k\right)$.
	\end{proof}

	\section{Prestable motives}\label{sec:prestable_motives}
	In this section, we study the \infcat of prestable motives, which is a universal
	(finitary) localizing invariant on additive categories with values in a presentable stable \infcat.
	We show that algebraic $K$-theory is corepresented by maps from the unit, similarly to the results of
	\cite{blumberg2013universal}.
	
The construction of this universal category requires a mild amount of
set-theoretic control.   We therefore begin with a few compactness lemmas which will
allow us to pass between the large and small versions of the additive categories
appearing below.
	We thank Maxime Ramzi for help with the following lemma.
	\begin{lem}\label{kcompmono}
		Let $\mathcal{C}$ be a presentable \infcat. Then there exists a sufficiently large regular cardinal $\kappa$ such that the full subcategory $\mathcal{C}^\kappa \subset \mathcal{C}$ of $\kappa$-compact objects is closed under subobjects (i.e., the domain of a $(-1)$-truncated morphism into an object of $\mathcal{C}^\kappa$ is itself in $\mathcal{C}^\kappa$). 
	\end{lem}
	
	\begin{proof}
		Since $\mathcal{C}$ is a presentable \infcat, it arises as an accessible localization of a presheaf \infcat. Let $L \colon \mathcal{P}(\mathcal{D}) \rightleftarrows \mathcal{C} \colon i$ be the corresponding adjunction for some small \infcat \mcd.
		We may choose a regular cardinal $\kappa$ large enough such that the following three conditions are simultaneously satisfied:
		\begin{enumerate}
			\item $\kappa > \omega$, i.e., $\kappa$ is an uncountable regular cardinal.
			\item The inclusion functor $i \colon \mathcal{C} \hookrightarrow \mathcal{P}(\mathcal{D})$ is $\kappa$-accessible and preserves $\kappa$-compact objects\footnote{so that $\mathcal{C}^\kappa= i^{-1}(\mathcal{P}(\mathcal{D})^\kappa)$.}.
			\item $\mathcal{D}$ is $\kappa$-small. (Specifically, the set of equivalence classes of objects in $\mathcal{D}$ is $\kappa$-small, and for any pair of objects $c, d \in \mathcal{D}$, the mapping anima $\operatorname{Map}_{\mathcal{D}}(c, d)$ is $\kappa$-small.)
		\end{enumerate}
		Such a $\kappa$ exists. Specifically, regarding condition (2), see \cite[Remark 5.4.2.13]{htt}.
		We claim that, under conditions (1) and (3), an object $F \in \mathcal{P}(\mathcal{D})$ is $\kappa$-compact if and only if for all $d \in \mathcal{D}$, the anima $F(d)$ is $\kappa$-small.
		
		To see necessity, note that the evaluation functor $\operatorname{Ev}_d \colon \mathcal{P}(\mathcal{D}) \to \mathcal{S}$ admits a right adjoint $R_d \colon \mathcal{S} \to \mathcal{P}(\mathcal{D})$ given by $$R_d(X)(c) \simeq \operatorname{Map}_{\mathcal{S}}(\operatorname{Map}_{\mathcal{D}}(c, d), X).$$ Since $\mathcal{D}$ is $\kappa$-small, $\operatorname{Map}_{\mathcal{D}}(c, d)$ is a $\kappa$-compact anima. Consequently, the functor corepresenting this anima preserves $\kappa$-filtered colimits, meaning $R_d$ preserves $\kappa$-filtered colimits. Therefore, its left adjoint $\operatorname{Ev}_d$ preserves $\kappa$-compact objects.
		
		For sufficiency, suppose $F(d)$ is $\kappa$-small for all $d \in \mathcal{D}$. We have an equivalence $F \simeq \operatorname{colim}_{\mathcal{D}_{/F}} y(d)$. Since $\mathcal{D}$ is $\kappa$-small, the slice \infcat $\mathcal{D}_{/F}$ is $\kappa$-small. Therefore, $F$ is exhibited as a $\kappa$-small colimit of completely compact objects, implying $F$ itself is $\kappa$-compact.
		
		Now, let $X \in \mathcal{C}$ be a $\kappa$-compact object, and let $Y \hookrightarrow X$ be a monomorphism in $\mathcal{C}$. By condition (2), $i$ preserves and reflects $\kappa$-compacts. 
		Since $i$ is a right adjoint, it preserves limits and hence monomorphisms. However, by the argument above, we see that $\kappa$-compacts in $\mcp(\mcd)$ are closed under monomorphisms, which completes the proof.
	\end{proof}
	We next apply this observation to the additive situation. 
	
	\begin{lem}\label{addsubcatkcomp}
\begin{enumerate}
	\item Let $\kappa>\omega$ be an (arbitrary) uncountable regular cardinal.  If an idempotent complete small additive \infcat $\mca$ is $\kappa$-compact in \catadidem, then so is any idempotent complete additive full subcategory of  $\mca$.	
	\item  There exists a large enough uncountable regular cardinal $\lambda>\omega$ such that if a \dualaddinfcat $\mcc$ is $\lambda$-compact in \praddbl, then so is any closed subcategory of it.	
\end{enumerate}	
	\end{lem}
	\begin{proof}
	(1) This follows from the same argument as \cite[Proposition 3.4 \& Corollary 3.7]{ramzi2025every}.

	(2) This follows from \cref{kcompmono}, because a closed subcategory forms a monomorphism in \praddbl.
	\end{proof}
	The preceding compactness results imply that short exact sequences can be
	approximated by short exact sequences of compact objects. 
	\begin{prop}\label{filshortext}
		\begin{enumerate}
			\item Let $\kappa>\omega$ be an (arbitrary) uncountable regular cardinal.  	Any short exact sequence in $\catadidem$ can be written as a $\kappa$-filtered colimit of short exact sequences of $\kappa$-compact objects in $\catadidem$.
			\item  There exists a large enough uncountable regular cardinal $\kappa>\omega$ such that any short exact sequence in $\praddbl$ can be written as a $\kappa$-filtered colimit of short exact sequences of $\kappa$-compact objects in $\praddbl$.
		\end{enumerate}	
		
	\end{prop}
	\begin{proof}
	Suppose given a short exact sequence $\mca\to\mcb\to\mcc$ in $\prad^\star$, where $\star \in \{ \op{dbl}, \op{at} \}$. Since $\prad^\star$ is $\kappa$-presentable, we can write $\mcb\simeq\colim \mcb_\alpha$ as a $\kappa$-filtered colimit of $\kappa$-compact objects in $\prad^\star$. We define $\mca_\alpha=\mca\times_{\mcb}\mcb_\alpha$ and $\mcc_\alpha=\mcb_\alpha/\mca_\alpha$. By \cref{dualker1}, we see that $\mca_\alpha\to\mcb_\alpha$ is a closed subcategory (when $\star=\op{at}$ it is trivial). Since $\prad^\star$ is $\kappa$-presentable, $\mca\to\mcb\to\mcc$ is equivalent to the $\kappa$-filtered colimit $$\colim_\alpha(\mca_\alpha\to\mcb_\alpha\to\mcc_\alpha)$$ of short exact sequences in $\prad^\star$. It suffices to show that each $\mca_\alpha$ is $\kappa$-compact (then so is $\mcc_\alpha$), but that follows from \cref{addsubcatkcomp}. 
	\end{proof}
	We can now construct the universal (stable) finitary localizing invariant on \dualaddinfcats.
	
\begin{thm}\label{defmotpst}
	There exists a  universal finitary localizing invariant to a presentable stable \infcat $$\mcu^{\op{cont}}:\praddbl\to \motpst$$ in the following sense: for any (potentially large) cocomplete stable \infcat \mce, the following functor is an equivalence
	$$\funct^L(\motpst,\mce)\xrightarrow{\sim}\Loc_\omega(\praddbl,\mce).$$
\end{thm}
\begin{proof}Let $\kappa$ be a large enough regular cardinal as in \cref{filshortext}(2).
	We define $$\motpst:= \funct^{\op{cofil},\times}(\praddblop,\opsp)[<\Sigma^n\cofib(h_\mca\to h_\mcb)\to \Sigma^nh_{\mcb/\mca}|\,\mca\hookrightarrow\mcb\in (\praddbl)^{\kappa}, n\in\mathbb{Z}>^{-1} ].$$  It is not hard to show that it satisfies the desired universal property, using the same philosophy as \cite{blumberg2013universal}. 	 
\end{proof}
\begin{de}
	We will refer to \motpst as the stable \infcat of \textbf{prestable motives}.
\end{de}

\begin{rem}
	As in the case of stable motives, it does not matter whether one starts with \praddbl or \pradat. More precisely, there is a natural equivalence $\motpst\simeq \motpst^\prime$, where
	$$
	\motpst^\prime=\mcp_{\Sigma}(\catadidemomega;\opsp)\big[<\Sigma^n\cofib(h_\mca\to h_\mcb)\to \Sigma^nh_{\mcb/\mca}\mid\,\mca\hookrightarrow\mcb\in \catadidemomegaone,\, n\in\mathbb{Z}>^{-1} \big]
	$$
	is associated to the universal finitary localizing invariant $\mcu:\catadidem\to \motpst^\prime$. This follows from the equivalence
	$$
	\Loc_\omega(\praddbl,\mce)\xrightarrow{\sim}\Loc_\omega(\pradat,\mce)
	$$
	of \cref{smalllargelocinv}.
	\end{rem}

The theorem above constructs a stable universal category.  One may also ask
whether there is an intrinsically prestable version of the same construction.
\begin{quest}
	Is the (unstable) \infcat of  prestable motives $\funct^\times(\catadidemop,\spgeq)/\sim $ constructed by the BGT construction  prestable or not?
\end{quest}

Before proving the corepresentability statement for $K$-theory, we
first isolate its splitting analogue.  This intermediate category is obtained
by imposing only the relations coming from right split exact sequences, and it
recovers connective algebraic $K$-theory by the usual group-completion
mechanism.
\begin{de}
	We define $$\motsplcn:= \mcp_{\Sigma}(\catadidemomega;\spgeq)[< h_{S_2\mcc} \to h_{\mcc\times\mcc}\,|\,\mcc\in \catadidemomega>^{-1} ]\simeq \mcp_{\Sigma}(\catadidemomega[S_{\op{spl}}^{-1}];\spgeq)$$ to be the (prestable) \infcat of splitting prestable motives with the universal finitary group-like splitting invariant $\mcu_{\op{spl}}:\catadidem\to\motsplcn$, where the second equivalence follows from the fact that $$\catadidemomega\to\catadidemomega[S_{\op{spl}}^{-1}]$$ is a semi-additive localization. Therefore \motsplcn  is a presentable prestable \infcat. We denote $\motspl:=\opsp(\motsplcn)$.
\end{de}

\begin{prop}\label{splmapping}
	Let $$\phi: \funct^{\times}_\omega(\catadidem,\mcs)\rightleftarrows \op{Split}_\omega(\catadidem,\spgeq) : \Omega^\infty$$ be the adjunction such that $\phi$ is the left adjoint. Then $\mcu_{\op{spl}}(\spgeqcproj)\in \motsplcn$ is compact projective and the natural transformation from the core functor\footnote{We denote by $\spgeqcproj$ the full subcategory of \spgeq spanned by compact projective objects.} 
	$$(-)^\simeq\simeq \mapp_{\catadidem}(\spgeqcproj,-)\to \Omega^\infty \unmap_{\motsplcn}(\mcu_{\op{spl}}(\spgeqcproj),\mcu_{\op{spl}}(-))$$  exhibits $$\phi((-)^\simeq)\simeq \unmap_{\motsplcn}(\mcu_{\op{spl}}(\spgeqcproj),\mcu_{\op{spl}}(-)).$$
	In other words, $\unmap_{\motsplcn}(\mcu_{\op{spl}}(\spgeqcproj),\mcu_{\op{spl}}(-))$ is the universal finitary group-like localizing invariant on \catadidem under $(-)^\simeq$. Furthermore, we have a natural equivalence $$\unmap_{\motsplcn}(\mcu_{\op{spl}}(\spgeqcproj),\mcu_{\op{spl}}(-))\xrightarrow{\sim} \kcn(\op{Stab}(-)).$$
\end{prop}
\begin{proof}
	Consider the following diagram.
	$$
	\begin{tikzcd}
		{\funct^L(\mcp_{\Sigma}(\catadidemomega),\cmon)} \arrow[d, "\sim"] \arrow[r, shift left=2] & {\funct^L(\mcp_{\Sigma}(\catadidemomega),\spgeq)} \arrow[l, hook', shift left] \arrow[d, "\sim"] \arrow[r, shift left] & {\funct^L(\motsplcn,\spgeq)} \arrow[l, hook', shift left=2] \arrow[d, "\sim"] \\
		{\funct_\omega^\times(\catadidem,\cmon)} \arrow[r, shift left=2]                                & {\funct_\omega^\times(\catadidem,\spgeq)} \arrow[l, hook', shift left] \arrow[r, shift left]                                & {\Split_\omega(\catadidem,\spgeq)} \arrow[l, hook', shift left=2]               
	\end{tikzcd}$$
	Let $h_1\in \mcp_{\Sigma}(\catadidemomega)$ be the  functor represented by $\spgeqcproj$. Then $\underline{h}^{\op{B}^\infty(h_1)}$ lies in $$\funct^L(\mcp_{\Sigma}(\catadidemomega;\spgeq),\spgeq),$$ because we have the following diagram
	$$\begin{tikzcd}
		\mcp_{\Sigma}(\catadidemomega;\spgeq) \arrow[d, "\Omega^\infty"] \arrow[r, "\underline{h}^{\op{B}^\infty(h_1)}"] & \spgeq \arrow[d, "\Omega^\infty"] \\
		\mcp_{\Sigma}(\catadidemomega) \arrow[r, "h^{h_1}"]                                            & \mcs            
	\end{tikzcd}
	$$
	where the bottom, left and right functors preserve sifted colimits. Since the connective $K$-theory is splitting, we obtain a factorization:
	$$\begin{tikzcd}
		& \motsplcn \arrow[rd, "\underline{h}^{\mcu_{\op{spl}}(\spgeqcproj)}", dashed] &        \\
		\mcp_{\Sigma}(\catadidemomega;\spgeq) \arrow[ru] \arrow[rr, "\underline{h}^{\op{B}^\infty (h_1)}"] &                                           & \spgeq
	\end{tikzcd}$$
	Finally by the Yoneda lemma, we have $$\phi((-)^\simeq)\simeq \underline{h}^{\mcu_{\op{spl}}(\spgeqcproj)}_{\motsplcn}(\mcu_{\op{spl}}(-))=\unmap_{\motsplcn}(\mcu_{\op{spl}}(\spgeqcproj),\mcu_{\op{spl}}(-)).$$
\end{proof}

\begin{rem}
	\cref{splmapping} also works for the universal $\kappa$-finitary splitting invariant for any regular cardinal $\kappa$
	$$
	\catadidem \to \mot^{\mathrm{cn}}_{\kappa\text{-}\mathrm{spl}},
	$$
	where
	$$
	\mot^{\mathrm{cn}}_{\kappa\text{-}\mathrm{spl}}
	:= \mcp_{\Sigma}(\catadidemkappa;\spgeq)
	\bigl[
	< h_{S_2\mcc} \to h_{\mcc\times\mcc}
	\,\vert\, \mcc\in \catadidemkappa >
	^{-1}
	\bigr]
	\simeq
	\mcp_{\Sigma}(\catadidemkappa[S_{\kappa\text{-}\mathrm{spl}}^{-1}];\spgeq).
	$$
	
\end{rem}
To pass from connective $K$-theory to nonconnective $K$-theory, we use a
delooping procedure based on a Calkin-type construction.  The following
abstract lemma packages the formal properties of such a construction that are
needed for the comparison.
\begin{lem}\label{big}
	Let $\op{big}(-):\catadidem\to \catadidem$ be an $(\infty,2)$-functor preserving finite products and right splitting sequences.  Let $\op{Id}_{\catadidem}\hookrightarrow\op{big}(-)$ be a fully faithful $(\infty,2)$-transformation.  Then $\calk(-)$ preserves finite products and right splitting sequences too, where $\calk(\mca):= \op{big}(\mca)/\mca$.
	Suppose further that each $\op{big}(\mca)$ admits countable coproducts.
	Then we have a natural equivalence of spectra $$\colimit_n\, \Omega^n\kcn(\stab(\calk^n(-))) \xrightarrow{\sim}K(\stab(-)).$$	
\end{lem}
\begin{proof}
	Let $\mca\xrightarrow{F}\mcb\xrightarrow{G}\mcc$ be a right splitting sequence in \catadidem. By assumption, we can deduce that $\calk(-)$ is an $(\infty,2)$-functor, therefore $\calk(F)$ is a Bousfield colocalization. Since the sequence $$\calk(\mca\xrightarrow{F}\mcb\xrightarrow{G}\mcc)$$ is the cofiber of cofiber sequences in \catadidem, it is also a cofiber sequence. Consequently, $\calk(\mca\xrightarrow{F}\mcb\xrightarrow{G}\mcc)$ is a right splitting sequence.
	
	Now suppose further that each $\op{big}(\mca)$ admits countable coproducts. Then by \cref{swindle}, we have $\kcn(\stab(\op{big}(-)))=K(\stab(\op{big}(-)))=0$. Therefore, given $\mca\in\catadidem$, the cofiber sequence 
	$$\mca\to\op{big}(\mca)\to \calk(\mca)$$
	induces a natural transformation $\kcn(\stab(\mca))\to \Omega\kcn(\stab(\calk(\mca)))$. Now consider the following commutative diagram.
	$$\begin{tikzcd}
		\kcn(\stab(\mca)) \arrow[r] \arrow[d] & \Omega\kcn(\stab(\calk(\mca))) \arrow[d] \\
		K(\stab(\mca)) \arrow[r, "\sim"]      & \Omega K(\stab(\calk(\mca)))            
	\end{tikzcd}$$
	For any $i\in\mathbb{Z}$, when $N$ is large enough we have $$\begin{aligned}
		\pi_i \colimit_n\, \Omega^n \kcn(\stab(\calk^n(\mca)))
		&\simeq \pi_i \Omega^N \kcn(\stab(\calk^N(\mca))) \\
		&= \pi_{i+N}\kcn(\stab(\calk^N(\mca)))
		\simeq \pi_{i+N}K(\stab(\calk^N(\mca))).
	\end{aligned}$$
\end{proof}
We now introduce the particular enlargement which will play the role of
$\op{big}(-)$ for additive categories:
\begin{de}\label{bigpsigma}
	Let $\mca\in\catadidem$. We denote by  $\mcp_{\Sigma}^{\omega_1}(\mca)$ the cocompletion of \mca adding all countable coproducts and idempotent colimits without changing finite coproducts. It can be identified with the smallest full subcategory of $\mcp_{\Sigma}(\mca)$ that contains \mca and is closed under countable colimits and retracts. 
\end{de}

\begin{rem}
	It induces the following adjunction 
	$$\mcp_{\Sigma}^{\omega_1}: \catadidem \rightleftarrows\Cat^{\omega_1\text{-}\sqcup}_{\op{ad}},$$ where $\Cat^{\omega_1\text{-}\sqcup}_{\op{ad}}$ denotes the \infcat of small idempotent-complete additive \infcats that admit countable coproducts. 
	
	The forgetful functor $\Cat^{\omega_1\text{-}\sqcup}_{\op{ad}}\to \catadidem$ preserves $\omega_1$-filtered colimits and so does the endofunctor $\mcp_{\Sigma}^{\omega_1}:\catadidem\to\catadidem$.
\end{rem}

\begin{prop}
	The endofunctor $\mcp_{\Sigma}^{\omega_1}:\catadidem\to\catadidem$ with the unit of adjunction $\op{Id}_{\catadidem}\hookrightarrow \mcp_{\Sigma}^{\omega_1}$ satisfies the conditions in \cref{big}. 
\end{prop}
\begin{proof}
	Since both categories in the adjunction $\mcp_{\Sigma}^{\omega_1}: \catadidem \rightleftarrows\Cat^{\omega_1\text{-}\sqcup}_{\op{ad}}$ are semi-additive, the endofunctor $\mcp_{\Sigma}^{\omega_1}:\catadidem\to\catadidem$ preserves finite products.
	Let $\mca\xrightarrow{F}\mcb\xrightarrow{G}\mcc$ be a right splitting sequence in \catadidem. Then $\mcp_{\Sigma}(\mca\xrightarrow{F}\mcb\xrightarrow{G}\mcc)$ is a right splitting sequence in \praddbl. Therefore we have a recollement $$ii^R(x)\to x\to p^Rp(x)$$ for any $x\in\mcp_{\Sigma}(\mcb)$, where $i:=\mcp_{\Sigma}(F)$ and $p:=\mcp_{\Sigma}(G)$. It suffices to show that $$\mcp_{\Sigma}^{\omega_1}(\mcc)\simeq \op{quot}(\mcp_{\Sigma}^{\omega_1}(\mca)\to\mcp_{\Sigma}^{\omega_1}(\mcb))$$ in the sense of \cref{addquot}. Now given a morphism $f:x\to y\in \mcp_{\Sigma}^{\omega_1}(\mcb)$ such that $p(x)\to p(y)$ is an equivalence. By the recollement we have the following diagram,
	$$\begin{tikzcd}
		ii^R(x) \arrow[r] \arrow[d] & x \arrow[d, "f"] \arrow[r, "u_x"] & p^Rp(x) \arrow[d, "\sim"] \arrow[l, "s_x", dashed, shift left=2] \\
		ii^R(y) \arrow[r]           & y \arrow[r]                       & p^Rp(y) \arrow[l, "fs_x", dashed, shift left=2]                 
	\end{tikzcd}$$
	where the existence of a section $s_x$ is because $p^Rp(x)$ is projective in $\mcp_{\Sigma}(\mcb)$ by the splitting sequence. Consequently,  $x\to y$ can be identified with the sum $$ii^R(x)\oplus p^Rp(x)\to ii^R(y)\oplus p^Rp(y),$$ which lies in  the $W$ of \cref{addquot}.
\end{proof}
The cofiber sequence defining the Calkin construction can now be transported
to the level of splitting motives:
\begin{cor}\label{diagram}
	The cofiber sequence $$\op{Id}_{\catadidem}\hookrightarrow \mcp_{\Sigma}^{\omega_1}\to \calk$$ in $\funct(\catadidem,\catadidem)$ induces the following diagram in $\funct^L(\motomegaonespl,\motomegaonespl)$.
	$$\begin{tikzcd}
		\op{Id} \arrow[d] \arrow[r] & \op{big}_l\simeq0 \arrow[d] \\
		0 \arrow[r]                 & \calk_l                    
	\end{tikzcd}$$
\end{cor}
\begin{rem}
	Beware that the above square is not necessarily a cofiber sequence in $$\funct^L(\motomegaonespl,\motomegaonespl),$$ because for a given $\mca \in \catadidem$,  
	$$\mca \hookrightarrow \mcp_{\Sigma}^{\omega_1}(\mca)\to \calk(\mca)$$ is neither left nor right splitting in \catadidem. However, applying $L$ to the above square we get a cofiber sequence in $\funct^L(\motomegaonespl,\motpst)$.
\end{rem}
We now combine the splitting corepresentability result with the Calkin
delooping construction to give the desired
corepresentability theorem in \motpst.
\begin{thm}\label{main4}
	Let $$\varphi: \funct^{\times}_\omega(\catadidem,\mcs)\rightleftarrows \Loc_\omega(\catadidem,\opsp) : \Omega^\infty$$ be the adjunction such that $\varphi$ is the left adjoint to $\Omega^\infty$. Then $\mcu(\spgeqcproj)\in \motpst$ is compact and the natural transformation from the core functor 
	$$(-)^\simeq\simeq \mapp_{\catadidem}(\spgeqcproj,-)\to \Omega^\infty \unmap_{\motpst}(\mcu(\spgeqcproj),\mcu(-))$$  exhibits $$\varphi((-)^\simeq)\simeq \unmap_{\motpst}(\mcu(\spgeqcproj),\mcu(-)).$$
	In other words, $\unmap_{\motpst}(\mcu(\spgeqcproj),\mcu(-))$ is the universal finitary stable localizing invariant under $(-)^\simeq$ (on \catadidem). Furthermore, we have a natural equivalence $$\unmap_{\motpst}(\mcu(\spgeqcproj),\mcu(-))\xrightarrow{\sim} K(\op{Stab}(-)).$$
\end{thm}
\begin{proof}
	Consider the following diagram.
	$$\begin{tikzcd}
		{\funct^L(\motomegaonespl,\opsp)} \arrow[d, "\sim"] \arrow[r, "L_!",shift left] & {\funct^L(\motpst,\opsp)} \arrow[l, "L^*",hook', shift left=2] \arrow[d, "\sim"] \\
		{\Split_{\omega_1}(\catadidem,\opsp)} \arrow[r, shift left=2]                    & {\Loc_\omega(\catadidem,\opsp)} \arrow[l, hook', shift left]            
	\end{tikzcd}$$
	Since $\mcu_{\op{spl}}^{\omega_1}(\spgeqcproj)\in \motomegaonesplcn$ is compact projective, we have $$\unmap_{\motomegaonespl}(\mcu_{\op{spl}}^{\omega_1}(\spgeqcproj),\mcu_{\op{spl}}^{\omega_1}(-))\simeq\unmap_{\motomegaonesplcn}(\mcu_{\op{spl}}^{\omega_1}(\spgeqcproj),\mcu_{\op{spl}}^{\omega_1}(-))\xrightarrow{\sim} \kcn(\op{Stab}(-)).$$
	Since $L:\motomegaonesplcn\to \motpst$ is a localization by construction, the representable functor $$h:=\unmap_{\motpst}(\mcu(\spgeqcproj),-):\motpst\to \opsp$$ is obtained by the left Kan extension  of the following representable functor along $L$ $$\motomegaonesplcn\xrightarrow{\kcn^{\op{spl}}:=\unmap_{\motomegaonesplcn}(\mcu_{\op{spl}}^{\omega_1}(\spgeqcproj),-)}\opsp.$$
	Let $$K^{\op{pst}}:\motpst\to\opsp$$ denote the colimit-preserving functor induced by the finitary localizing invariant $K(\stab(-)):\catadidem\to \opsp$, i.e., we have $K^{\op{pst}}(\mcu(-))\simeq K(\stab(-))$.
	By the Yoneda lemma, there exists a unique transformation $\theta:h\to K^{\op{pst}}$ making the following diagram  commute.
	$$\begin{tikzcd}
		\kcn^{\op{spl}} \arrow[d] \arrow[rd] &                  \\
		L^*h \arrow[r, "L^*\theta"]          & L^* K^{\op{pst}}
	\end{tikzcd}$$
	Since 
	the functor $h\circ\mcu:\catadidem\to\opsp$ is a 
	localizing invariant, we have the following factorization.
	$$\begin{tikzcd}
		& \kcn^{\op{spl}} \arrow[d] \arrow[rd] \arrow[ld]                            &                  \\
		{L^*K^{\op{pst}}\simeq\colim_n \Omega^n \kcn^{\op{spl}}(\calk_l^n(-))} \arrow[r,"\gamma"', dashed] & {L^*h\simeq \colim_n \Omega^n h(L(\calk_l^n(-)))} \arrow[r, "L^*\theta"'] & L^* K^{\op{pst}}
	\end{tikzcd}$$
	By the universal property, the composition $\gamma \circ L^*\theta: L^*h\longrightarrow L^*h$ is equivalent to the identity. Therefore $L^*h$ is a retract of $L^*K^{\op{pst}}$ and hence preserves colimits. It implies that $h$ preserves colimits too, thus $L_!\kcn^{\op{spl}} \simeq h$ and $\mcu(\spgeqcproj)$ is compact in \motpst. 
	
	By the retract above, for any $\mca\in\catadidem$ we have that $\pi_n h(\mcu(\mca))\xrightarrow{\pi_n \theta} \pi_n K^{\op{pst}}(\mcu(\mca))$ is an isomorphism when $n\geq 0$. By the following diagram of spectra,
	$$
	\begin{tikzcd}
		h(\mcu(\mca)) \arrow[r, "\sim"] \arrow[d, "\theta"] & \colim_n \Omega^n h(\mcu(\calk^n(\mca))) \arrow[d, "\theta"] \\
		K^{\op{pst}}(\mcu(\mca)) \arrow[r, "\sim"]           & \colim_n \Omega^n K^{\op{pst}}(\mcu(\calk^n(\mca)))          
	\end{tikzcd}
	$$
	we conclude that $\pi_n h(\mcu(\mca))\xrightarrow{\pi_n \theta} \pi_n K^{\op{pst}}(\mcu(\mca))$ is an isomorphism for any $n\in\mathbb{Z}$. Since both $h$ and $K^{\op{pst}}$ preserve colimits and $\theta$ becomes an equivalence when restricted to \catadidem, the transformation $\theta$ is an equivalence as desired. 
\end{proof}

\subsection{Additive Calkin construction}
The preceding argument used the Calkin construction as a formal delooping
device.  We now  prove that it is compatible with short exact sequences.  This provides a
categorical realization of the suspension operation on prestable motives.
 The main result is stated as follows:
\begin{thm}\label{addcalk}
	Let
	$	\mathcal{A} \longrightarrow \mathcal{B} \longrightarrow \mathcal{C}$	be a short exact sequence in \catadidem. Then
	\[
	\mathcal{A}^{\mathrm{big}} \longrightarrow \mathcal{B}^{\mathrm{big}} \longrightarrow \mathcal{C}^{\mathrm{big}}
	\]
	is a short exact sequence in \catadidem. Here we define $(-)^{\mathrm{big}}=\mcp_{\Sigma}^{\omega_1}(-)$, which appeared in \cref{bigpsigma}.
	Moreover, $$\calk(\mca)\to\calk(\mcb)\to\calk(\mcc)$$ is short exact in \catadidem too, 
	where we define $\calk(-)=(-)^{\op{big}}/(-)$.
\end{thm}
Before the proof, we recall the notion of a weight structure.
\begin{de}\label{wstructure}
	A \textbf{weight structure} on a stable \infcat $\mathcal{C}$ is the data of two idempotent closed full subcategories $\left(\mathcal{C}_{w \geq 0}, \mathcal{C}_{w \leq 0}\right)$ of $w$-connective and, respectively, $w$-coconnective objects such that:
	\begin{enumerate}
	\item $\mathcal{C}_{w \geq 0}$ is closed under suspension, and $\mathcal{C}_{w \leq 0}$ is closed under desuspensions.
	\item  The mapping spectrum $\unmap(x, y)$ is connective if $x \in \mathcal{C}_{w \leq 0}, y \in \mathcal{C}_{w \geq 0}$.
	\item For any object $x \in \mathcal{C}$ there exists a cofiber sequence
	$$
	x_{w\leq 0} \rightarrow x \rightarrow x_{w\geq 1},
	$$
	where $x_{w\leq 0} \in \mathcal{C}_{w \leq 0}$ and $\Sigma^{-1} x_{w\geq 1} \in \mathcal{C}_{w \geq 0}$. We call these weight truncations of $x$.
	We warn the reader that despite the notation, the weight truncations of $x$ are not uniquely determined by $x$.
	\end{enumerate}
	We call $\mcc^{w\text{-}\heartsuit}=\mathcal{C}_{w \geq 0}\bigcap \mathcal{C}_{w \leq 0}$ the weight heart or w-heart of \mcc.
\end{de}
We will use the following existence theorem for compactly generated weight
structures.  
\begin{thm}[{\cite[Theorem 1.12]{levy2025c}}]\label{cgwstructure}
 Let $\mathcal{C}\in \prl_{\op{st},\omega}$ be a compactly generated stable presentable \infcat. Then for any set of compact objects $S \subset \mathcal{C}^\omega$, there is a unique weight structure on $\mathcal{C}$ such that
$$
\mathcal{C}_{w \geq 0}=\{y \in \mathcal{C}\mid \unmap(x, y) \text{ is connective for all } x \in S\}.
$$

\end{thm}

\begin{proof}
	We define 
	$$
	\mathcal{C}_{w \leq 0}=\{x \in \mathcal{C}\mid \unmap(x, y) \text { is connective for all } y \in \mathcal{C}_{w \geq 0}\} .
	$$
	It is easy to verify conditions (1), and (2) of a weight structure, so it remains to check condition (3). Given $x_0\in\mcc$, we can take $y_0 \rightarrow x_0$ to be the direct sum of all homotopy classes of maps from all the negative suspensions of objects $\bigcup_{n\geq0} \Sigma^{-n}S$ to $x_0$, and let $x_1=\cofib(y_0\to x_0).$
	We can inductively define $y_i \rightarrow x_i$ this way and let $$x_{i+1}=\cofib(y_i\to x_i).$$ Compactness of $S$ then shows that $\operatorname{colim}_i x_i$ is in $\mathcal{C}_{w \geq 1}$. On the other hand, the fiber of $x_0 \rightarrow  \operatorname{colim}_i x_i$ is a filtered colimit of extensions of $y_i$, so is in $\mathcal{C}_{w\leq 0}$.
\end{proof}
For later use, we record an explicit description of the coconnective part of
the compactly generated weight structure.
\begin{rem}\label{wleqgen}
	If $x_0\in \mcc_{w\leq0}$, then by construction in \cref{cgwstructure}, $x_0$ is a retract of $\fib(x_0 \rightarrow  \operatorname{colim}_i x_i)$. Therefore, $\mcc_{w\leq0}$ can be identified with the smallest full subcategory $\mcc'$ of $\mcc$ satisfying the following:
	\enu{
		\item $\mcc'$ contains $S$.
	\item The inclusion $\mcc'\hookrightarrow\mcc$ is closed under small coproducts, finite limits, retracts and extensions.
	\item For any sequence $z_\bullet:\mathbb{N}\to \mcc'$ such that each $\cofib(z_i\to z_{i+1})$ lies in $\mcc'$, the colimit $\colim_i z_i$ lies in $\mcc'$ too.
	}
\end{rem}
We first identify the stable hull of the additive enlargement with the
weight-bounded part of the corresponding \(\omega_1\)-compact Ind-category.
\begin{lem}\label{calklem1}
 Let $\mca\in\catadidem$. The \infcat \(\ind(\stab(\mathcal{A}))^{\omega_1}\) inherits a weight structure from $\ind(\stab(\mathcal{A}))$. Furthermore, the natural functor
	\[
	\stab(\mathcal{A}^{\mathrm{big}})\longrightarrow \ind(\stab(\mathcal{A}))^{\omega_1}
	\]
	is fully faithful and identifies the source with the weight-bounded objects.
\end{lem}  

\begin{proof}
	Applying \cref{cgwstructure} to $S=\mca$, we obtain a weight structure on $\ind(\stab(\mathcal{A}))$, whose connective part identifies with the inclusion $\mcp_\Sigma(\mca)\hookrightarrow \ind(\stab(\mathcal{A}))$, and whose w-heart identifies with projective objects in $\mcp_\Sigma(\mca)$, namely $$\ind(\stab(\mathcal{A}))^{w\text{-}\heartsuit}\simeq\mcp_\Sigma(\mca)^{\op{proj}}=\mcp_{\Sigma}^{\op{small}\text{-}\sqcup}(\mca).$$ We wish to show that this weight structure restricts to  $\ind(\stab(\mathcal{A}))^{\omega_1}$. After that,  its w-heart can be identified with $\mca^{\op{big}}$, and the weight-bounded objects of $\ind(\stab(\mathcal{A}))^{\omega_1}$ coincide with  \(\stab(\mathcal{A}^{\mathrm{big}})\) by \cite[Corollary 3.4]{Sosnilo_2019}, as desired.
	
Let $\mct=\ind(\stab(\mca))$. First, consider the case where $x \in \mct^{\omega_1}$ is a countable direct sum of compact objects, i.e., $x \simeq \bigoplus_{i=0}^\infty x_i$ with $x_i \in \stab(\mca)$. By construction, the weight structure restricts to $\stab(\mca)$, so each $x_i$ admits a weight decomposition:
\[
W_i \longrightarrow x_i \longrightarrow Z_i
\]
where $W_i \in \stab(\mca)_{w \leq 0}$ and $Z_i \in \stab(\mca)_{w \geq 1}$. Taking the countable direct sum of these exact triangles yields:
\[
\bigoplus_{i=0}^\infty W_i \longrightarrow \bigoplus_{i=0}^\infty x_i \longrightarrow \bigoplus_{i=0}^\infty Z_i.
\]
Since both $\mathcal{T}_{w \leq 0}$  and $\mathcal{T}_{w \geq 1}$ are closed under arbitrary coproducts (as the weight structure is compactly generated), we deduce that $\bigoplus W_i \in \mathcal{T}_{w \leq 0}$ and $\bigoplus Z_i \in \mathcal{T}_{w \geq 1}$. Furthermore, as countable coproducts of compact objects, these truncations remain in $\mathcal{T}^{\omega_1}$. Thus, any countable direct sum of compact objects admits a valid weight decomposition in $\mathcal{T}^{\omega_1}$.

Now, let $x$ be a general object in $\mathcal{T}^{\omega_1}$. Since $x$ is $\omega_1$-compact, it can be represented as a sequential colimit of compact objects: $x \simeq \colim_{i \in \mathbb{N}} x_i$, where $x_i \in \stab(\mca)$.
In a stable \infcat, this sequential colimit is canonically equivalent to the cofiber of the shift map in the Milnor sequence (mapping telescope). This yields a standard exact triangle:
\[
\bigoplus_{i=0}^\infty x_i \longrightarrow x \longrightarrow \Sigma\left(\bigoplus_{i=0}^\infty x_i\right).
\]
Let $A = \bigoplus_{i=0}^\infty x_i$ and $B = \left(\bigoplus_{i=0}^\infty x_i\right)[1]$. Both $A$ and $B$ are countable direct sums of compact objects, so by the previous discussion, they admit weight decompositions in $\mathcal{T}^{\omega_1}$:
\[
W_A \longrightarrow A \longrightarrow Z_A,\qquad
W_B \longrightarrow B \longrightarrow Z_B.
\]
 To construct a weight decomposition for $x$, we must check the vanishing of the obstruction to completing the $3 \times 3$ diagram, which lies in the mapping anima $\Hom_{\mathcal{T}}(\Sigma^{-1}W_B, Z_A)$.
Observe that $W_B \in \mathcal{T}_{w \leq 0}$, so its desuspension satisfies $W_B[-1] \in \mathcal{T}_{w \leq -1} \subset \mathcal{T}_{w \leq 0}$. 
However, by the orthogonality axiom of weight structures, the obstruction vanishes:
\[
\pi_0\unmap_{\mathcal{T}}(W_B[-1], Z_A) = 0.
\]
Therefore we obtain the following diagram:
$$\begin{tikzcd}
	\Sigma^{-1}W_B \arrow[d] \arrow[r, dashed] & W_A \arrow[d] \arrow[r, dashed] & W_x \arrow[d] \\
	\Sigma^{-1}B \arrow[r] \arrow[d]           & A \arrow[d] \arrow[r]           & x \arrow[d]   \\
	\Sigma^{-1}Z_B \arrow[r, dashed]           & Z_A \arrow[r, dashed]           & Z_x          
\end{tikzcd}$$
where $W_x=\cofib(\Sigma^{-1}W_B\to W_A)$ and $Z_x=\cofib(W_x\to x)$.
Since $\mathcal{T}_{w \leq 0}$ is closed under extensions, and both $W_A, W_B \in \mathcal{T}_{w \leq 0}$, it follows that $W_x \in \mathcal{T}_{w \leq 0}$. Similarly we have $Z_x \in \mathcal{T}_{w \geq 1}$.  Because $A, B$ and their weight truncations lie in $\mathcal{T}^{\omega_1}$, and $\mathcal{T}^{\omega_1}\hookrightarrow\mct$ is closed under extensions, we conclude that $W_x, Z_x \in \mathcal{T}^{\omega_1}$.
This completes the proof that the weight structure on $\ind(\stab(\mca))$ restricts to $\ind(\stab(\mca))^{\omega_1}$.
\end{proof}

The second ingredient concerns quotients.  In order to prove exactness for the
Calkin construction, we need to know that passing to a quotient is compatible
with restricting to the bounded part for the relevant weight structures.

\begin{lem}\label{calklem2}
	Let \(\mathcal{D}\) be a weighted stable \infcat and let \(\mathcal{K}\subset \mathcal{D}\) be a weighted full subcategory (i.e. the weight structure restricts to \mck). Let \(\mathcal{D}^b\) and \(\mathcal{K}^b\) denote the subcategories of weight-bounded objects. Then the canonical functor (between the Verdier quotients)
	\[
	\mathcal{D}^b/\mathcal{K}^b \longrightarrow \mathcal{D}/\mathcal{K}
	\]
	is fully faithful.
\end{lem}  

\begin{proof}
	We exploit the weight-truncation. Fix \(x\in\mathcal{D}^b\). 
	Let \(\mathcal{K}^+\) denote the full subcategory of objects of \(\mathcal{K}\) that are bounded above (in weight). For any map \(u\!:x\to k\) with \(k\in\mathcal{K}\) the composite \(x\to k\) factors through a weight truncation \(k_{w_{\leq N}}\to k\) for some large enough \(N\). Thus the inclusion  $\mck^{+}_{x/}\subset \mck_{x/}$ is final because both of them are cofiltered \infcats. Passing to colimits which compute mapping anima in the localization shows that
	$$
\colimit_{k \in (\mathcal{K}^{+}_{x/})^{\opp}} \text{Map}_{\mathcal{D}}(\text{fib}(x \to k), y)\xrightarrow{\sim} \colimit_{k \in (\mathcal{K}_{x/})^{\opp}} \text{Map}_{\mathcal{D}}(\text{fib}(x \to k), y) $$ and hence
$$\mapp_{\mathcal{D}/\mathcal{K}^{+}}(x,y)\xrightarrow{\sim}  \text{Map}_{\mathcal{D}/\mathcal{K}}(x, y)$$
	for every \(x\in\mathcal{D}^b\) and all \(y\in \mcd\). The dual argument shows that
	$$
	\colimit_{k \in \mathcal{K}^{b}_{/y}} \text{Map}_{\mathcal{D}}(x, \cofib(k\to y))\xrightarrow{\sim}	\colimit_{k \in \mathcal{K}^{+}_{/y}} \text{Map}_{\mathcal{D}}(x, \cofib(k\to y)) $$
	for any $y\in \mcd^b$ and $x\in\mcd$, and hence
	\(\mathcal{D}^b/\mathcal{K}^b\to\mathcal{D}/\mathcal{K}^+\) is fully faithful. Combining these identifications yields the claimed full faithfulness of
	\(\mathcal{D}^b/\mathcal{K}^b\to\mathcal{D}/\mathcal{K}\).
\end{proof}

\begin{proof}[Proof of \cref{addcalk}]
	Full faithfulness of \(\mathcal{A}^{\mathrm{big}}\to \mathcal{B}^{\mathrm{big}}\) is clear, so it suffices to show that the induced functor (from the quotient in \catadidem)
	\[
	\mathcal{B}^{\mathrm{big}}/\mathcal{A}^{\mathrm{big}} \longrightarrow \mathcal{C}^{\mathrm{big}}
	\]
	is an equivalence.
 It suffices to prove that the functor  (from the quotient in \catper)
\[
\stab(\mathcal{B}^{\mathrm{big}})/\stab(\mathcal{A}^{\mathrm{big}})\longrightarrow \stab(\mathcal{C}^{\mathrm{big}})
\]
is an equivalence. The target is generated by the image, so it is enough to check full faithfulness. Consider the commutative diagram
\[
\begin{tikzcd}
	\stab(\mathcal{B}^{\mathrm{big}})/\stab(\mathcal{A}^{\mathrm{big}}) \ar[r] \ar[d] &
	\ind(\stab(\mathcal{B}))^{\omega_1}/\ind(\stab(\mathcal{A}))^{\omega_1} \ar[d] \\
	\stab(\mathcal{C}^{\mathrm{big}}) \ar[r] & \ind(\stab(\mathcal{C}))^{\omega_1},
\end{tikzcd}
\]
where the bottom horizontal functor is fully faithful by  \cref{calklem1}. Note that $$\ind(\stab(\mathcal{A}))^{\omega_1}\hookrightarrow\ind(\stab(\mathcal{B}))^{\omega_1}$$ is a weighted full subcategory by \cref{wleqgen}. Therefore the top horizontal functor is fully faithful by \cref{calklem1} and \cref{calklem2}.  The right vertical functor is fully faithful because \(\ind(-)^{\omega_1}:\catper\to \catper\) preserves short exact sequences (see \cite[Proposition A.26]{ramzi2024dualizable}). Hence the left vertical functor is fully faithful as well, which completes the proof.

For the Calkin operation, consider the following diagram:
$$\begin{tikzcd}
	\mcp_\Sigma(\mca) \arrow[d] \arrow[r, hook]            & \mcp_\Sigma(\mcb) \arrow[d] \arrow[r]            & \mcp_\Sigma(\mcc) \arrow[d]            \\
	\mcp_\Sigma(\mca^{\op{big}}) \arrow[r, hook] \arrow[d] & \mcp_\Sigma(\mcb^{\op{big}}) \arrow[r] \arrow[d] & \mcp_\Sigma(\mcc^{\op{big}}) \arrow[d] \\
	\mcp_\Sigma(\calk(\mca)) \arrow[r, "i"]                & \mcp_\Sigma(\calk(\mcb)) \arrow[r]               & \mcp_\Sigma(\calk(\mcc))              
\end{tikzcd}$$
Since the first and the second rows are short exact in \praddbl, so is the bottom row. Thus it suffices to show $i$ is fully faithful. Since the left above square is vertically right adjointable, so is the left bottom square by a similar argument of \cite[Proposition A.16]{ramzi2024dualizable}. Consequently, $i$ is fully faithful and $$\calk(\mca)\to\calk(\mcb)\to\calk(\mcc)$$ is short exact in \catadidem.
\end{proof}
A consequence of this is the stability of the \infcat of (additive) localizing invariants:
\begin{cor}
 Let $\mathcal{E}$ be an additive \infcat with finite limits and $E:\catadidem\to\mce$ be a localizing invariant. There is a natural equivalence $\Omega E\left(\operatorname{Calk}_\kappa(\mathcal{M})\right) \simeq E(\mathcal{M})$. Moreover, the forgetful functor $$\operatorname{Loc}(\catadidem,\operatorname{Sp}(\mathcal{E})) \rightarrow \operatorname{Loc}(\catadidem,\mathcal{E})$$ is an equivalence.
\end{cor}
\begin{proof}
	Let $\mcm\in\catadidem$ be an arbitrary object. By \cref{addcalk}, there is a natural fiber-cofiber sequence $\mathcal{M} \rightarrow \mcm^{\op{big}} \rightarrow \operatorname{Calk}(\mathcal{M}),$ which induces a natural fiber sequence $$E(\mathcal{M}) \rightarrow 0 \rightarrow E\left(\operatorname{Calk}(\mathcal{M})\right)$$ by \cref{swindle}, i.e. it induces a natural equivalence as claimed.
	
It is clear that $\operatorname{Loc}(\opsp(\mathcal{E})) \simeq \opsp(\operatorname{Loc}(\mathcal{E}))$, compatibly with the forgetful functor to $\operatorname{Loc}(\catadidem,\mathcal{E})$. In particular, it suffices to prove that $\operatorname{Loc}(\catadidem,\mathcal{E})$ is stable.
	Because $\operatorname{Loc}(\catadidem,\mathcal{E})$ is pointed and has finite limits, by 
	\cite[Proposition 1.4.2.11.(3)]{ha}, it suffices to prove that $\Omega$ induces an equivalence on it, but $\Omega$ is given by postcomposition with $\Omega$ and it commutes with precomposition by $\operatorname{Calk}$ (which preserves short exact sequences by \cref{addcalk}), and they are inverses to one another by the natural equivalence $\Omega E\left(\operatorname{Calk}_\kappa(\mathcal{M})\right) \simeq E(\mathcal{M})$.
\end{proof}
We conclude with a question concerning the dualizable additive Calkin construction. It is not clear whether either of the constructions $\mcp_{\Sigma}(\mcc^{\omega_1})/\mcc$ or $\mcp_{\Sigma}(\cflatw)/\mcc$ preserves dualizable additive short exact sequences, and consequently, whether they yield a well-behaved Calkin construction.

\begin{quest}
	Does there exist a natural Calkin construction for \dualaddinfcats?
\end{quest}
\subsection{Canonical $t$-structure on \motpst}
Although \motpst is stable, it is generated by motives of additive categories,
which should be regarded as connective objects.  This suggests a canonical
accessible \(t\)-structure whose connective part is generated under colimits
and extensions by the basic motives \(\mcu(\mca)\).

\begin{de}
	Let $\motpstcn\subset \motpst$ denote the smallest full subcategory closed under small colimits and extensions, containing those motives $\{\mcu(\mca)\mid \mca\in\catadidem\}$ represented by small additive \infcats. 
	By \cite[Lemma 1.4.4.11]{ha}, it defines the connective part of an accessible $t$-structure, since $\motpstcn$ is generated under small colimits and extensions by the set $\left\{\mcu(\mca)\mid \mca\in (\catadidem)^{\omega_1}\right\}$.
\end{de}
We first check that this definition is compatible with the large formulation
used earlier.  Namely, the motive of a dualizable additive \infcat is already
connective, and in fact is represented by a small additive model:
\begin{prop}\label{tmot1}
	For any \dualaddinfcat \mcm, the motive $\mcu^{\op{cont}}(\mcm)$ lies in the connective part \motpstcn. In fact, there always exists a small additive \infcat \mca such that $\mcu^{\op{cont}}(\mcm)\simeq \mcu(\mca)$ in \motpst.
\end{prop}
	\begin{proof}
	Consider the short exact sequence of \dualaddinfcats
	$$\mcm\hookrightarrow \mcc\to \mcd$$
	where $\mcc=\mcp_{\Sigma}(\mcm^{\omega_1})$ and $\mcd=\mcp_{\Sigma}(\mcm^{\omega_1})/\mcm$.
	We obtain a diagram where both rows are short exact in \praddbl (by \cref{connpullbackofdualadd}):
	$$\begin{tikzcd}
		\mcm \arrow[d, Rightarrow, no head] \arrow[r] & \mcc\times_{\mcd}\mcc \arrow[d] \arrow[r] & \mcc \arrow[d] \\
		\mcm \arrow[r]                                & \mcc \arrow[r]                            & \mcd          
	\end{tikzcd}
	$$
	By \cref{swindle}, we have $\mcu^{\op{cont}}(\mcc)=0$. Therefore the induced map $$\mcu^{\op{cont}}(\mcm)\to \mcu^{\op{cont}}(\mcc\times_{\mcd}\mcc)$$ is an equivalence. However, \cref{diagpullbackcpgen} indicates that $\mcc\times_{\mcd}\mcc$ is compact projectively generated, as desired.
	\end{proof}
	
	\begin{rem}
			\cref{tmot1} shows that the $t$-structure on $\motpst$ generated by the dualizable additive \infcats under colimits and extensions is the same as that generated by the compact projectively generated additive \infcats. Thus the \(t\)-structure does not depend on whether one uses the dualizable
			large model or the small compact-projective model as generators.
	\end{rem}
	The resulting \(t\)-structure is not degenerate.  The following observation,
	using \(\op{THH}\), gives a simple way to detect nontrivial objects in its
	heart and shows that the universal functor is far from essentially
	surjective.
	\begin{rem}\label{rmk:thhmot}
		The localizing invariant $\op{THH}: \catadidem \to \opsp$
		induces a right $t$-exact functor $\motpst \to \opsp$ by \cite[Theorem 1.3.4.11]{ha}.
		Since there exist additive categories with $\pi_0\op{THH} \neq 0$, $\pi_0\op{THH}$ induces a non-trivial 
		functor $\motpst^\heartsuit \to \mathrm{Ab}$. 
		
		In particular, this implies $\mcu:\catadidem\to \motpst$ is \emph{not} essentially surjective.
	\end{rem}

\subsection{Questions}
We finish this section with a few open questions. 
In \cite{ramzi2025every} it has been proved that the universal finitary localizing invariant of stable \infcats $\catper \to \mcm_{\op{loc}}$ is a Dwyer--Kan localization. 
This result heavily relied on the existence of the categorical desuspension, given by the so-called Grayson construction, which is not available in the additive setting, and, in particular, we can see that 
$\catadidem\to \motpst$ is not even essentially surjective (see \cref{rmk:thhmot}). However, we still do not know whether the connective analog of this has any chance of being true:
\begin{quest} 
	Does $\catadidem$ admit the structure of a \infcat of cofibrant objects? 
	Is $\mcu|_{\geq 0}: \catadidem\to \motpstcn$ a Dwyer--Kan localization?
\end{quest}

At this moment we do not know any interesting computations of mapping spaces in $\motpst$ except \cref{main4}. 
On the other hand, in the stable setting Efimov proved an explicit formula for computing mapping spaces in $\mcm_{\op{loc}}$ (see \cite[Theorem~0.4]{efimov2025rigidity}). 
It is possible, however, that additive theory does not introduce any new complexity. In that case the answers to the following questions should be positive:
\begin{quest}\label{quest1}
	Is the induced functor $\motpstcn\to \mcm_{\op{loc}}$ fully faithful?
\end{quest}
\begin{quest}\label{quest2}
	Does every finitary localizing invariant $\catadidem\to \opsp$ extend to $\catper$?
\end{quest}
Note that due to monoidality of $\motpst \to \mcm_{\op{loc}}$ and \cref{main4}, the answer to \cref{quest1} and \cref{quest2} is true if  the following question holds:
\begin{quest}
	Is $\motpst$ rigid over $\opsp$?
\end{quest}

In \cite{efimov2025rigidity}, it was also shown that the \infcat of localizing motives
is rigid over $\opsp$. We would be surprised if this were true for $\motpst$, but do not know it to be false. 
Finally, we expect the following to be true, which in the stable case was proved by Efimov using the aforementioned formula for mapping spaces in $\mcm_{\op{loc}}$:

\begin{quest} 
 	Does $\mcu: \catadidem\to \motpst$  preserve products?
\end{quest}

We were unable to resolve the question below, but believe a positive answer:
\begin{quest}
	Does there exist a non-connective object $X$  in $\motpst$ such that its $\op{THH}$ is connective?
\end{quest}

	\appendix
	\section{Examples about small additive short exact sequences}\label{appendixA}
	In this appendix, we explicitly construct pathological counterexamples to rigorously demonstrate that additive homological epimorphisms and Karoubi projections are not stable under pullbacks in the small additive setting.
	\begin{exam}\label{exa1}
		Additive homological epimorphisms are not closed under pullback. 
		
		For example,
		Let $X = D^3$ be the unit 3-ball; this space is contractible. Let $Y = S^2 = \partial D^3$ be its boundary sphere. Set
		$A = C(D^3, \mathbb{C})$
		to be the ring of complex-valued continuous functions on $X$, and let
		$$I = \{\, f \in A \mid f(x) = 0 \text{ for all } x \in S^2 \,\}$$
		be the ideal of functions vanishing on the boundary. One checks that $I$ is an idempotent ideal: for rings of continuous functions, whenever the zero set is closed, the corresponding ideal satisfies $I^2 = I$ (using a partition of unity or a “taking roots” trick, e.g. $f = f^{1/3} \cdot f^{2/3}$). The quotient is
		$A/I \cong C(S^2, \mathbb{C}).$ Furthermore, we have $A/I^\infty\simeq A/I$ because $A/I \otimes^L_A A/I\simeq A/I$ (or by the fact that $I$ is a pure ideal).
		By Swan’s theorem, finitely generated projective $A$-modules correspond to complex vector bundles over $D^3$. Since the base space is contractible, all such bundles are trivial; hence
		$$\mathrm{Proj}_A \simeq \{ \text{free modules } A^{\oplus k} \}.$$
		Any $P \in \mathrm{Proj}_A$ restricts to the boundary $A/I$ as a trivial bundle $(A/I)^{\oplus k}$, so its first Chern class vanishes: $c_1(P) = 0$.
		For $A/I$: the sphere $S^2 \cong \mathbb{C}P^1$ admits nontrivial line bundles. Let $L$ denote the tautological line bundle on $S^2$, or any line bundle with
		$c_1(L) = 1 \in H^2(S^2;\mathbb{Z}) \cong \mathbb{Z}.$
		
		We define $\mathcal{A}:=\mathrm{Proj}_{\mathbb{Z}}$, and define $F: \mca \to \mathrm{Proj}_{A/I}$ given by $\mathbb{Z}\mapsto L$.
		Now consider the category
		$$\mathcal{P}= \mathcal{A} \times_{\mathrm{Proj}_{A/I}} \mathrm{Proj}_A.$$
		We claim $ \mathcal{A} \times_{\mathrm{Proj}_{A/I}} \mathrm{Proj}_A=0$ and hence $$0=\mcp_{\Sigma}( \mathcal{A} \times_{\mathrm{Proj}_{A/I}} \mathrm{Proj}_A)\to \mcp_{\Sigma}(\mca)\neq 0$$ \textbf{is not a connective internal localization}, even though $\mcp_{\Sigma}(\mathrm{Proj}_A)\to \mcp_{\Sigma}(\mathrm{Proj}_{A/I})$ is.
		
		An object of $\mathcal{P}$ is a triple $(N, P, \alpha)$, where:
		$N \in \mathcal{A}$, 
		$P \in \mathrm{Proj}_A$;
		$\alpha: F(N) \xrightarrow{\sim} P \otimes_A (A/I)$ is an isomorphism of $A/I$-modules. Since $P$ is a projective $A$-module and $A$ comes from a contractible space, $P$ must be free, so $P \cong A^{\oplus m}$. Restricting to $A/I$, we would have an isomorphism
		$$F(N)=L^{\oplus n} \cong (A/I)^{\oplus m}.$$
		Comparing first Chern classes,
		$$c_1\big((A/I)^{\oplus m}\big) = 0, \qquad c_1\big(L^{\oplus n}\big) = c_1(L) + \cdots + c_1(L) = n \cdot 1 = n.$$
		Thus an isomorphism can exist only if $n = 0$ (and hence $m = 0$). Therefore the only object of $\mathcal{P}$ is the zero object $(0,0,0)$, and the category $\mathcal{P}$ is equivalent to $0$.
	\end{exam}

	\begin{exam}\label{exa2}
		Additive Karoubi projections are not closed under pullback.
		
		For example, let $k$ be a field such that $\text{char}(k) \neq 2$ and $R=k[x,y]$ and $I=(y^2-x^3+x)\subset R$.  Consider the (noncommutative) ring
		\[
		\Lambda=\begin{pmatrix}R & R \\ I & R\end{pmatrix},
		\]
		and the additive category $\mathcal{B}=\mathrm{Proj}(\Lambda)$ of finitely generated projective left $\Lambda$-modules. Two distinguished projective left $\Lambda$-modules are
		\[
		P_1=\begin{pmatrix}R\\ I\end{pmatrix},\qquad P_2=\begin{pmatrix}R\\ R\end{pmatrix},
		\]
		and they generate \mcb under finite sums. 
		We set $\mathcal{A}=\mathbf{add}(P_2)\subset \mathcal{B}$.
		The full subcategory $\mathcal{A}$ is governed by the trace ideal of $P_2$. Let
		\[
		e_2=\begin{pmatrix}0 & 0 \\ 0 & 1\end{pmatrix}\in \Lambda,\qquad P_2\cong \Lambda e_2.
		\]
		Then the trace ideal is $J=\Lambda e_2\Lambda$, and an explicit multiplication shows
		\[
		\Lambda e_2=\begin{pmatrix}0 & R\\ 0 & R\end{pmatrix},\qquad
		J=\begin{pmatrix}0 & R\\ 0 & R\end{pmatrix}\begin{pmatrix}R & R\\ I & R\end{pmatrix}
		=\begin{pmatrix}I & R\\ I & R\end{pmatrix}.
		\]
		Since $I^2\subset I$, one checks
		\[
		J^2=\begin{pmatrix}I & R\\ I & R\end{pmatrix}\begin{pmatrix}I & R\\ I & R\end{pmatrix}
		=\begin{pmatrix}I^2+I & I+R\\ I^2+I & I+R\end{pmatrix}
		=\begin{pmatrix}I & R\\ I & R\end{pmatrix}=J,
		\]
		so $J$ is a two-sided idempotent ideal. 
		Furthermore, $\Lambda\to \Lambda/J$ is a homological epimorphism: $\Lambda/J$ is derived idempotent over $\Lambda$, i.e.
		\[
		\Lambda/J\otimes^{\mathbb{L}}_{\Lambda}\Lambda/J\cong \Lambda/J,
		\]
		because $\operatorname{Tor}_i^{\Lambda}(\Lambda/J,\Lambda/J)=0$ for $i\geq 1$. One convenient projective resolution of the right $\Lambda$-module $(\Lambda/J)_R\cong (R/I,0)$ is
		\[
		0\longrightarrow P_2^{\mathrm{row}}=(I,R)\xrightarrow{\iota} P_1^{\mathrm{row}}=(R,R)\longrightarrow (R/I,0)\longrightarrow 0.
		\]
		Tensoring on the right with $(\Lambda/J)_L$ and using $e_2(\Lambda/J)=0$ gives
		\[
		P_2^{\mathrm{row}}\otimes_{\Lambda}(\Lambda/J)_L\cong e_2(\Lambda/J)_L=0,\qquad
		P_1^{\mathrm{row}}\otimes_{\Lambda}(\Lambda/J)_L\cong e_1(\Lambda/J)_L\cong R/I,
		\]
		so $\operatorname{Tor}_1^{\Lambda}(\Lambda/J,\Lambda/J)=0$ and higher $\operatorname{Tor}$ groups vanish as well.  Consequently, $\Lambda/J$ can be identified with $\Lambda/J^\infty$ that appeared in \cref{almostalg}.

		Denote
		\[
		\mathrm{Proj}(\Lambda,J):= \ker(\mathrm{Proj}(\Lambda)\to\mathrm{Proj}(\Lambda/J))\simeq\{P\in \mathrm{Proj}(\Lambda)\mid JP=P\}.
		\]
		Then $JP_2=P_2$, so $P_2\in \mathrm{Proj}(\Lambda,J)$ and hence $\mathcal{A}\subset \mathrm{Proj}(\Lambda,J)$. On the other hand, we will show that $\mathrm{Proj}(\Lambda,J)\subset \mca$ and hence $$\mathrm{Proj}(\Lambda,J)=\mca.$$ In particular, \mca is idempotent-complete.
		
		Now given a projective left $\Lambda$-module $P$ such that $JP = P$. Then $P$ decomposes as
		\[
		P \cong P_1^{\oplus m} \oplus P_2^{\oplus n}
		\]
		for some $m,n \geq 0$. Since $JP = P$,  by distributivity,
		\[
		J\big(P_1^{\oplus m} \oplus P_2^{\oplus n}\big) 
		= (J P_1)^{\oplus m} \oplus (J P_2)^{\oplus n}
		\]
		must equal $P_1^{\oplus m} \oplus P_2^{\oplus n}$. By Krull–Schmidt uniqueness (or a simple componentwise rank comparison), each summand must be preserved, forcing in particular
		\[
		J P_1 \cong P_1 \quad \text{if } m>0.
		\]
		Since
		\[
		J P_1
		=
		\begin{pmatrix} I & R \\ I & R \end{pmatrix}
		\begin{pmatrix} R \\ I \end{pmatrix}
		=
		\begin{pmatrix} I\cdot R + R\cdot I \\ I\cdot R + R\cdot I \end{pmatrix}
		=
		\begin{pmatrix} I + I \\ I + I \end{pmatrix}
		=
		\begin{pmatrix} I \\ I \end{pmatrix}.
		\]
		and
		\[
		P_1=\begin{pmatrix} R \\ I \end{pmatrix}
		\quad \text{and} \quad
		J P_1=\begin{pmatrix} I \\ I \end{pmatrix},
		\]
		the equality $J P_1 = P_1$ would force $R = I$ in the first component, which leads to a contradiction because $I \subsetneqq R$. That proves $m=0$ and $P\in\mca$.
		
		Heuristically $\mathrm{quot}(\mca\to\mcb)$ behaves like the image of $\mathrm{Proj}(\Lambda)$ inside $\mathrm{Proj}(\Lambda/J)$; it misses the non-diagonal projectives of $\mathrm{Proj}(\Lambda/J)$. More specifically, one can  find that $$\mathrm{quot}(\mca\to\mcb)\simeq \op{Free}(\Lambda/J)\simeq \op{Free}(R/I),$$ because \mcb is generated by $P_1, P_2$ under finite sums and $
		\Lambda/J\cong \begin{pmatrix}R/I & 0\\ 0 & 0\end{pmatrix}.
		$
		Its Karoubi envelope $$\mcb/\mca:=\mathrm{quot}(\mca\to\mcb)^{\natural}\simeq \op{Proj}(R/I)$$ restores the missing (non-diagonal) projectives. Now  consider 
		\[
		\mathcal{E} :=  \op{Proj}(\mathbb{Z}),
		\]
		and the additive functor $\mce \to \mcb/\mca$ given by $\mathbb{Z}\mapsto L=(x,y)/I \in \op{Proj}(R/I)$, where $L$ corresponds with the canonical bundle  and is a non-free projective $R/I$-module. By design, 
		\mce is idempotent-complete and the additive functor $$\mce\times_{\mcb/\mca} \mcb\to \mce$$
		factors through the zero category due to the following pullback squares.
		$$\begin{tikzcd}
			\mce\times_{\mcb/\mca} \mcb \arrow[d] \arrow[r] & 0 \arrow[d] \arrow[r] & \mce \arrow[d]         \\
			\mcb \arrow[r]                                         & \mathrm{quot}(\mca\to\mcb) \arrow[r]   & \mathrm{quot}(\mca\to\mcb)^{\natural}
		\end{tikzcd}$$
		Therefore $\mce\times_{\mcb/\mca} \mcb\to \mce$ is no longer an additive Karoubi projection.
	\end{exam}
	
\section{Compact projectively assembled \infcats}\label{appendixB}
To support the analysis of \dualaddinfcats in the main text, we provide a self-contained exposition of compact projectively assembled \infcats. 
	\begin{de}
	We define \catsift to be the (very large) \infcat whose objects are large \infcats with small sifted colimits and whose morphisms are sifted-colimit-preserving functors. 
	\end{de}
		\begin{de}
		Let $\mcc\in\catsift$. 
		\begin{enumerate}[label=(\arabic*),font=\normalfont]
			\item We say that an object $x\in \mcc$ is (compact) projective if the functor $\mapp_{\mcc}(x,-):\mcc\to\mcs$ preserves geometric realizations (respectively, small sifted colimits).
			\item We say a map $f:x\to y$ in \mcc is compact projective if    for any small sifted diagram
			$z_\bullet : I \to \mcc$, there exists a diagonal filler in the following commutative square of anima:  
			
			$$
			\begin{tikzcd}
				{\colimit_{I}\mapp(y,z_i)} \arrow[d] \arrow[r] & {\mapp(y,\colimit_{I}z_i)} \arrow[d] \arrow[ld, dashed] \\
				{\colimit_{I}\mapp(x,z_i)} \arrow[r]           & {\mapp(x,\colimit_{I}z_i)}                             
			\end{tikzcd}$$
		\end{enumerate}
	\end{de}
	\begin{de}
		Let $\mcc\in \catsift$. We say $\mcc$ is compact projectively generated if $\mcc^{\op{cproj}}$ is small and $\mcc\simeq \mcp_{\Sigma}(\mcc^{\op{cproj}})=\mcp_{\emptyset}^{\op{sift}}(\mcc_0)$, i.e.  the sifted cocompletion. We say \mcc is compact projectively assembled if it is a retract of a compact projectively generated \infcat in \catsift.
	\end{de}
	\begin{rem}
		Let $\mcc\in \catsift$. (Compact) projective objects in \mcc are closed under retracts.
	\end{rem}
	\begin{rem}
		If $\mcc\simeq \mcp_{\Sigma}(\mcd)$ for a small \infcat \mcd, then for any $d\in \mcd$ the Yoneda image $h_d\in \mcp_{\Sigma}(\mcd)$ is compact projective.
	\end{rem}
	\begin{lem}[{cf. \cite[Proposition 10.1]{rezk2021generalizing}}]\label{siftrep}
		Let $\mathcal{C}_0$ be a small \infcat. If $F \in \mcc=\mathcal{P}_\Sigma(\mathcal{C}_0)$, then the \infcat of elements $\int F := \mcc_0\times_\mcc \mcc_{/F}$ is a small sifted \infcat.
	\end{lem}
	
	\begin{proof}
		Let $Y:\mcc_0^{\opp}\to \mcp(\mcc_0^{\opp})$ denote the Yoneda embedding of $\mcc_0^{\opp}$, and let $h^A=\mapp_{\mcc_0}(A,-)\in \mcp(\mcc_0^{\opp})$ denote the corepresentable functor for any $A\in\mcc_0$. Let $\op{Lan}_Y(-):\funct(\mcc_0^\opp,\mcs)\to \funct(\mcp(\mcc_0^{\opp}),\mcs)$ denote the left Kan extension along $Y$. For any objects $A, B \in \mathcal{C}_0$, we claim that the canonical map
		\[
		\begin{tikzcd}
			\op{Lan}_Y(F)(h^A\times h^B) \arrow[d, Rightarrow, no head] \arrow[r]                                                            & \op{Lan}_Y(F)(h^A)\times\op{Lan}_Y(F)(h^B) \arrow[d, Rightarrow, no head] \\
			{\colim_{(X, x) \in \int F} \left( \mapp_{\mathcal{C}_0}(A, X) \times \mapp_{\mathcal{C}_0}(B, X) \right)} \arrow[r, "\theta_F"] & F(A)\times F(B)                                                          
		\end{tikzcd}
		\]
		is an equivalence in \mcs, since the corresponding statement holds when $F$ is a representable presheaf and the class of $F$ satisfying this property is closed under sifted colimits.
		
		To prove that $\int F$ is sifted, it suffices by the Joyal-Lurie criterion to show that for any pair of objects $(A, a)$ and $(B, b)$ in $\int F$, the \infcat of cospans under them, denoted by
		\[
		\mathcal{D} = (\int F)_{(A,a)/} \times_{\int F} (\int F)_{(B,b)/},
		\]
		is weakly contractible, meaning its classifying anima $|\mathcal{D}|$ is equivalent to a point $\ast$.
		
		Fix a base point $(a, b) \in F(A) \times F(B)$. Since $\theta_F$ is an equivalence of anima, its homotopy fiber at $(a, b)$ is contractible:
		$
		\fib_{(a,b)}(\theta_F) \simeq \ast.
		$
		In the topos $\mathcal{S}$, colimits are universal and thus commute with pullbacks (and consequently, with homotopy fibers). Therefore, we can compute the homotopy fiber of the colimit by taking the colimit of the object-wise homotopy fibers:
		\[
		\ast \simeq \colim_{(X, x) \in \int F} \fib_{(a,b)}\left( \eta_{(X,x)} \right),
		\]
		where $$\eta_{(X,x)}\colon \Map_{\mathcal{C}_0}(A, X) \times \Map_{\mathcal{C}_0}(B, X) \to F(A) \times F(B)$$ is the canonical map induced by the Yoneda extension.
		
		Recall that the canonical projection $\int F \to \mathcal{C}_0$ is a right fibration. A fundamental property of right fibrations implies that the homotopy fiber of the map $\Map_{\mathcal{C}_0}(A, X) \xrightarrow{u \mapsto F(u)(x)} F(A)$ over the point $a \in F(A)$ is canonically equivalent to the mapping anima $\Map_{\int F}((A, a), (X, x))$. Since the homotopy fiber of a product anima is canonically equivalent to the product of the homotopy fibers, we obtain:
		\[
		\fib_{(a,b)}(\eta_{(X,x)}) \simeq \Map_{\int F}((A, a), (X, x)) \times \Map_{\int F}((B, b), (X, x)).
		\]
		Substituting this equivalence back into our colimit expression, we have:
		\[
		|\mcd|\simeq \colim_{(X, x) \in \int F} \left( \Map_{\int F}((A, a), (X, x)) \times \Map_{\int F}((B, b), (X, x)) \right)\simeq \colim_{(X, x) \in \int F} \fib_{(a,b)}\left( \eta_{(X,x)} \right)\simeq *.
		\]
	\end{proof}
	\begin{rem}
		From the proof we can see that a presheaf $F\in \mcp(\mcc_0)$ lies in the sifted cocompletion $\mcp_{\Sigma}(\mcc_0)$ if and only if the natural map 	$$\op{Lan}_Y(F)(h^A\times h^B) \xrightarrow{\theta_F^{(A,B)}}\op{Lan}_Y(F)(h^A)\times\op{Lan}_Y(F)(h^B)$$ is an equivalence for all $A,B\in\mcc_0$.
		In particular, $\mcp_\Sigma(\mcc_0)$ can be identified with the following pullback:
		$$\begin{tikzcd}[row sep=2em, column sep=8em]
			\mcp_\Sigma(\mcc_0) \arrow[d, hook] \arrow[r]                      & {\prod_{(A,B)\in \mcc_0^2}\op{Iso}(\Delta^1,\mcs)} \arrow[d, hook] \\
			\mcp(\mcc_0) \arrow[r, "{F\mapsto \prod_{(A,B)}\theta_F^{(A,B)}}"] & {\prod_{(A,B)\in \mcc_0^2}\funct(\Delta^1,\mcs)}                  
		\end{tikzcd}$$
		
		Furthermore, since the bottom and right arrows above are accessible (actually sifted-colimit-preserving) functors between accessible \infcats, it follows that $\mcp_{\Sigma}(\mcc_0)$ is accessible for any small \infcat $\mcc_0$ (cf. \cite[Proposition 14.1]{rezk2021generalizing}).
	\end{rem}
	\begin{rem}\label{siftrep2}
		If $\mcc_0$ has finite coproducts, the proof of \cref{siftrep} becomes easier. In this case  $\mcc_0\times_\mcc \mcc_{/F}$ admits finite coproducts too,  so its diagonal is a right adjoint and hence cofinal, which  deduces that $\mcc_0\times_\mcc \mcc_{/F}$ is sifted.
	\end{rem}
	\begin{prop}
Let $\mcc\in \catsift$. Then $\mcc$ is compact projectively generated if and only if it is locally small and is generated under small sifted colimits by a small set $S$ of compact projective objects. Moreover, in this case, $\mcc^{\op{cproj}}$ can be identified with the idempotent completion of $S$.
	\end{prop}
\begin{proof}
	The ``only if'' direction is clear. For the ``if'' direction, by assumption we can show that $$\mcp_{\Sigma}(S)\simeq \mcc.$$ Given a compact projective object $x\in \mcc^{\op{cproj}}$. By \cref{siftrep}, $x$ can be written as a small sifted colimit  $x\simeq \colimit_I x_i$ such that each $x_i\in S$. Therefore $\op{id}:x\to x$ factors through some $x_i\to x$, and hence $x$ is a retract of some $x_i$. Consequently, $\mcc^{\op{cproj}}$ can be identified with the idempotent completion of $S$ and hence small. Since we have the inclusions $\mcp_{\Sigma}(S)\hookrightarrow \mcp_{\Sigma}(\mcc^{\op{cproj}})\hookrightarrow \mcc$, we see that $\mcp_{\Sigma}(\mcc^{\op{cproj}})\simeq \mcc$.
\end{proof}
Now we start to investigate compact projectively assembled \infcats.
\begin{prop}\label{cpasscriterion1}
	Let $\mathcal{C}\in\catsift$. Then \mcc is compact projectively assembled if and only if  \mcc is accessible and the colimit functor 
	$$c_\kappa:\mcp_{\Sigma}(\mcc^\kappa)\to\mcc$$ admits a left adjoint for some (large enough) regular cardinal $\kappa$.
\end{prop}
\begin{proof}
	This follows from the same argument as in \cite[Theorem 21.1.2.10]{sag}.
\end{proof}
\begin{de}
		Let $\mathcal{C}\in\catsift$. We say an object $x\in\mcc$ is \textbf{compact projectively exhaustible} if $x\simeq \colim_{n}x_n$ can be written as a sequential colimit such that each transition map $x_n\to x_{n+1}$ is a compact projective map.
\end{de}
\begin{rem}\label{preserveprojmap1}
	 If $F: \mathcal{C} \rightarrow \mathcal{D} \in\catsift$  is fully faithful, then it reflects  compact projective maps. If $F: \mathcal{C} \rightarrow \mathcal{D}\in\catsift$ has a sifted-colimit-preserving right adjoint, then $F$ preserves compact projective maps and compact projectively exhaustible objects.
\end{rem}
	\begin{prop}\label{a13}
   Let $\mathcal{C}\in\catsift$ such that \mcc is $\kappa$-accessible for a regular cardinal $\kappa$. Let $x\simeq\colim_n x_n$ be a compact projective exhaustion such that each $x_n$ is $\kappa$-compact. For any $Y\in\mcp_\Sigma(\mcc^\kappa)$ the following is an equivalence
		$$\mapp_{\mcp_\Sigma(\mcc^\kappa)}(X, Y) \to \mapp_{\mcc}(x, c(Y))$$
		where $c:\mcp_\Sigma(\mcc^\kappa)\to \mcc$ is the colimit functor and $X\simeq \colim_n h_{x_n}$.
		
	\end{prop}
	\begin{proof}
		By \cref{siftrep} we can find a small sifted diagram $y_\bullet:I\to \mcc^\kappa$ such that $Y\simeq \colimit_I h_{y_i}$. Thus the map in question is the canonical map
		$$
		\lim_{n} \operatorname{colim}_I \operatorname{Map}_\mcc\left(x_n, y_i\right) \rightarrow \lim_{n} \operatorname{Map}_\mcc\left(x_n, \operatorname{colim}_I y_i\right) .
		$$
		It is indeed an equivalence by the filler property of a compact projective map, see \cite[Lemma 2.24]{ramzi2024dualizable} for a similar argument.
	\end{proof}
\begin{rem}\label{partadj}
	Let $\mathcal{C}\in\catsift$ such that \mcc is $\kappa$-accessible for some (large enough) regular cardinal $\kappa$. Recall that the colimit functor $c:\mcp_\Sigma(\mcc^\kappa)\to\mcc$ is  a left adjoint. However, \cref{a13} indicates that with respect to any compact projective exhaustible object $x \in \mcc$, the functor $c$ also behaves like a \emph{right} adjoint. In categorical terms, $c$ admits a (partial) left adjoint evaluated at $x$, with the corresponding lifted object $X \simeq \colim_n h_{x_n}$ serving as this left adjoint value. 
\end{rem}
	
	\begin{prop}\label{cpasscriterion}
	 Let $\mcc\in\catsift$. If  \mcc is accessible and is generated by compact projectively exhaustible objects under small sifted colimits, then \mcc is compact projectively assembled.
	\end{prop}

	\begin{proof}
		Because $\mcc$ is accessible, we can choose a regular cardinal $\kappa$ large enough such that all objects in $\mathcal{E}$ are $\kappa$-compact (i.e., $\mathcal{E} \subset \mcc^\kappa$). For each generator $e \in \mathcal{E}$, \cref{a13} ensures that the functor $$\operatorname{Map}_\mcc(e, c_\kappa(-)): \mcp_\Sigma(\mcc^\kappa) \to \mathcal{S}$$ is corepresentable, which means the colimit functor $c_\kappa$ admits a (partial) left adjoint at $e$. Because $\mathcal{E}$ generates $\mcc$ under small sifted colimits, we conclude that $\operatorname{Map}_\mcc(x, c_\kappa(-))$ is corepresentable for any $x\in\mcc$, as desired.
	\end{proof}
	\begin{quest}
	Is the converse of	\cref{cpasscriterion} true?
	\end{quest}
		\subsection{Projective maps}
Inspired by 
the atomic maps defined in \cite{ramzi2024locally}, we investigate relations between dualizability and projective maps in this section. We refer the reader to \cite[\textsection2]{ramzi2024locally} for atomic maps in general enriched context.

\begin{prop}[{\cite[Proposition 2.3]{ramzi2024locally}}]
	Let $\mathcal{M}, \mathcal{N}$ be  presentable additive \infcats. Then the inclusion $\operatorname{Fun}^L(\mathcal{M}, \mathcal{N}) \rightarrow \operatorname{Fun}^{\op{add}}(\mathcal{M}, \mathcal{N})$ admits a right adjoint.
\end{prop}
\begin{de}
	Let $\mathcal{M}\in\prlad$, and $x \in \mathcal{M}$. We let  $\underline{\op{at}}_{\mathcal{M}}(x,-)$ denote the image of $\underline{\operatorname{Map}}_{\mathcal{M}}(x,-): \mathcal{M} \rightarrow \spgeq$ under the right adjoint to the inclusion $\operatorname{Fun}^L(\mathcal{M}, \spgeq) \rightarrow \operatorname{Fun}^{\op{add}}(\mathcal{M}, \spgeq)$; i.e. $$\underline{\op{at}}_{\mathcal{M}}(x,-) \rightarrow \underline{\operatorname{Map}}_{\mathcal{M}}(x,-)$$ is the terminal colimit-preserving functor with a map to $\underline{\operatorname{Map}}_{\mathcal{M}}(x,-)$. We also let  $\op{at}_{\mathcal{M}}(x,-):= \Omega^\infty\underline{\op{at}}_{\mathcal{M}}(x,-)$ denote the underlying functor to anima.
	
	We say a morphism $f: x \rightarrow y$ in $\mathcal{M}$ is $\spgeq$-atomic, if the classifying morphism $\mathbb{S} \rightarrow \unmap_{\mathcal{M}}(x, y)$ admits a lift to  $\unat_{\mathcal{M}}(x, y)$ along the canonical map  $\unat_{\mathcal{M}}(x, y) \rightarrow \unmap_{\mathcal{M}}(x, y)$.
	
\end{de}
\begin{prop}
	Let $\mathcal{M}\in\prlad$ and let $f$ be a morphism in \mcm. Then $f$ is \spgeq-atomic if and only if the  map between corepresentable functors $h^y\to h^x \in \funct(\mcm,\mcs)$ factors through a functor $C:\mcm\to\mcs$ which preserves sifted colimits and finite products. 
\end{prop}
\begin{proof}
	It suffices to observe that $\funct^L(\mcm,\spgeq)\simeq\funct^{\op{sift,\times}}(\mcm,\spgeq)\simeq\funct^{\op{sift,\times}}(\mcm,\mcs)$.
\end{proof}
\begin{de}\label{atomicmaps}
	Let $\mathcal{M}\in\prlad$ and let $f$ be a morphism in \mcm. 
	\begin{enumerate}[label=(\arabic*),font=\normalfont]
		\item We say $f$ is strongly compact projective if $h^y\to h^x \in \funct(\mcm,\mcs)$ factors through a functor $C:\mcm\to\mcs$ which preserves sifted colimits.
		\item We say $f$ is compact projective if    for any small sifted diagram
		$z_\bullet : I \to \mcm$, there exists a diagonal filler in the following commutative square of anima:  
		
		$$
		\begin{tikzcd}
			{\colimit_{I}\mapp(y,z_i)} \arrow[d] \arrow[r] & {\mapp(y,\colimit_{I}z_i)} \arrow[d] \arrow[ld, dashed] \\
			{\colimit_{I}\mapp(x,z_i)} \arrow[r]           & {\mapp(x,\colimit_{I}z_i)}                             
		\end{tikzcd}$$
		We say $f$ is projective if the above filler exists when $I=\Delta^\opp$.

		\item We say $f: x \rightarrow y$  is $\spgeq$-atomically presentable if $x$ can be written as a weighted colimit $\operatorname{colim}_I^W f$ for some small $\spgeq$-enriched category $I$, weight $W$ and diagram $f$ so that the map $\alpha: \operatorname{colim}_I^W f \rightarrow y$ corresponds, under adjunction, to a map of weights $W \rightarrow \unmap_{\mathcal{M}}(f(-), y)$ that factors through  $\unat_{\mathcal{M}}(f(-), y)$.
	\end{enumerate}
\end{de}
\begin{de}\label{atomicexhau}
	Let $\mathcal{M}\in\prlad$.
	We say an object $x\in\mcm$ is compact projectively (resp. projectively, 
	 resp. \spgeq-atomically, resp. \spgeq-atomically presentable) exhaustible if $x$ can be written as a sequential colimit such that each transition map is compact projective (resp. projective,
	 resp. \spgeq-atomic, resp. \spgeq-atomically presentable).
	
\end{de}
\begin{rem}\label{presreflcpmap}
	Let $\mathcal{M}\in\prlad$.
	\enu{

		\item It is clear that any \spgeq-atomic map is strongly compact projective, and that any strongly compact projective map is compact projective.  
		
		\item The collections of (strongly) compact projective maps and \spgeq-atomic maps  are 2-sided ideals: if $f$ is a (strongly) compact projective map (resp. \spgeq-atomic map), then for any $g, h$ for which it makes sense, so is $g f h$. In particular, any map factoring through a compact projective object is \spgeq-atomic.
		\item If $F: \mathcal{M} \rightarrow \mathcal{N} \in\prlad$  is fully faithful, then it reflects (strongly) compact projective maps and \spgeq-atomic maps. If $F: \mathcal{M} \rightarrow \mathcal{N}\in\prlad$ is an internal left adjoint, then it preserves (strongly) compact projective maps and \spgeq-atomic maps.
	}
\end{rem}
\begin{prop}[{\cite[Corollary 2.13]{ramzi2024locally}}]
	Let $\mathcal{M}$ be a \dualaddinfcat and $x \in \mathcal{M}^\kappa$ for some uncountable regular $\kappa$, with $\hat{h}: \mathcal{M} \rightarrow \mathcal{P}_{\Sigma}\left(\mathcal{M}^\kappa\right)$ the left adjoint to the colimit functor $c: \mathcal{P}_{\Sigma}\left(\mathcal{M}^\kappa\right) \rightarrow \mathcal{M}$, and let $h$ be the right adjoint to this canonical functor. Then we have the identification
	$$
	\unat_{\mathcal{M}}(x,-) \simeq \unmap_{\mathcal{P}_{\Sigma}\left(\mathcal{M}^\kappa\right)}(h(x), \hat{h}(-))
	$$
	and the map $\unat_{\mathcal{M}}(x,-)  \rightarrow \unmap_{\mathcal{M}}(x,-)$ is given by applying $c$ and using $c \circ \hat{h} \xrightarrow{\sim} c \circ h \xrightarrow{\sim} \operatorname{id}_{\mathcal{M}}$.
\end{prop}
\begin{prop}\label{dualaddcprojmaps}
	Let $\mcm\in\prlad$. Then the following are equivalent:
	\enu{
		
		\item  $\mcm$ is a \dualaddinfcat.
		\item  $\mcm$ is compact projectively assembled.
		\item \mcm is generated by \spgeq-atomically exhaustible objects under small colimits.
		\item  \mcm is generated by compact projective exhaustible objects under small colimits.

	}
\end{prop}
\begin{proof}
	$(1)\Longleftrightarrow  (2)$: This follows from \cref{ram149}.
	
	$(1)\implies (3)$: This follows by combining \cref{flatgenerate} and \cref{flatcpexhau}.
	
	$(3)\implies (4)$: Trivial.
	
	$(4)\implies (2)$: This follows from the same argument as \cref{cpasscriterion}.
\end{proof}
\begin{lem}\label{a25}
	In a \dualaddinfcat \mcm, any compact projective map is an \spgeq-atomic map.
\end{lem}
\begin{proof}
Let $f\in\mcm$ be a compact projective map.	Since the hat Yoneda $\hat{h}:\mcm\hookrightarrow \mcp_{\Sigma}(\mcm^{\omega_1})$ is a fully faithful internal left adjoint, it preserves and reflects both compact projective maps and \spgeq-atomic maps. Since $\hat{h}(f)$ is compact projective, it factors through a compact projective object. Therefore $\hat{h}(f)$ is an \spgeq-atomic map and so is $f$.
\end{proof}
\begin{prop}
	Let $F:\mcm\to\mcn \in\prlad$. Assume \mcm, \mcn are dualizable additive. Then the following are equivalent:
	\enu{
		
		\item  $F$ is an internal left adjoint.
		\item $F$ preserves \spgeq-atomic maps.
		\item  $F$ preserves compact projective maps.

	}
\end{prop}

\begin{proof}
	$(1)\implies (2):$	 This follows from \cref{presreflcpmap}.
	
	$(2)\implies (3):$ This follows from \cref{a25}.

	$(3)\implies (1):$ This follows by the same argument as \cite[Corollary 2.43]{ramzi2024dualizable}.
\end{proof}
\section{$\kappa$-projective generation}\label{kappaprojgeneration}
In this appendix, we study projective generation, rather than merely compact projective generation, and  large weight structures.

Throughout, we fix a regular cardinal $\kappa$.
\begin{de}
	Let $\mcc$ be a small \infcat which admits $\kappa$-small coproducts. We let $\mcp_{\Sigma_\kappa}(\mcc)$ denote the full subcategory of $\mcp(\mcc)$ spanned by those
	functors $\mcc^\opp \to\mcs$ which preserve $\kappa$-small products.
\end{de}
\begin{thm}\label{thma2}
	Let $\mathcal{C}$ be a small \infcat which admits $\kappa$-small coproducts. Then
	\enu{
		\item  The \infcat $\mathcal{P}_{\Sigma_\kappa}(\mathcal{C})$ is an accessible localization of $\mathcal{P}(\mathcal{C})$.
		\item The Yoneda embedding $j: \mathcal{C} \rightarrow \mathcal{P}(\mathcal{C})$ factors through $\mathcal{P}_{\Sigma_\kappa}(\mathcal{C})$. Moreover, $j$ carries $\kappa$-small coproducts in $\mathcal{C}$ to $\kappa$-small coproducts in $\mathcal{P}_{\Sigma_\kappa}(\mathcal{C})$.
		\item Let $\mc{D}$ be a presentable \infcat and let
		$$
		\mathcal{P}(\mathcal{C}) \underset{G}{\stackrel{F}{\rightleftarrows}} \mathcal{D}
		$$
		be a pair of adjoint functors. Then $G$ factors through $\mathcal{P}_{\Sigma_\kappa}(\mathcal{C})$ if and only if $f=F \circ j: \mathcal{C} \rightarrow \mathcal{D}$ preserves $\kappa$-small coproducts.
		\item The full subcategory $\mathcal{P}_{\Sigma_\kappa}(\mathcal{C}) \subset \mathcal{P}(\mathcal{C})$ is closed under small $\kappa$-filtered colimits.
		\item The \infcat $\mathcal{P}_{\Sigma_\kappa}(\mathcal{C})$ is $\kappa$-compactly generated.
		\item 
		If $\mcc$ is \emph{additive}, then the full subcategory $\mathcal{P}_{\Sigma_\kappa}(\mathcal{C}) \subset \mathcal{P}(\mathcal{C})$ is also closed under geometric realizations.
	}
	
\end{thm}

\begin{warning}
	Beware that if \mcc is not additive, then $\mathcal{P}_{\Sigma_\kappa}(\mathcal{C}) \subset \mathcal{P}(\mathcal{C})$ is not necessarily closed under geometric realizations. 
\end{warning}
\begin{proof}
	(1)-(3) follow from the construction of the formal completion; see \cite[Proposition 5.3.6.2]{htt}. 
	
	For (4), this follows because $\kappa$-small products commute with $\kappa$-filtered colimits in \mcs; see \cite[Proposition 5.3.3.3]{htt}.
	
	For (5), it suffices to observe that $\mathcal{P}_{\Sigma_\kappa}(\mathcal{C})=\mcp^{\op{small}}_{\kappa\text{-}\sqcup}(\mcc)\simeq \op{Ind}_\kappa(\mcp^{\kappa\text{-}\op{small}}_{\kappa\text{-}\sqcup}(\mcc))$.
	
	For (6), it suffices to observe that in $\spgeq$ infinite products commute with geometric realizations  by \cref{prodcommutegeom}.
\end{proof}
\begin{lem}[{\cite[Lemma 5.5.8.13]{htt}}]
	Let $\mathcal{C}$ be a small \infcat and let $X$ be an object of $\mathcal{P}(\mathcal{C})$. Then there exists a simplicial object $Y_{\bullet}: \mathrm{N}(\boldsymbol{\Delta})^{o p} \rightarrow \mathcal{P}(\mathcal{C})$ with the following properties:
	\enu{
		\item The colimit of $Y_{\bullet}$ is equivalent to $X$.
		\item  For each $n \geq 0$, the object $Y_n \in \mathcal{P}(\mathcal{C})$ is equivalent to a small coproduct of objects lying in the image of the Yoneda embedding $j: \mathcal{C} \rightarrow \mathcal{P}(\mathcal{C})$.
	}
	
\end{lem}

\begin{prop}\label{propa5}
	Let $\mathcal{C}$ be a small \infcat which admits $\kappa$-small coproducts and let $X \in \mathcal{P}_{\Sigma_\kappa}(\mathcal{C})$. Then there exists a simplicial object $U_{\bullet}: \mathrm{N}(\boldsymbol{\Delta})^{o p} \rightarrow \operatorname{Ind}_\kappa(\mathcal{C})$ whose colimit\footnote{Beware such colimit is not necessarily preserved by inclusion $\mathcal{P}_{\Sigma_\kappa}(\mathcal{C}) \subset \mathcal{P}(\mathcal{C})$, unless $\kappa=\omega$ or \mcc is additive.} in $\mathcal{P}_{\Sigma_\kappa}(\mathcal{C})$ is $X$.
	
\end{prop}
\begin{proof}
	The proof of \cite[Lemma 5.5.8.14]{htt} also works for arbitrary regular cardinal $\kappa$.
\end{proof}
\begin{rem}
	Let $\mathcal{C}$ be a small \infcat which admits $\kappa$-small coproducts. We have a canonical equivalence $$\op{Ind}_\kappa(\mcc) \simeq \mcp_{\kappa\text{-}\sqcup}^{\kappa\text{-}\op{fil},\kappa\text{-}\sqcup}(\mcc).$$
	
	To see this, recall that
	in the \infcat $\mathcal{S}$ of anima , $\kappa$-filtered colimits commute with $\kappa$-small limits (and thus with $\kappa$-small products). Because limits and colimits in presheaf categories are computed pointwise, $\mcp_{\Sigma_\kappa}(\mcc)$ is closed under $\kappa$-filtered colimits in $\mcp(\mcc)$. Therefore, the $\kappa$-filtered colimit completion of $\mcc$, denoted $\op{Ind}_\kappa(\mcc)$, sits entirely inside $\mcp_{\Sigma_\kappa}(\mcc)$.
	
	Furthermore, the subcategory $\op{Ind}_\kappa(\mcc) \subset \mcp_{\Sigma_\kappa}(\mcc)$ is  closed under $\kappa$-small coproducts. Consequently, it exactly satisfies the universal property of the joint free completion $\mcp_{\kappa\text{-}\sqcup}^{\kappa\text{-}\op{fil},\kappa\text{-}\sqcup}(\mcc)$.
\end{rem}
\begin{cor}
	Let $\mathcal{C}$ be a small \infcat which admits $\kappa$-small coproducts.
	If  \mcc is additive, then we have a natural identification  $$\mathcal{P}_{\Sigma_\kappa}(\mathcal{C})\simeq \mathcal{P}_{\emptyset}^{\kappa\text{-}\mathrm{fil},\Delta^\opp}(\mathcal{C}).$$
\end{cor}
\begin{proof}
	Since \mcc is additive, the inclusion $\mathcal{P}_{\Sigma_\kappa}(\mathcal{C})\subset \mathcal{P}(\mathcal{C})$ is closed under $\kappa$-filtered colimits and geometric realizations. Therefore, the following inclusion holds inside $\mcp(\mcc)$ $$\mathcal{P}_{\emptyset}^{\kappa\text{-}\mathrm{fil},\Delta^\opp}(\mathcal{C})\subset\mathcal{P}_{\Sigma_\kappa}(\mathcal{C}).$$
	On the other hand, the converse inclusion follows from \cref{propa5}.
\end{proof}
\begin{de}
	Let $\mathcal{C}$ be a \infcat which admits geometric realizations of simplicial objects. We will say that an object $P \in \mathcal{C}$ is \textbf{projective} if the functor $\mathcal{C} \rightarrow \mathcal{S}$ corepresented by $P$ commutes with geometric realizations.
	
	We say an object $X\in\mcc$ is $\kappa$-projective if it is both $\kappa$-compact and projective.
\end{de}
\begin{prop}
	Let \mca be a small additive \infcat which admits $\kappa$-small coproducts. For any $a\in \mca$, the representable presheaf $h_a\in \mcp_{\Sigma_\kappa}(\mca)$ is $\kappa$-projective. 
	Furthermore, $\mcp_{\Sigma_\kappa}(\mca)^{\kappa\text{-}\op{proj}}$ can be identified with the idempotent completion of $\mca$.
\end{prop}
\begin{proof}
	The first statement follows from \cref{thma2}(6).
	
	For the last statement, let $P\in \mcp_{\Sigma_\kappa}(\mca)$ be a $\kappa$-projective object. By \cref{propa5},  there exists a simplicial object $U_{\bullet}: \mathrm{N}(\boldsymbol{\Delta})^{o p} \rightarrow \operatorname{Ind}_\kappa(\mathcal{A})$ whose colimit in $\mathcal{P}_{\Sigma_\kappa}(\mathcal{A})$ is $P$, therefore  $\id: P\to P$ factors through some  $U_k$. Since $U_k\simeq \colim_\alpha h_{a_\alpha}$ is a $\kappa$-filtered colimit of representable objects. The map $P\to U_k$ factors through some $h_{a_\alpha}$, hence $P$ is a retract of representable object, as desired.
\end{proof}
\begin{prop}\label{coprodproj}
	If \mcc is an additive \infcat which admits geometric realizations, then projective objects in \mcc are closed under any coproduct which exists in \mcc.
\end{prop}
\begin{proof}
	This is because  infinite products commute with geometric realizations in $\spgeq$ by \cref{prodcommutegeom}.
\end{proof}
\begin{prop}\label{kcprojgen}
	Let $\mathcal{A}$ be a small additive \infcat which admits $\kappa$-small coproducts, $\mathcal{D}$ a \infcat which admits $\kappa$-filtered colimits and geometric realizations, and $F: \mathcal{P}_{\Sigma_\kappa}(\mathcal{A}) \rightarrow \mathcal{D}$ a functor obtained  from  the left Kan extension of $f=F \circ j: \mathcal{A} \rightarrow \mathcal{D}$, where $j: \mathcal{A} \rightarrow \mathcal{P}_{\Sigma_\kappa}(\mathcal{A})$ denotes the Yoneda embedding. Consider the following conditions:
	\enu{
		\item  The functor $f$ is fully faithful.
		\item The essential image of $f$ consists of  $\kappa$-projective objects of $\mc{D}$.
		\item The \infcat $\mc{D}$ is generated by the essential image of $f$ under $\kappa$-filtered colimits and geometric realizations.
	}
	If (1) and (2) are satisfied, then $F$ is fully faithful. Moreover, $F$ is an equivalence if and only if (1), (2), and (3) are satisfied.
	
\end{prop}
\begin{proof}
	It basically follows from the same argument as \cite[Proposition 5.3.5.11]{htt}.
\end{proof}

\begin{prop}
	Let $\mathcal{A}$ be a small additive \infcat which admits $\kappa$-small coproducts. Then  $\mathcal{P}_{\Sigma_\kappa}(\mathcal{A})$ is a complete presentable prestable \infcat. We also have $\mathcal{P}_{\kappa\text{-}\sqcup}^{\op{small}\text{-}\sqcup,\op{idem}}(\mca)\simeq \mathcal{P}_{\Sigma_\kappa}(\mca)^{\op{proj}}$.
\end{prop}

\begin{proof}
	Since the $\funct^{\kappa\text{-}\times}(\mca^\opp,\mcs)\simeq\funct^{\kappa\text{-}\times}(\mca^\opp,\spgeq)\hookrightarrow\funct^{\kappa\text{-}\times}(\mca^\opp,\opsp)$ is closed under finite colimits and extensions, we see that $\mathcal{P}_{\Sigma_\kappa}(\mca)$ is prestable. By \cref{coprodproj}, we obtain the inclusion $\mathcal{P}_{\kappa\text{-}\sqcup}^{\op{small}\text{-}\sqcup,\op{idem}}(\mca)\subset \mathcal{P}_{\Sigma_\kappa}(\mca)^{\op{proj}}$. For the converse inclusion, given a projective object $X\in\mathcal{P}_{\Sigma_\kappa}(\mca)^{\op{proj}}$. Then there exists a $\pi_0$-epimorphism $\bigsqcup_\alpha P_\alpha\to X$ from a small coproduct of $\kappa$-projective objects, which makes $X$ be a retract of $\bigsqcup_\alpha P_\alpha$.
\end{proof}

\begin{rem}
	Note that $\mathcal{P}_{\Sigma_\kappa}(\mca)$ is not necessarily Grothendieck and hence not dualizable additive.
\end{rem}
\begin{prop}\label{kproj}
	Let $\mcc$ be a presentable additive \infcat. If $\mcc$ is generated by projectives under small colimits, then $\mcc\simeq \mathcal{P}_{\Sigma_\kappa}(\mcc^{\kappa\text{-}\op{proj}})$ for some regular cardinal $\kappa$.
\end{prop}
\begin{proof}
	Let $\mcp$ denote the class of all projective objects in $\mcc$. By assumption, $\mcc$ is the smallest full subcategory of itself containing $\mcp$ and closed under all small colimits. We denote the colimit closure of any class $\mathcal{A} \subset \mcc$ as $\closure{\mathcal{A}}$. Thus, our assumption is $\closure{\mcp} = \mcc$.
	
	Since $\mcc$ is a \textit{presentable} \infcat, there exists a small set of objects $G = \{X_i\}_{i \in I}$ (where $I$ is a small indexing set) such that $G$ generates $\mcc$ under small colimits, i.e., $\closure{G} = \mcc$.

	Define a subcategory $\mathcal{U}$ as the union of colimit closures of all small subsets of $\mcp$:
	\[ \mathcal{U} = \bigcup_{S \subset \mcp, \, S \text{ is small}} \closure{S} \]
	We claim that $\mathcal{U}$ is closed under small colimits. Let $D: J \to \mathcal{U}$ be a small diagram. Since $J$ is essentially small, for each vertex $j \in J$, the object $D(j)$ belongs to $\closure{S_j}$ for some small subset $S_j \subset \mcp$. Define $S^* = \bigcup_{j \in J} S_j$. Since $J$ is a small set and each $S_j$ is a small set, $S^*$ is also a small subset of $\mcp$. 
	Clearly, $D(j) \in \closure{S_j} \subset \closure{S^*}$ for all $j$. Since $\closure{S^*}$ is closed under small colimits, $\colim D \in \closure{S^*} \subset \mathcal{U}$. 
	Thus, $\mathcal{U}$ is closed under small colimits and contains $\mcp$. By the minimality of the closure, we have $\mathcal{U} = \closure{\mcp} = \mcc$.
	
	Since $G \subset \mcc = \mathcal{U}$, for each generator $X_i \in G$, there exists a small subset $S_i \subset \mcp$ such that $X_i \in \closure{S_i}$. Let:
	\[ \Sigma = \bigcup_{i \in I} S_i \]
	Since $I$ is a small set and each $S_i$ is small, $\Sigma$ is a \textit{small set} of projective objects. Furthermore, $G \subset \closure{\Sigma}$. Since $\closure{G} = \mcc$, it follows that $\closure{\Sigma} = \mcc$. Thus, $\mcc$ is generated by a small set of projectives $\Sigma$.
	
	In a presentable \infcat, every object is $\kappa$-compact for some regular cardinal $\kappa$. For each $P \in \Sigma$, let $\kappa_P$ be a regular cardinal such that $P$ is $\kappa_P$-compact. Since $\Sigma$ is a small set, we can define a successor cardinal:
	\[ \kappa = \left( \sup_{P \in \Sigma} \kappa_P \right)^+ \]
	Then $\kappa$ is a regular cardinal, and every $P \in \Sigma$ is $\kappa$-compact. Therefore, $\Sigma$ is a small set of $\kappa$-compact projectives that generates $\mcc$ under small colimits. By \cref{kcprojgen}, we get $\mcc\simeq \mathcal{P}_{\Sigma_\kappa}(\mcc^{\kappa\text{-}\op{proj}})$.
\end{proof}
\subsection{Large weight structures}
We briefly study presentable weight structures on large, non-compactly generated stable \infcats. We prove that a presentable $t$-category generated by projectives naturally forms a hypercomplete presentable weight structure.
\begin{de}
	Let $\mcc\in\prlst$. We say that an accessible $t$-structure $(\mcc,\mcc_{\geq0})$ is a \textbf{presentable weight structure} if it forms  the connective part of a weight structure\footnote{Note that once such a weight structure exists, then it is unique.} (see \cref{wstructure}).
\end{de}

\begin{lem}\label{wheartdiscribe}
	Let $(\mcc,\mcc_{\geq0})$ be a stable \infcat with a $t$-structure such that $\mcc_{\geq0}$ admits geometric realizations. If $(\mcc,\mcc_{\geq0})$ forms  a weight structure, then $\mcc^{w\text{-}\heartsuit}$ can be  identified with $\mcc_{\geq0}^{\op{proj}}$, the full subcategory of projective objects in $\mcc_{\geq0}$.
\end{lem}
\begin{proof}
	It follows from the description of projective objects in terms of $t$-structure; see \cite[Proposition 3.7]{hattt}.
\end{proof}
\begin{prop}\label{hycompleteimpliespsigamak}
	Let $(\mcc,\mcc_{\geq0})$ be a presentable stable \infcat equipped with a presentable  weight structure. Then:
	\enu{
		\item The inclusion $\mcc_{\geq0}\hookrightarrow\mcc$ is closed under small products, i.e. the $t$-structure satisfies $\mathrm{AB4}^*$.
		\item 
		If $\mcc_{\geq0}$ is hypercomplete, then	$\mcc_{\geq0}$ is generated by projectives under small colimits, and hence $\mcc_{\geq0}\simeq \mcp_{\Sigma_\kappa}(\mcc_{\geq0}^{\kappa\text{-}\op{proj}})$ for some $\kappa$ by \cref{kproj}. 
	} 
\end{prop}
\begin{proof}
	(1) It follows from $\mathrm{AB4}^*$ of \spgeq.\\
	(2) Since by \cref{wheartdiscribe} the w-heart of \mcc can be identified with $\mcc_{\geq0}^{\op{proj}}$, the result follows from 	\cite[Proposition 1.17]{levy2025c}.
\end{proof}
\begin{prop}\label{tcatiswcat}
	Let $(\mcc,\mcc_{\geq0})$ be a presentable stable \infcat with an accessible $t$-structure. If projectives in $\mcc_{\geq0}$ generate \mcc as a localizing subcategory, then $(\mcc,\mcc_{\geq0})$ forms a hypercomplete  presentable weight structure.
\end{prop}
\begin{proof}
	By the same argument as \cref{kproj}, there exists a $\kappa$ such that $\mcc_{\geq0}^{\kappa\text{-}\op{proj}}$ generates \mcc as a localizing subcategory. 
	By \cite[Theorem 6.3]{nikolaus2026unbounded}, there exists a weight structure on $\mcc$ such that $$\mathcal{C}_{w \geq 0}=\{y \in \mathcal{C}\mid \unmap(x, y) \text{ is connective for all } x \in \mcc_{\geq0}^{\kappa\text{-}\op{proj}}\}.$$
	By the generation assumption, we conclude that $\mcc_{w\geq0}=\mcc_{\geq0}$.
	
	We now claim it is hypercomplete. Let $x\in\bigcap_{n}\mcc_{\geq n}$. Since each projective object lies in the w-heart, there is no non-zero map from shifts of ($\kappa$-)projective objects to $x$, therefore $x=0$ by the generation assumption.
\end{proof}
\begin{rem}
	Note that \cref{tcatiswcat} is not an ``if and only if'' statement in general. 
	
	Indeed, let $\mca$ be a small additive \infcat. Let $\mcb=\opsp(\mcp_{\Sigma}(\mca))$ equipped with a presentable weight structure whose connective part is $\mathcal{B}_{\geq0}=\mcp_{\Sigma}(\mca)$, and let $\mathcal{E}$ be a nonzero presentable stable \infcat. Set
	$
	\mcc:=\mathcal{B}\times \mathcal{E}.
	$
	Equip $\mcc$ with the accessible $t$-structure
	\[
	\mcc_{\geq0}:=\mathcal{B}_{\geq0}\times 0,
	\qquad
	\mcc_{\leq0}:=\mathcal{B}_{\leq0}\times \mathcal{E}.
	\]
	Then this $t$-structure is again the connective part of a presentable weight structure, namely
	\[
	\mcc_{w\geq0}:=\mathcal{B}_{\geq0}\times 0,
	\qquad
	\mcc_{w\leq0}:=\mathcal{B}_{w\leq0}\times \mathcal{E}.
	\]
	Since the $t$-structure on $\mathcal{B}$ is hypercomplete, the $t$-structure on $\mcc$ is also hypercomplete, since
	\[
	\bigcap_{n\in \mathbb{Z}}\mcc_{\geq n}
	=
	\left(\bigcap_{n\in \mathbb{Z}}\mathcal{B}_{\geq n}\right)\times 0
	=
	0.
	\]
	Moreover,
	$
	\mcc_{\geq0}^{\op{proj}}
	\simeq
	\mathcal{B}_{\geq0}^{\op{proj}}\times 0.
	$
	Thus, 
	$
	\op{Loc}_{\mcc}\bigl<\mcc_{\geq0}^{\op{proj}}\bigr>
	=
	\mathcal{B}\times 0
	\subsetneq
	\mcc,
	$
	because $\mathcal{E}\neq 0$. Hence the projective objects of $\mcc_{\geq0}$ need not generate $\mcc$ as a localizing subcategory, even though $(\mcc,\mcc_{\geq0})$ is a hypercomplete presentable weight structure.
\end{rem}
\begin{quest}
	Let $(\mcc,\mcc_{\geq0})$ be a presentable stable \infcat equipped with an accessible $t$-structure. Under what precise categorical condition does it form a (hypercomplete) weight structure?
\end{quest}
\section{Complete proof of Efimov's formula}\label{appen3}
The goal of this section is to provide a complete proof of the following formula.

\begin{prop}\label{appenc1}
	Let $R$ be an adic \einfring, and let $\mcd=\op{Stab}(\op{Proj}_R^{\op{cpl},\omega_1})$. Let $\nucr\xhookrightarrow{i}\ind(\mcd)$ be the canonical inclusion. Let $P=(\bigoplus_{\mathbb{N}} R)^{\wedge}\in \op{Proj}_R^{\op{cpl},\omega_1}$. Then there exists an equivalence from the mapping spectrum $$\unmap_{\ind(\mcd)}(P,ii^R(P))\simeq \big((\prod_{\mathbb{N}} R)\otimes_R (\bigoplus_{\mathbb{N}} R)\big)^\wedge.$$
\end{prop}

\begin{lem}[{\cite[Proposition A.1]{efimov2025localizinginvariantsinverselimits}}]\label{efiA1}
 Let $p: \mathcal{I} \rightarrow \mathbb{N}^{o p}$ be a cocartesian fibration such that the fibers $\mathcal{I}_n$ are directed posets, $n \in \mathbb{N}$. Denote by $f_{m, n}: \mathcal{I}_m \rightarrow \mathcal{I}_n$ the transition maps, $m \geq n$. Suppose
that the maps $f_{n+1, n}$ are cofinal for all $n \geq 0$. We write the objects of $\mathcal{I}$ as pairs $(n, i_n)$, where $n \in \mathbb{N}, \, i_n \in \mathcal{I}_n$.

Let $\mathcal{C}$ be a presentable \infcat which satisfies  $\mathrm{AB5}$ (i.e. filtered colimits commute with finite limits) and $\mathrm{AB6}$ for countable products (for example, $\mathcal{C}$ can be any compactly assembled presentable \infcat). Let $G: \mathcal{I} \rightarrow \mathcal{C}$ be a functor. Then we have the following isomorphism:
$$
 \underset{\varphi: \mathbb{N} \rightarrow \mathcal{I}^{\vee}}{\colim }\, \underset{n \leq m}{\lim} G\big(n, f_{m, n}(\varphi(m))\big)\xrightarrow{\sim}\varprojlim_n \underset{i_n \in \mathcal{I}_n}{\colim}\, G(n, i_n)  .
$$
Here $\varphi$ runs through the directed poset of sections of the cartesian fibration $p^{\vee}: \mathcal{I}^{\vee} \rightarrow \mathbb{N}$.
\end{lem}
\begin{thm}\label{formulamain}
	Let $\mcc$ be a presentable stable\footnote{Note that any presentable stable \infcat trivially satisfies AB5.} 
	\infcat satisfying $\mathrm{AB6}$ (i.e. a dualizable stable \infcat). Let 
	$M\simeq \lim_n M_n$ be the inverse limit of some tower with $M_0=0$. Then we have a natural 
	identification $$\lim_{n}(\bigoplus_{\mathbb{N}} M_n)\simeq \colim_{f\in F} \prod_{i\in \mathbb{N}} I_{f(i)}, $$ 
	where $I_k=\fib(M\to M_k)$ and $$F = \{f \in \mathrm{Fun}(\mathbb{N}, \mathbb{N})^{op} \mid f(i) \to \infty\}$$ 
	denotes the filtered poset consisting of those increasing functions which converge to infinity, under the decreasing pointwise order.
\end{thm}
\begin{proof}
 We proceed by constructing a suitable cocartesian fibration $p: \mci \to \mathbb{N}^{op}$ and a functor $G: \mci \to \mathcal{C}$.
	For each $n \in \mathbb{N}$, define the poset $\mci_n$ as the set of increasing sequences:
	$$\mci_n = \{c: \mathbb{N} \to \{0, 1, \dots, n\} \mid c(i) = n \text{ for } i \gg 0\}.$$
	We equip $\mci_n$ with the reverse pointwise partial order: $c \leq d$ if and only if $c(i) \geq d(i)$ for all $i \in \mathbb{N}$. This condition ensures $\mci_n$ is a directed poset.
	For $m \geq n$, define the transition maps $f_{m,n}: \mci_m \to \mci_n$ by truncation:
	$$f_{m,n}(c)(i) = \min(n, c(i)).$$
	These transition maps are clearly surjective, which guarantees that the maps $f_{n+1,n}$ are cofinal, satisfying the assumption of \cref{efiA1}.
	
	Next, we define the functor $G: \mci \to \mathcal{C}$ on objects $(n, c)$ by:
	$$G(n, c) = \prod_{i \in \mathbb{N}} \mathrm{fib}(M_n \to M_{c(i)}).$$
	Observe that since $c \in \mci_n$, we have $c(i) = n$ for $i \gg 0$. Consequently, the fiber $\mathrm{fib}(M_n \to M_{c(i)}) \simeq 0$ for sufficiently large $i$. Note that a countable product with only finitely many non-zero terms is equivalent to a direct sum. Thus, we have the equivalence:
	$$G(n, c) \simeq \bigoplus_{i \in \mathbb{N}} \mathrm{fib}(M_n \to M_{c(i)}).$$
	
	Now we evaluate the right-hand side  of \cref{efiA1}:
	\[
	 \lim_n \mathrm{colim}_{c \in \mci_n} G(n, c).
	\]
	Taking the filtered colimit over \(\mci_n\) corresponds to \(c(i)\) eventually stabilizing to \(0\), giving:
	\[
	\mathrm{colim}_{c \in \mci_n} G(n, c)
	\simeq
	\bigoplus_{i \in \mathbb{N}}
	\mathrm{colim}_{c \in \mci_n}
	\mathrm{fib}(M_n \to M_{c(i)})
	\simeq
	\bigoplus_{i \in \mathbb{N}}
	\mathrm{fib}(M_n \to 0)
	\simeq
	\bigoplus_{i \in \mathbb{N}} M_n.
	\]
	Taking the inverse limit over \(n\) yields the desired right-hand side of our main theorem:
	\[
\lim_n \mathrm{colim}_{c \in \mci_n} G(n, c) \simeq \lim_n \left( \bigoplus_{i \in \mathbb{N}} M_n \right).
	\]
	
	To evaluate the left-hand side, we analyze the directed poset of sections
	\(\phi: \mathbb{N} \to \mci^\vee\). We first note that the full subcategory of Cartesian sections is cofinal to all sections. However, $$\funct^{\op{Car}}_{/\mathbb{N}}(\mathbb{N},\mci^\vee)\simeq\lim_n \mci_n\simeq 	F =
	\{f \in \mathrm{Fun}(\mathbb{N}, \mathbb{N})^{op} \mid f(i) \to \infty\}.
	$$
	Substituting this into the left-hand side of \cref{efiA1}, we obtain:
	\[
	\mathrm{colim}_{f \in F}
	\lim_n G(n, \min(n, f)).
	\]
	Expanding the definition of the functor \(G\):
	\[
	\mathrm{colim}_{f \in F}
	\lim_n
	\prod_{i \in \mathbb{N}}
	\mathrm{fib}(M_n \to M_{\min(n, f(i))}).
	\]
 For a fixed \(i\) and sufficiently large \(n \geq f(i)\), the term
	\(\min(n, f(i))\) stabilizes exactly to \(f(i)\). Therefore, the inverse limit of the fibers computes as:
	\[
	\lim_n
	\mathrm{fib}(M_n \to M_{\min(n, f(i))})
	\simeq
	\mathrm{fib}\left(\lim_n M_n \to \lim_n M_{f(i)}\right)
	\simeq
	\mathrm{fib}(M \to M_{f(i)}).
	\]
	By definition, this fiber is exactly \(I_{f(i)}\). Thus, the left-hand side evaluates to:
	\[
		\mathrm{colim}_{f \in F}
	\lim_n G(n, \min(n, f))
	\simeq
	\mathrm{colim}_{f \in F}
	\prod_{i \in \mathbb{N}} I_{f(i)},
	\]
	as desired.
\end{proof}
We thank Yifan Jin for help with the following lemma.
\begin{lem}\label{formulapopvee}
	Let $R$ be an adic \einfring. Let  $\mcd=\op{Stab}(\op{Proj}_R^{\op{cpl},\omega_1})$. Let $P=(\bigoplus_{\mathbb{N}} R)^{\wedge}\in \op{Proj}_R^{\op{cpl},\omega_1}$. Then $P^{\opp,\vee}\in \ind(\mcd^\opp)$ can be identified with\footnote{Note that this  coproduct is taken in $\ind(\mcd^\opp)$ instead of internally in $\mcd^\opp$.} $$P^{\opp,\vee}\simeq \lim_n (\bigoplus_{\mathbb{N}} R^\opp/I^n),$$
	where $R^\opp/I^n=\big(\bigotimes_{1\leq i\leq k}\fib(R\xrightarrow{x_i^n}R)\big)^\opp.$
\end{lem}
\begin{proof}
Let $I=(x_1,\dots,x_k)\subset \pi_0R$ be a finitely generated ideal of definition. Let us denote 
\[
R/I^n := \operatorname{Kos}(R;x_1^n,\ldots,x_k^n)=\bigotimes_{1\leq i\leq k} R/x_i^n
\]
and \[
(R/I^n)^\vee := \underline{\operatorname{Hom}}_R(R/I^n,R).
\]
Note that we have $R^\opp/I^n=(R/I^n)^{\vee,\opp}$.

Now given $d \in \mcd$. Then the result follows from the following computation of mapping spectra.
\[
\begin{aligned}
	\underline{\operatorname{Map}}_{\operatorname{Ind}(\mathcal{D}^{\mathrm{op}})}
	\left(d^{\mathrm{op}},P^{\mathrm{op},\vee}\right)
	&=
	\underline{\operatorname{Map}}_{\mathcal{D}^{\mathrm{op}}}
	\left(d^{\mathrm{op}}\otimes P^{\mathrm{op}},R_I^{\wedge,\mathrm{op}}\right) \\[4pt]
	&\simeq
	\underline{\operatorname{Map}}_{\mathcal{D}^{\mathrm{op}}}
	\left(
	d^{\mathrm{op}}\otimes P^{\mathrm{op}},
	\lim_n R^{\mathrm{op}}/I^n
	\right) \\[4pt]
	&\simeq
	\lim_n
	\underline{\operatorname{Map}}_{\mathcal{D}^{\mathrm{op}}}
	\left(
	d^{\mathrm{op}}\otimes P^{\mathrm{op}},
	R^{\mathrm{op}}/I^n
	\right) \\[4pt]
	&\simeq
	\lim_n
	\underline{\operatorname{Map}}_{\mcd}
	\left(
	\left(R/I^n\right)^\vee,
	d\otimes P
	\right) \\[4pt]
	&\simeq
	\lim_n
	\underline{\operatorname{Map}}_{\modrcpl}
	\left(
	\left(R/I^n\right)^\vee,
	\big(\bigoplus_{\mathbb N} d\big)_{I}^{\wedge}
	\right) \\[4pt]
	&\simeq
	\lim_n
	\bigoplus_{\mathbb N}
	\underline{\operatorname{Map}}_{\modrcpl}
	\left(
	\left(R/I^n\right)^\vee,
	d
	\right) \\[4pt]
	&\simeq
	\lim_n
	\bigoplus_{\mathbb N}
	\underline{\operatorname{Map}}_{\mcd^{\mathrm{op}}}
	\left(
	d^{\mathrm{op}},
	R^{\mathrm{op}}/I^n
	\right) \\[4pt]
	&\simeq
	\lim_n
	\underline{\operatorname{Map}}_{\operatorname{Ind}(\mathcal{D}^{\mathrm{op}})}
	\left(
	d^{\mathrm{op}},
	\bigoplus_{\mathbb N} R^{\mathrm{op}}/I^n
	\right) \\[4pt]
	&\simeq
	\underline{\operatorname{Map}}_{\operatorname{Ind}(\mathcal{D}^{\mathrm{op}})}
	\left(
	d^{\mathrm{op}},
	\big(\bigoplus_{\mathbb N} R^{\mathrm{op}}\big)_I^{\wedge}
	\right).
\end{aligned}
\]
Note that the first equivalence follows from the formula $$R^\wedge_I\simeq \big(\colim_n (R/I^n)^\vee\big)^\wedge_I$$ in \modrcpl.
\end{proof}
\begin{rem}
	Note that 
	 $\stab(\mca^\opp)\simeq\stab(\mca)^\opp$, which follows from the fact that $\opsp\otimes-:\prlad\to \prlst$ preserves dual data.
\end{rem}

\begin{prop}\label{popveenuclear}
	Let $R$ be an adic \einfring. Let  $\mcd=\op{Stab}(\op{Proj}_R^{\op{cpl},\omega_1})$. Then for any $d\in \mcd$, we have $d^{\opp,\vee}$ is a nuclear object in $\op{Ind}(\mcd^\opp)$.
\end{prop}
\begin{proof}
By \cite[Proposition 1.33]{efimov2025localizinginvariantsinverselimits}, it suffices to show that for any $c\in \mcd$, the following comparison is an equivalence in $\ind(\mcd^\opp)$
$$
c^{\opp,\vee}\otimes d^{\opp,\vee}\to (c\otimes d)^{\opp,\vee} .
$$
Since $P=(\bigoplus_{\mathbb{N}} R)^{\wedge}\in \op{Proj}_R^{\op{cpl},\omega_1}$ generates \mcd as a stable subcategory, it suffices to show that  $$
P^{\opp,\vee}\otimes P^{\opp,\vee}\to (P\otimes P)^{\opp,\vee} 
$$ is an equivalence. 

By \cref{formulapopvee}, we can identify $$P^{\opp,\vee}\simeq \lim_n (\bigoplus_{\mathbb{N}} R^\opp/I^n).$$ By virtue of \cref{withoutlosscomplete}, we may assume that $R$ is complete without loss of generality. Let $I^nR^\opp$ denote the fiber of the projection $R^\opp=R^{\wedge,\opp}_I\to R^\opp/I^n$. By \cref{formulamain}, we have
\[
\begin{aligned}
	P^{\opp,\vee}\otimes P^{\opp,\vee}
	&\simeq 
	\left(\lim_n \bigoplus_{\mathbb{N}} R^\opp/I^n\right)
	\otimes
	\left(\lim_n \bigoplus_{\mathbb{N}} R^\opp/I^n\right) \\[4pt]
	&\simeq
	\left(
	\operatorname*{colim}_{f\in F}
	\prod_{i\in \mathbb{N}} I^{f(i)}R^\opp
	\right)
	\otimes
	\left(
	\operatorname*{colim}_{f\in F}
	\prod_{i\in \mathbb{N}} I^{f(i)}R^\opp
	\right) \\[4pt]
	&\simeq
	\operatorname*{colim}_{(f,g)\in F\times F}
	\left(
	\prod_{i\in \mathbb{N}} I^{f(i)}R^\opp
	\otimes
	\prod_{j\in \mathbb{N}} I^{g(j)}R^\opp
	\right) \\[4pt]
	&\simeq
	\operatorname*{colim}_{(f,g)\in F\times F}
	\prod_{i\in \mathbb{N}}
	\prod_{j\in \mathbb{N}}
	\left(
	I^{f(i)}R^\opp\otimes I^{g(j)}R^\opp
	\right) \\[4pt]
	&\simeq
	\lim_{m,n}
	\left(
	\bigoplus_{\mathbb{N}} R^\opp/I^m
	\otimes
	\bigoplus_{\mathbb{N}} R^\opp/I^n
	\right) \\[4pt]
	&\simeq
	\left(
	\bigoplus_{\mathbb{N}} R^\opp
	\otimes
	\bigoplus_{\mathbb{N}} R^\opp
	\right)^\wedge_I
	\\[4pt]
	&=(P\otimes P)^{\opp,\vee} .
\end{aligned}
\]
as desired.
\end{proof}

\begin{proof}[Proof of \cref{appenc1}]
	By \cref{popveenuclear}, we see that $\mcd$ satisfies the  condition in \cite[Proposition 1.33]{efimov2025localizinginvariantsinverselimits}, and hence it satisfies the  condition in \cite[Proposition 1.30]{efimov2025localizinginvariantsinverselimits}. Therefore, we have $$\unmap_{\ind(\mcd)}(P,ii^R(P))\simeq \unmap_{\ind(\mcd)}(\mb{1},P^\vee\otimes P).$$
	By \cite[Proposition 1.31]{efimov2025localizinginvariantsinverselimits}, there is a natural identification
	$$\unmap_{\ind(\mcd)}(\mb{1},P^\vee\otimes P)\simeq\unmap_{\ind(\mcd^\opp)}(\mb{1},P^\opp\otimes P^{\opp,\vee}).$$
	
By virtue of \cref{withoutlosscomplete}, we may assume that $R$ is complete without loss of generality. 
By \cref{formulamain} and \cref{formulapopvee}, we obtain
\[
\begin{aligned}
	P^\opp\otimes P^{\opp,\vee}
	&\simeq
	\left(\prod_{\mathbb N} R^\opp\right)
	\otimes
	\operatorname*{colim}_{f\in F}
	\prod_{i\in \mathbb{N}} I^{f(i)}R^\opp \\[4pt]
	&\simeq
	\operatorname*{colim}_{f\in F}
	\left(
	\prod_{\mathbb N}
	\prod_{i\in \mathbb{N}} I^{f(i)}R^\opp
	\right) \\[4pt]
	&\simeq
	\operatorname*{colim}_{f\in F}
	\left(
	\prod_{i\in \mathbb{N}}
	\left(
	I^{f(i)}\prod_{\mathbb N} R^\opp
	\right)
	\right) \\[4pt]
	&\simeq
	\lim_n
	\left(
	\bigoplus_{\mathbb{N}}
	\prod_{\mathbb{N}} R^\opp/I^n
	\right).
\end{aligned}
\]
Because $$\unmap_{\ind(\mcd^\opp)}\left(\mb{1},R^\opp/I^n\right)=\unmap_{\mcd^\opp}\left(R^\opp,(R/I^n)^{\vee,\opp}\right)=\unmap_{\modrcpl}\left((R/I^n)^{\vee},R\right)\simeq \unmap_{\modrcpl}\left(R,R/I^n\right),$$
we get the desired equivalence of spectra
\[
\unmap_{\ind(\mcd^\opp)}
\left(\mb{1},P^\opp\otimes P^{\opp,\vee}\right)
\simeq
\big((\prod_{\mathbb{N}} R)\otimes_R (\bigoplus_{\mathbb{N}} R)\big)^\wedge.
\]

\end{proof}

\printbibliography
\bigskip

\textsc{Ishan Levy, Department of Mathematics, Institute for Advanced Study, USA}

\emph{Email address:} \href{mailto:ishanl@ias.edu}{\texttt{ishanl@ias.edu}}
\bigskip

\textsc{Jiacheng Liang, Department of Mathematics, Johns Hopkins University, USA}

\emph{Email address:} \href{mailto:jliang66@jhu.edu}{\texttt{jliang66@jhu.edu}}

\bigskip
\textsc{Vladimir Sosnilo, RIKEN iTHEMS, Wako, Saitama, 351-0198, Japan}

\emph{Email address:} \href{mailto:vsosnilo@gmail.com}{\texttt{vsosnilo@gmail.com}}\;\;\;\href{mailto:vladimir.sosnilo@riken.com}{\texttt{vladimir.sosnilo@riken.com}}
\end{document}